\documentclass[onefignum,onetabnum]{siamonline190516}

\usepackage{lipsum}
\usepackage{amsfonts}
\usepackage{graphicx}
\usepackage{epstopdf}
\usepackage{algorithmic}
\usepackage{tikz}

\usepackage{preamble-SIADS}

\ifpdf
  \DeclareGraphicsExtensions{.eps,.pdf,.png,.jpg}
\else
  \DeclareGraphicsExtensions{.eps}
\fi

\usepackage{enumitem}
\usepackage{datetime}
\setlist[enumerate]{leftmargin=.5in}
\setlist[itemize]{leftmargin=.5in}

\newsiamremark{remark}{Remark}
\newsiamremark{hypothesis}{Hypothesis}
\crefname{hypothesis}{Hypothesis}{Hypotheses}
\newsiamthm{claim}{Claim}

\headers{Dynamics of a state-dependent delay-differential
equation}{Gedeon, Humphries, Mackey, Walther and Wang}

\title{Dynamics of a state-dependent delay-differential
equation\thanks{Submitted \today\funding{This
        work was supported by the Natural Sciences and Engineering Research Council (NSERC) of Canada (ARH and MCM) and the Alexander von Humboldt Stiftung of Germany (H-OW).
        TG was partially supported by NSF grant DMS-1951510.}}}
        
\author{Tom\'{a}\v{s} Gedeon\thanks{Department of Mathematics, Montana State University, Bozeman, MT 59717 (\email{gedeon@math.montana.edu})}
\and
Antony R. Humphries\thanks{Department of Mathematics and Statistics and Department of Physiology, McGill University, Montreal, QC, Canada H3A 0B9 (\email{tony.humphries@mcgill.ca})}
\and
Michael C. Mackey\thanks{Departments of Physiology, Physics and Mathematics \& Statistics, McGill University, 3655 Promenade Sir William Osler, Montreal, Quebec H3G 1Y6, Canada (\email{michael.mackey@mcgill.ca})}
        \and
Hans-Otto Walther\thanks{Mathematisches Institut, Universit{\"{a}}t Giessen, Arndtstrasse 2, 35392 Giessen, Germany
        (\email{Hans-Otto.Walther@math.uni-giessen.de}.)}
\and Zhao Wang\thanks{Department of Mathematics and Statistics, McGill University, Montreal, QC, Canada H3A 0B9 (\email{ zhao.wang3@mail.mcgill.ca}.)}
}

\usepackage{amsopn}

\begin{document}

\maketitle

\begin{abstract}
We present a detailed study of  a scalar differential equation with threshold state-dependent delayed feedback. This equation arises as a simplification of a gene regulatory model. There are two monotone nonlinearities in the model: one describes the dependence of delay on state, and the other is the feedback  nonlinearity. Both increasing and decreasing nonlinearities are considered.
Our analysis is exhaustive both analytically and numerically as we examine the bifurcations  of the system for  various combinations of increasing and decreasing nonlinearities. We identify rich bifurcation patterns including Bautin, Bogdanov-Takens, cusp, fold, homoclinic, and Hopf bifurcations whose existence  depend on the derivative signs of nonlinearities. Our analysis confirms many of these patterns in the limit where the nonlinearities are  switch-like and change their value abruptly at a threshold.
Perhaps one of the most surprising findings is the existence of a Hopf bifurcation to a periodic solution when the nonlinearity is monotone increasing and the time delay is a decreasing function of the state variable.
\end{abstract}



\section{Introduction} \label{sec:intro}

In considering mathematical models for the dynamics of biological feedback systems, the occurrence of delays is almost ubiquitous and this leads to very interesting biological, mathematical, and modeling problems.  These delays arise  because of the 
time required to propagate feedback signals in biological systems.  For physiological examples one need only think of the conduction time of action potentials in a neural feedback circuit \cite{an1981complex,mackey1984dynamics}, the time required to complete DNA synthesis, mitosis and cytokinesis in cell replication \cite{de2019dynamics}, the significant time required to produce mature cells in the hematopoietic system 
\cite{foley2009dynamic,Craig2016}, as well as a myriad of  others \cite{glass2020clocks}.  Examples in engineering and physical settings are rich and abundant ranging from control of  ships  \cite{minorsky1922directional}, vibration control \cite{moon2008chaotic}, and the delays due to the relativistic propagation of signals at the speed of light \cite{wheeler45,wheeler49,LHR2012}.

In this paper we study the dynamics of the positive solutions of the scalar state-dependent delay differential equation (DDE)
\be \label{eq:basic}
x'(t)  = \beta e^{-\mu\tau (t )}
 \frac{v(x(t )))}{v(x (t- \tau (t )))}
 g(x (t -\tau (t ))) - \gamma x (t)
\ee
where for $t\ge0$ the delay $\tau(t)$ is defined by the threshold condition
\be \label{eq:thres}
a =\int^0_{-\tau (t)}v(x(t+s))ds=\int^t_{t-\tau (t)}v(x(s))ds.
\ee
We assume that the constants $\beta, \mu, \gamma$ and $a$ are all positive.

We will consider the system \eqref{eq:basic},\eqref{eq:thres} with both smooth and piecewise
constant functions $v(x)$ and $g(x)$. In the smooth case, we take the functions to be
Hill functions
\be \label{eq:vghill}
g(x) =\dfrac{g^{-}\theta_g^n + g^{+} x^n}{\theta_g^n+x^n}, \qquad
v(x) =\dfrac{v^-\theta_v^m + v^{+} x^m}{\theta_v^m+x^m},
\ee
where the exponents $m$ and $n$ are strictly positive real numbers.
We also assume that $v(x)$ is strictly positive and bounded away from zero:
\be \label{eq:vbounds}
0<v_0\leq v(x)\leq v_U, \qquad\text{where}\quad v_0=\min\{v^-,v^+\}, \quad
v_U=\max\{v^-,v^+\}.
\ee
This ensures that the delay defined by \eqref{eq:thres} satisfies $\tau(t)\in[a/v_U,a/v_0]$, and is thus both bounded and bounded away from zero.
Moreover, applying Leibnitz's rule to \eqref{eq:thres} with $v(x)\geq v_0>0$ shows that
$\tfrac{\phantom{t}d}{dt}(t-\tau(t))\geq v_0/v_U >0$ so $t-\tau(t)$ is a strictly monotonically  increasing function of $t$. Consequently to pose \eqref{eq:basic}-\eqref{eq:thres} as an initial value problem (IVP) it is sufficient to provide an initial function $x(t)=\phi(t)$ for $t\in[-\tau(0),0]$ where $\tau(0)$ is defined by \eqref{eq:thres}.
We will take the function $g$ to be non-negative and bounded with
\be \label{eq:gbounds}
0 < g_0\leq g(x)\leq g_U, \qquad\text{where}\quad g_0=\min\{g^-,g^+\}, \quad
g_U=\max\{g^-,g^+\}.
\ee

We are interested in the dynamics when the Hill coefficients $m$, $n$ are small or large, and also consider the case of piecewise constant functions obtained by taking the limits of $v(x)$ and $g(x)$ as $m$ and $n\to\infty$. In that case we take the limiting functions to be set-valued at the threshold points with
\begin{gather}  \label{eq:gpwconst}
g(x) = \left\{  \begin{array}{cl}
g^-, & x < \theta_g \\
{[g_0,g_U]}, & x = \theta_g \\
g^+, & x > \theta_g
\end{array}\right. \\
\label{eq:vpwconst}
v(x) = \left\{  \begin{array}{cl}
v^-, & x < \theta_v \\
{[v_0,v_U]}, & x = \theta_v \\
v^+, & x > \theta_v
\end{array}\right.
\end{gather}
Equation \eqref{eq:gpwconst} is derived from \eqref{eq:vghill} by regarding the curve $\{x,g(x)\}$ as a subset of $\mathbb{R}^2$ and taking the limit as $n\to\infty$ using the Hausdorff metric, and \eqref{eq:vpwconst} is derived similarly. There is a long history of incorporating set-valued functions into differential equations, resulting in differential inclusions for ordinary differential equations \citep{A&C,A&F}, and in DDEs. Mallet-Paret and Nussbaum (\citep{JMPRNIII,JMPRN11}) consider set-valued limiting solutions in their study of slowly oscillating periodic solutions.

As shown in Appendix \ref{app-reduction}, the system \eqref{eq:basic}-\eqref{eq:thres} arises as a quasi-steady state reduction of the
state-dependent delayed extension of the Goodwin \citep{Goodwin1963} operon model studied in \cite{ghmww2020,wendy-msc,ifacs22}.
Thus the model is taken to describe the regulation of a bacterial operon in which the cells are growing exponentially at a rate $\mu$ and have finite transcriptional and translational velocities that are potentially dependent on the state of the system.
Then the function $g$ denotes the production flux of messenger RNA (mRNA) while  $v$ is the velocity of translation/transcription of the mRNA strand to produce the end product (protein).  $x$ represents the dimensionless effector concentration.

The function $g(x)$ in \eqref{eq:vghill} is monotone increasing when
$g^-<g^+$, which we refer to as a positive feedback case. In the operon context (\citep{Goodwin1963}) this corresponds to an {\it inducible} operon.  The classical example would be the {\it lac} operon regulating bacterial utilization of lactose as an energy source. Conversely, when $g^->g^+$ the function $g(x)$ is monotone decreasing, which we refer to as a negative feedback case, and in the operon setting it would correspond to a {\it repressible} operon.  The immediate example that comes to mind is the {\it tryp} operon regulating the production of the amino acid tryptophan.

In \cite{ghmww2020} it is argued that the transcription velocity for an operon should be an increasing function when $g$ is a decreasing function, and vice versa, while the translation
velocity should always be monotonically decreasing. In the current study, we do not make this assumption and instead consider the different possible  combinations of increasing/decreasing/constant $g$ and $v$ to explore potential dynamics. To avoid confusion we use $g\uparrow$ to indicate that the function $g$ is monotonically  increasing
or equivalently that $g^-<g^+$. Then $g\downarrow$
denotes that $g$ is monotonically decreasing or equivalently that $g^->g^+$, while we use $g\leftrightarrow$ to denote that $g^-=g^+$ and so $g$ is a constant function. With analogous definitions for the function $v$ we denote
different cases of these pairs of functions by $(g\uparrow,v\uparrow)$, $(g\uparrow,v\downarrow)$, $(g\uparrow,v\leftrightarrow)$, etc.

Smith \cite{Smith91,Smith93} showed that a DDE with a threshold delay can be converted through a time transformation to a distributed delay DDE with a constant delay, so the theory of 
those equations is applicable to this model.
Some authors have used this transformation to study threshold models as distributed constant delay DDEs (see for example \cite{KCP16,TIY18}), but other times the existence of the time transformation to constant delay DDE has been used as an excuse to just ignore the threshold delay and treat the delay as if constant.  We will demonstrate that the model \eqref{eq:basic} with the threshold delay \eqref{eq:thres} can display very different dynamics than the same model \eqref{eq:basic} with a constant delay. In this work we will tackle \eqref{eq:basic} with the threshold delay \eqref{eq:thres} directly without transforming to a constant delay DDE. We do this because we are also interested in problems with multiple delays for which the time transformation does not result in a constant delay DDE, and also because the time transformation is trajectory dependent, which creates difficulties in comparing solutions of the time transformed model with the original model. We also point out that when $\mu\ne0$, the term $e^{-\mu\tau}$ in \eqref{eq:basic} would cause the time transformed equation to be a distributed DDE with constant delay, which can create issues when linearizing (see \cite{wendy-phd}).
Finally, we believe it is interesting that the analysis and numerics of these problems can be tackled directly for the problem as formulated.

The outline of this paper is as follows. Section~\ref{sec:semiflow} discusses the semiflow generated by \eqref{eq:basic}-\eqref{eq:thres}. The following Section~\ref{sec:positivity} considers issues related to the positivity of solutions and the existence of a global attractor, and then Section~\ref{sec:linearization} examines the linearization of the system at one of the steady states and the stability of the steady state.

Sections~\ref{sec:onehill} and \ref{sec:twohill} contain a detailed investigation of the rich dynamics generated by the system \eqref{eq:basic}-\eqref{eq:thres}. We combine the verification of essential necessary conditions for local bifurcations from stationary points with numerical studies of one- and two-parameter bifurcations. 
Our numerical techniques are extensively documented in \cite{ghmww2020,wendy-msc,ifacs22}, and summarised in Appendix~\ref{app-numerics}. 
Concerning 
Theorems~\ref{thm:gdown}, \ref{thm:gup}, \ref{thm:vdown} and~\ref{thm:vup}
about fold- and Hopf-bifurcations
the reader should be aware that for most statements we limit ourselves to proving only those parts which make the result plausible and serve as a basis for numerical investigation.
In particular, we do not verify non-degeneracy conditions for the bifurcations. Theorems on Hopf bifurcation for differential equations with state-dependent delay can be found in \cite{Eichmann06,HuWu10,Sieber12}.

We begin in Section~\ref{sec:onehill} by examining the simpler situation in which there is only one nonlinearity, with one of $g$ or $v$ being constant. In Sections~\ref{sec:gdownvconst}
and~\ref{sec:gupvconst} we consider the constant delay cases $(g\downarrow,v\leftrightarrow)$ and $(g\uparrow,v\leftrightarrow)$ respectively. Then in Sections~\ref{sec:gconstvdown} and~\ref{sec:gconstvup}
we consider $(g\leftrightarrow,v\downarrow)$ and $(g\leftrightarrow,v\uparrow)$. In these two cases $v$ and thus $\tau$ are varying, resulting in state-dependent delays. We find that the dynamics are considerably richer and more surprising in the state-dependent delay cases, with qualitative differences in the dynamics depending on whether the growth rate, $\mu$, or the decay rate, $\gamma$, is larger.

We continue in this vein in Section~\ref{sec:twohill} but  considering both  nonlinearities, $g$ and $v$. In Section~\ref{sec:twothetas} we begin by considering the case where $\theta_g \neq \theta_v$.
In principle there should be four cases to consider: $(g\uparrow,v\uparrow)$, $(g\uparrow,v\downarrow)$,$(g\downarrow,v\uparrow)$ and $(g\downarrow,v\downarrow)$, but in practice we find that the dynamics is determined by the cases studied in Section~\ref{sec:onehill}. For example with
$(g\uparrow,v\uparrow,\theta_g\neq\theta_v)$ we find that the dynamics and bifurcations from the steady state $\xi$ are given by the previously studied cases $(g\uparrow,v\leftrightarrow)$ for $\xi\approx\theta_g$
and by $(g\leftrightarrow,v\uparrow)$ for $\xi\approx\theta_v$. Consequently, we illustrate just two of these cases, $(g\downarrow,v\uparrow,\theta_g \neq \theta_v)$ and
$(g\uparrow,v\uparrow,\theta_g \neq \theta_v)$ to show how the dynamics relates to the
previous examples of Section~\ref{sec:onehill}. The case $(g\uparrow,v\uparrow,\theta_g \neq \theta_v)$ is interesting as it can result in up to five co-existing steady states, three of which are stable.

If $|\theta_g-\theta_v|\ll1$ then both functions $g$ and $v$ influence the dynamics, so in
Section~\ref{sec:onetheta} we consider the dynamics when $\theta_g = \theta_v$.
In the cases $(g\uparrow,v\uparrow)$ and $(g\downarrow,v\downarrow)$ both functions are increasing or decreasing,
and no new dynamics arise, beyond what was already seen in Section~\ref{sec:onehill}. However, the cases
$(g\uparrow,v\downarrow,\theta_g = \theta_v)$ and $(g\downarrow,v\uparrow,\theta_g = \theta_v)$ reveal surprising
dynamics in limiting cases.

Section \ref{sec:summary} contains summary remarks as well as comments on possible extensions of this work.

Finally, as noted above the appendices contain the reduction of the model of \cite{ghmww2020} to the situation we consider here (Appendix~\ref{app-reduction}) as well as a brief elaboration in 
Appendix~\ref{app-other} of five previously published models that fall within the context of this paper,
along with a summary and references for our numerical techniques (Appendix~\ref{app-numerics}).

\section{The semiflow of differentiable solution operators generated by the system \eqref{eq:basic}-\eqref{eq:thres}}\label{sec:semiflow}

For delay differential equations a familiar state space is given by the space of continuous functions on a compact interval, see e.g. \cite{HVL,DvGVLW}.
In case of variable, state-dependent delays, however, there is a specific lack of smoothness which means that in general the initial value problem is not well-posed for only continuous initial data, not to speak of, say, smoothness of solutions with respect to initial data and linearization (\cite{W1,HKWW}).

Recall the definition of the segment $x_t$ of a map $x:I\to\mathbb{R}$
for  which the domain $I$ contains the interval $[t-r,t]$ with $t \in \mathbb{R}$ and $r>0$ : $x_t$ is the map $[-r,0]\to\mathbb{R}$ given by $x_t(s)=x(t+s)$ for $-r\le s\le0$. In other words, the restriction of $x$ to $[t-r,t]$ is shifted to the interval $[-r,0]$. In this section we reformulate the system \eqref{eq:basic},\eqref{eq:thres}
as a delay differential equation
\be
x'(t)=G(x_t)
\label{eq:form G}
\ee
with a functional $G:C^1\to\mathbb{R}$ on the Banach space $C^1=C^1([-r,0],\mathbb{R})$ of continuously differentiable maps $[-r,0]\to\mathbb{R}$, for some $r>0$ which is to be determined. The norm on $C^1$ is given by
$$
|\phi|_1=\max_{-r\le t\le 0}|\phi(t)|+\max_{-r\le t\le 0}|\phi'(t)|,
$$

We will also need the Banach space $C\supset C^1$ of continuous maps $[-r,0]\to\mathbb{R}$, with the norm given by
$$
|\phi|=\max_{-r\le t\le 0}|\phi(t)|.
$$
We shall verify the hypotheses from \cite{W1,W2,HKWW} which guarantee the existence, uniqueness, and differentiability, with respect to initial data, of solutions to an initial value problem which is associated with \eqref{eq:form G} in a submanifold of the space $C^1$.

We make the following assumptions.
The function $g:\mathbb{R}\to(0,\infty)$ is continuously differentiable with
$$
0<\inf\,g(\mathbb{R})\le\sup\,g(\mathbb{R})=g_U.
$$
The function $v:\mathbb{R}\to[v_0,\infty)$, with $v_0>0$, is continuously differentiable.
We fix a number $r>a/v_0$ and notice that $r$ is an a-priori bound for $\tau(t)$ in Eq.~\eqref{eq:thres}.

Next we rewrite the system Eq. \eqref{eq:basic},\eqref{eq:thres}
in a form which is more convenient for our purpose.
Using segment notation,  Eq.
\eqref{eq:thres} becomes
\be
a=\int^0_{-\tau(t)}v(x_t(s))ds\label{eq:thres-2}
\ee
with the segment $x_t\in C$. More generally, we consider the equation
\begin{equation}
a=\int^0_{-u}v(\phi(s))ds\label{eq:thres-3}
\end{equation}
for $u\in[0,r]$ and $\phi\in C$. Using positivity of the function $v$ and the Intermediate Value Theorem we infer  that for every $\phi\in C$ there is a uniquely determined solution $u=\delta(\phi)\in(0,r)$ of Eq.~\eqref{eq:thres-3}.
This yields a map $\delta:C\to(0,r)$.

\begin{prop} \label{prop:1}
The map $\delta:C\to(0,r)$ is continuously differentiable with
$$
D\delta(\phi)\chi=-\frac{\int^0_{-\delta(\phi)}v'(\phi(s))\chi(s)ds}{v(\phi(-\delta(\phi)))}.
$$
In case $\phi(s)=\xi$ for all $s\in[-r,0]$,
$$
\delta(\phi)=\frac{a}{v(\xi)}
$$
and
$$
D\delta(\phi)\chi=-\frac{v'(\xi)}{v(\xi)}\int^0_{-a/v(\xi)}\chi(s)ds.
$$
\end{prop}

Before giving the proof recall from \cite[page 47]{W1} or \cite[page 466]{HKWW}
 that the evaluation map
$$
ev_C:C\times[-r,0]\ni(\chi,u)\mapsto\chi(u)\in\mathbb{R}.
$$
is continuous (but not
locally Lipschitz continuous, let alone differentiable),
and that the restricted
evaluation map
$$
ev:C^1\times(-r,0)\ni(\phi,u)\mapsto\phi(u)\in\mathbb{R}
$$
is continuously differentiable with
\be
D\,ev(\phi,u)(\hat{\phi},\hat{u})=D_1ev(\phi,u)\hat{\phi}+D_2ev(\phi,u)\hat{u}=\hat{\phi}(u)+\phi'(u)\hat{u},
\label{eq:eval deriv}
\ee
where $D_1$ and $D_2$ denote partial derivatives with respect to the argument in $C^1$ and in $(-r,0)$, respectively.
The substitution operator
$$
V:C\ni\phi\mapsto v\circ\phi\in C
$$
is continuously differentiable with
$$
(DV(\phi)\hat{\phi})(s)=v'(\phi(s))\hat{\phi}(s)\quad\text{for all}\quad\hat{\phi}\in C,\quad s\in[-r,0],
$$
see for example \cite[Appendix IV, Lemma 1.5]{DvGVLW}.

\begin{proof}[Proof of Proposition~\ref{prop:1}]
For every $\phi\in C$ the value  $u=\delta(\phi)$ is the unique solution of the equation
$h(u,\phi)=0$, where
$h:(0,r)\times C\to\mathbb{R}$
is given by
$$
h(u,\phi)=a-\int^0_{-u}v(\phi(s))ds=a - ev(I(V(\phi))),-u)
$$
with the continuous linear integration operator $I:C\to C^1$ defined by
$$
(I\psi)(t)=\int^0_t\psi(s)ds.
$$
The map $h$ is continuously differentiable with
$$
D_1h(u,\phi)1=-v(\phi(-u))<0
$$
and
\begin{multline*}
D_2h(u,\phi)\chi  =  -D_1ev(I(V(\phi)),-u)DI(V(\phi))DV(\phi)\chi
 = -(DI(V(\phi))DV(\phi)\chi)(-u)\\
 =  -(I(DV(\phi)\chi))(-u)=-\int^0_{-u}v'(\phi(s))\chi(s)ds.
\end{multline*}
The Implicit Function Theorem applies at every $(\delta(\phi),\phi)\in(-r,0)\times C$ and yields that locally,
$\delta$ is given by a continuously differentiable map. Differentiation of  the equation $h(\delta(\phi),\phi)=0$ gives
\begin{displaymath}
D\delta(\phi)\chi  =  -\frac{D_2h(\delta(\phi),\phi)\chi}{D_1h(\delta(\phi),\phi)1}
 =  -\frac{\int^0_{-\delta(\phi)}v'(\phi(s))\chi(s)ds}{v(\phi(-\delta(\phi)))}.
\end{displaymath}
In the case $\phi$ is constant with value $\xi\in\mathbb{R}$, Eq.~\eqref{eq:thres-3}
gives $a=\delta(\phi)v(\xi)$, and by the previous formula,
\begin{displaymath}
D\delta(\phi)\chi=-\frac{v'(\xi)}{v(\xi)}\int^0_{-a/v(\xi)}\chi(s)ds.
\end{displaymath}
\end{proof}

Using the delay functional $\delta:C\to(0,r)$ the  system
\eqref{eq:basic},\eqref{eq:thres-2}
is reduced to the equation
\be
x'(t) = \beta\,e^{-\mu\,\delta(x_t)}\frac{v(x(t))}{v(x(t-\delta(x_t)))}g(x(t-\delta(x_t)))-\gamma\,x(t)\label{eq:single-1}
\ee
with segments $x_t\in C$. Eq. \eqref{eq:single-1} is of the form
$x'(t)=G_C(x_t)\in\mathbb{R}$ with the map $G_C:C\to\mathbb{R}$ given by
\be
G_C(\phi)  = \beta\,e^{-\mu\,\delta(\phi)}\frac{v(\phi(0))}
{v(\phi(-\delta(\phi)))}g(\phi(-\delta(\phi)))-\gamma\,\phi(0)
\label{eq:single-2}
\ee
We observe that $G_C$ is continuous. The restriction $G$ of $G_C$ to $C^1$ is continuously differentiable because, for $\phi\in C^1\subset C$, we have
$ev_C(\phi,-\delta(\phi))=ev(\phi,-\delta(\phi))$.  Here
$ev:C^1\times(0,r)\to\mathbb{R}$ is continuously differentiable, the map $\delta$ is continuously differentiable, and the  evaluation map
$ev_0:C\ni\chi\mapsto\chi(0)\in\mathbb{R}$ is linear and continuous.

To simplify the calculations below we now introduce the
continuous map
$$
E_C:C\to\mathbb{R},\quad E_C(\phi)=ev_C(\phi,-\delta(\phi)),
$$
and the continuously differentiable map
$$
E:C^1\to\mathbb{R},\quad E(\phi)=ev(\phi,-\delta(\phi)),
$$
with the derivative at $\phi\in C^1$ given by
\begin{eqnarray*}
	DE(\phi)\hat{\phi}& = & D_1ev(\phi,-\delta(\phi))\hat{\phi}+D_2ev(\phi,-\delta(\phi))D(-\delta)(\phi)\hat{\phi}\\	
	& = & \hat{\phi}(-\delta(\phi))-\phi'(-\delta(\phi))D\delta(\phi)\hat{\phi}	
\end{eqnarray*}
for all $\hat{\phi}\in C^1$. Notice that the right hand side of the previous equation makes sense also for arguments $\chi\in C$ instead of $\hat{\phi}\in C^1$. Thus they define linear extensions $D_eE(\phi):C\to\mathbb{R}$ of $DE(\phi):C^1\to\mathbb{R}$.  Using the continuity of the map $ev_C$, and the fact that differentiation $C^1\ni\phi\mapsto\phi'\in C$ is linear and continuous we obtain the next result.

\begin{prop}
	The map
	$C^1\times C\ni(\phi,\chi)\mapsto D_eE(\phi)\chi\in\mathbb{R}$
is continuous.	
\end{prop}

Incidentally, in the case $\phi\in C^1$ is constant with value $\xi\in\mathbb{R}$ we have
$$
E(\phi)=\xi,\quad DE(\phi)\hat{\phi}=\hat{\phi}(-a/v(\xi)).
$$
With the linear continuous evaluation map $ev_0:C\ni\phi\mapsto\phi(0)\in\mathbb{R}$ we obtain for the restriction $G$ of $G_C$
\be
G(\phi)=  \beta\,e^{-\mu\delta(\phi)}\frac{v(ev_0\phi)}{v(E(\phi))}g(E(\phi))
-\gamma\,ev_0\phi\label{eq:map G}.
\ee
In the sequel we will show that the initial value problem
\begin{equation}
x'(t)=G(x_t)\quad\text{for}\quad t>0,\quad x_0=\phi\label{eq:ivp}
\end{equation}
is well-posed on the set
$$
X_G=\{\phi\in C^1:\phi'(0)=G(\phi)\},
$$
which is a continuously differentiable submanifold of codimension $1$ in the space $C^1$. This result  follows from results in  \cite{W1,W2,HKWW}
	provided that  the following two assertions are verified:
\begin{enumerate}
	\item $X_G\neq\emptyset$;  and
	\item $G$ has the property that
	
	\medskip
	
	{\it each derivative $DG(\phi):C^1\to\mathbb{R}$, $\phi\in C^1$, has a linear extension $D_eG(\phi):C\to\mathbb{R}$, and the map
		$$
		C^1\times C\ni(\phi,\hat{\phi})\mapsto D_eG(\phi)\hat{\phi}\in\mathbb{R}
		$$
		is continuous.}
\end{enumerate}

\medskip

Property 
 (2) is a version of being {\it almost Fr\'echet differentiable}
from \cite{M-PNP}.
In case the delay functional $\delta:C\to(0,r)$ is bounded away from zero, which happens to be true for $v$ bounded also from above as in the subsequent sections, the manifold $X_G$ is simply a graph in the space $C^1$, given by a continuously differentiable map from an open subset of the closed hyperplane $\{\phi\in C^1:\phi'(0)=0\}$ into a complementary line in $C^1$, see \cite[Theorem 2.4]{walther2021solution}.
\medskip

We now proceed to the proof of properties (1) and (2). We first show that  $X_G\neq\emptyset$. The continuous map
$$
[0,\infty)\ni\xi\mapsto\beta\,e^{-\mu\,a/v(\xi)}-\gamma\,\xi\in\mathbb{R}
$$
is positive at $\xi=0$ and tends to $-\infty$ for $\xi\to\infty$. Therefore the Intermediate Value Theorem yields a zero $\zeta>0$ of this map. The constant function $\phi\in C^1$ with value $\zeta$ satisfies $\phi'(0)=0=G(\phi)$, so it belongs to the set $X_G$.
This finishes the proof of property (1).

\medskip

To prepare for the proof of the extension property (2)
we compute the derivatives $DG(\phi)$, $\phi\in C^1$.
For $\delta:C\to\mathbb{R}$ we use the fact that restrictions of differentiable maps $m:C\to\mathbb{R}$ to $C^1$ remain differentiable, with derivatives $D(m|C^1)(\phi):C^1\to\mathbb{R}$ being restrictions of the derivatives $Dm(\phi):C\to\mathbb{R}$, $\phi\in C^1\subset C$. It follows that
\begin{multline} \label{eq:DG}
DG(\phi)\hat{\phi} =   -\gamma\,\hat{\phi}(0)+\mbox{}\\
\shoveright{\beta\biggl\{\frac{1}{[v(E(\phi))]^2}
\left[v'(\phi(0))\hat{\phi}(0)\cdot v(E(\phi))
 -v(\phi(0))v'(E(\phi))DE(\phi)\hat{\phi}\right] e^{-\mu\delta(\phi)}g(E(\phi))}\\
+ \frac{v(\phi(0))}{v(E(\phi))}\left[
-\mu\,e^{-\mu\delta(\phi)}D\delta(\phi)\hat{\phi}\cdot g(E(\phi))
 +e^{-\mu\delta(\phi)}g'(E(\phi))DE(\phi)\hat{\phi}\right]\biggr\}.
\end{multline}
Now we are ready to verify property   (2). In the formula for $DG(\phi)\hat{\phi}$, 
replace the real number $DE(\phi)\hat{\phi}$  by $D_eE(\phi)\chi$ with $\chi\in C$, replace the function $\hat{\phi}$ by $\chi$, and replace $\hat{\phi}(0)$ by $\chi(0)$. This defines $D_eG(\phi)\chi\in\mathbb{R}$ for $\phi\in C^1$ and $\chi\in C$ so that the maps $D_eG(\phi):C\to\mathbb{R}$, $\phi\in C^1$, are linear. Using the continuous differentiability of $\delta:C\to\mathbb{R}$ and Proposition \ref{prop:1} 
one shows that the map
$C^1\times C\ni(\phi,\chi)\mapsto D_eG(\phi)\chi\in\mathbb{R}$
is continuous.
This finishes the proof of property  (2).

\medskip

With (1) and (2) verified,
results from  \cite{W1,W2,HKWW} apply and yield the following.

The set $X_G$ is a continuously differentiable submanifold of the Banach space $C^1$, with codimension 1.  Each $\phi\in X_G$ uniquely determines a maximal continuously differentiable solution $x:[-r,t_x)\to\mathbb{R}$, $0<t_x\le\infty$, of the initial value problem
\eqref{eq:ivp}.  That is, $x$ is continuously differentiable and satisfies $x_0=\phi$ and $x'(t)=G(x_t)$ for all $t\in(0,t_x)$, and any other continuously differentiable function $y:[-r,t_y)\to\mathbb{R}$, $0<t_y\le \infty$,  which satisfies $y_0=\phi$ and $y'(t)=G(y_t)$ for all $t\in(0,t_y)$ is a restriction of $x$. All segments $x_t$, $0\le t<\infty$, belong to $X_G$ (because of Eq.~\eqref{eq:ivp}). Write $x^{\phi}=x$ and $t_{\phi}=t_x$.
Let
$$
\Omega_G=\{(t,\phi)\in[0,\infty)\times X_G:0\le t<t_{\phi}\}.
$$
The equation
$$
S_G(t,\phi)=x_t^{\phi},
$$
defines a continuous semiflow  $S_G:\Omega_G\to X_G$. For each $t\ge0$ the set
$$
\Omega_{G,t}=\{\phi\in X_G:t<t_{\phi}\}
$$
is an open subset of $X_G$ (possibly empty), $\Omega_{G,0}=X_G$, and each map
$$
S_{G,t}:\Omega_{G,t}\ni\phi\mapsto S_G(t,\phi)\in X_G,\quad t\ge0,
$$
on a non-empty domain is continuously differentiable.

\medskip

Moreover, the restriction of the semiflow $S_G$ to the open subset $\{(t,\phi)\in\Omega_G:r<t\}$ of the manifold $\mathbb{R}\times X_G$ is continuously differentiable \citep{W2}.

\section{Positivity, dissipativity, global attractor}\label{sec:positivity}

In addition to the assumptions made in Section~\ref{sec:semiflow},
we assume in this section that the function $v$ is also bounded from above by a real number $v_U\ge v_0$.
Using Eq.~\eqref{eq:thres-3} we infer
$$
\frac{a}{v_U}\le\delta(\phi)\le\frac{a}{v_0}
$$
for all $\phi\in C$.

\begin{prop}\label{prop 4}
	For every $c>0$ there exists $c'>0$ such that for all $\phi\in X_G$ with $|\phi(t)|\le c$ on $[-r,0]$ and for all $t\in[-r,t_{\phi})$ we have $|x^{\phi}(t)|\le c'$.
\end{prop}

\begin{proof}
Let $\phi\in X_G$ with $|\phi(t)|\le c$ on $[-r,0]$  be given, and set $x=x^{\phi}$.
The first term on the right hand side of Eq.~\eqref{eq:single-1} is positive and bounded by the constant
$$
d^U=\beta\,g_U\,\frac{v_U}{v_0}.
$$
The variation-of-constants formula yields
\begin{align*}
|x(t)| & = \bigg|x(0)e^{-\gamma\,t} + \! \left. \int_0^t\! e^{-\gamma(t-s)}\! \left[\frac{\beta v(x(s))}{v(x(s-\delta(x_s))}
	e^{-\mu\,\delta(x_s)}g(x(s-\delta(x_s)))\right]^{\!}ds\right|\\
	& \le c+e^{-\gamma\,t}\frac{d^U}{\gamma}\left(e^{\gamma\,t}-1\right)
 \le c+\frac{d^U}{\gamma}\quad\text{for}\quad 0\le t<t_{\phi}.
\end{align*}
Set $c'=c+\frac{d^U}{\gamma}$. With $|x(t)|=|\phi(t)|\le c$ on $[-r,0]$ we obtain $|x(t)|\le c'$ for all $t\in[-r,t_{\phi})$.
\end{proof}

Observe that for any (continuously differentiable) solution $x:[-r,t_x)\to\mathbb{R}$ of Eq.~\eqref{eq:single-1} and for any $t\in[0,t_x)$ the first term on the right hand side of Eq. \eqref {eq:single-1} belongs to the interval
$$
[d^L,d^U]=\left[\beta\frac{v_0}{v_U}e^{-\mu\, a/v_0}g_0,\beta\, g_U\,\frac{v_U}{v_0}\right],
$$
and in the case
	$$
	x(t)>\frac{d^U}{\gamma}\quad\text{we have}\quad x'(t)\le d^U-\gamma\,x(t)<0,
	$$
	while for
	$$
	x(t)<\frac{d^L}{\gamma}\quad\text{we have}\quad x'(t)\ge d^L-\gamma\,x(t)>0.
	$$
Set
$$
Q=\left[\frac{d^L}{\gamma},\frac{d^U}{\gamma}\right]\subset\mathbb{R}
$$
and
$$
R=\{\phi\in C^1:\phi([-r,0])\subset Q\}.
$$

\begin{prop}\label{prop 5}(Global existence, absorption and positive invariance, positivity)
	
	\medskip
	\noindent
	(i) For all $\phi\in X_G$, $t_{\phi}=\infty$.
	
	\medskip
	\noindent
	(ii) For every neighbourhood $N$ of $Q$ in $\mathbb{R}$ and for each $\phi\in X_G$ there exists $t(\phi,N)\in[0,\infty)$ with $x^{\phi}(t)\in N$ for all $t\ge t(\phi,N)$.
	
	\medskip
	\noindent
	(iii) If $\phi\in X_G\cap R$ then $x^{\phi}(t)\in Q$ for all $t\ge0$.
	
	\medskip
	\noindent
	(iv) If $\phi\in X_G$ is strictly positive then $x^{\phi}(t)>0$ for all $t\ge-r$.
\end{prop}

\begin{proof}
	\begin{description}
		\item[(i)]
		Let $\phi\in X_G$ be given. From Proposition \ref{prop 4} the solution $x=x^{\phi}$ is bounded.
		Using this and Eq.~\eqref{eq:single-1} we infer that $x'$ is bounded. It follows that $x$ is Lipschitz continuous. Assume now $t_{\phi}<\infty$. Then Lipschitz continuity yields that $x$ has a limit $\xi\in\mathbb{R}$ at $t=t_{\phi}$ and $x$ extends to a continuous map $\hat{x}:[-r,t_{\phi}]\to\mathbb{R}$. From uniform continuity on the compact interval $[-r,t_{\phi}]$ it follows that the curve $[0,t_{\phi}]\ni t\mapsto \hat{x}_t\in C$ is continuous. Using this and the equation
		$$
		x'(t)=G(x_t)=G_C(\hat{x}_t)\quad\text{for}\quad0\le t<t_{\phi}
		$$
		with the continuous map $G_C:C\to\mathbb{R}$, we also conclude that $x'$ has a limit $\eta\in\mathbb{R}$ at $t=t_{\phi}$. It follows that $\hat{x}$ is continuously differentiable (with $\hat{x}'(t_{\phi})=\eta$), and $\hat{x}'(t_{\phi})=G_C(\hat{x}_{t_{\phi}})=G(\hat{x}_{t_{\phi}})$. In particular,
		$\psi=\hat{x}_{t_{\phi}}$ belongs to $X_G$, and defines a maximal solution $x^{\psi}:[0,t_{\psi})\to\mathbb{R}$ of Eq.~\eqref{eq:form G}, with $0<t_{\psi}\le\infty$. From the semiflow properties, it follows that when $t_{\psi}=\infty$ we have $t_{\phi}=\infty$, in contradiction to the assumption above, while in the case $t_{\psi}<\infty$ we get $t_{\phi}\ge t_{\phi}+t_{\psi}$, which contradicts $t_{\psi}>0$.

		\medskip
		\item[(ii)]
		Let a neighbourhood $N$ of $Q$ in $\mathbb{R}$ and $\phi\in X_G$ be given. Set $x=x^{\phi}$. There exists  $\epsilon>0$ so that
		for
		$$
			d_{-\epsilon}=\frac{d^L}{\gamma}-\epsilon\quad\text{and}\quad
			d_{+\epsilon}=\frac{d^U}{\gamma}+\epsilon
	$$
	we have $d_{-\epsilon}>0$ and $N\supset[d_{-\epsilon},d_{+\epsilon}]$.

		{\bf (ii).1.}  Proof that when $x(t)\le d_{+\epsilon}$ for some $t\ge0$ we have
		$$
		x(s)\le d_{+\epsilon}\quad\text{for all}\quad s\ge t.
		$$
		Otherwise $x(s)>d_{+\epsilon}=\frac{d^U}{\gamma}+\epsilon$ for some $s>t$. For the smallest $u\in[t,s]$ with $x(u)=x(s)$ we have $0\le x'(u)$ and, on the other hand,
		$$
		x'(u)<d^U-\gamma\,x(u)=d^U-\gamma\,x(s)<-\epsilon\,\gamma<0.
		$$
		
		{\bf (ii).2.}  Proof that when $x(s)>d_{+\epsilon}$ for some $s\ge0$ there exists $t>s$ with
		$$
		x(t)\le d_{+\epsilon}.
		$$
		Otherwise $x(t)>d_{+\epsilon}=\frac{d^U}{\gamma}+\epsilon$ on $[s,\infty)$. Hence
		$$
		x'(t)<d^U-\gamma\,x(t)<-\epsilon\,\gamma<0\quad\text{on}\quad[s,\infty),
		$$
		and consequently $x(t)\to-\infty$ as $t\to\infty$, in contradiction to the assumption.
		
		\medskip
		{\bf (ii).3.}  It follows that there exists $t^{\ast}\ge0$ with
		$$
		x(t)\le d_{+\epsilon}\quad\text{for all}\quad t\ge t^{\ast}.
		$$
		Similarly one finds $t(\phi,N)\ge t^{\ast}$ with
		$$
		x(t)\ge d_{-\epsilon}\quad\text{for all}\quad t\ge t(\phi,N).
		$$
		Hence
		$$
		x(t)\in[d_{-\epsilon},d_{+\epsilon}]\subset N\quad\text{for all}\quad t\ge t(\phi,N).
		$$

		\item[(iii)] The proof of assertion (iii) begins with the assumption that for a given $\phi\in X_G\cap R$ there exists $s>0$ with $x^{\phi}(s)>\frac{d^U}{\gamma}$, and is then accomplished by a simplified version of arguments as in Part (ii).1.
		
		\medskip
		
		\item[(iv)]  Let $\phi\in X_G$ be given with $\phi(t)>0$ for all $t\in[-r,0]$. Set $x=x^{\phi}$. The assumption $x(t)\le0$ for some $t>0$ leads to a smallest $t>0$ with $x(t)=0$. Necessarily, $x'(t)\le0$ while Eq.~\eqref{eq:single-1} yields $x'(t)>0$. It follows that $x(t)>0$ for all $t\ge -r$.
	\end{description}
\end{proof}

The solution manifold $X_G$ is a closed subset of the space $C^1$, and thereby a complete metric space with respect to the metric given by the norm on $C^1$.   The next result implies that the semiflow $S_G$ on the complete metric space $X_G$ is point dissipative as defined in \cite{hale88}.

\begin{cor}\label{cor 1}
	There is a bounded open subset $B_G$  of the submanifold $X_G\subset C^1$, with
	$$
	\phi(t)>0\quad\text{for all}\quad \phi\in B_G,\quad t\in[-r,0],
	$$
	such that for every $\phi\in X_G$ there exists $t(\phi)\ge0$ with
	$$
	S_G(t,\phi)\in B_G\quad\text{for all}\quad t\ge t(\phi).
	$$
\end{cor}

\begin{proof}
Choose $c>0$ so that $N=(0,c)$ is a neighbourhood of $Q$. Set $\tilde{R}=\{\phi\in C^1:\phi([-r,0])\subset N\}$.
	Let $\phi\in X_G$ be given. Choose $t(\phi,N)$ according to Proposition \ref{prop 5} (ii).
	From Eq.~\eqref{eq:single-1} we see that  the map $G$
	sends the set $\tilde{R}$ (which is not a bounded subset of $C^1$) into a bounded subset of $\mathbb{R}$, say, into
	$(-b,b)$ for some $b>0$. It follows that for all $t\ge t(\phi,N)+r$ we have $|(x^{\phi})'(t)|<b$. For $t\ge t(\phi,N)+2r$ we obtain $x^{\phi}_t\in\{\psi\in X_G\cap\tilde{R}:|\psi'|<b\}=B_G$. The set $B_G$ is an open and bounded
	subset of $X_G$, with $0<\phi(t)$ for all $\phi\in B_G,\,t\in[-r,0]$. Set $t(\phi)=t(\phi,N)+2r$.
\end{proof}

Recall from \cite{hale88} the definition of a global attractor of a semiflow, which in the case of our semiflow $S_G$ is equivalent to saying that a subset $A_G\subset X_G$ is a global attractor if it is
\begin{itemize}
		\item compact,
		\item invariant in the sense that for every $\phi\in A_G$ there exists an entire flowline\footnote{An entire flowline is a curve $\xi:\mathbb{R}\to X_G$ with
			$\xi(t+s)=S_G(t,\xi(s))$ for all $t\ge0$ and $s\in\mathbb{R}$} $\xi:\mathbb{R}\to X_G$ with $\xi(0)=\phi$ and $\xi(\mathbb{R})\subset A_G$, and
		\item if $A_G$ attracts every bounded set $B\subset X_G$ in the sense that given an open neighbourhood  $U\supset A_G$ of $A_G$ in $X_G$ there exists $t_{B,U}\ge0$ such that
		$$
		S_G([t_{B,U},\infty)\times B)\subset U.
		$$
	\end{itemize}

Theorem 3.4.8 of \citep{hale88}
guarantees the existence of such a global attractor provided the semiflow is point-dissipative and there exists $t_1\ge0$ so that the semiflow $S_G$ is \emph{completely continuous for $t\ge t_1$}. The property of being completely continuous (for $t\ge t_1$) is explained after
Lemma 3.2.1 in \citep{hale88}.
It means that
for every $t\ge t_1$ and for every bounded set $B$ with
$S_G([0,t]\times B)$ bounded the set $S_G(\{t\}\times B)$ is precompact.

To prove  the latter property, it is sufficient to verify the slightly stronger property that for every bounded set $B\subset X_G$
\begin{description}
	\item[(i)]
	 there exists $t_1\ge0$
	so that for every $t\ge t_1$ the set $S_G(\{t\}\times B)$ is precompact.
\end{description}

\begin{thm}\label{thm 4-1}
	The semiflow $S_G$ has a global attractor $A_G\subset X_G$, with $\phi(t)>0$ for all $\phi\in A_G,\,t\in[-r,0]$.
\end{thm}

\begin{proof}
	1. We first show that for every bounded subset $B\subset X_G$ there exists $c_B>0$ with
	$$
	|x^{\phi}(t)|\le c_B\quad\text{and}\quad |(x^{\phi})'(t)|\le c_B\quad\text{for all}\quad\phi\in B\quad\text{and}\quad t\ge-r.
	$$
	Let $B\subset X_G$ be bounded (with respect to the norm of the space $C^1$). Proposition \ref{prop 4} guarantees the existence of a constant $c_{B,0}$ with $|x^{\phi}(t)|\le c_{B,0}$ for all $\phi\in B$, $t\ge-r$. Then  \eqref{eq:map G} shows that the set
	$\{G(x^{\phi}_t)\in\mathbb{R}:\phi\in B,t\ge0\}$ is bounded, and Eq.~\eqref{eq:form G} gives that the set
	$\{(x^{\phi})'(t)\in\mathbb{R}:\phi\in B,t\ge0\}$ is bounded. Also the set $\{\phi'(t)\in\mathbb{R}:\phi\in B,-r\le t\le0\}$ is bounded.
	
	\medskip
	
	2. Claim: For every bounded subset $B\subset X_G$ the set $S_G(\{r\}\times B)\subset X_G$ has compact closure in $C^1$.
	
	\medskip
	
	Proof: (a)  Let $B\subset X_G$ be a bounded subset of $C^1$. Due to Part 1 the sets $\{x^{\phi}(t)\in\mathbb{R}:\phi\in B,-r\le t\le r\}$ and
	$\{(x^{\phi})'(t)\in\mathbb{R}:\phi\in B, -r\le t\le r\}$ are bounded. Using the Mean Value Theorem we see that in particular the set $S_G(\{r\}\times B)$ is equicontinuous. As it also is bounded in $C$, the Ascoli-Arz\`{e}la Theorem implies that its closure in $C$ is compact.
	
	\medskip
	
	(b) We turn to the set $\{S_G(r,\phi)'\in C:\phi\in B\}$ of derivatives, which is bounded in $C$, and proceed to show that it is also equicontinuous. As in Part (a) one sees that the closure $K$ of the set
	$$
	\{S_G(t,\phi)\in C:\phi\in B,0\le t\le r\}
	$$
	in the space $C$ is compact. The map $G:C^1\to\mathbb{R}$ is the restriction of the continuous map $G_C:C\to\mathbb{R}$
	which is uniformly continuous on the compact set $K\subset C$. Using the boundedness of the set $\{(x^{\phi})'(t)\in\mathbb{R}:\phi\in B, -r\le t\le r\}$ and the Mean Value Theorem one finds that the curves
	$$
	[0,r]\ni t\mapsto S_G(t,\phi)\in C,\quad\phi\in B,
	$$
	are uniformly Lipschitz continuous, hence equicontinuous. Now
	let $t_0\in [0,r]$ and $\epsilon>0$ be given. There exists $\delta'>0$ with
	$$
	|G_C(\phi)-G_C(\psi)|\le\epsilon\quad\text{for all}\quad\phi,\psi\quad\text{in}\quad S_G([0,r]\times B)\quad\text{with}\quad|\phi-\psi|\le\delta',
	$$
	due to uniform continuity of $G_C$ on $K$. Due to equicontinuity there exists $\eta>0$ with
	$$
	|S_G(t,\phi)-S_G(t_0,\phi)|\le\delta'\quad\text{for all}\,\phi\in B\,\text{and}\, t\in[0,r]\,\text{with}\,|t-t_0|<\eta.
	$$
	Hence
	\begin{align}
	|(x ^{\phi})'(t)-(x^{\phi})'(t_0)| & =|G(S_G(t,\phi))-G(S_G(t_0,\phi))|\nonumber\\
	& =|G_C(S_G(t,\phi))-G_C(S_G(t_0,\phi))|<\epsilon\nonumber
	\end{align}
	for all $\phi\in B$ and $t\in[0,r]$ with $|t-t_0|<\eta$.
	
	\medskip
	
	(c) The Ascoli-Arz\`{e}la Theorem implies that the closure of $\{S_G(r,\phi)'\in C:\phi\in B\}$  in $C$ is compact.
	For the closure of $S_G(\{r\}\times B)$ in $C^1$ to be compact it is sufficient to show that every sequence of points
	$\phi_j\in S_G(\{r\}\times B)$, $j\in\N$, has a subsequence which converges in $C^1$. Let a sequence $(\phi_j)_1^{\infty}$ in $S_G(\{r\}\times B)$ be given. Part a) implies that there is a subsequence which converges in $C$ to some $\phi\in C$. Part b) shows that the subsequence has a secondary subsequence  so that the derivatives of the latter subsequence converge in $C$ to some $\psi\in C$. It follows that $\phi\in C^1$ with $\phi'=\psi$, which in turn yields convergence of the  secondary subsequence to $\phi\in C^1$ as $k\to\infty$.
	
	\medskip
	
	3. We now show that for every bounded subset $B\subset X_G$ and for every $t\ge r$ the set $S_G(\{t\}\times B)\subset X_G$ has compact closure in $C^1$.  Let $B\subset X_G$ be bounded and let $t\ge r$. For every $\phi\in B$,
	$$
	S_G(t,\phi)=S_G(t-r,S_G(r,\phi)),
	$$
	hence
	$$
	S_G(\{t\}\times B)=S_G(t-r,\cdot)(S_G(\{r\}\times B)).
	$$
	Use that the closure of $S_G(\{r\}\times B)$ in $C^1$ is compact and belongs to $X_G$ (since $X_G$ is a closed subset of $C^1$), and that the map $S_G(t-r,\cdot)$ is continuous, and conclude that the closure of $S_G(\{t\}\times B)$ in $C^1$ is contained in a compact subset of $X_G\subset C^1$.
	
	\medskip
	
	4.   Point dissipativity from Corollary \ref{cor 1} in combination with condition (i) from the previous Part 3 of the proof  yield existence of a global attractor, see the remarks preceding Theorem~\ref{thm 4-1}.
	
	\medskip
	
	5. Finally we show  that for all $\phi\in A_G$ and $t\in[-r,0]$ we have $\phi(t)>0$. Let $\phi\in A_G$ be given. There exists a solution $x:\mathbb{R}\to\mathbb{R}$ of Eq.~\eqref{eq:single-1} with $x_0=\phi$ and all segments $x_s$, $s\in\mathbb{R}$, in the compact set $A_G$. It suffices to deduce $x(t)>0$ for all $t\in\mathbb{R}$. Proof of this: First, observe that $x$ is bounded. Assume $x(t)\le0$ for some $t\in\mathbb{R}$. In case $x(t)=0$ Eq.~\eqref{eq:single-1} yields $x'(t)>0-\gamma\,x(t)=0$. It follows that $x(u)<0$ for some $u\in(-\infty,t)$. In case $x(t)<0$ set $u=t$. For every $s\le u$ the variation-of-constants formula  yields
	$$
	x(u)\ge x(s)e^{-\gamma(u-s)}+0,
	$$
	hence
	$$
	x(s)\le x(u)e^{\gamma(u-s)}\qquad(\to-\infty\quad\text{as}\quad s\to-\infty),
	$$
	and we arrive at a contradiction to the boundedness of $x$.
\end{proof}

\section{Linearization}\label{sec:linearization}

We now turn to linearization. At a point $\phi\in X_G$ the tangent space of the manifold $X_G$ is given by
$$
T_{\phi}X_G=\{\chi\in C^1:\chi'(0)=DG(\phi)\chi\}.
$$
For $\phi\in\Omega_{G,t}$ the derivative
$$
DS_{G,t}(\phi):T_{\phi}X_G\to T_{S_{G,t}(\phi)}X_G
$$
is given by
$$
DS_{G,t}(\phi)\chi=w^{\phi,\chi}_t
$$
where $w^{\phi,\chi}=w$ is the unique continuously differentiable solution $[-r,t_{\phi})\to\R$ of the IVP
\bea
w'(t) & = & DG(S_G(t,\phi))w_t\quad\text{for}\quad t>0,
\label{eq:varsyst1}\\
w_0 & = & \chi\in T_{\phi}X_G.
\label{eq:varsyst2}
\eea
Equation \eqref{eq:varsyst1}  is called the linear variational equation along the solution $x^{\phi}$ or  along the flowline
$$
S_G(\cdot,\phi):[0,t_{\phi})\ni t\mapsto S_G(t,\phi)\in X_G.
$$

Suppose that $\phi\in X_G$ is a stationary point of the semiflow $S_G$, that is, $t_{\phi}=\infty$ and $S_G(t,\phi)=\phi$ for all $t\ge0$. Then $\phi$ is constant
since for every $t\ge0$, $x^{\phi}(t)=x^{\phi}_t(0)=S_G(t,\phi)(0)=\phi(0)$, hence
\[ \phi(s)=S_G(r,\phi)(s)=x^{\phi}_r(s)=x^{\phi}(r+s)=\phi(0)\quad \mbox{  for each } s\in[-r,0]. \]

Let $\xi\in\R$ denote the value of $\phi$. Then $x^{\phi}(t)=\xi$ for all $t\ge-r$. To obtain the linear variational equation along this constant solution in terms of $g,v,a,\beta,\mu,\gamma$ we compute the values
$DG(\phi)\hat{\phi},\,\,\hat{\phi}\in C^1$ from the formula \eqref{eq:DG}. Using Proposition \ref{prop:1}
(for the values and for the derivatives of the map
$\delta$ in case of constant arguments) and the calculation of $DE(\phi)\hat{\phi}$ right after Proposition  \ref{prop:1}
(in case of constant arguments), we find
\begin{align}
DG(\phi)\hat{\phi} & = -\gamma \hat{\phi}(0)
+\beta\Bigg\{\frac{1}{[v(\xi)]^2}\left[v'(\xi)\hat{\phi}(0)v(\xi)
-v(\xi)v'(\xi)\hat{\phi}(-a/v(\xi))\right]e^{-\mu a/v(\xi)}g(\xi)\notag\\
& \qquad +\left[-\mu e^{-\mu a/v(\xi)}
\left(-\frac{v'(\xi)}{v(\xi)}\int^0_{-a/v(\xi)}\hat{\phi}(s)ds\right) g(\xi)+ e^{-\mu a/v(\xi)}g'(\xi)\hat{\phi}(-a/v(\xi))\right]\!\Bigg\}
\notag\\
& =  -\gamma \hat{\phi}(0)+A \hat{\phi}(0)
+\mu A \int^0_{-a/v(\xi)}\hat{\phi}(s)ds
+\left(\beta e^{-\mu a/v(\xi)}g'(\xi)-A\right)\hat{\phi}(-a/v(\xi))
\label{eq:DG-1-equi}
\end{align}
with
\be \label{eq:A}
A=\beta\frac{v'(\xi)}{v(\xi)}e^{-\mu a/v(\xi)}g(\xi).
\ee
The variational equation \eqref{eq:varsyst1}
along the constant solution $x^{\phi}:[-r,\infty)\ni t\mapsto\xi\in\R$  becomes
\be
w'(t) =  -\gamma w(t)+A w(t)
+ \left(\beta e^{-\mu a/v(\xi)}g'(\xi)-A\right)w(t-a/v(\xi))
 +\mu A \int^0_{-a/v(\xi)}w(t+s)ds
\label{eq:varsyst1stat}
\ee

The derivatives $T_{G,t}=DS_{G,t}(\phi)$, $t\ge0$, form a strongly continuous semigroup on the closed hyperplane $T_{\phi}X_G$ of the space $C^1$. This semigroup is given by $T_{G,t}\chi=T_{G,e,t}\chi$, where $T_{G,e,t}:C\to C$ is the solution operator associated with the {\it classical} initial value problem
\bea
w'(t) & = & D_eG(\phi)w_t\quad\text{for}\quad t>0,
\label{eq:ivponC1}\\
w_0 & = & \chi\in C,
\label{eq:ivponC2}
\eea
with the continuous linear functional $L:C\to\R$, $L=D_eG(\phi)$, as in the monographs, e.g., \cite{HVL,DvGVLW}. Recall  that by definition the solution $w:[-r,\infty)\to\R$ of the initial value problem
\eqref{eq:ivponC1},\eqref{eq:ivponC2} is only continuous, with the restriction
$w|_{[0,\infty)}$ continuously differentiable and satisfying Eq. \eqref{eq:ivponC1}.

\medskip

The extended derivative $D_eG(\phi):C\to\R$ in the case just considered (where $\phi\in X_G$ is a stationary point with value $\xi$) is given by \eqref{eq:DG-1-equi}, now for $\hat{\phi}\in C$. Therefore the equation \eqref{eq:ivponC1}  coincides with Eq. \eqref{eq:varsyst1stat}, considered for continuous maps $[-r,\infty)\to\R$ whose restrictions to $[0,\infty)$ are differentiable and satisfy Eq. \eqref{eq:varsyst1stat} for all $t\ge0$.

The stability of the zero solution of the linear variational equation
\eqref{eq:varsyst1stat} is determined by the spectrum $\sigma\subset\C$ of the generator of the semigroup $(T_{G,t})_{t\ge0}$ on $T_{\phi}X_G\subset C^1$, which coincides with the spectrum $\sigma_e\subset\C$ of the generator of the semigroup on $C$.

The spectrum $\sigma_e$ consists of the solutions $\lambda\in\C$ of the {\it characteristic equation}, which is obtained from the ansatz $\R\ni t\mapsto e^{\lambda\,t}\in\C$  for a complex-valued  solution of Eq.
	\eqref{eq:varsyst1stat} as follows.
	We write down Eq. \eqref{eq:varsyst1stat} for $w:t\mapsto e^{\lambda t}$, multiply by $e^{-\lambda t}$,  and obtain the equation
$$
\lambda=-\gamma+A + \left(\beta\,e^{-\mu\,a/v(\xi)}g'(\xi)-A\right)e^{-\lambda\,a/v(\xi)}+\mu\,A\,\int^0_{-a/v(\xi)}e^{\lambda\, s}ds,
$$
or equivalently,
\be
0=\lambda+\gamma-A- \left(\beta\,e^{-\mu\,a/v(\xi)}g'(\xi)-A\right)e^{-\lambda\,a/v(\xi)}-\mu\,A\,\int^0_{-a/v(\xi)}e^{\lambda\,s}ds.\label{eq:char}
\ee

To investigate the stability of steady states in different special cases, we make explicit various forms of \eqref{eq:char}. 

If we study the case with a constant delay, i.e. $v$ is a constant function with $v(\xi)=v^-=v^+$ in \eqref{eq:vghill} and \eqref{eq:vpwconst}, then $v'=0$ and hence $A=0$. The characteristic equation reduces to
\begin{equation}
	\lambda=-\gamma+ \beta\,e^{-\mu\,a/v}g'(\xi)e^{-\lambda\,a/v}.
	\label{eq:char_A0g'}
\end{equation}

Similarly, if we study the case with $g$ constant by setting $g^-=g^+$ in \eqref{eq:vghill} and \eqref{eq:gpwconst}, then $g'(\xi)=0$ and the characteristic equation is of the form
\be
\lambda =-\gamma+A -Ae^{-\lambda a/v(\xi)}+\mu A\int^0_{-a/v(\xi)}e^{\lambda s}ds
= -\gamma + A(1- e^{-\lambda a/v(\xi)})\left(1 + \frac{\mu}{\lambda}\right).
	\label{eq:char_Ag'0}
\ee

When we study the limiting case where both $g$ and $v$ are piecewise constant as defined in \eqref{eq:gpwconst} and \eqref{eq:vpwconst}, then there are intervals where $A=v'(\xi)=0$. For a steady state $\xi$ in such an interval, the characteristic equation simplifies to
\begin{equation} \label{eq:char_A0g'0}
\lambda = -\gamma,
\end{equation}
where there is a unique negative real characteristic root and the steady state is stable.

\section{Dynamics with one Hill function}
\label{sec:onehill}

At a steady state $x(t)\equiv\xi\in\R$, equation \eqref{eq:thres} reduces to
\be \label{eq:steadydelay}
\tau(\xi)=\frac{a}{v(\xi)},
\ee
so the delay is still state-dependent. 
Equation \eqref{eq:basic} becomes
\be \label{eq:h}
0  = h(\xi):=\beta e^{-\mu\tau(\xi)}g(\xi) - \gamma\xi
\ee
at a steady state.
With $g(x)$ and $v(x)$ defined by \eqref{eq:vghill}
we have
that
$$\beta g^-\geq h(0)=\beta e^{-\mu a/v^-}g^-\geq 0,
\qquad h(x)
\leq\beta e^{-\mu a/v_U}g_U-\gamma x
\leq\beta g_U-\gamma x.$$
Consequently any steady state satisfies $\xi \in  [0,\beta g_U/\gamma]$, and
there is always at least one such steady state.

The steady states occur at the zeros of $h(\xi)$, which from \eqref{eq:h} occur at the intersections of
$\beta e^{-\mu\tau(\xi)}g(\xi)$ and $\gamma\xi$. Thus the number of steady states depends on the behavior of the term $e^{-\mu\tau(\xi)}g(\xi)$ in \eqref{eq:h}.

We begin by considering the simplified setting where either $g$ or $v$ is a constant function, while the other one is either  a Hill function defined  in \eqref{eq:vghill},
or a piecewise constant function \eqref{eq:gpwconst} or \eqref{eq:vpwconst}.
This leads to the four cases discussed below.

With $v^\pm:=v^-=v^+$ in \eqref{eq:vghill} and \eqref{eq:vpwconst} 
we obtain $v(\xi)=v^\pm$ independent of the value of $\xi$ and consequently
a constant delay
$\tau=a/v^\pm$, and equation \eqref{eq:basic} reduces to
\be \label{eq:constbasic}
x'(t)  = \beta e^{-\mu\tau}g(x (t -\tau)) - \gamma x(t),
\ee
which is a constant delay DDE with a monotone feedback nonlinearity. We consider this case first for decreasing and increasing $g$ in Sections~\ref{sec:gdownvconst} and~\ref{sec:gupvconst}, respectively.
The results which we obtain correspond to semi-local properties of equilibria which are familiar for the constant delay equation \eqref{eq:constbasic}
with a sufficiently smooth nonlinearity $g$.  If $g$ is strictly decreasing then there is a single equilibrium solution, from which periodic solutions bifurcate off in a sequence of Hopf bifurcations when a parameter multiplying $g$ grows to infinity. These periodic solutions can be distinguished by their oscillation frequencies. Stable periodic orbits occur only at the lowest possible frequency, for so-called \emph{slowly oscillating periodic solutions} \cite{sausage}. If $g$ is increasing then multiple equilibria are possible. Hopf bifurcations from these equilibria yield periodic orbits which are all unstable.
For more detailed information about the numerous local and global results on solutions  of autonomous delay differential equations which were achieved during the past decades see, for example, the survey \cite{HOW14}.

Sections~\ref{sec:gdownvconst} and~\ref{sec:gupvconst} contain results which are not covered by the existing theory.  We consider the limiting cases as the smooth function $g$ approaches a piecewise constant function, as these help to understand better the changes which occur in the dynamics when certain parameters grow to infinity. This also affords us the opportunity to present, in the simpler setting of constant delays, the methods we will subsequently use when the delays  are state-dependent in Sections~\ref{sec:gconstvdown} and~\ref{sec:gconstvup}.

With $g^-=g^+$ in \eqref{eq:vghill} and \eqref{eq:gpwconst}, $g$ is a constant function, while $v^-\ne v^+$ results in a state-dependent delay. We consider the two cases of $v$ decreasing or increasing in
Sections~\ref{sec:gconstvdown} and~\ref{sec:gconstvup}, respectively.
The general case
where $g$ and $v$ are both non-constant is studied in Section~\ref{sec:twohill}.

\subsection{Constant delay with decreasing $g$ $(g\downarrow,v\leftrightarrow)$}
\label{sec:gdownvconst}

In this section we  study \eqref{eq:basic}-\eqref{eq:thres} with a constant delay $\tau$ and decreasing $g$.
The DDE reduces to \eqref{eq:constbasic} with $\tau=a/v^\pm$ on
setting $v^\pm=v^-=v^+$ in \eqref{eq:vghill} or \eqref{eq:vpwconst}.
We require $g^->g^+$ in \eqref{eq:vghill} or \eqref{eq:gpwconst} to ensure that $g$ is decreasing.

\begin{figure}[htp!]
	\centering
	\begin{tabular}{cc}
		\begin{tikzpicture}[scale = 1.8]
		\tikzstyle{line} = [-,very thick]
		\tikzstyle{dotted line}=[.]
		\tikzstyle{arrow} = [->,line width = .4mm]
		\tikzstyle{unstable} = [red]
		\tikzstyle{stable} = [blue]
		\tikzstyle{pt} = [circle,draw=black,fill = black,minimum size = 1pt];
		\tikzstyle{bifpt} = [circle,draw=red,fill = red,minimum size = 1pt];
		\tikzstyle{upt} = [circle,draw=black,minimum size = 1pt];
		\def\arrlen{.3}
		
		\draw[line] (0,2) to (1.2,2);
		\draw[line] (1.2,2) to (1.2,1);
		\draw[line] (1.2,1) to (2.6,1);
		\draw[dotted line] (0,1) to (1.2,1);
		\draw[dotted line] (1.2,0) to (1.2,2);
		
		\draw[line] (0,0.2) to (0.5,2.5);
		\draw[line] (0,0.1) to (1,2.5);
		\draw[line] (0,0) to (1.5,2.5);
		\draw[line] (0,-0.08) to (1.83,2.5);
		\draw[line] (0,-0.165) to (2.13,2.5);
		\draw[line] (0,-0.25) to (2.55,2.5);
		\draw[line] (0.1,-0.25) to (2.6,2);
		\draw[line] (0.2,-0.25) to (2.6,1.4);
		
		\node at (1.2,-.2) {$\theta_g$};
		
		\node at (-.6,2) {$\beta e^{-\mu\tau}g^-$};
		\node at (-.6,1) {$\beta e^{-\mu\tau}g^+$};
		
		\node[bifpt] at (1.2,1) (8) [] {};
		\node[bifpt] at (1.2,2) (8) [] {};
		\node[upt] at (1.2,1.62) (8) [] {};
		\node[upt] at (1.2,1.33) (8) [] {};
		\node[pt] at (1.5,1) (8) [] {};
		\node[pt] at (2,1) (8) [] {};
		\node[pt] at (0.8,2) (12) [] {};
		\node[pt] at (0.4,2) (12) [] {};
		
		\end{tikzpicture}
		
	\end{tabular}
	\caption{Steady states of \eqref{eq:basic}, given by \eqref{eq:h}, occur at the intersections of $\xi\mapsto\beta e^{-\mu\tau}g(\xi)$ and $\xi\mapsto\gamma\xi$.  These are illustrated for various $\gamma$ in the limiting case of \eqref{eq:gpwconst} and \eqref{eq:vpwconst} with 
$v^\pm= v^-=v^+$, so the delay $\tau=a/v^\pm$ is constant, and $g^->g^+$, so $g$ is monotonically decreasing:  $(g \downarrow, v \leftrightarrow)$.}
	\label{fig:1}
\end{figure}
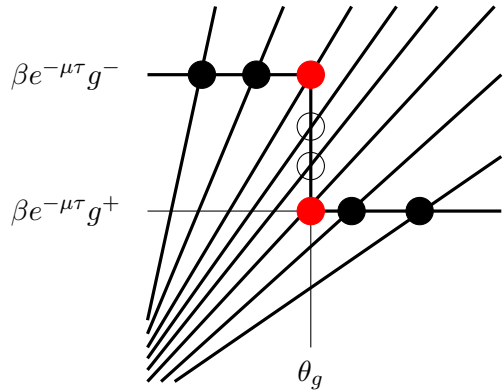

As noted above, if $h(0) \geq 0$ with $g$ decreasing we have $h'(\xi) = \beta e^{-\mu\tau} g'(\xi)-\gamma < 0$.
Therefore, with smooth $g$ defined in \eqref{eq:vghill},
for any fixed values of the parameters there is always exactly one steady state $\xi$. We note for later use that rearranging \eqref{eq:h}, for any fixed value of $\xi>0$, there is also a unique value of $\gamma$ for which $\xi$ is a steady state.

Figure~\ref{fig:1} illustrates the uniqueness of the steady state with a piecewise constant $g$ defined by \eqref{eq:gpwconst} as $\gamma$ varies. At the `corners', the steady state $\xi=\theta_g$ satisfies \eqref{eq:h} with $g=g^+$ and $g=g^-$ respectively, which gives rise to
\begin{equation}
	\gamma_1=\frac{\beta e^{-\mu\tau} g^+}{\theta_g} \quad \text{and} \quad \gamma_2=\frac{\beta e^{-\mu\tau} g^-}{\theta_g}.
	\label{eq:corners_vconst}
\end{equation}
We refer to the steady state where $e^{-\mu\tau}g(\xi)$ and $\gamma\xi$ intersect on the vertical line segment of the curve as a \textit{singular steady state}. The singular steady state exists for $\gamma\in(\gamma_1,\gamma_2)$, where the bounds on $\gamma$ are given by \eqref{eq:corners_vconst}.

Next, we consider the stability of the steady state for the piecewise constant function $g$ defined in \eqref{eq:gpwconst}.
If $\gamma \in (0, \gamma_1 ) \cup (\gamma_2, +\infty)$, the intersection occurs on the horizontal line segments.
Then
$A=g'(\xi)=0$ and the characteristic equation \eqref{eq:char} reduces to \eqref{eq:char_A0g'0}, and hence the
steady state must be stable.
This implies that in the limiting case only a singular steady state may be unstable. However, the characteristic equation is not defined in this situation
and we therefore consider the smooth nonlinearity $g$ defined in \eqref{eq:vghill} for large $n$. As a consequence of the choices made in this section, $A=0$ and  $g'(\xi)<0$ with $n>0$.

It follows immediately from \eqref{eq:char_A0g'} that 
there are no real non-negative characteristic values when $g'(\xi)<0$. So,
suppose $\lambda=\alpha\pm i\omega$, $\omega > 0$ is a solution of the characteristic equation \eqref{eq:char_A0g'}.
Then for $\lambda= \alpha+i\omega$ we have
\begin{displaymath}
\alpha+i\omega = -\gamma+\beta e^{-\mu\tau}g'(\xi)e^{-(\alpha+i\omega)\tau}
	= -\gamma+\beta e^{-\mu\tau}g'(\xi)e^{-\alpha\tau}\left[\cos(\omega\tau)-i\sin(\omega\tau)\right]. 
\end{displaymath}
Equating real and imaginary parts respectively, we obtain
\begin{gather} \label{eq:real}
\alpha+\gamma = \beta e^{-\mu\tau}g'(\xi)e^{-\alpha\tau}\cos(\omega\tau),\\
\omega = -\beta e^{-\mu\tau}g'(\xi)e^{-\alpha\tau}\sin(\omega\tau).
\label{eq:imag}
\end{gather}
Now suppose that $\lambda=\alpha+ i\omega$ is a root of the characteristic equation with $\alpha\geq0$. Then
\begin{displaymath}
\gamma\leq\alpha+\gamma\leq\beta e^{-\mu\tau}e^{-\alpha\tau}|g'(\xi)\cos(\omega\tau)|\leq
\beta e^{-\mu\tau}|g'(\xi)|.
\end{displaymath}
Consequently, if
\be \label{eq:ssstable}
\gamma>\beta e^{-\mu\tau}|g'(\xi)|
\ee
then $\alpha<0$ for all the characteristic values and the steady state is asymptotically stable.
Notice, that since \eqref{eq:h} is also satisfied at a steady state, equation \eqref{eq:ssstable}
is equivalent to
\be \label{eq:xigdashgone}
\left|\frac{\xi g'(\xi)}{g(\xi)}\right|<1.
\ee
We arrive at a sufficient condition for stability of the unique steady state: when \eqref{eq:xigdashgone} is satisfied at a steady state, the steady state is asymptotically stable.

The function $\xi \mapsto  \xi g'(\xi)/g(\xi)$ 
plays a central role in the analysis, so
we study its properties here. Differentiating $g$ in \eqref{eq:vghill} gives
\be \label{eq:gdashxi}
g'(\xi)= \frac{n\theta_g^n(g^{+}-g^{-}) \xi^{n-1}}{(\theta_g^n+\xi^n)^2}
\ee
and hence
\be \label{eq:xigdashghill}
\frac{\xi g'(\xi)}{g(\xi)}=\frac{n(g^+-g^-)(\xi/\theta_g)^n}{(1+(\xi/\theta_g)^n)(g^-+g^+(\xi/\theta_g)^n)}.
\ee
Now, let
\be \label{eq:fxpr}
f(x,p,r)=\frac{p(1-r)x^p}{(1+x^p)(r+x^p)}
\ee
and note from \eqref{eq:xigdashghill} that $f(\xi/\theta_g,n,g^-/g^+)=\xi g'(\xi)/g(\xi)$.
The following proposition will be essential in this and following sections.

\begin{prop} \label{prop:fxpr}
Let $f(x,p,r):\R_{>0}^3\to\R$ be defined by \eqref{eq:fxpr} then
\begin{enumerate}
\item
For fixed $p>0$, $r>0$,
$$\lim_{x\to0}f(x,p,r)=\lim_{x\to\infty}f(x,p,r)=0.$$
\item
For fixed $r>0$ with $r\neq1$ we have $\lim_{p\to\infty}|f(1,p,r)|=\infty$, with
$$f(1,p,r)=\frac{p(1-r)}{2(1+r)},$$
and $|f(1,p,r)|\leq p/2$.
\item
For fixed $r>0$ and fixed $x>0$ with $x\ne1$,
$$\lim_{p\to\infty}f(x,p,r)=0.$$
Moreover, for fixed $r>0$ and any fixed $x^-\in(0,1)$ and $x^+>1$
$$\lim_{p\to\infty}\Bigl(\max_{x\in[0,x^-]}|f(x,p,r)|\Bigr)=
\lim_{p\to\infty}\Bigl(\max_{x\in[x^+,\infty)}|f(x,p,r)|\Bigr)=0.$$
\item
For fixed $r>0$ with $r\ne1$ and fixed $p>0$, regarding $|f(x,p,r)|$ as a function of $x$ only,
$|f(x,p,r)|$ has a unique global maximum at $x=r^{1/2p}$ with
$$f(r^{1/2p},p,r)=\frac{p(1-r^{1/2})}{1+r^{1/2}},$$
and $|f(r^{1/2p},p,r)|\leq p$.
\end{enumerate}
\end{prop}

\begin{proof}
Follows using elementary differentiation and algebra.
\end{proof}

\sloppy{Proposition~\ref{prop:fxpr} shows that $\xi\mapsto\xi g'(\xi)/g(\xi)$ is a unimodal function which approaches zero uniformly as $n\to0$, but which resembles a delta-function
with its peak at $\xi=\theta_g(g^-/g^+)^{1/2n}\to\theta_g$
as $n\to\infty$.}

Now \eqref{eq:xigdashgone} and Proposition~\ref{prop:fxpr}(1) show that if $\xi\ll\theta_g$ or
for $\xi\gg\theta_g$ the steady state must be asymptotically stable.
From \eqref{eq:xigdashgone}, a necessary condition for the steady state to be unstable is
that $|\xi g'(\xi)/g(\xi)|\ge1$, but from
Proposition~\ref{prop:fxpr}(4) we have
\begin{displaymath}
\left|\frac{\xi g'(\xi)}{g(\xi)}\right|\leq|f(r^{1/2n},n,r)|=\frac{n|1-r^{1/2}|}{1+r^{1/2}}
\end{displaymath}
where $r=r_g=g^-/g^+$. Thus a necessary condition for the steady state to be unstable is that
$|f(r^{1/2n},n,r)|\ge1$. For $g$ decreasing, the value  $r_g=g^-/g^+>1$, and this necessary condition for instability becomes $n>1$ and
\be \label{eq:ratnec}
r_g=\frac{g^-}{g^+}\geq\left(1+\frac{2}{n-1}\right)^2.
\ee

Since the steady state is unique, a steady state bifurcation cannot occur, and we therefore investigate the existence of Hopf bifurcations.  At a Hopf bifurcation, $\lambda=\pm i\omega$,
and \eqref{eq:real} and \eqref{eq:imag}
reduce to
\begin{gather} \label{eq:realhopf}
\gamma = \beta e^{-\mu\tau}g'(\xi)\cos(\omega\tau),\\
\omega = -\beta e^{-\mu\tau}g'(\xi)\sin(\omega\tau).
\label{eq:imaghopf}
\end{gather}
As $\beta, \gamma, \omega, e^{-\mu\tau} > 0$ and $g'(\xi) < 0$, we must have
$\cos(\omega\tau) < 0 < \sin(\omega\tau)$, and hence
$\omega\tau \in (\pi/2+2k\pi, \pi+2k\pi)$ for $k \in \mathbb{N}$. We denote by $\omega_k$ any solution of  \eqref{eq:realhopf} and \eqref{eq:imaghopf} with the property  that
$\omega_k\tau \in (\pi/2+2k\pi, \pi+2k\pi)$ for $k \in \mathbb{N}$.

Note that at a Hopf bifurcation, in addition to \eqref{eq:realhopf} and \eqref{eq:imaghopf}, equation \eqref{eq:h} must also be satisfied.
These three equations  can be rearranged as
\begin{gather}
\gamma = \frac{\beta}{\xi} e^{-\mu\tau}g(\xi),  \label{eq:gam} \\
\omega\cot(\omega\tau)=-\gamma, \label{eq:omega}  \\
|g'(\xi)| =\frac{\sqrt{\gamma^2+\omega^2}}{\beta e^{-\mu\tau}}.  \label{eq:gdash}
\end{gather}

We will consider these equations sequentially to show that $\gamma$ and $\omega$ can be regarded as functions of $\xi$, resulting in a single equation to solve for $\xi$.
First  note that for a smooth nonlinearity $g(\xi)$, and for arbitrary  $\xi>0$, equation \eqref{eq:gam} gives a unique value
of $\gamma= \gamma(\xi)>0$.
Moreover, since $g$ is monotonically decreasing, we also obtain that
$\xi\mapsto\gamma(\xi)$ is monotonically decreasing.
Because of the remarks following Proposition~\ref{prop:fxpr} we will be particularly interested in the cases where $\xi=\theta_g$ and $\xi=\theta_g(g^-/g^+)^{1/2n}$.
Using \eqref{eq:gam}, \eqref{eq:corners_vconst} and \eqref{eq:vghill}
it is easy to see that $\gamma(\theta_g)=(\gamma_1+\gamma_2)/2$ and
$\gamma(\theta_g(g^-/g^+)^{1/2n})=\sqrt{\gamma_1\gamma_2}/(g^-/g^+)^{1/2n}$.
This shows that the $\gamma$ value corresponding to the extremum of $\xi g'(\xi)/g(\xi)$ converges to the geometric mean of $\gamma_1$ and $\gamma_2$ as $n\to\infty$, while the $\gamma$ value corresponding to $\xi=\theta_g$ is equal to the arithmetic mean of $\gamma_1$ and $\gamma_2$ independent of the value of $n$.

Next, for a given $\gamma>0$, it follows from the properties of the cotangent function that the equation
\eqref{eq:omega} gives a sequence of solutions $\{\omega_k\}_{k\geq0}$
with $\omega_k\tau \in (\pi/2+2k\pi, \pi+2k\pi)$ for $k=0,1,2,\ldots$, with each
$\omega_k$ uniquely defined as a function of $\gamma$.
Recall that $\tau$ is constant because we are considering $v^-=v^+=v$.  Thus
the values  $\omega_k \to \infty$ as $k \to \infty$.  It follows  that solutions  $\omega_k$ of the equation \eqref{eq:omega}
for fixed $\gamma$ must satisfy 
$\cot(\omega_k\tau)\to0$ and hence $\omega_k\tau\to\pi/2+2k\pi$ as $k\to\infty$.

Since $g$ is decreasing, the last equation  \eqref{eq:gdash} then becomes
\be \label{eq:gdash2}
g'(\xi) =-\frac{\sqrt{\gamma^2+\omega_k^2}}{\beta e^{-\mu\tau}}.
\ee
Notice here that through \eqref{eq:gam} we have $\gamma$ as a function of $\xi$, while from \eqref{eq:omega} we can regard each $\omega_k$ as an implicitly defined function of $\gamma$ and hence of $\xi$. Thus it remains only to solve
\eqref{eq:gdash2} for $\xi$, or more precisely,  we need to solve for a
$\xi_k$ for every $\omega_k$.

Combining \eqref{eq:realhopf} and \eqref{eq:gam} we see that at a Hopf bifurcation
\be \label{eq:xigdashghopf}
\frac{\xi g'(\xi)}{g(\xi)} \cos(\omega\tau)=1.
\ee
Thus, the function $\xi\mapsto\xi g'(\xi)/g(\xi)$ also plays a role for Hopf bifurcations.
Consequently, instead of solving \eqref{eq:gdash2} directly for $\xi$, we 
proceed by combining \eqref{eq:gam} and \eqref{eq:gdash2} which leads to
$$\beta e^{-\mu\tau}=\frac{\gamma\xi}{g(\xi)}=-\frac{\sqrt{\gamma^2+\omega_k^2}}{g'(\xi)},$$
and hence
\be \label{eq:xigdashg}
\frac{\xi g'(\xi)}{g(\xi)}=-\sqrt{1+(\omega_k/\gamma)^2}.
\ee
Combining \eqref{eq:xigdashg} and \eqref{eq:xigdashghill}, it remains to find $\xi$ that solves
\be \label{eq:xigdashgeq}
\frac{n(g^+-g^-)(\xi/\theta_g)^n}{(1+(\xi/\theta_g)^n)(g^-+g^+(\xi/\theta_g)^n)}
=-\sqrt{1+(\omega_k/\gamma)^2}.
\ee

We already considered the behaviour of the left-hand side of \eqref{eq:xigdashgeq} in Proposition~\ref{prop:fxpr}. The behaviour of the right-hand side is
considered in the following proposition.

\begin{prop} \label{prop:sxi}
Let $s_k(\xi)=\sqrt{1+(\omega_k/\gamma)^2}$ where $\gamma=\gamma(\xi)$ is defined by \eqref{eq:gam} and $\omega_k$ satisfies $\omega_k\tau\in(\pi/2+2k\pi,\pi+2k\pi)$ and is a function of $\gamma$ and hence of $\xi$ through \eqref{eq:omega}. Then
\begin{enumerate}
\item
$\left( 1+\left(\frac{\pi(2k+1/2)}{\beta\tau e^{-\mu\tau}g_U}\right)^2\xi^2 \right)^{\frac12}
=m_k(\xi) \leq s_k(\xi) \leq M_k(\xi)=
\left( 1+\left(\frac{\pi(2k+1)}{\beta\tau e^{-\mu\tau}g_0}\right)^2\xi^2 \right)^{\frac12}.$
\item
$1 < s_k(\xi) < 1+\frac{\pi(2k+1)}{\beta\tau e^{-\mu\tau}g_0}\xi$ for all $\xi>0$.
\end{enumerate}
\end{prop}

\begin{proof}
Using \eqref{eq:gam}
\begin{displaymath}
s_k(\xi) =
\left( 1+\left(\frac{\omega_k}{\beta e^{-\mu\tau}g(\xi)}\right)^2\xi^2 \right)^{\frac12}.
\end{displaymath}
Recalling \eqref{eq:gbounds} and also using the bounds on $\omega_k\tau$,
the inequalities in Proposition \ref{prop:sxi}(1), and the expressions for the bounds $m_k(\xi)$ and $M_k(\xi)$ follow easily,
while Proposition \ref{prop:sxi}(2)  is weaker than Proposition \ref{prop:sxi}(1).
\end{proof}

Recall that for a Hopf bifurcation, we need to solve \eqref{eq:xigdashg} for $\xi$, but this is the same as solving $f(\xi/\theta_g,n,g^-/g^+)=-s_k(\xi)$, where the relevant properties of $f$ and $s_k$ are  stated in Propositions~\ref{prop:fxpr} and~\ref{prop:sxi}. Propositions~\ref{prop:sxi}(1) defines bands $[m_k(\xi),M_k(\xi)]$
within which each $s_k(\xi)$ lies. Since $|s_k(\xi)|>1$ and 
$\lim_{\xi\to0}f(\xi/\theta_g,n,g^-/g^+)=\lim_{\xi\to\infty}f(\xi/\theta_g,n,g^-/g^+)=0$, for all $\xi$ sufficiently large or small we have $f(\xi/\theta_g,n,g^-/g^+)>-s_k(\xi)$. On the other hand, considering $\xi=\theta_g$, by Proposition~\ref{prop:fxpr}(2),
$f(1,n,g^-/g^+)\to-\infty$ as $n\to\infty$ while $s_k(\theta_g)\leq M_k(\theta_g)$ is bounded. Consequently, for $n$ sufficiently large $f(1,n,g^-/g^+)<-s_k(\xi)$. It follows that there are at least two points $\xi_k^- < \theta_g < \xi_k^+$ for which 
$f(\xi_k^{\pm}/\theta_g,n,g^-/g^+)=-s_k(\xi_k^{\pm})$, and hence which solve \eqref{eq:xigdashgeq}. With the corresponding values of
$\gamma$ and $\omega_k$ defined by \eqref{eq:gam} and \eqref{eq:omega} this defines two solutions of
\eqref{eq:gam}--\eqref{eq:gdash}.

We already noted that the steady state must be stable, and hence cannot undergo a Hopf bifurcation
for $\xi$ sufficiently small or large.
On the other hand, Hopf bifurcations must occur for $n$ sufficiently large, 
as
for large enough $n$ the function $f$ will pierce
through the band $s_k(\xi)\in[m_k(\xi),M_k(\xi)]$ for $\xi\approx1$.
In particular,
a sufficient (but not necessary) condition for this to occur is that
$M_k(\theta_g)<-f(1,n,g^-/g^+)$, or equivalently that
\be \label{eq:nsuffgdown}
n>2\frac{(g^-/g^+)+1}{(g^-/g^+)-1}
\left( 1+\left(\frac{\pi\theta_g(2k+1)}{\beta\tau e^{-\mu\tau}g_0}\right)^2 \right)^{\frac12}.
\ee
This condition follows from evaluation of the function $f$ at $x=1$. A
more complicated but tighter bound can be derived using the maximum of $f$.
Therefore another sufficient condition to ensure that the $k$-th Hopf bifurcation occurs
is that
\begin{displaymath}
M_k(\theta_g(g^-/g^+)^{1/(2n)})<-f((g^-/g^+)^{1/(2n)},n,g^-/g^+).
\end{displaymath}

While there is a unique steady state $\xi$ for the case of decreasing $g$ with constant $\tau$, the location and properties of this steady state will depend on the values of the other parameters. The following theorem collects together our results for this case.

\begin{thm} \label{thm:gdown}
Let $\xi$ be the steady state of the DDE \eqref{eq:basic},\eqref{eq:thres} with constant delay $\tau$
and nonlinearity $g$ defined by \eqref{eq:vghill} with $g$ monotonically decreasing (so $g^->g^+$). Then: 
\begin{enumerate}
\item
If $ |\frac{\xi g'(\xi)}{g(\xi)}|<1$, then the steady state  $\xi$ is asymptotically stable.
\item
If $n\leq1$, or $n>1$ and $r_g=\frac{g^-}{g^+}<\Bigl(1+\frac{2}{n-1}\Bigr)^2$, then the steady state
$\xi$ is asymptotically stable.
\item
For any fixed $n > 1$, and for $0<\xi\ll\theta_g$ or $\xi\gg\theta_g$, or equivalently for $\gamma\gg\gamma_2$ or
$0<\gamma\ll\gamma_1$, the steady state is asymptotically stable.
\item
For any fixed $\xi\ne\theta_g$, let $\gamma= \gamma(n,\xi)$ be the value of $\gamma$ such that \eqref{eq:h} is satisfied and hence $\xi$ is a steady state. 
Or, for any fixed
$\gamma$ with $0<\gamma<\gamma_1$ or $\gamma>\gamma_2$ let $\xi=\xi(n,\gamma)$
satisfy  \eqref{eq:h} and hence
be a steady state.
Then $\xi$ is asymptotically stable for all $n$ sufficiently large.
\item
Let \eqref{eq:nsuffgdown} be satisfied for fixed $n\geq2$ and fixed $k\geq0$. Then as $\gamma$ is varied
\begin{enumerate}
\item
There are two families of (at least) $k+1$ Hopf bifurcations. One exists for
$\gamma<\sqrt{\gamma_1\gamma_2}/(g^-/g^+)^{1/2n}$
and the other
for $\gamma>(\gamma_1+\gamma_2)/2$. In the first family, the characteristic values $\lambda_j=i\omega_j$ satisfy $\omega_j\tau\in(\pi/2+2j\pi,\pi+2j\pi)$ for $j=0,1,\ldots,k$ and
cross  the imaginary axis from left to right as $\gamma$ is increased, while in the second family they cross the imaginary axis from right to left.
\item
For $\xi\in[\theta_g,\theta_g(g^-/g^+)^{1/(2n)}]$ or equivalently for $\gamma\in[\sqrt{\gamma_1\gamma_2}/(g^-/g^+)^{1/2n},(\gamma_1+\gamma_2)/2]$
the steady state is unstable with at least $k+1$ pairs of complex conjugate characteristic values
$\lambda_j=\alpha_j\pm i\omega_j$ with $\alpha_j>0$ and $\omega_j\tau\in(\pi/2+2j\pi,\pi+2j\pi)$
for $j=0,1,\ldots,k$.
\end{enumerate}
\item
Let $\gamma\in(\gamma_1,\gamma_2)$ be fixed. Then as $n$ is increased there is an infinite
sequence of Hopf bifurcations where the real part of $\lambda_k=\alpha_k\pm i\omega_k$ becomes positive
with $\omega_k\tau\in(\pi/2+2k\pi,\pi+2k\pi)$ for $k=0,1,2,\ldots$.
\end{enumerate}
\end{thm}

\begin{proof} 
(1) and (2) were already shown immediately after the proof of Proposition~\ref{prop:fxpr}.

The first part of (3) then follows from Proposition~\ref{prop:fxpr}(1),
since $\lim_{\xi\to0}f(\xi/\theta_g,n,g^-/g^+)=\lim_{\xi\to\infty}f(\xi/\theta_g,n,g^-/g^+)=0$
implies $|\xi g'(\xi)/g(\xi)| < 1 < s_k(\xi)$
and $\xi g'(\xi)/g(\xi)=-s_k(\xi)$ cannot hold
for $\xi$ sufficiently
small ($0<\xi\ll\theta_g$) or large ($\xi\gg\theta_g$).
The second part follows
from \eqref{eq:gam} on noting that $g$ monotonically decreasing implies that $\gamma(\xi)$ is monotonically decreasing with
$\lim_{\xi\to0}\gamma(\xi)=+\infty$ and $\lim_{\xi\to\infty}\gamma(\xi)=0$.

For the first part of (4), consider a steady state at a fixed value of $\xi$ as $n$ is varied, with the other parameters fixed except for $\gamma=\gamma(n, \xi)$ which is determined by \eqref{eq:h}. The result then follows directly from the first part of 
Proposition~\ref{prop:fxpr}(3).
To prove the second part of (4) note that when $g^->g^+$ from \eqref{eq:vghill} we have
$g^->g(\xi)>g^+$ which using \eqref{eq:corners_vconst} is equivalent to
\begin{displaymath}
\gamma_2 > \frac{\beta e^{-\mu\tau}g(\xi)}{\theta_g} > \gamma_1.
\end{displaymath}
Fix $\gamma>0$ and
using \eqref{eq:gam}, we further rewrite this 
as
\be \label{eq:gambds}
\frac{\gamma_2}{\gamma} > \frac{\xi}{\theta_g} > \frac{\gamma_1}{\gamma}.
\ee
Consequently, for $\xi=\xi(n,\gamma)$ satisfying \eqref{eq:h},
if $\gamma<\gamma_1$ then
 we have
$\xi>\theta_g\gamma_1/\gamma>\theta_g$
while $\gamma>\gamma_2$ implies $\xi<\theta_g\gamma_2/\gamma<\theta_g$.
The result then follows from the second part of Proposition~\ref{prop:fxpr}(3).

For (5), equation \eqref{eq:nsuffgdown} implies that
\begin{displaymath}
-f(1,n,g^-/g^+) > M_k(\theta_g) > M_{k-1}(\theta_g) > \ldots > M_0(\theta_g).
\end{displaymath}
However, $s_j(\xi)\geq1$ for all $\xi\in\R$  and $\lim_{x\to0}f(x,n,g^-/g^+)=\lim_{x\to\infty}f(x,n,g^-/g^+)=0$.
Consequently for each $j\in\{0,1,\ldots,k\}$ there are at least 
two values of $\xi$ which solve
$-f(\xi/\theta_g,n,g^-/g^+)=s_j(\xi)$. The largest such $\xi$ with $\xi>\theta_g(g^-/g^+)^{1/(2n)}$
and the smallest with $\xi<\theta_g$
define the required Hopf bifurcations. Since, as already noted, $\gamma(\xi)$ defined by \eqref{eq:gam} is a monotonically decreasing function of $\xi$ with
$\gamma(\theta_g)=(\gamma_1+\gamma_2)/2$ and $\gamma(\theta_g(g^-/g^+)^{1/(2n)})=\sqrt{\gamma_1\gamma_2}/(g^-/g^+)^{1/2n}$,
the result follows.

\begin{figure}[thp!]
	\centering
	\vspacefig \includegraphics[scale=0.5]{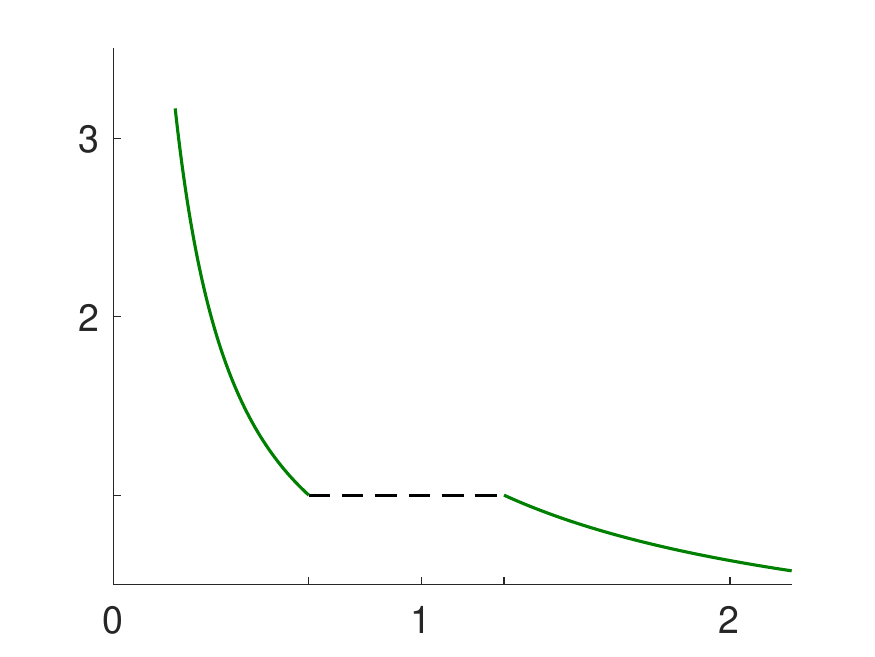}\hspace*{0.5em}\includegraphics[scale=0.5]{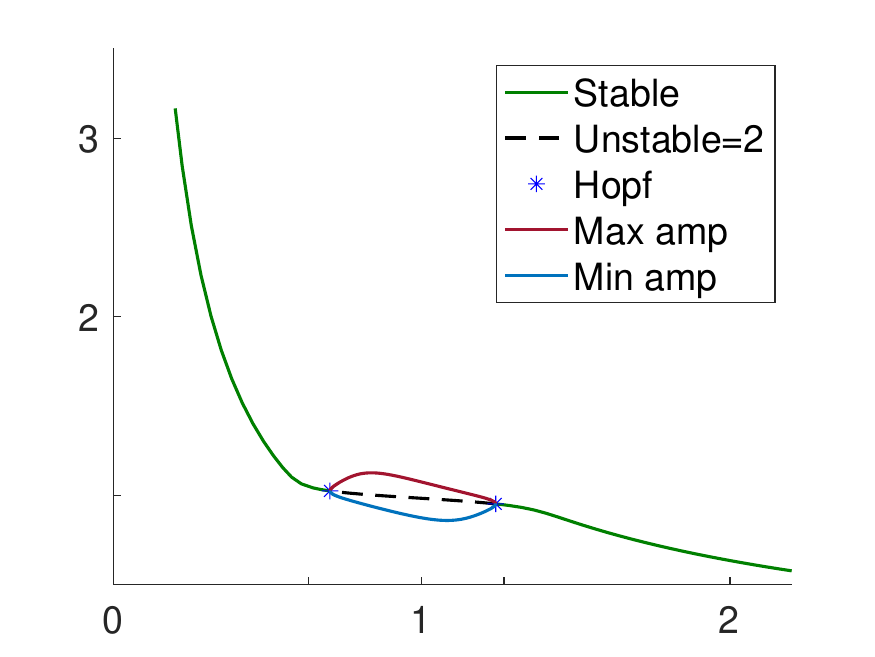}
	\put(-406,140){\rotatebox{90}{$x$}}
	\put(-412,38){$\theta_g$}
	\put(-240,10){$\gamma$}
	\put(-357,7){$\gamma_1$}
	\put(-310,7){$\gamma_2$}
	\put(-333,135){$(a)$}
	\put(-135,135){$(b)$}
	\put(-190,140){\rotatebox{90}{$x$}}
	\put(-196,38){$\theta_g$}
	\put(-24,10){$\gamma$}
	\put(-140,7){$\gamma_1$}
	\put(-93,7){$\gamma_2$}\\
	\vspace*{-0.5em}
	\includegraphics[scale=0.5]{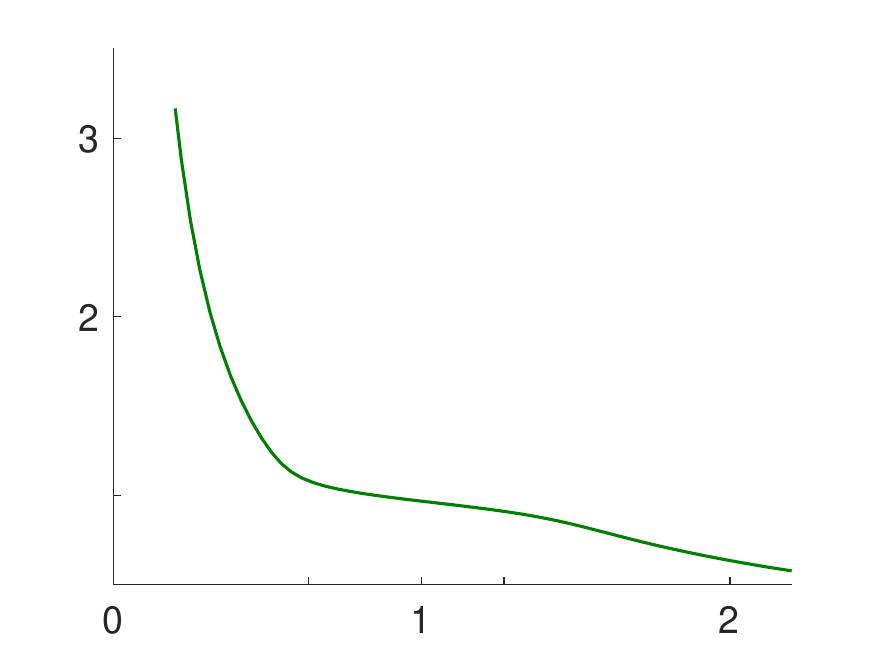}\hspace*{0.5em}\includegraphics[scale=0.5]{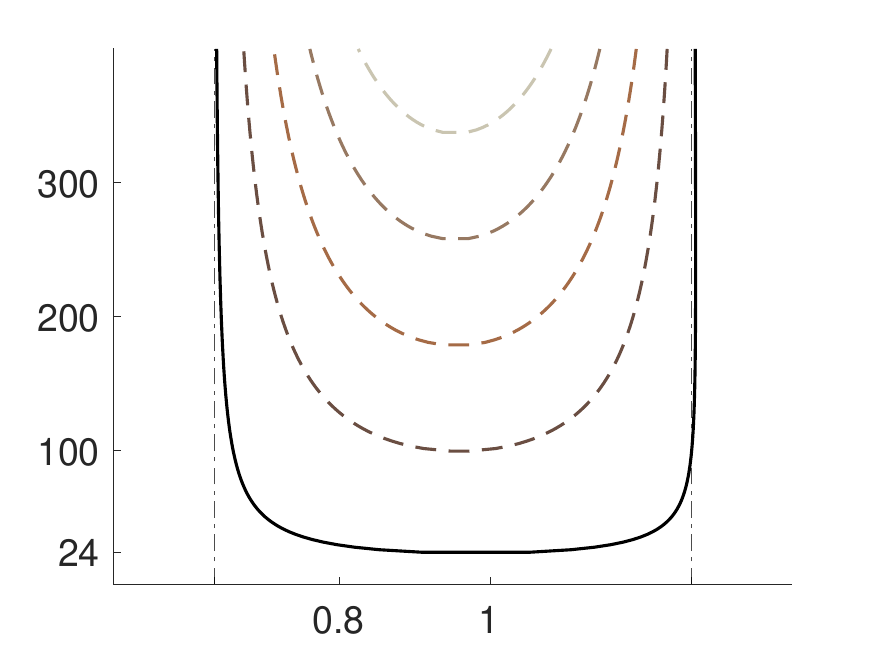}
	\put(-406,140){\rotatebox{90}{$x$}}
	\put(-412,38){$\theta_g$}
	\put(-240,10){$\gamma$}
	\put(-357,7){$\gamma_1$}
	\put(-310,7){$\gamma_2$}
	\put(-333,135){$(c)$}
	\put(-108,137){$(d)$}
	\put(-190,140){\rotatebox{90}{$n$}}
	\put(-24,10){$\gamma$}
	\put(-163,7){$\gamma_1$}
	\put(-48,7){$\gamma_2$}\\
	\vspace*{-0.5em}
	\includegraphics[scale=0.5]{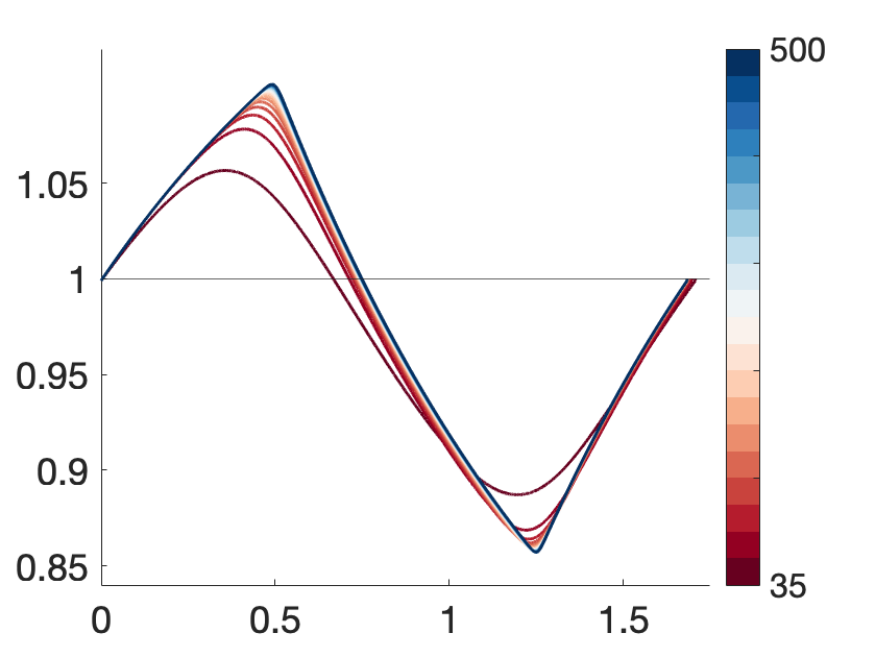}\hspace*{0.5em}\includegraphics[scale=0.5]{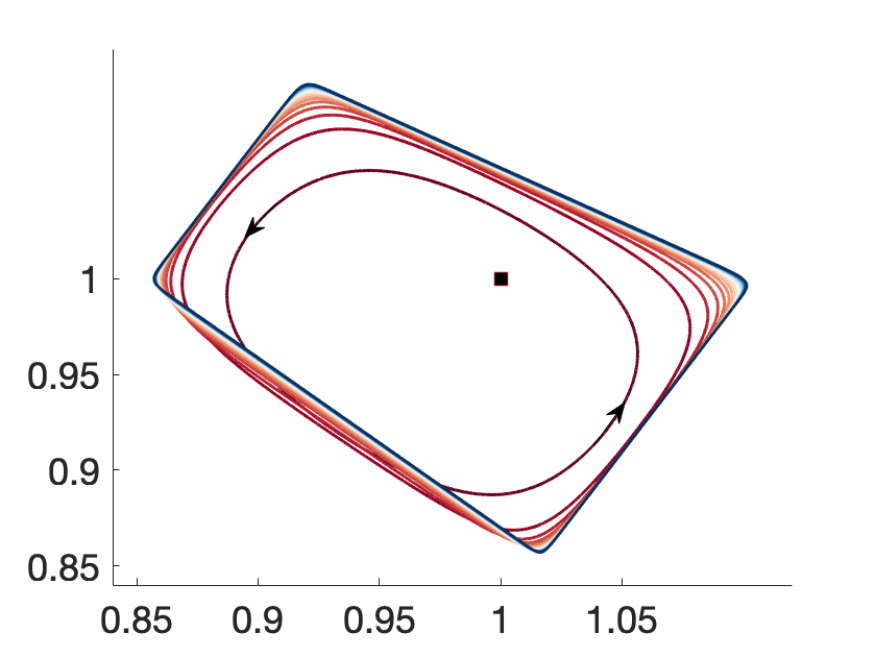}
		\put(-408,140){\rotatebox{90}{$x$}}
		\put(-260,10){$t$}
	    \put(-242,80){$n$}
		\put(-323,135){$(e)$}
		\put(-170,135){$(f)$}
		\put(-195,110){\rotatebox{90}{$x(t-\tau)$}}
		\put(-34,8){$x(t)$}\\
	\vspace*{-0.5em}

\caption{Bifurcations of \eqref{eq:basic}-\eqref{eq:thres}
for $(g\downarrow,v\leftrightarrow)$
with parameters $\beta=1.4$, $\mu=0.2$, $g^-=1$, $g^+=1/2$, $\theta_g=1$, $\gamma=a=1$, and $v=v^-=v^+=2$.
(a) The limiting case with $g$ defined by \eqref{eq:gpwconst} showing the stable (green solid line) and singular (black dashed line) steady states.
(b) With smooth nonlinearity $g$ defined by \eqref{eq:vghill} and $n=50$.
Solid lines denote stable objects including the stable steady state (green) and a stable limit cycle (represented by maximum and minimum of $x(t)$ on the periodic solution). Dashed lines represent unstable steady states which have two eigenvalues with positive real part (in black).
(c) As in (b) but with $n=23$.
(d) Two-parameter continuations in $n$ and $\gamma$ of the Hopf bifurcations defined by
\eqref{eq:gam}-\eqref{eq:gdash} with the other parameters as above.
Solid curves indicate the parts of the branch where there are no characteristic values with positive real part (and hence a stability change at the bifurcation), and dashed lines indicate the parts of the branch where there are already unstable characteristic values.
The outermost curve of Hopf bifurcations is associated with the stability change seen in (b). The dash-dotted vertical black lines denote $\gamma=\gamma_1$ and $\gamma=\gamma_2$, the locations of the Hopf bifurcations in the limiting case as $n\to\infty$. (e) Profiles of the stable periodic orbits from the outermost curve of Hopf bifurcations in (d) at $\gamma=1$ for different values of the continuation parameter $n$.
(f) The same periodic orbits as in (e), but now shown as a projection onto the plane $(x(t),  x(t-\tau))$ where $\tau=0.5$. The arrow indicates the direction of the flow.  The square denotes the singular steady state in the limiting case.
}
\label{fig:gdown_ex1}
\end{figure}
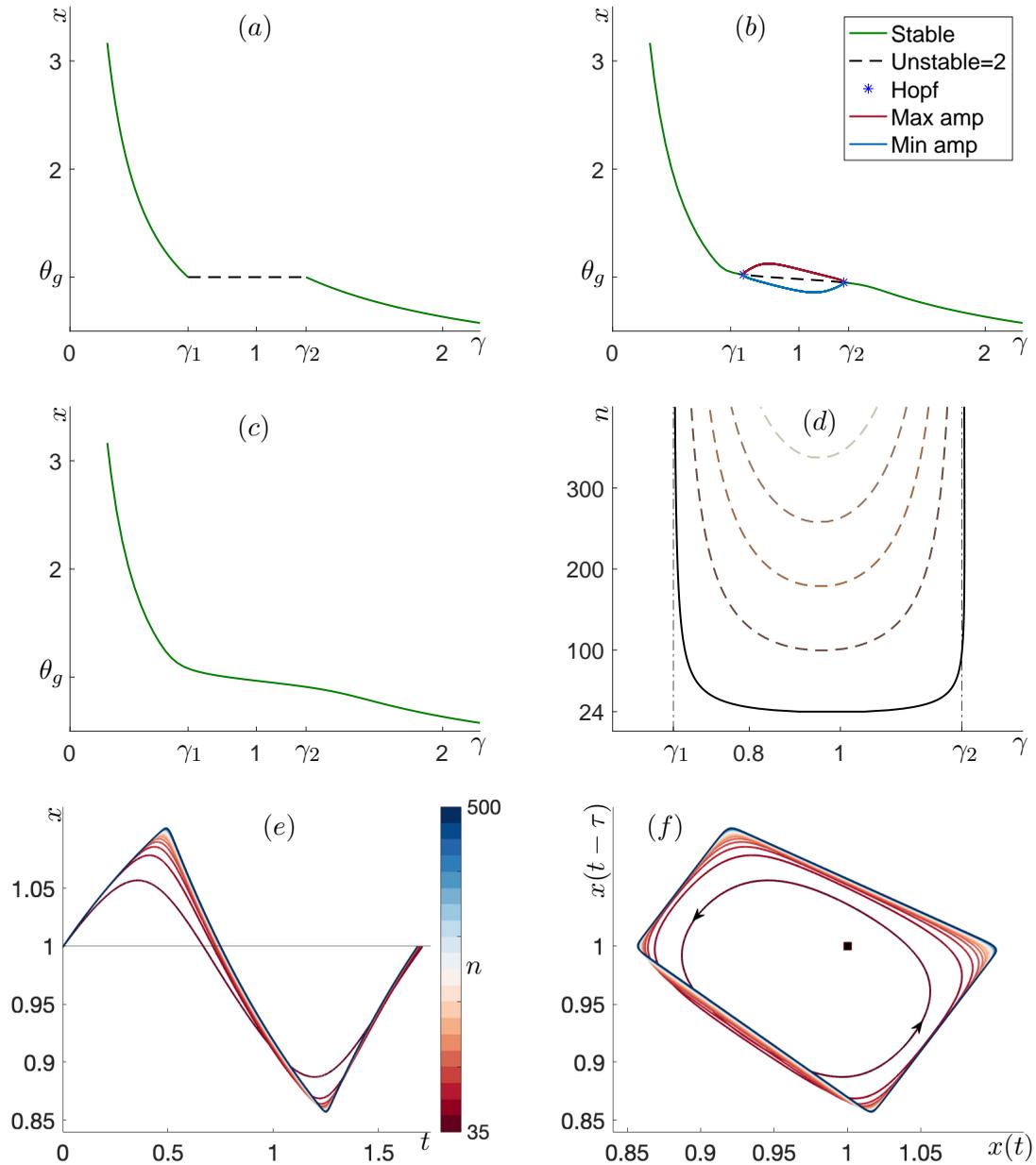

For (6) we consider the behaviour of the Hopf bifurcation points as $n\to\infty$.
Applying Proposition~\ref{prop:fxpr} (and recalling that $g_->g_+$ because $g$ is decreasing),
the function $|\xi g'(\xi)/g(\xi)|$ takes its maximum
at $\xi=\theta_g(g^-/g^+)^{1/2n}>\theta_g$,
while also $|\theta_g g'(\theta_g)/g(\theta_g)|=|f(1,n,g^-/g^+)|\gg1$. Consequently for fixed $k$ the two Hopf bifurcation points, $\xi_k^-$ and $\xi_k^+$ satisfy  $\xi_k^-<\theta_g<\theta_g(g^-/g^+)^{1/2n}<\xi_k^+$.
Using Proposition~\ref{prop:fxpr}(3) we conclude that $\xi_k^\pm\to\theta_g$ as $n\to\infty$.

Also for fixed $k$, as $n\to\infty$, the value of $\omega_k\leq(2k+1)\pi/\tau$ remains bounded,
as does $\gamma$ by (4).
Consequently for \eqref{eq:gdash2} to be satisfied it follows that $g'(\xi_k)$
must also remain bounded as $n\to \infty$.
Then in the limit as $n\to\infty$ the Hopf bifurcations
must converge to the ``corners'' of $g(\xi)$ where $\xi=\theta_g$ and $\gamma=\gamma_1$ or $\gamma=\gamma_2$.
Consequently
for any fixed $\gamma\in(\gamma_1,\gamma_2)$ as $n\to\infty$ there is an infinite sequence of Hopf bifurcations.
\end{proof}

Figure~\ref{fig:gdown_ex1} illustrates the behaviour of \eqref{eq:basic}-\eqref{eq:thres}
for $(g\downarrow,v\leftrightarrow)$.
Panels (b)-(f) were computed numerically using \texttt{ddebiftool} as described in  Appendix~\ref{app-numerics}.

Figure~\ref{fig:gdown_ex1}(a) and (b)  show the
similarities between the dynamics with the piecewise continuous nonlinearity \eqref{eq:gpwconst} and the
smooth Hill function $g$ defined in \eqref{eq:vghill} with $n=50$. The singular steady state in (a) becomes an unstable steady state in the smooth case, with a bubble of stable periodic orbits existing between the pair of Hopf bifurcations where the steady state changes stability. The stable periodic orbits at $\gamma=1$
for increasing values of $n$ are shown as profiles in Figure~\ref{fig:gdown_ex1}(e) and projected onto the
$(x(t),x(t-\tau))$-plane in Figure~\ref{fig:gdown_ex1}(f).
The apparent limiting behaviour that is revealed
is the topic of \cite{sausage}.

When the value of $n$ is decreased, the interval of $\gamma$ values between the Hopf bifurcations shrinks, until for $n$ sufficiently small the steady state is always stable, as seen in Figure~\ref{fig:gdown_ex1}(c). Interestingly, even though there is no bifurcation in this case, the graph in Figure~\ref{fig:gdown_ex1}(c) still has a plateau around where the singular steady states exist in
Figure~\ref{fig:gdown_ex1}(a).

Figure~\ref{fig:gdown_ex1}(d) shows two-parameter continuations in the $(\gamma,n)$ plane of the Hopf bifurcations. This reveals the Hopf bifurcations associated with $\omega_k$ for $k=0,1,\ldots,4$, with each successive Hopf bifurcation only existing for progressively larger values of $n$, as implied
by \eqref{eq:nsuffgdown}.
In particular there is no Hopf bifurcation for $n<24$ and a second Hopf bifurcation is only seen if $n>100$. This is why no Hopf bifurcation is seen in Figure~\ref{fig:gdown_ex1}(c) with $n=23$, and only one pair of Hopf bifurcations is seen in Figure~\ref{fig:gdown_ex1}(b) with $n=50$.

Figure~\ref{fig:gdown_ex1}(d) also illustrates Theorem~\ref{thm:gdown} (points 5 and 6)
where the additional Hopf bifurcations occurring as $n$ increases approach the vertical asymptotes
$\gamma=\gamma_1$ and $\gamma=\gamma_2$ in the limit as $n\to\infty$. Notice also the existence of Hopf bifurcations with $\gamma>\gamma_2$ in Figure~\ref{fig:gdown_ex1}(d); so it \emph{is} possible for the steady state to be unstable outside the interval $\gamma\in[\gamma_1,\gamma_2]$, albeit only for a finite range of values $n$ by Theorem~\ref{thm:gdown} (point 3).

\subsection{Constant delay with $g$ increasing $(g\uparrow,v\leftrightarrow)$}
\label{sec:gupvconst}

In this section we study \eqref{eq:basic}-\eqref{eq:thres} with constant delay again, but in contrast to the previous section we assume that $g$ is increasing.
We thus assume $g^-<g^+$ in \eqref{eq:vghill} and \eqref{eq:gpwconst}, with
$v^-=v^+=v^\pm$ so the delay $\tau=a/v^\pm$ is constant, independent of $\xi$.

When $g$ is increasing, it is possible for multiple steady states to coexist. For example, considering the limiting case where $g$ is given by \eqref{eq:gpwconst}, as shown in Figure~\ref{fig:2}, there are up to three coexisting steady states, including a singular steady state, as the slope of the line $\gamma\xi$ changes. The corners defined by \eqref{eq:corners_vconst} give rise to fold bifurcations due to a change in the number of steady states. Since these bifurcations involve a singular steady state they are not truly fold bifurcations but we use this term since, as we show below, they reflect the presence of true fold bifurcations for the smooth nonlinearity $g$ defined by \eqref{eq:vghill} with $g'(\xi) \gg 0$.
This is illustrated in Figure~\ref{fig:gup_ex1}(a) where there are two fold bifurcations between stable and singular steady states,
with the outer steady states stable, and the middle steady state is singular.

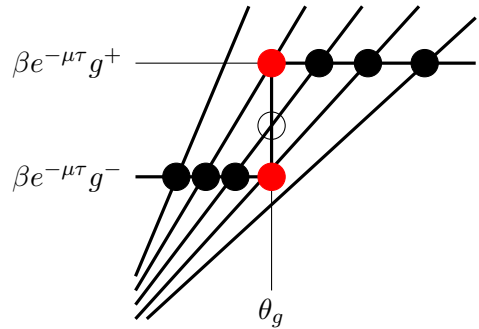
\begin{figure}[tp!]
\centering
\begin{tikzpicture}[scale = 1.5]
		\tikzstyle{line} = [-,very thick]
		\tikzstyle{dotted line}=[.]
		\tikzstyle{arrow} = [->,line width = .4mm]
		\tikzstyle{unstable} = [red]
		\tikzstyle{stable} = [blue]
		\tikzstyle{pt} = [circle,draw=black,fill = black,minimum size = 1pt];
		\tikzstyle{bifpt} = [circle,draw=red,fill = red,minimum size = 1pt];
		\tikzstyle{upt} = [circle,draw=black,minimum size = 1pt];
		\def\arrlen{.3}

		\draw[line] (0,1) to (1.2,1);
		\draw[line] (1.2,1) to (1.2,2);
		\draw[line] (1.2,2) to (3,2);
		\draw[dotted line] (0,2) to (1.2,2);
		\draw[dotted line] (1.2,0) to (1.2,2);
		
		\draw[line] (0,0.125) to (1,2.5);
		\draw[line] (0,0) to (1.5,2.5);
		\draw[line] (0,-0.125) to (2,2.5);
		\draw[line] (0,-0.25) to (2.5,2.5);
		\draw[line] (0.1,-0.25) to (3,2.4);
		
		\node at (1.2,-.2) {$\theta_g$};

		\node at (-.6,2) {$\beta e^{-\mu\tau}g^+$};
		\node at (-.6,1) {$\beta e^{-\mu\tau}g^-$};
		
		\node[bifpt] at (1.2,1) (8) [] {};
		\node[bifpt] at (1.2,2) (8) [] {};
		\node[upt] at (1.2,1.45) (8) [] {};
		\node[pt] at (1.62,2) (8) [] {};
		\node[pt] at (2.05,2) (8) [] {};
		\node[pt] at (2.55,2) (8) [] {};
		\node[pt] at (0.88,1) (12) [] {};
		\node[pt] at (0.62,1) (12) [] {};
		\node[pt] at (0.36,1) (12) [] {};		
\end{tikzpicture}
\caption{Illustration of how the number of steady states of \eqref{eq:basic} given by \eqref{eq:h} changes with the intersections of $\xi\mapsto\beta e^{-\mu\tau}g(\xi)$ and $\xi\mapsto\gamma\xi$. These are shown in the limiting case with $v^\pm =v^-=v^+$ so $\tau=a/v^\pm$ is constant, and
$g^-<g^+$ in \eqref{eq:gpwconst} so $g$ is piecewise constant and monotonically increasing:  $(g\uparrow,v\leftrightarrow)$.}
\label{fig:2}
\end{figure}

\begin{figure}[tp!]
\centering	\includegraphics[scale=0.5]{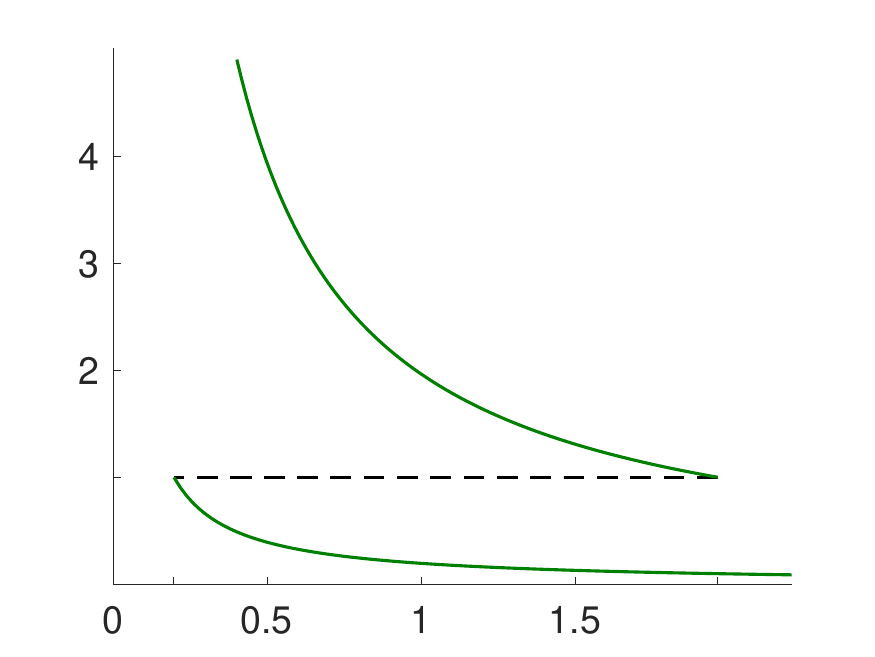}\hspace*{0.5em}\includegraphics[scale=0.5]{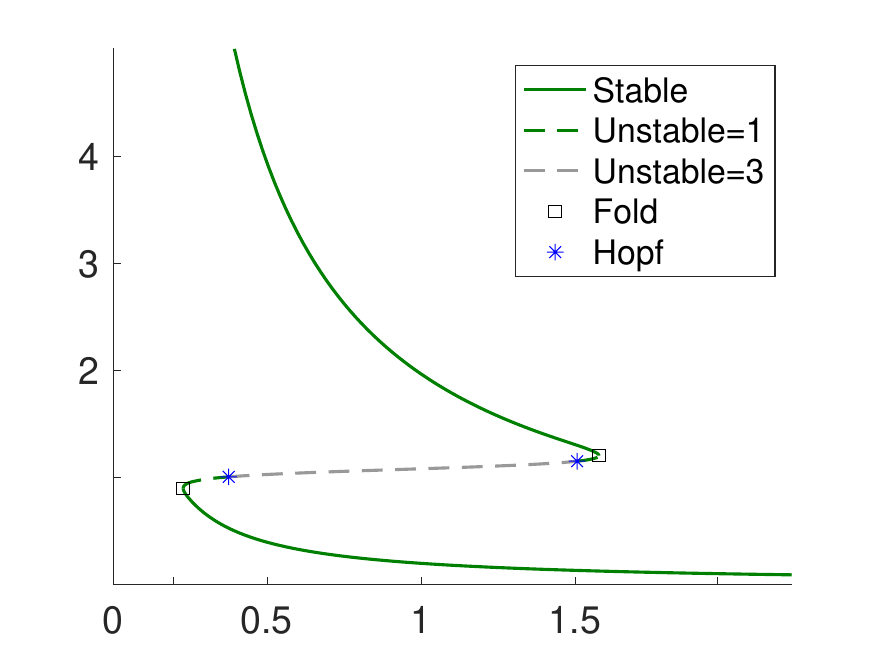}
	\put(-406,140){\rotatebox{90}{$x$}}
	\put(-412,40){$\theta_g$}
	\put(-240,10){$\gamma$}
	\put(-390,7){$\gamma_2$}
	\put(-258,7){$\gamma_1$}
	\put(-322,140){$(a)$}
	\put(-120,140){$(b)$}
	\put(-190,140){\rotatebox{90}{$x$}}
	\put(-196,40){$\theta_g$}
	\put(-24,10){$\gamma$}
	\put(-174,7){$\gamma_2$}
	\put(-42,7){$\gamma_1$}\\

	\includegraphics[scale=0.5]{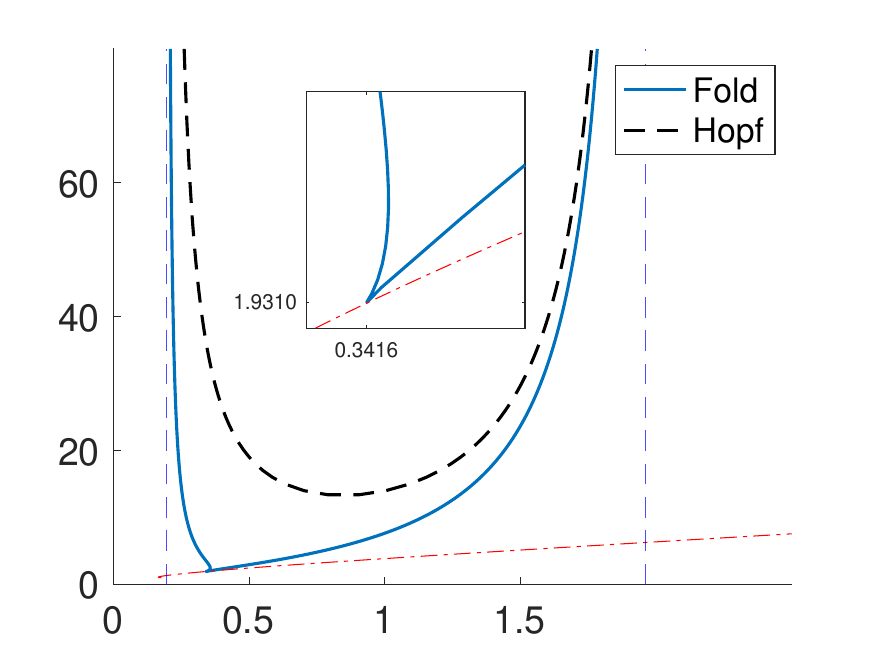}
	\put(-120,140){$(c)$}
	\put(-190,140){\rotatebox{90}{$n$}}
	\put(-24,10){$\gamma$}
	\put(-175,7){$\gamma_2$}
	\put(-59,7){$\gamma_1$}
\caption{Bifurcations of \eqref{eq:basic}-\eqref{eq:thres} with $(g \uparrow, v \leftrightarrow)$ and parameters
$\beta=2$, $\mu=0.02$, $g^-=0.1$, $g^+=1$, $\theta_g=1$, $\gamma=1$, $a=2$ and $v=v^-=v^+=2$.
(a) The limiting case with $g$ defined by \eqref{eq:gpwconst}. Stable steady states are shown as green solid lines, and the singular steady state as a black dashed line.
(b) With a smooth nonlinearity $g$ defined by \eqref{eq:vghill} with $n=30$.
(c) Two-parameter continuations in $n$ and $\gamma$ of the fold (blue) and the Hopf (black) bifurcations
with the other parameters as above.
The dashed vertical lines denote $\gamma=\gamma_1$ and $\gamma=\gamma_2$, the location of the fold bifurcations in the limiting case as $n\to\infty$.
The red dash-dotted curve denotes the bound on the fold bifurcations given by \eqref{eq:foldbound}.}
\label{fig:gup_ex1}
\end{figure}
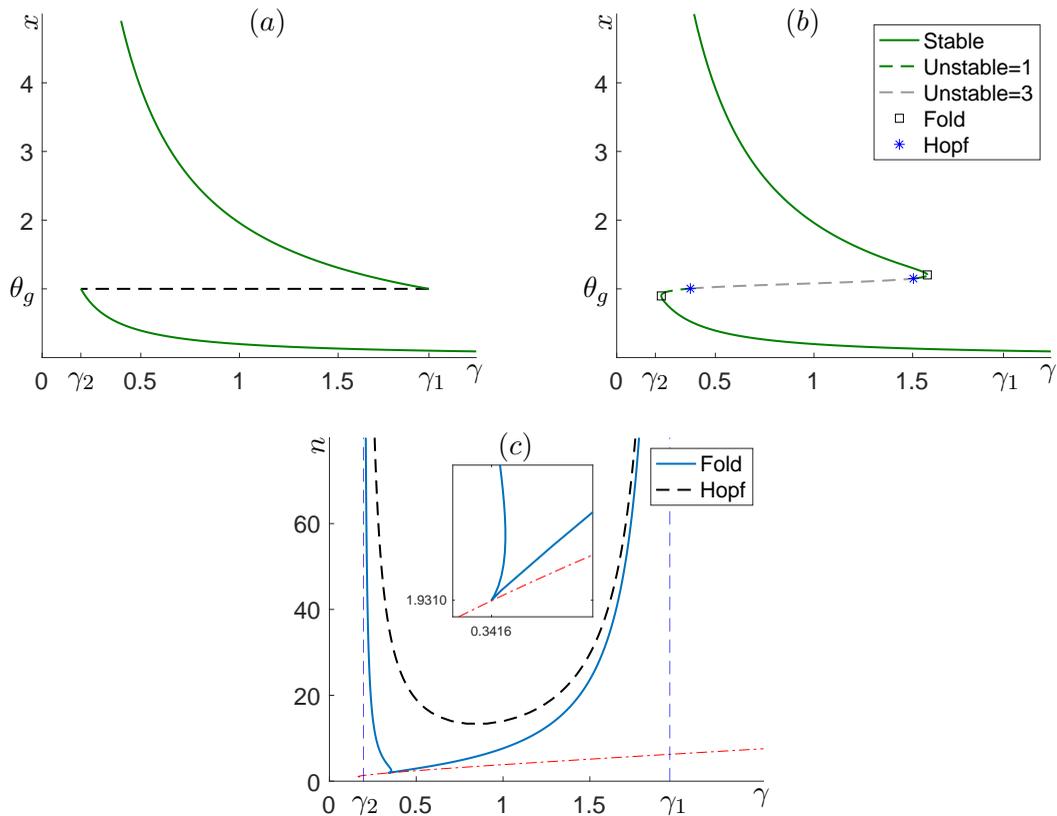

With $g$ defined by the piecewise constant function \eqref{eq:gpwconst}, as in Section~\ref{sec:gdownvconst}, for the non-singular steady states we have $A=g'(\xi)=0$. The characteristic equation is again of the form \eqref{eq:char_A0g'0}, so these steady states must be stable. Only the singular steady state may be unstable, and since it coexists with two stable steady states, it is natural to regard it as unstable even though the dynamical system is infinite dimensional and the characteristic equation is not defined at the singular steady state. The bifurcation diagram for this case is illustrated in Figure~\ref{fig:gup_ex1}(a).

For the smooth nonlinearity $g(\xi)$  defined by \eqref{eq:vghill},
from \eqref{eq:char_A0g'} the characteristic function is given by
\be \label{eq:charfgup}
\Delta(\lambda)=\lambda+\gamma-\beta e^{-\mu\tau}g'(\xi)e^{-\lambda\tau}.
\ee
Considering $\lambda\in\R$ first, note that $\Delta(\lambda)$ is asymptotic to $\lambda+\gamma$ for $\lambda\gg0$.
On the other hand, using \eqref{eq:h} we find
\be \label{eq:char0gup}
\Delta(0)=\gamma\left(1-\frac{\xi g'(\xi)}{g(\xi)}\right).
\ee
Consequently when
\be \label{eq:xigdashgunstab}
\frac{\xi g'(\xi)}{g(\xi)}>1
\ee
there is always a real characteristic value $\lambda>0$ and the steady state is unstable.
Moreover there is a characteristic value $\lambda=0$ if and only if
\be \label{eq:xigdashgfold}
\frac{\xi g'(\xi)}{g(\xi)}=1,
\ee
and the results developed in Section~\ref{sec:gdownvconst},
and in particular Proposition~\ref{prop:fxpr},
can be applied to the fold bifurcations.
Thus for $n\gg0$ there will be two fold bifurcations $\xi^\pm\approx\theta_g$: one with
$\gamma\approx\gamma_1$ and the other with $\gamma\approx\gamma_2$, where $\gamma_1$ and $\gamma_2$
are defined by \eqref{eq:corners_vconst} with $\gamma_2<\gamma_1$  since $g$ is increasing.

To find the fold bifurcations, equations
\eqref{eq:xigdashgfold} and \eqref{eq:gam} must be solved together.
We first do this numerically. 
Figure~\ref{fig:gup_ex1}(b) shows the resulting bifurcation diagram for the
smooth nonlinearity $g(\xi)$  defined by \eqref{eq:vghill} with $n=30$, revealing, as expected,
two smooth
fold bifurcations, with an intermediate branch of unstable steady states
between the stable steady states.

To further investigate when multiple steady states arise for smoothly increasing $g$,
note that a necessary condition for multiple coexisting steady states is
that  $\max_{\xi\geq0}\{h'(\xi)\} \ge 0$. This imposes a constraint on the parameters as follows. The function $g$ has a single point of inflection $\bar{x}$ with $g''(\bar{x}) = 0$, which can be computed by differentiating $g$ twice
to find
$$\bar{x} = \theta_g\left(\frac{n-1}{n+1}\right)^{1/n}.$$
Note that $\bar{x} > 0$ requires the restriction $n > 1$.
With $\bar{x}$, we can compute the maximal value of the derivative
$$
\max_{\xi\geq0}\{h'(\xi)\} = h'(\bar{x}) = \beta e^{-\mu\tau} g'(\bar{x}) - \gamma \\
	= \beta e^{-\mu\tau}(g^{+}-g^{-})\frac{(n+1)^{1+1/n}(1-n)^{1-1/n}}{4n\theta_g} -\gamma.
$$
Therefore, if there are multiple steady states, the parameters must satisfy
\begin{equation}
	\gamma \leq \beta e^{-\mu\tau}(g^{+}-g^{-})\frac{(n+1)^{1+1/n}(n-1)^{1-1/n}}{4n\theta_g}.
	\label{eq:foldbound}
\end{equation}

Figure~\ref{fig:gup_ex1}(c) shows a two-parameter continuation of the fold bifurcations for increasing $g$. For large $n$ the folds approach the asymptotes $\gamma=\gamma_1$ and $\gamma=\gamma_2$ defined by the limiting case.
As $n$ decreases the two fold bifurcations move closer together until they collide in a cusp bifurcation.
At the cusp, the three steady states coincide, and at this point $h(\xi)$ is a monotonically decreasing
function of $\xi$ with a zero of multiplicity three at the steady state $\bar{\xi}$. Consequently not only is
$h(\bar{\xi})=h'(\bar{\xi})=0$ at this point, but also the function $h'(\xi)$ attains its maximum at $\bar{\xi}$ and
the cusp point lies on the bounding curve defined by \eqref{eq:foldbound}, as seen in
Figure~\ref{fig:gup_ex1}(c).

The analysis of Hopf bifurcations from Section~\ref{sec:gdownvconst} can be repeated with only minor changes for
the case of increasing $g$. Characteristic values again satisfy
\eqref{eq:real} and \eqref{eq:imag}, and at a Hopf bifurcation
\eqref{eq:realhopf} and \eqref{eq:imaghopf}. Since $g'(\xi)>0$, we require
$\sin(\omega\tau) < 0<\cos(\omega\tau)$
and hence
$$\omega\tau \in (2k+3/2)\pi, (2k+2)\pi), \quad k \in \mathbb{N}.$$
Equations \eqref{eq:gam}, \eqref{eq:omega} and \eqref{eq:gdash}
can again be
considered sequentially, and reduced to a single equation to solve for $\xi$,
where for increasing $g$, equation \eqref{eq:gdash} becomes
\be \label{eq:gdashinc}
g'(\xi) =\frac{\sqrt{\gamma^2+\omega_k^2}}{\beta e^{-\mu\tau}}.
\ee
This last equation is most easily considered by combining it with \eqref{eq:gam}
to obtain
\be \label{eq:xigdashgup}
\frac{\xi g'(\xi)}{g(\xi)}=\sqrt{1+(\omega_k/\gamma)^2},
\ee
where $\xi g'(\xi)/g(\xi)$ is still given by \eqref{eq:xigdashghill} and Propositions~\ref{prop:fxpr} and~\ref{prop:sxi} both apply.

We conclude that as $\gamma$ is varied there is a sequence of Hopf bifurcations parameterized by a frequency $\omega_k \to \infty$ satisfying $\omega_k\tau \in ((2k+3/2)\pi,(2k+2)\pi)$ for $k\in\mathbb{N}$, where for each fixed $k$ there is a minimal $n$ at which this bifurcation exists. Furthermore, this minimal $n$ grows with $k$.
Additionally,  for any fixed $k$ as $n\to\infty$ the bifurcation points satisfy $\gamma\to\gamma_1$ and $\gamma\to\gamma_2$.
Figure~\ref{fig:gup_ex1}(b) shows that the steady states lose stability at the fold bifurcation and the resulting unstable steady state undergoes a Hopf bifurcation. This indicates that the resulting periodic orbits are unstable
and thus not consequential for the asymptotic dynamics in contrast to the case of decreasing $g$. The two-parameter continuation of the first Hopf bifurcation is shown in Figure~\ref{fig:gup_ex1}(c).

We collect our results from this section together in the following theorem. While there are many similarities between Theorems~\ref{thm:gdown} and~\ref{thm:gup}, the presence of fold bifurcations introduces some important differences. Note also that since $g$ is increasing, in this section we have
$r_g=g^-/g^+<1$ and also $\gamma_1>\gamma_2$.

\begin{thm} \label{thm:gup}
Let $\xi$ be a steady state of the DDE \eqref{eq:basic},\eqref{eq:thres} with constant delay $\tau$
and nonlinearity $g$ defined by \eqref{eq:vghill} with $g$ monotonically increasing (so $g^-<g^+$). Then
\begin{enumerate}
\item
The steady state $\xi$ is asymptotically stable if $\frac{\xi g'(\xi)}{g(\xi)}<1$, and unstable if
$\frac{\xi g'(\xi)}{g(\xi)}>1$.
\item
For any fixed $\xi\ne\theta_g$ let $\gamma= \gamma(n,\xi)$ be the value of $\gamma$ such that \eqref{eq:h} is satisfied and hence $\xi$ is a steady state.
Or, for any fixed
$\gamma$ with $0<\gamma<\gamma_2$ or $\gamma>\gamma_1$ let $\xi=\xi(n,\gamma)$ satisfy  \eqref{eq:h} and hence
be a steady state.
Then $\xi$ is asymptotically stable for all $n$ sufficiently large.
\item
If $n\leq1$, or $n>1$ and $r_g=\frac{g^-}{g^+}>\Bigl(1-\frac{2}{n+1}\Bigr)^2$,  then the steady state
$\xi$ is asymptotically stable.
\item
If $n>1$ and $r=\frac{g^-}{g^+}<\Bigl(1-\frac{2}{n+1}\Bigr)^2$, then
there exists $\xi^-<\theta_g(g^-/g^+)^{1/2n}<\xi^+$ and $\gamma(\xi^-)<\sqrt{\gamma_1\gamma_2}/(g^-/g^+)^{1/2n}<\gamma(\xi^+)$
such that as $\gamma$ is varied, there is one branch of stable steady states
with $\xi<\xi^-$ and $\gamma>\gamma(\xi^-)$ and another stable branch
with $\xi>\xi^+$ and $\gamma<\gamma(\xi^+)$. For $\gamma\in(\gamma(\xi^-),\gamma(\xi^+))$ the two stable branches of steady states co-exist with a branch of unstable steady states which exists between fold bifurcations at $(\xi,\gamma)=(\xi^-,\gamma(\xi^-))$ and $(\xi,\gamma)=(\xi^+,\gamma(\xi^+))$.
\item
Let fixed $n\geq2$ and fixed $k\geq0$ satisfy
\be \label{eq:nsuffgup}
n>2\frac{1+(g^-/g^+)}{1-(g^-/g^+)}
\left( 1+\left(\frac{2\pi\theta_g(k+1)}{\beta\tau e^{-\mu\tau}g_0}\right)^2 \right)^{\frac12}.
\ee
Then as $\gamma$ is varied
\begin{enumerate}
\item
There are two families of (at least) $k+1$ Hopf bifurcations
from the unstable steady-state.
One exists for $\gamma<\sqrt{\gamma_1\gamma_2}/(g^-/g^+)^{1/2n}$ and the other
for $\gamma>(\gamma_1+\gamma_2)/2$. In the first family, the characteristic values
$\lambda_j=\alpha_j\pm i\omega_j$ with $\omega_j\tau\in(3\pi/2+2j\pi,2\pi+2j\pi)$
for $j=0,1,\ldots,k$
cross the imaginary axis from left to right as $\gamma$ is increased, while in the second family they cross the imaginary axis from right to left.
\item
For $\xi\in[\theta_g(g^-/g^+)^{1/(2n)},\theta_g]$, or equivalently for $\gamma\in(\sqrt{\gamma_1\gamma_2}/(g^-/g^+)^{1/2n},(\gamma_1+\gamma_2)/2)$,
the unstable steady state has one positive real characteristic value and at least $k+1$ pairs of complex conjugate characteristic values
$\lambda_j=\alpha_j\pm i\omega_j$ with $\alpha_j>0$ and $\omega_j\tau\in(3\pi/2+2j\pi,2\pi+2j\pi)$
for $j=0,1,\ldots,k$.
\end{enumerate}
\item
Let $\gamma\in(\gamma_2,\gamma_1)$ be fixed. Then as $n$ is increased there is an infinite
sequence of Hopf bifurcations on the branch of unstable equilibria where the real part of $\lambda_k=\alpha_k\pm i\omega_k$ becomes positive
with $\omega_k\tau\in(3\pi/2+2k\pi,2\pi+2k\pi)$.
\end{enumerate}
\end{thm}

\begin{proof}
The first part of (1) follows from \eqref{eq:xigdashgone}, while the second part follows from
\eqref{eq:charfgup}--\eqref{eq:xigdashgunstab}.

The proof of the first part of (2) is identical to the proof of the first part of (4) of Theorem~\ref{thm:gdown}.
To show the second part of (2),
note that when $g^+>g^-$ from \eqref{eq:vghill} we have $g^+>g(\xi)>g^-$ which, using \eqref{eq:corners_vconst},
is equivalent
to
\begin{displaymath}
 \gamma_1 > \frac{\beta e^{-\mu\tau}g(\xi)}{\theta_g} > \gamma_2.
\end{displaymath}
Fix $\gamma>0$ and using \eqref{eq:gam}, we further rewrite this as
\begin{displaymath}
 \frac{\gamma_1}{\gamma} > \frac{\xi}{\theta_g} > \frac{\gamma_2}{\gamma}.
\end{displaymath}
Hence,
for $\xi=\xi(n,\gamma)$ satisfying \eqref{eq:h}
if $\gamma>\gamma_1$ then
$\xi<\theta_g\gamma_1/\gamma<\theta_g$.
On the other hand,  if
$\gamma<\gamma_2$ then $\xi>\theta_g\gamma_2/\gamma>\theta_g$.
The result follows from (1) and the second part of Proposition~\ref{prop:fxpr}(3)

Statement (3) is shown similarly to the corresponding result
in Theorem~\ref{thm:gdown}, noting that since $g$ is increasing $r_g=g^-/g^+<1$, which results in a different inequality
than the one found in \eqref{eq:ratnec}.

To show (4) consider the curve of steady states $(\xi,\gamma(\xi))$ for $\xi>0$ where $\gamma(\xi)$ is defined by
\eqref{eq:gam}. Using Proposition~\ref{prop:fxpr}, the conditions of (4) imply that $\xi g'(\xi)/g(\xi)$
has a maximum value larger than $1$ at $\xi=\theta_g(g^-/g^+)^{1/2n}$. Let
$\xi^-<\theta_g(g^-/g^+)^{1/2n}<\xi^+$ be the points where
$\xi^- g'(\xi^-)/g(\xi^-)=\xi^+ g'(\xi^+)/g(\xi^+)=1$. Then the steady state is stable for $\xi<\xi^-$ and
$\xi>\xi^+$, and unstable for $\xi\in(\xi^-,\xi^+)$.
Now differentiating \eqref{eq:gam} we find that
\be \label{eq:gamdash}
\gamma'(\xi)=-\frac{\beta}{\xi^2}e^{-\mu\tau}g(\xi)+\frac{\beta}{\xi}e^{-\mu\tau}g'(\xi)
=\frac{\beta}{\xi^2}e^{-\mu\tau}g(\xi)\left[\frac{\xi g'(\xi)}{g(\xi)}-1\right]
=\frac{\gamma(\xi)}{\xi}\left[\frac{\xi g'(\xi)}{g(\xi)}-1\right].
\ee
Thus $\gamma(\xi)$ is a decreasing function of $\xi$ when the steady state is stable, and an increasing function when it is unstable. From this (4) follows.

The proofs of (5) and (6) are similar to the proof of Theorem~\ref{thm:gdown}, with the main difference being that for a Hopf bifurcation from the unstable steady state we require
$\omega_k\tau\in(3\pi/2+2k\pi,2\pi+2k\pi)$, while for $g$ increasing a sufficient condition to obtain
a solution  of \eqref{eq:xigdashgup} is that $M_k(\theta_g)<f(1,n,g^-/g^+)$.
\end{proof}

It would be interesting if the Hopf and fold bifurcations could exchange positions on the branch, so that the steady state lost stability in a Hopf bifurcation instead of a fold bifurcation. However, this cannot happen with a constant delay because,  as the theorem shows, stability is always lost when $\xi g'(\xi)/g(\xi)=1$ at which point there is a zero characteristic value giving rise to a fold bifurcation.

\subsection{State-dependent delay with $v$ decreasing and $g$ constant $(g\leftrightarrow,v\downarrow)$}\label{sec:gconstvdown}

To understand the dynamics of \eqref{eq:basic}-\eqref{eq:thres} with a state-dependent delay,  
we impose $g^\pm:=g^-=g^+$ in \eqref{eq:vghill} and \eqref{eq:gpwconst},
so $g(\xi)=g^\pm$ is a constant function.
We first consider a decreasing function $v$ with $v^->v^+$ in \eqref{eq:vghill} and \eqref{eq:vpwconst}.

\begin{figure}[tp!]
	\centering
		\begin{tikzpicture}[scale = 1.8]
		\tikzstyle{line} = [-,very thick]
		\tikzstyle{dotted line}=[.]
		\tikzstyle{arrow} = [->,line width = .4mm]
		\tikzstyle{unstable} = [red]
		\tikzstyle{stable} = [blue]
		\tikzstyle{pt} = [circle,draw=black,fill = black,minimum size = 1pt];
		\tikzstyle{bifpt} = [circle,draw=red,fill = red,minimum size = 1pt];
		\tikzstyle{upt} = [circle,draw=black,minimum size = 1pt];
		\def\arrlen{.3}

		\draw[line] (0,2) to (1.2,2);
		\draw[line] (1.2,2) to (1.2,1);
		\draw[line] (1.2,1) to (2.6,1);
		\draw[dotted line] (0,1) to (1.2,1);
		\draw[dotted line] (1.2,0) to (1.2,2);
		
		\draw[line] (0,0.2) to (0.5,2.5);
		\draw[line] (0,0.1) to (1,2.5);
		\draw[line] (0,0) to (1.5,2.5);
		\draw[line] (0,-0.08) to (1.83,2.5);
		\draw[line] (0,-0.165) to (2.13,2.5);
		\draw[line] (0,-0.25) to (2.55,2.5);
		\draw[line] (0.1,-0.25) to (2.6,2);
		\draw[line] (0.2,-0.25) to (2.6,1.4);
		
		\node at (1.2,-.2) {$\theta_v$};

		\node at (-.6,2) {$\beta g e^{-\mu\tau^-}$};
		\node at (-.6,1) {$\beta g e^{-\mu\tau^+}$};
		
		\node[bifpt] at (1.2,1) (8) [] {};
		\node[bifpt] at (1.2,2) (8) [] {};
		\node[upt] at (1.2,1.62) (8) [] {};
		\node[upt] at (1.2,1.33) (8) [] {};
		\node[pt] at (1.5,1) (8) [] {};
		\node[pt] at (2,1) (8) [] {};
		\node[pt] at (0.8,2) (12) [] {};
		\node[pt] at (0.4,2) (12) [] {};
		
\end{tikzpicture}
\caption{Steady states of \eqref{eq:basic} are given by \eqref{eq:h}, and hence  occur at the intersections of $\xi\mapsto\beta e^{-\mu\tau(\xi)}g$ and $\xi\mapsto\gamma\xi$.  These are illustrated for various $\gamma>0$ in the limiting case of \eqref{eq:gpwconst} and \eqref{eq:vpwconst} with $g^\pm= g^-=g^+$, so $g(\xi)=g^\pm$ is a constant function, and $v^->v^+$, so $v$ is 
piecewise constant and monotonically decreasing : $(g\leftrightarrow,v\downarrow)$. Then $\tau(\xi)=a/v(\xi)$
is state-dependent; $\tau(\xi)=\tau^-=a/v^-$ for $\xi<\theta_v$, $\tau(\xi)=\tau^+=a/v^+$ for $\xi>\theta_v$ and $\tau(\xi)$ is set-valued when $\xi=\theta_v$.}
\label{fig:3}
\end{figure}

It is convenient to let $\tau^+:=a/v^+$ and $\tau^-:=a/v^-$.
Then at a steady state $\xi$ with $v^->v^+$, equation \eqref{eq:steadydelay}
implies $\tau(\xi)$ is an increasing function of $\xi$
with $\tau^-<\tau(\xi)<\tau^+$, so $e^{-\mu\tau(\xi)}$ is a decreasing function of $\xi$.
Under these circumstances, $h(\xi)$ is monotonically decreasing and equation \eqref{eq:h} has exactly one solution, and so there is always a unique steady state.

Figure~\ref{fig:3} illustrates the uniqueness of the steady state
in the limiting case when $v$ defined by \eqref{eq:vpwconst} is piecewise constant.
The steady states at the corners are associated with
\begin{equation}
\gamma_3=\frac{\beta g^\pm e^{-\mu\tau^+}}{\theta_v} \quad\text{and}\quad
\gamma_4=\frac{\beta g^\pm e^{-\mu\tau^-}}{\theta_v}.
\label{eq:corners_gconst}
\end{equation}
As discussed in Section~\ref{sec:gdownvconst}, the steady state is stable if $\gamma \in (0,\gamma_3) \cup (\gamma_4, \infty)$ as then $A=g'(\xi)=0$ and the characteristic equation \eqref{eq:char}
reduces to \eqref{eq:char_A0g'0} with exactly one negative real characteristic value. Consequently, the steady state may only be
unstable in the singular case for $\gamma \in (\gamma_3, \gamma_4)$.

Since $A$, which is given by \eqref{eq:A}, is undefined for singular steady states, the characteristic equation \eqref{eq:char_Ag'0} cannot be used to study the stability of the singular steady states. So, instead we consider the stability of the steady states for the smooth velocity nonlinearity $v$ defined in \eqref{eq:vghill}. There will be some similarities to the analysis in Section~\ref{sec:gdownvconst}, but the problem studied in this section with state-dependent delay is significantly more complicated than the constant delay problem considered before.

With $g$ constant and $v$ defined by \eqref{eq:vghill} the characteristic equation is of the form \eqref{eq:char_Ag'0}. Recalling the definition of $A$ from \eqref{eq:A} and using \eqref{eq:h},
we obtain at a steady state $\xi$
\begin{equation} \label{eq:Axi}
A= \beta \frac{v'(\xi)}{v(\xi)}e^{-\mu\tau(\xi)} g^\pm = \gamma\frac{\xi v'(\xi)}{v(\xi)}<0,
\end{equation}
since $v^->v^+$ implies $v'(\xi)<0$. It follows that the right-hand side of \eqref{eq:char_Ag'0} is negative when $\lambda\geq0$, and so there are no non-negative real characteristic values. This is not surprising since we already know that the steady state is unique, and therefore there are no steady-state bifurcations.

Furthermore, since $v(\xi)$ is a Hill function, it follows that
\begin{equation} \label{eq:xivdashvf}
\frac{\xi v'(\xi)}{v(\xi)}=f(\xi/\theta_v,m,v^-/v^+),
\end{equation}
where $f$ is defined by \eqref{eq:fxpr}. Thus Proposition~\ref{prop:fxpr} will be  relevant in what follows.

To investigate the stability of the steady state we consider complex characteristic values.
Let $\lambda=\alpha+i\omega$, $\omega>0$ then \eqref{eq:char_Ag'0} implies
\begin{alignat*}{2}
	\alpha+i\omega =-\gamma&+A(1-e^{-(\alpha+i\omega)\tau(\xi)})(1+\frac{\mu}{\alpha+i\omega})\\
=-\gamma&+\frac{A}{\alpha^2+\omega^2}\big[(1-e^{-\alpha\tau(\xi)}\cos{\omega\tau(\xi)})(\alpha^2+\omega^2+\mu \alpha)+e^{-\alpha\tau(\xi)}\sin{\omega\tau(\xi)}\mu \omega]&&\\
	& +i\frac{A}{\alpha^2+\omega^2}[e^{-\alpha\tau(\xi)}\sin{\omega\tau(\xi)}(\alpha^2+\omega^2+\mu \alpha)-(1-e^{-\alpha\tau(\xi)}\cos{\omega\tau(\xi)})\mu \omega]&&.
\end{alignat*}
Equating the real and imaginary parts yields
\begin{align*}
	\alpha+\gamma &=\frac{A}{\alpha^2+\omega^2}\Big[(1-e^{-\alpha\tau(\xi)}\cos{\omega\tau(\xi)})(\alpha^2+\omega^2+\mu \alpha)+e^{-\alpha\tau(\xi)}\sin{\omega\tau(\xi)}\mu \omega\Big], \\
	\omega &=\frac{A}{\alpha^2+\omega^2}\Big[e^{-\alpha\tau(\xi)}\sin{\omega\tau(\xi)}(\alpha^2+\omega^2+\mu \alpha)-(1-e^{-\alpha\tau(\xi)}\cos{\omega\tau(\xi)})\mu \omega\Big].
\end{align*}
Isolating $e^{-\alpha\tau(\xi)}\sin{\omega\tau(\xi)}$ and $1-e^{-\alpha\tau(\xi)}\cos{\omega\tau(\xi)}$ respectively gives
\begin{align}
	A\big(1-e^{-\alpha\tau(\xi)}\cos{\omega\tau(\xi)}\big)\big[(\alpha+\mu)^2+\omega^2\big] &= (\alpha+\gamma)(\alpha^2+\omega^2+\mu \alpha)-\mu \omega^2, \label{eq:cos}\\
	Ae^{-\alpha\tau(\xi)}\sin{\omega\tau(\xi)}\big[(\alpha+\mu)^2+\omega^2\big] &= \omega\Big((2\alpha+\gamma)\mu+\alpha^2+\omega^2\Big). \label{eq:sin}
\end{align}

Notice that since $A<0$, for $\alpha\geq0$ the left-hand side of \eqref{eq:cos} is non-positive.
On the other hand, when
\be \label{eq:gamgtmu}
\gamma>\mu
\ee
and $\alpha\geq0$ then the right-hand side of \eqref{eq:cos} is strictly positive. Consequently if \eqref{eq:gamgtmu}
holds, then all the characteristic values must have $\alpha<0$ and the steady state is asymptotically stable.

Next we show that the steady state is also stable if $\gamma\leq\mu$ and
\begin{equation}
	\left|\frac{\xi v'(\xi)}{v(\xi)}\right| < \frac{1}{\mu\tau(\xi)}.
	\label{eq:xivdashvmutau}
\end{equation}
To do so, for contradiction suppose that $\lambda=\alpha+i\omega$ is
a characteristic value with $\alpha\geq0$ and $\omega>0$
where $\gamma\leq\mu$ and \eqref{eq:xivdashvmutau} holds.
Then $|\sin{\omega\tau(\xi)}|\leq|\omega\tau(\xi)|$ and \eqref{eq:xivdashvmutau} implies
\begin{equation} \label{eq:ineqomegamu}
\Big|\frac{\xi v'(\xi)}{v(\xi)}\sin{\omega\tau(\xi)}\Big|
< \frac{\omega}{\mu}.
\end{equation}
Consequently, using \eqref{eq:Axi},
\begin{align*}
\Big|Ae^{-\alpha\tau(\xi)}\sin{\omega\tau(\xi)}\big[(\alpha+\mu)^2+\omega^2\big]\Big|
& < \frac{\gamma\omega}{\mu}\big((\alpha+\mu)^2+\omega^2\big) \\
& = \omega\Big((2\alpha+\mu)\gamma+\frac{\gamma}{\mu}(\alpha^2+\omega^2)\Big) \\
& \leq \omega\big((2\alpha+\gamma)\mu+(\alpha^2+\omega^2)\big).
\end{align*}
However,  this contradicts the assumption that $\lambda=\alpha+i\omega$ satisfies \eqref{eq:sin}. We thus conclude that the steady state is asymptotically stable whenever \eqref{eq:gamgtmu} or \eqref{eq:xivdashvmutau} is satisfied.

Since the right-hand side of \eqref{eq:xivdashvmutau} is bounded below by $1/\mu\tau^+$, it follows from
Proposition~\ref{prop:fxpr}(1) that the steady state is stable for $0<\xi\ll \theta_v$ and
for $\xi\gg\theta_v$.

To determine when the steady state may be unstable we investigate the basic spectral condition for Hopf bifurcation, namely, the existence of a pair of complex conjugate eigenvalues on the imaginary axis. A proof that Hopf bifurcations actually occur for our equations  would, of course, require  in addition that a pair of eigenvalues crosses the imaginary axis at nonzero speed, that a nonresonance condition is satisfied, and furthermore that the right hand side of the delay differential equation has certain higher order smoothness properties, see \cite{Eichmann06,HuWu10,Sieber12} for the case of state-dependent delays.

Assume that $\lambda=\pm i\omega$, $\omega>0$ solves equation \eqref{eq:char_Ag'0}.
Then with $\alpha=0$ equations \eqref{eq:cos} and \eqref{eq:sin} reduce to
\begin{align}
A(1-\cos(\omega\tau))(\omega^2+\mu^2)=\omega^2 (\gamma-\mu),
\label{eq:sin_cancelled}
\\
A\sin(\omega\tau)(\omega^2+\mu^2) = \omega(\omega^2+\gamma\mu). \label{eq:cos_cancelled}
\end{align}
Since $A<0$ and the right-hand side of \eqref{eq:cos_cancelled} is positive, at any Hopf bifurcation we must have  $\sin(\omega\tau)<0$ to satisfy \eqref{eq:cos_cancelled}.
Moreover, the left-hand side of
\eqref{eq:sin_cancelled} is negative, and so a Hopf bifurcation is only possible if the
right-hand side is also negative, that is if
\be \label{eq:gam<mu}
\gamma<\mu.
\ee

In Section~\ref{sec:gdownvconst} it was so simple to rewrite \eqref{eq:real} and \eqref{eq:imag} as \eqref{eq:omega} and \eqref{eq:gdash} that we did so without comment.
Equation \eqref{eq:gdash} involves the derivative $g'(\xi)$; here the analogous term is $v'(\xi)$ which is part of $A$.
We want to rewrite
\eqref{eq:sin_cancelled} and \eqref{eq:cos_cancelled} as one equation for $\omega$ which is independent
of $v'(\xi)$ and one equation for $v'(\xi)$ which contains no trigonometric functions. To accomplish this we make use of half-angle formulae.

Let $U = \omega\tau/2$ then
\begin{displaymath}
\tan\frac{\omega\tau}{2}=\tan U = \frac{2\sin^2 U}{2\sin U\cos U}=\frac{1-\cos(\omega\tau)}{\sin(\omega\tau)},
\end{displaymath}
and hence
using \eqref{eq:cos_cancelled} and \eqref{eq:sin_cancelled}
\be \label{eq:tanU}
\tan\frac{\omega\tau}{2}=\frac{\omega(\gamma-\mu)}{\omega^2+\gamma\mu}.
\ee

We next simplify \eqref{eq:cos_cancelled} using another half-angle formula.
Still with $U =\omega\tau/2$, from the standard formula
$$
\sin\omega\tau=\frac{2\tan U}{1+\tan^2U},
$$
on substituting for $\tan U$ from \eqref{eq:tanU} we obtain
$$
\sin\omega\tau=\frac{2\omega(\gamma-\mu)(\omega^2+\gamma\mu)}{(\omega^2+\gamma^2)(\omega^2+\mu^2)}.
$$
Substituting this into \eqref{eq:cos_cancelled}, rearranging and using \eqref{eq:Axi} gives
\be \label{eq:AU}
\frac{\xi v'(\xi)}{v(\xi)}=\frac{\omega^2+\gamma^2}{2\gamma(\gamma-\mu)}.
\ee

At a Hopf bifurcation equations \eqref{eq:h}, \eqref{eq:sin_cancelled} and \eqref{eq:cos_cancelled}
must all be satisfied.
This is equivalent to solving
\be \label{eq:gamv}
\gamma = \frac{\beta g^\pm}{\xi} e^{-\mu\tau(\xi)},
\ee
along with \eqref{eq:tanU} and
\eqref{eq:AU}.

We will follow similar steps as in Section~\ref{sec:gdownvconst}, and 
consider \eqref{eq:gamv},\eqref{eq:tanU} and \eqref{eq:AU} sequentially, using the first two equations to define $\gamma$ and $\omega_k$ as functions of $\xi$, so that it only remains to solve \eqref{eq:AU} for $\xi$.
But because of the state-dependent delay and the  constraint \eqref{eq:gam<mu} the situation is not as simple as in the constant delay case considered in the previous two sections.

Note first that for any $\xi>0$ equation \eqref{eq:gamv}
gives a unique value of $\gamma= \gamma(\xi)$. Moreover, since as already noted, $\tau$ is monotonically increasing it follows that $\gamma$ is a monotonically decreasing function of $\xi$.

For $\gamma$ satisfying \eqref{eq:gam<mu}, the right-hand side of \eqref{eq:tanU} is negative. Then because of the properties of the $\tan$ function in \eqref{eq:tanU} there will be at least one solution $\omega \tau$ to
\eqref{eq:tanU} satisfying $\omega\tau \in((2k+1)\pi,(2k+2)\pi)$ for $k=0,1,2,\ldots$. We denote by $\omega_k$ any solution of \eqref{eq:tanU} for which $\omega_k\tau(\xi)\in((2k+1)\pi,(2k+2)\pi)$.

At this point, we have defined $\gamma$ and $\omega_k$ as functions of $\xi$. We still need to solve
for $\xi$, or  $\xi_k$ from \eqref{eq:AU}.
We already considered the behaviour of the left-hand side of \eqref{eq:AU} in
Proposition~\ref{prop:fxpr}, so we now consider the behaviour of the right-hand side. Let
\be \label{eq:rk}
r_k(\gamma)=\frac{\omega_k^2+\gamma^2}{2\gamma(\gamma-\mu)}.
\ee
Usually we will take $\gamma=\gamma(\xi)$, defined by \eqref{eq:gamv}, but in the proposition below we consider
$r_k(\gamma)$ for general $\gamma$. Note that a solution of \eqref{eq:AU} corresponds to a solution
of $f(\xi/\theta_v,m,v^-/v^+)=r_k(\gamma(\xi))$, and since $v$ is decreasing both $f$ and $r_k$ will be negative at such a solution, thus from
\eqref{eq:rk} we require $\gamma<\mu$.

\begin{prop} \label{prop:rgamma}
Let $r_k(\gamma)$ be defined by \eqref{eq:rk} and $\omega_k$ satisfy
$\omega_k\tau\in((2k+1)\pi,(2k+2)\pi)$ for integer $k\geq0$. Then for $\gamma\in(0,\mu)$ it holds that
\begin{enumerate}
\item
$$|r_k(\gamma)|\geq2\left(\frac{\omega_k}{\mu}\right)^2.$$
\item
$r_k(\gamma)\to-\infty$ as $\gamma\to0$ or $\gamma\to\mu$,
\item
$$|r_k(\gamma)| \leq \frac{(2k+2)^2(\pi/\tau^-)^2+\gamma^2}{2\gamma(\mu-\gamma)}.$$
\end{enumerate}
\end{prop}

\begin{proof}
The proof is elementary.
\end{proof}

While there is a unique steady state $\xi$ for the case of decreasing $v$ with constant $g$, just as in 
Section~\ref{sec:gdownvconst}
the location and properties of this steady state will depend on the values of the other parameters. The following theorem collects together our results for this case.

\begin{thm} \label{thm:vdown}
Let $\xi$ be the steady state of the DDE \eqref{eq:basic},\eqref{eq:thres} with $g(\xi)=g^\pm$ constant,
and $v$ monotonically decreasing, so $v^->v^+$ and the state-dependent delay $\tau$ defined by \eqref{eq:vghill}
evaluated at the steady state is $\tau(\xi)=a/v(\xi)$.
Then
\begin{enumerate}
\item
If $\left|\frac{\xi v'(\xi)}{v(\xi)}\right|<\frac{1}{\mu\tau(\xi)}$ then the steady state  $\xi$ is asymptotically stable.
\item
The steady state
$\xi$ is asymptotically stable
if $m\leq1/(\mu\tau(\xi))$, or if both  $m>1/(\mu\tau(\xi))$ and
$r=r_v=\frac{v^-}{v^+}<\Bigl(1+\frac{2}{m\mu\tau(\xi)-1}\Bigr)^2$.
\item
For any fixed $m>0$, and for $0<\xi\ll\theta_v$ or $\xi\gg\theta_v$, or equivalently for $\gamma\gg\gamma_4$ or
$0<\gamma\ll\gamma_3$, the steady state is asymptotically stable.
\item
If $\gamma>\mu$ the steady state is asymptotically stable.
\item
For any fixed $\xi\ne\theta_v$ let $\gamma= \gamma(m,\xi)$ be the value of $\gamma$ such that \eqref{eq:h} is satisfied and hence $\xi$ is a steady state.
Or, for any fixed
$\gamma$ with $0<\gamma<\gamma_3$ or $\gamma>\gamma_4$ let $\xi=\xi(m,\gamma)$ satisfy  \eqref{eq:h} and hence
be a steady state.
Then $\xi$ is asymptotically stable for all $m$ sufficiently large.
\item
Let $\mu>\gamma_3$.
\begin{enumerate}
\item
For any $k>0$, for all $m=m(k)$ sufficiently large there are
two families of (at least) $k+1$ Hopf bifurcations as $\gamma$ is varied.
In the first family, the characteristic values
$\lambda_j=\alpha_j\pm i\omega_j$ with $\omega_j\tau\in((2j+1)\pi,(2j+2)\pi)$ for $j=0,1,\ldots,k$
cross the imaginary axis from left to right as $\gamma$ increases, while in the second family they cross the imaginary axis from right to left.
\item
Let $\gamma\in(\gamma_3,\min\{\gamma_4,\mu\})$ be fixed. Then as $m$ is increased there is an infinite
sequence of Hopf bifurcations where the real part of $\lambda_k=\alpha_k\pm i\omega_k$ becomes positive
with $\omega_k\tau\in((2k+1)\pi,(2k+2)\pi)$.
\end{enumerate}
\end{enumerate}
\end{thm}

\begin{proof}
Statements (1) and (4) were already shown; see equations \eqref{eq:Axi}, \eqref{eq:xivdashvmutau} and \eqref{eq:gamgtmu}.

Claim (2) follows from (1) using Proposition~\ref{prop:fxpr}, since
\begin{displaymath}
\left|\frac{\xi v'(\xi)}{v(\xi)}\right|\leq|f(r_v^{1/2m},m,r_v)|=\frac{m|1-r_v^{1/2}|}{1+r_v^{1/2}}
\end{displaymath}
where $r_v=v^-/v^+>1$.

Claim (3) also follows from (1), similarly to the proof of Theorem~\ref{thm:gdown}(3).

Statement (5) is derived similarly to Theorem~\ref{thm:gdown} (4), by using \eqref{eq:gamv} and
\eqref{eq:corners_gconst} to show that
$e^{-\mu\tau^+}<e^{-\mu\tau(\xi)}<e^{-\mu\tau^-}$ implies that
$\gamma_3/\gamma<\xi/\theta_v<\gamma_4/\gamma$.

Finally, (6) is more delicate to prove. As noted before Proposition~\ref{prop:rgamma},
to find a Hopf bifurcation we need to solve
$f(\xi/\theta_v,m,v^-/v^+)=r_k(\gamma(\xi))$. Since for $\xi\ne\theta_v$ the function $f(\xi/\theta_v,m,v^-/v^+)\to0$ as $m\to\infty$,
we begin by considering $\xi=\theta_v$. Recall that
$\gamma(\xi)$ defined by \eqref{eq:gamv} is a monotonically decreasing function of $\xi$. Moreover, using
\eqref{eq:vghill}
we see that $\tau(\theta_v)=2a/(v^-+v^+)\in(\tau^-,\tau^+)$, which is independent of the value of $m$. Then
\be \label{eq:gamthetav}
\gamma(\theta_v)=\frac{\beta g^\pm}{\theta_v}e^{-\mu\tau(\theta_v)} \in (\gamma_3,\gamma_4),
\ee
and $\gamma(\theta_v)$ is also independent of the value of $m$.

Now there are two cases to consider. First consider the case where $\gamma(\theta_v)<\mu$.
If
\be \label{eq:msuffvdown}
m > 2\frac{(v^-/v^+)+1}{(v^-/v^+)-1}
\left( \frac{(2k+2)^2(\pi/\tau(\theta_v))^2+\gamma(\theta_v)^2}{2\gamma(\theta_v)(\mu-\gamma(\theta_v))}\right),
\ee
then
\begin{align*}
f(1,m,v^-/v^+)=-\frac{m(v^-/v^+-1)}{2(v^-/v^++1)}
 & < \left( \frac{(2k+2)^2(\pi/\tau(\theta_v))^2+\gamma(\theta_v)^2}{2\gamma(\theta_v)(\gamma(\theta_v)-\mu)}\right)\\
& \leq r_j(\gamma(\theta_v)), \quad j=0,1,\ldots,k.
\end{align*}
Here the equality comes from definition of $f$, the strict inequality from \eqref{eq:msuffvdown} and the last inequality follows from a similar argument that proves Proposition~\ref{prop:rgamma}(3), the only difference being that here we use the actual value of $\tau(\theta_v)$ in the inequality, rather than the bound $\tau^-$.

With this inequality as the starting point, we examine what happens when we increase $\xi$ away from $\theta_v$.
If $\xi$ is increased then $\gamma(\xi)$ decreases with
$\lim_{\xi\to\infty}\gamma(\xi)=0$. But $r_j(\gamma(\xi))$ and $f(\xi/\theta_v,m,v^-/v^+)$ are both
continuous functions of $\xi$ with
$\lim_{\xi\to\infty}r_j(\gamma(\xi))=-\infty$ and
$\lim_{\xi\to\infty}f(\xi/\theta_v,m,v^-/v^+)=0$.
Consequently for each $j=0,1,\ldots,k$ there exists a $\xi$ such that
$f(\xi/\theta_v,m,v^-/v^+)=r_j(\gamma(\xi))$.

If instead $\xi$ is decreased from $\theta_v$
then $\lim_{\xi\to0}\gamma(\xi)=+\infty$, so $\gamma(\xi)>\mu$ for $\xi$ sufficiently small.
However, for $\gamma(\xi)<\mu$ we have
$\lim_{\gamma\nearrow\mu}r_j(\gamma)=-\infty$, while $f(\xi/\theta_v,m,v^-/v^+)$ is bounded, so again
for each $j=0,1,\ldots,k$ there exists a $\xi$ such that
$f(\xi/\theta_v,m,v^-/v^+)=r_j(\gamma(\xi))$.

Solutions of this equation define the Hopf bifurcation points, which
gives the required Hopf bifurcations when $\gamma(\theta_v)<\mu$.
To summarize the argument up to this point, for a fixed $k$ and any large enough  $m=m(k)$ satisfying \eqref{eq:msuffvdown}, we found two families of $k+1$ Hopf bifurcations, one for $\xi_j < \theta_v$ and one for $\xi_j > \theta_v$, by finding appropriate $\gamma(\xi_j)$ that satisfy $f(\xi_j/\theta_v,m,v^-/v^+)=r_j(\gamma(\xi_j))$ for each $j=0, \ldots, k$. See Figure~\ref{fig:vdown_ex2}(d) for illustration of these families as functions of the parameters $m$ and $\gamma$.

Now consider the more delicate case
where $\gamma_3<\mu\leq \gamma(\theta_v)<\gamma_4$. The above argument fails in that case as
$|f(\xi/\theta_v,m,v^-/v^+)|\gg0$ for $\xi=\theta_v$ but the corresponding $\gamma$ is $\gamma(\theta_v)$,
with $\gamma(\theta_v)>\mu$, and by (4) the steady state would be asymptotically stable.
Instead, noting that $\gamma(\xi)$ defined by \eqref{eq:gamv} is monotonically decreasing, this function is invertible and we can instead consider $\xi = \xi(\gamma) $ as a function of $\gamma$. Fix $\gamma\in(\gamma_3,\mu)$, and consider the behaviour as $m\to\infty$. In this case the function $v(\xi)$ defined by \eqref{eq:vghill}
approaches the piecewise constant function \eqref{eq:vpwconst}, and the steady-state function $h(\xi)$
(recall \eqref{eq:h}) approaches the case illustrated in Figure~\ref{fig:3}.
Since $\gamma$ is fixed with $\gamma\in(\gamma_3,\mu)\subset(\gamma_3,\gamma_4)$,
we find $\xi\to\theta_v$ and $v'(\xi)\to-\infty$ while $v(\xi)$ and $\xi$ remain bounded and bounded away from zero. Consequently,  for $m$ sufficiently large
$\frac{\xi v'(\xi)}{v(\xi)}<r_j(\gamma)$ for $j=0,1,\ldots,k$. From here the argument proceeds as in the
case $\gamma(\theta_v)<\mu$. Statement (6)b also follows trivially in the case that
$\gamma_3<\mu\leq \gamma(\theta_v)<\gamma_4$.
\end{proof}

Although Theorem~\ref{thm:vdown}(6) is stated for one-parameter continuation in $m$ or $\gamma$, we can also draw conclusions for two-parameter continuation of the Hopf bifurcations in the $(\gamma,m)$-parameter plane.

The argument used in the proof of Theorem~\ref{thm:vdown}(6) in the case when  $\mu\in(\gamma_3,\gamma_4)$
shows that for any $\gamma_- > \gamma_3$ the $k$-th Hopf bifurcation occurs for $\gamma\in(\gamma_3,\gamma_-)$ for all $m$ sufficiently large. Then because of Theorem~\ref{thm:vdown}(5)
the $k$-th Hopf bifurcation approaches $\gamma=\gamma_3$ as $m\to\infty$, and in the $(\gamma,m)$-parameter plane the left side of the Hopf bifurcation curves asymptote to $\gamma=\gamma_3$ as $m\to\infty$. Similarly
for any $\gamma_+ < \mu$ the other instance of the $k$-th Hopf bifurcation occurs for $\gamma\in(\gamma_+,\mu)$ for all $m$ sufficiently large, and the right
side of the Hopf bifurcation curves asymptote to $\gamma=\mu$ as $m\to\infty$.

In the case that $\mu>\gamma_4$, a similar argument can be applied to show that the $k$-th Hopf bifurcation curve asymptotes to $\gamma=\gamma_3$ and $\gamma=\gamma_4$ in the
$(\gamma,m)$-parameter plane.

There are nevertheless differences between
the cases where
$\gamma$ converges to $\gamma_3$ or $\gamma_4$ as $m\to\infty$ and the case where
$\gamma\to\mu$ as $m\to\infty$.
To see this consider for fixed $k$ the limit as
$m \to \infty$ when $v$ approaches the piecewise constant function \eqref{eq:vpwconst}.
Then from Theorem~\ref{thm:vdown}, we have $\xi\to\theta_v$, and $\gamma\in[\gamma_3,\min\{\gamma_4,\mu\}]$.
Since $\tau(\xi)\in[a/v_U,a/v_L]$ and $\omega_k\tau(\xi)\leq(2k+2)\pi$ it also follows that $\omega_k$ is bounded. Thus the numerator of the right-hand side of \eqref{eq:AU} also remains bounded.
Now there are two cases to consider.

First suppose that as $m\to\infty$ and $\xi\to\theta_v$ that $v'(\xi)$ becomes unbounded,
that is $v'(\xi)\to-\infty$,
or equivalently that $A \to -\infty$. Then the left-hand side of \eqref{eq:AU} becomes unbounded in the limit as $m\to\infty$.
Since  the numerator of the right-hand side is bounded,
we must have  $\gamma -\mu = \mathcal{O}(1/A)\to0$ to satisfy
equation~\eqref{eq:AU}.
To summarize, if $v'(\xi) \to-\infty$ as $m\to\infty$ we must have that $\gamma\to\mu$ in this limit.

On the other hand, if $A<0$ remains finite as $m\to \infty$, because $\gamma$, $\xi$ and $v(\xi)$ are bounded and bounded away from zero in the limit, the only possibility in \eqref{eq:Axi}   is that $v'(\xi)$ also remains finite.
But as $m\to\infty$ the function $v(\xi)$ approaches a step function, and the only places where $v'(\xi)$ is non-zero and finite are near the corners of the limiting velocity nonlinearity.
Consequently, the only possibility for a Hopf bifurcation to exist for arbitrary $m$ is that the steady state at which this Hopf bifurcation happens converges to the corners of the limiting velocity nonlinearity. That is,
in the limit as $m \to \infty$ with $A<0$ finite, we must have that $\gamma\to\gamma_3$ or $\gamma\to\gamma_4$
with $\xi\to\theta_v$, where $\gamma_3$ and $\gamma_4$ are defined by \eqref{eq:corners_gconst}.

Below we illustrate the different possible behaviours allowed by Theorem~\ref{thm:vdown} in the three cases: $\mu<\gamma_3$, $\mu\in(\gamma_3,\gamma_4)$ and $\mu>\gamma_4$.

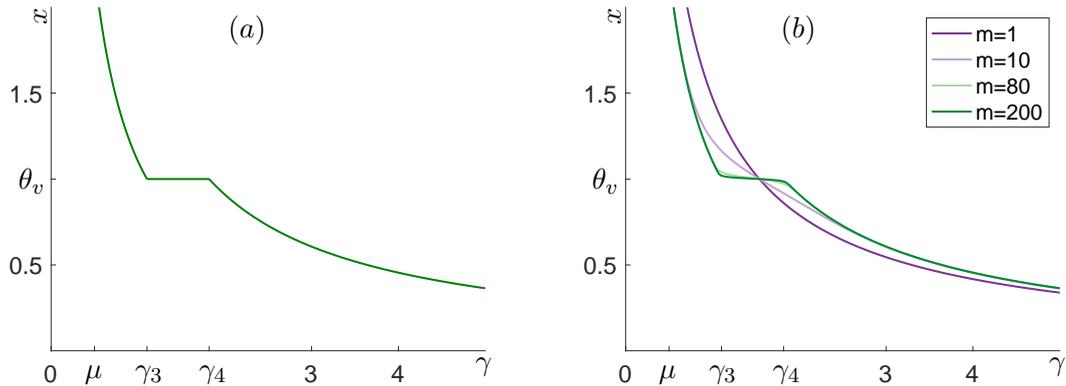
\begin{figure}[tp!]
	\centering \includegraphics[scale=0.5]{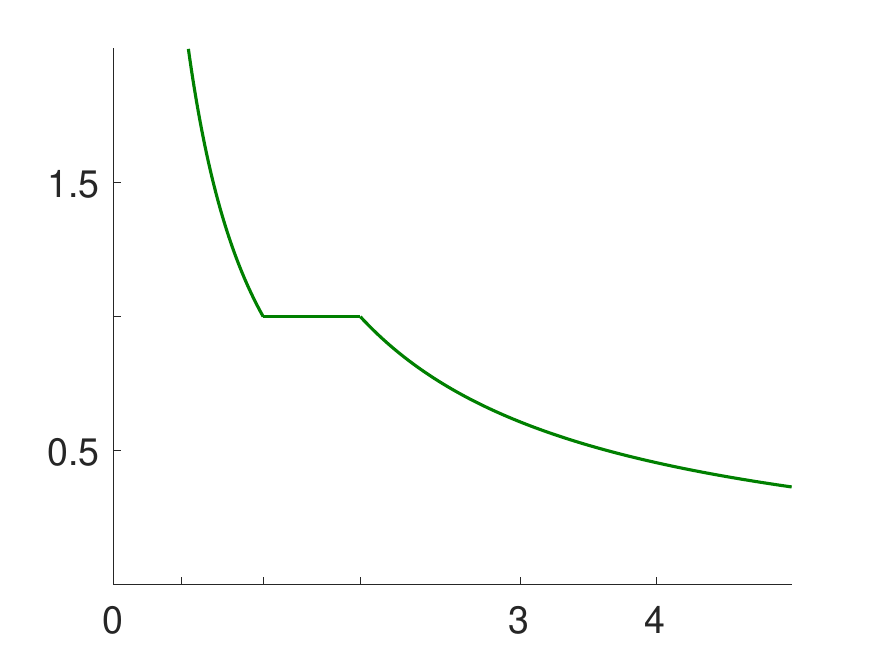}\hspace*{0.5em}\includegraphics[scale=0.5]{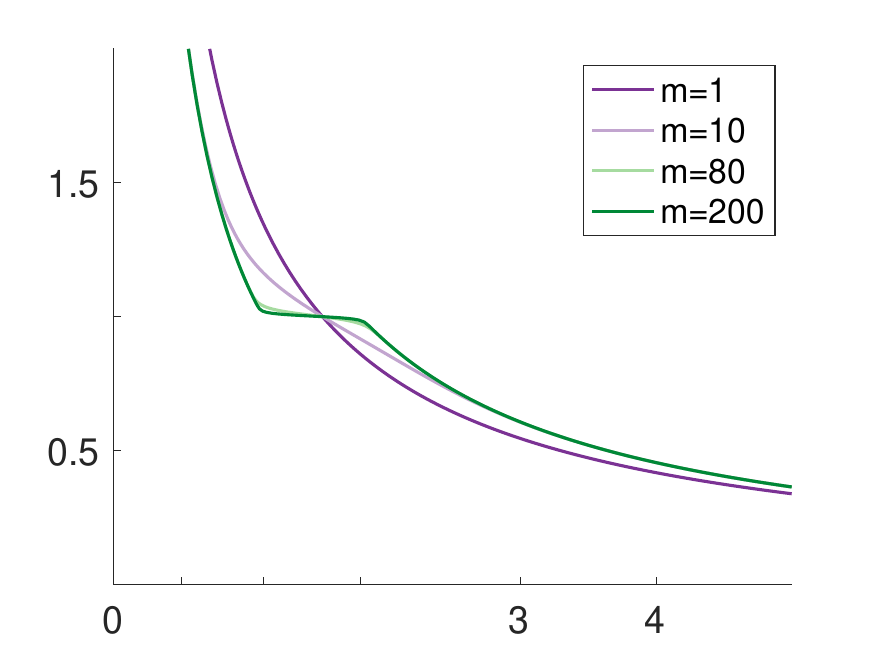}
	\put(-406,140){\rotatebox{90}{$x$}}
	\put(-412,78){$\theta_v$}
	\put(-240,10){$\gamma$}
	\put(-387,7){$\mu$}
	\put(-368,7){$\gamma_3$}
	\put(-344,7){$\gamma_4$}
    \put(-333,135){$(a)$}
	\put(-125,135){$(b)$}
	\put(-190,140){\rotatebox{90}{$x$}}
	\put(-196,78){$\theta_v$}
	\put(-24,10){$\gamma$}
	\put(-171,7){$\mu$}
	\put(-152,7){$\gamma_3$}
	\put(-128,7){$\gamma_4$}
	\caption{Bifurcation diagram of \eqref{eq:basic}-\eqref{eq:thres}
		with $(g \leftrightarrow, v \downarrow)$ and parameters $\beta=3, \mu=0.5, g^-=g^+=1, \gamma=1, \theta_v=1, a=2, v^-=2$ and $v^+=1$. (a) The limiting case with $v$ defined by \eqref{eq:vpwconst}. The stable steady state is shown as a green solid line.
		(b) With smooth velocity nonlinearity $v$ defined by \eqref{eq:vghill} with $m=1$, $10$, $80$, $200$ as indicated by color. The steady state is always stable. }
	\label{fig:vdown_ex3}
\end{figure}

{\bf Case 1:}
We begin with the case $\mu<\gamma_3$. By Theorem~\ref{thm:vdown} (point 4) the steady state must be stable whenever $\gamma>\mu$, while for $\gamma\leq\mu$ we have $\gamma\leq\mu<\gamma_3$ so by Theorem~\ref{thm:vdown}(point 5) the steady state is stable for all $m$ sufficiently large, or by Theorem~\ref{thm:vdown} (point 3) it is stable for all $\gamma$ sufficiently small.

Figure~\ref{fig:vdown_ex3} shows an example of the behaviour of \eqref{eq:basic}-\eqref{eq:thres} for $(g \leftrightarrow, v \downarrow)$ with $\mu < \gamma_3 < \gamma_4$.
Panel (b) shows the smooth case for several different values of $m$, which reveals that the steady state is always stable. Panel (a) shows the behavior in the limiting case with \eqref{eq:vpwconst}.
In this case the singular steady state can only become a stable steady state for large finite $m$.

In the model \eqref{eq:basic}-\eqref{eq:thres} we consider the parameters $\gamma$ and $\mu$ to be independent, but depending on how the model is derived, that may not be the case.
For example, the current model \eqref{eq:basic}-\eqref{eq:thres} can be derived as a reduction of the
operon model in \cite{ghmww2020}. In the model in \cite{ghmww2020} the parameter equivalent to $\gamma$ is an effective removal rate which is the sum of the actual degradation rate with the
dilution because of the growth rate $\mu$. Consequently,
in the model of \cite{ghmww2020}, and in similar systems modelling gene regulatory dynamics in a growing cell, we obtain the natural parameter constraint $\gamma>\mu$, and expect
dynamics corresponding to Figure~\ref{fig:vdown_ex3}.

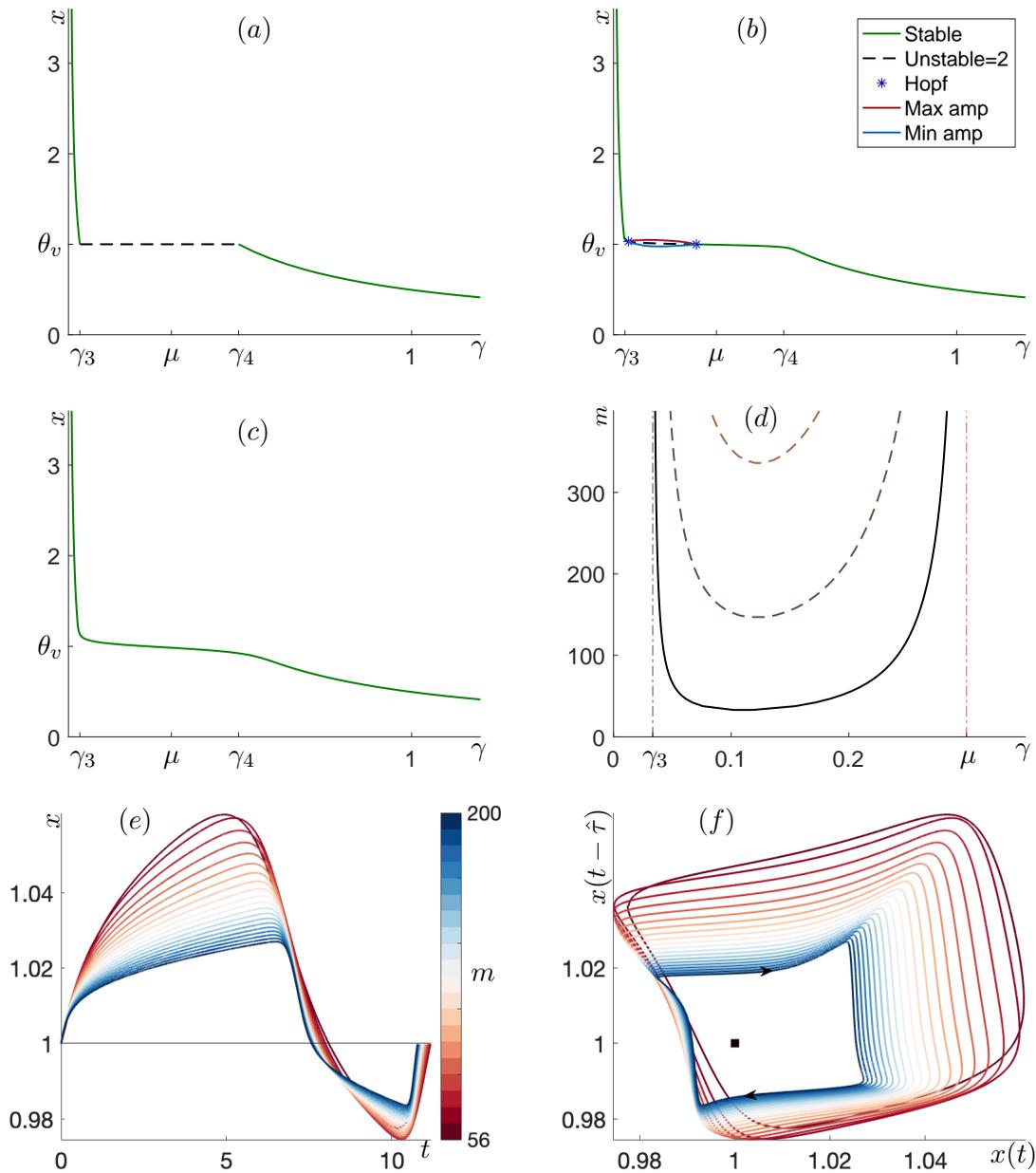
\begin{figure}[htp!]
	\centering
	\vspacefig
	\includegraphics[scale=0.5]{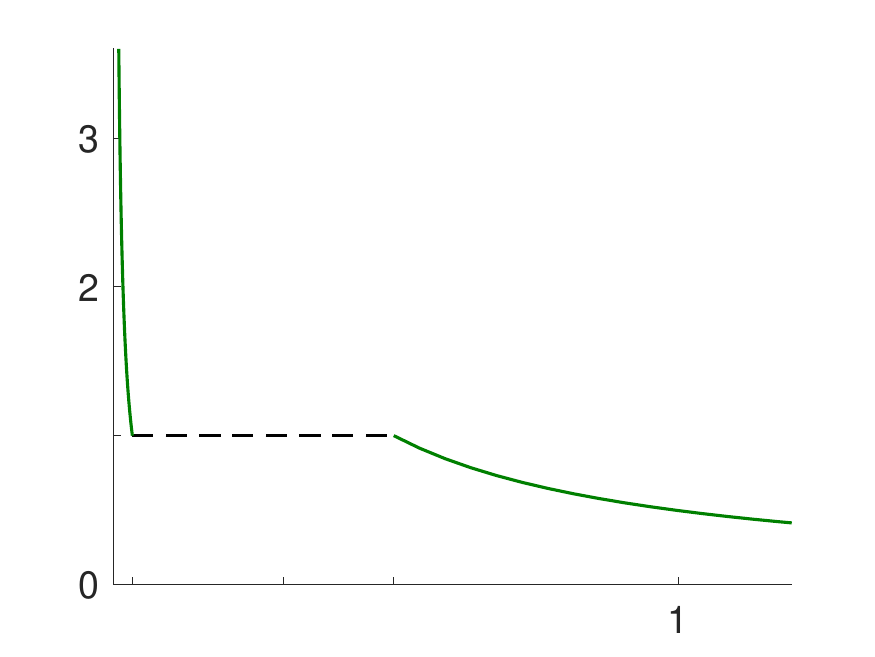}\hspace*{0.5em}\includegraphics[scale=0.5]{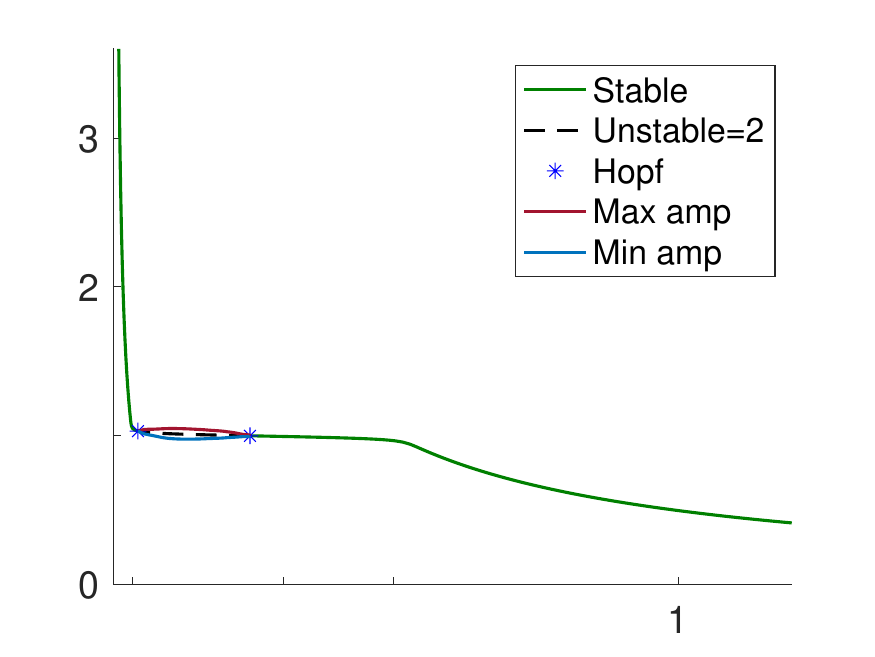}
	\put(-406,140){\rotatebox{90}{$x$}}
	\put(-412,50){$\theta_v$}
	\put(-240,10){$\gamma$}
	\put(-399,7){$\gamma_3$}
	\put(-362,7){$\mu$}
	\put(-336,7){$\gamma_4$}
    \put(-333,135){$(a)$}
	\put(-135,135){$(b)$}
	\put(-190,140){\rotatebox{90}{$x$}}
	\put(-196,50){$\theta_v$}
	\put(-24,10){$\gamma$}
	\put(-183,7){$\gamma_3$}
	\put(-146,7){$\mu$}
	\put(-120,7){$\gamma_4$}\\
	\vspace*{-0.5em}
	\includegraphics[scale=0.5]{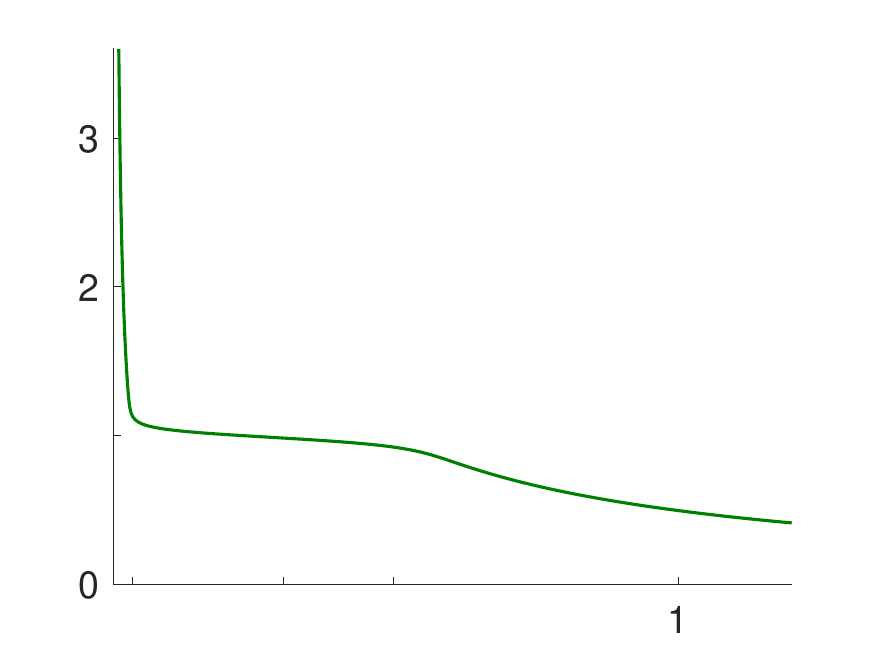}\hspace*{0.5em}\includegraphics[scale=0.5]{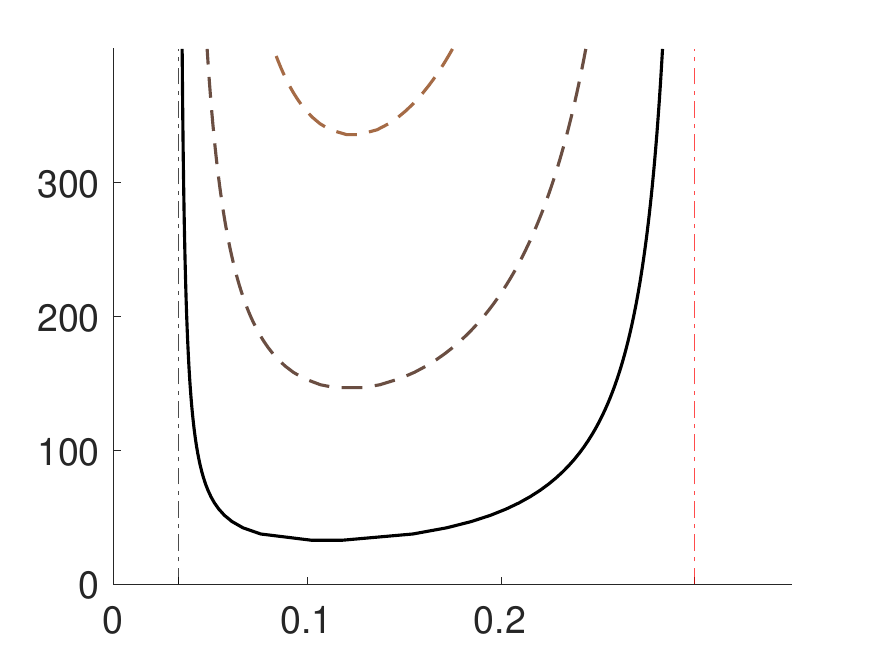}
	\put(-406,140){\rotatebox{90}{$x$}}
	\put(-412,50){$\theta_v$}
	\put(-240,10){$\gamma$}
	\put(-399,7){$\gamma_3$}
	\put(-362,7){$\mu$}
	\put(-336,7){$\gamma_4$}
    \put(-333,135){$(c)$}
	\put(-132,140){$(d)$}
	\put(-190,140){\rotatebox{90}{\footnotesize$m$}}
	\put(-24,10){$\gamma$}
	\put(-172,7){$\gamma_3$}
	\put(-46,7){$\mu$}\\
	\vspace*{-0.5em}
	\includegraphics[scale=0.5]{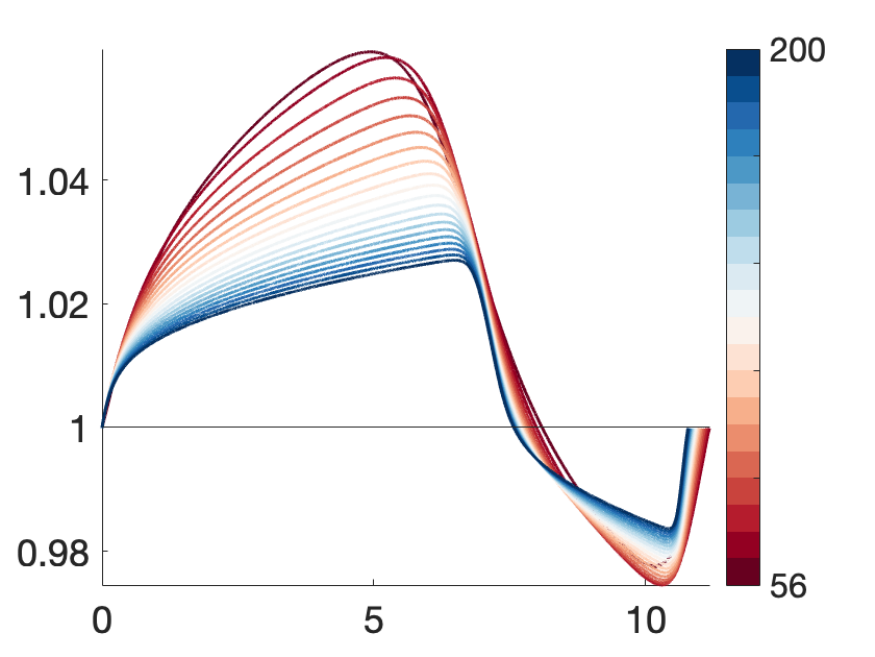}\hspace*{0.5em}\includegraphics[scale=0.5]{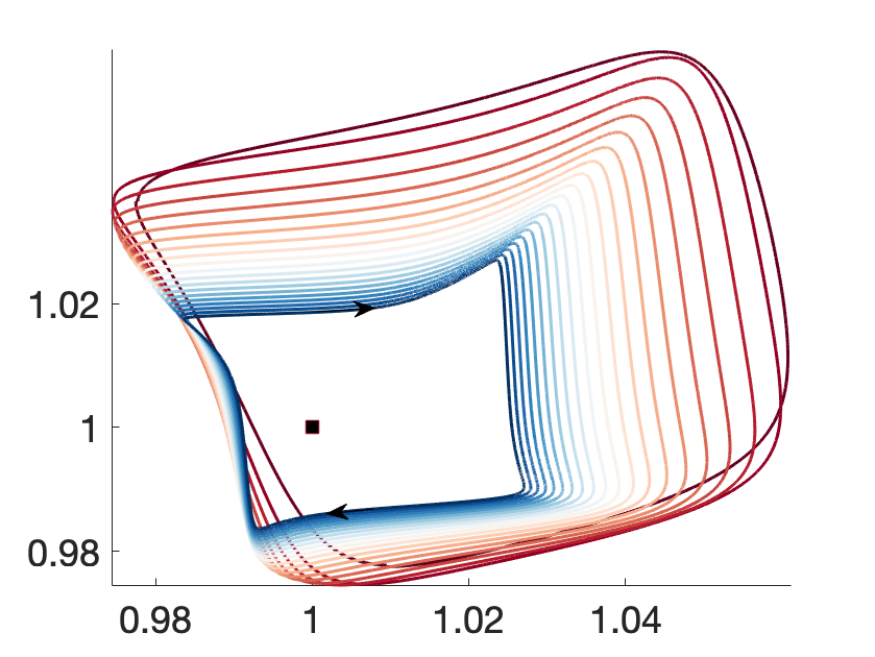}
	\put(-408,140){\rotatebox{90}{$x$}}
	\put(-260,10){$t$}
   \put(-240,80){$m$}
	\put(-380, 140){$(e)$}
	\put(-150, 140){$(f)$}
	\put(-195,110){\rotatebox{90}{$x(t-\hat\tau)$}}
	\put(-34,8){$x(t)$}
	\caption{Bifurcation diagram of \eqref{eq:basic}-\eqref{eq:thres}
		with $(g \leftrightarrow, v \downarrow)$ and parameters $\beta=3, \mu=0.3, g^-=g^+=1, \gamma=1, \theta_v=1, a=3, v^-=0.5$ and  $v^+=0.2$.
		(a) The limiting case with $v$ defined by \eqref{eq:vpwconst}. The stable steady state is shown as a green solid line, and the singular steady state as a black dashed line.
		(b) With a smooth velocity nonlinearity $v$ defined by \eqref{eq:vghill} with $m=100$. Solid lines represent stable objects including the stable steady state (in green), and envelope of the periodic orbit (in red and blue). Dashed lines represent unstable steady states which have two eigenvalues with positive real part (in black).
		(c) As in (b) but with $m=32$.
		(d) Two-parameter continuations in $m$ and $\gamma$ of the Hopf bifurcations (shown as solid curves) with the other parameters as above. The outermost curve of Hopf bifurcations is associated with the stability change seen in (b). The black dash-dotted line denotes $\gamma=\gamma_3=0.0333$ and the red dash-dotted line denotes $\gamma=\mu=0.3$, the location of the Hopf bifurcations in the limiting case as $m\to\infty$.
Note that $\gamma(\theta_v)=0.2293 < \mu<\gamma_4=0.4959$.
	(e) Profiles of the stable periodic orbits from the Hopf bifurcations in (b) at $\gamma=0.13$
           and $m\in[56,200]$ as indicated by the color map.
		(f) Projection of the phase space dynamics into the $(x(t),  x(t-\hat\tau))$ plane at $\gamma=0.13$ where $\hat\tau=\tau(\theta_v)=60/7$. The arrow indicates the direction of the flow. The square marks the unstable steady state in the limiting case at the threshold. }
	\label{fig:vdown_ex2}
\end{figure}

{\bf Case 2:}
Next,  in
Figure~\ref{fig:vdown_ex2}, we illustrate  the behavior of \eqref{eq:basic}-\eqref{eq:thres} for $(g
\leftrightarrow, v \downarrow)$ with $\mu\in(\gamma_3,\gamma_4)$.  Panel (a) depicts the behavior of the limiting case with piecewise constant $v$ defined by \eqref{eq:vpwconst}. Panel (b) shows the case of the smooth velocity nonlinearity \eqref{eq:vghill} with $m=100$.
The steady state close to $x=\theta_v$ undergoes a pair of Hopf bifurcations creating a bubble of stable periodic orbits which coexist with the unstable steady state for an interval of values of $\gamma$ which is a subinterval of $(\gamma_3,\mu)$.
As required by Theorem~\ref{thm:vdown}(4) the steady state is asymptotically stable for $\gamma>\mu$.
Thus, in contrast to the previous case, the singular steady state may become either a stable or unstable steady state for very large finite values of $m$.
Panel (c) is similar to panel (b) but for $m=32$. In this case the steady state is always stable, and
no Hopf bifurcations are seen.

Figure~\ref{fig:vdown_ex2}(d) shows a two-parameter continuation of the first three Hopf bifurcations illustrating Theorem~\ref{thm:vdown}(4-6).
The steady state is stable below and outside the outermost Hopf curve and unstable otherwise,
and  Hopf bifurcations appear sequentially for increasing values of $m$.
The bifurcation curve of first Hopf bifurcation (corresponding to $k=0$ in the analysis above) is clearly seen to asymptote to $\gamma=\gamma_3$ and $\gamma=\mu$ as $m\to\infty$. The subsequent Hopf bifurcations also approach these limits as $m$ becomes larger, but do so more slowly.

For the parameter values shown in Figure~\ref{fig:vdown_ex2} the inequalities
$\gamma_3<\gamma(\theta_v)< \mu<\gamma_4$ hold. Consequently, \eqref{eq:msuffvdown} applies and gives a sufficient condition on $m$ to ensure that the $k$-th Hopf bifurcation arises. For the given parameters (see the caption of Figure~\ref{fig:vdown_ex2}) these sufficient conditions are approximately $m>84$, $316$ and $703$ for $k=0,1,2$ respectively, whereas in Figure~\ref{fig:vdown_ex2} the corresponding Hopf bifurcation curves have minimal $m$ values of approximately $m=34$, $147$, and $336$. Therefore, at least in this case, the sufficient condition is not tight.
No Hopf bifurcations are seen for $m<34$ in panel (d), which explains why no Hopf
bifurcations were detected for $m=32$ in (c).

As $m$ increases, the stable periodic orbits created in the Hopf bifurcation at $\gamma=0.13\in(\gamma_3,\mu)$ are shown in Figure~\ref{fig:vdown_ex2}(e) and (f). Panel (e) shows the profiles of these orbits, while in (f) we show their projection into the $(x(t),x(t-\hat\tau))$ plane. Note that since the delay $\tau$ is state-dependent we do not use the actual delay for this projection, but we took $\hat\tau=\tau(\theta_v)$.
The limiting behaviour as $m\to\infty$ is quite different from the constant delay case considered in Section~\ref{sec:gdownvconst}. Not only is the shape of the profile different, but comparing Figure~\ref{fig:vdown_ex2}(f) and Figure~\ref{fig:gdown_ex1}(f) we see that the direction of rotation of the orbits, as indicated by the arrows, is reversed in the two examples.

\begin{figure}[tp!]
	\centering
	\vspacefig
	\includegraphics[scale=0.5]{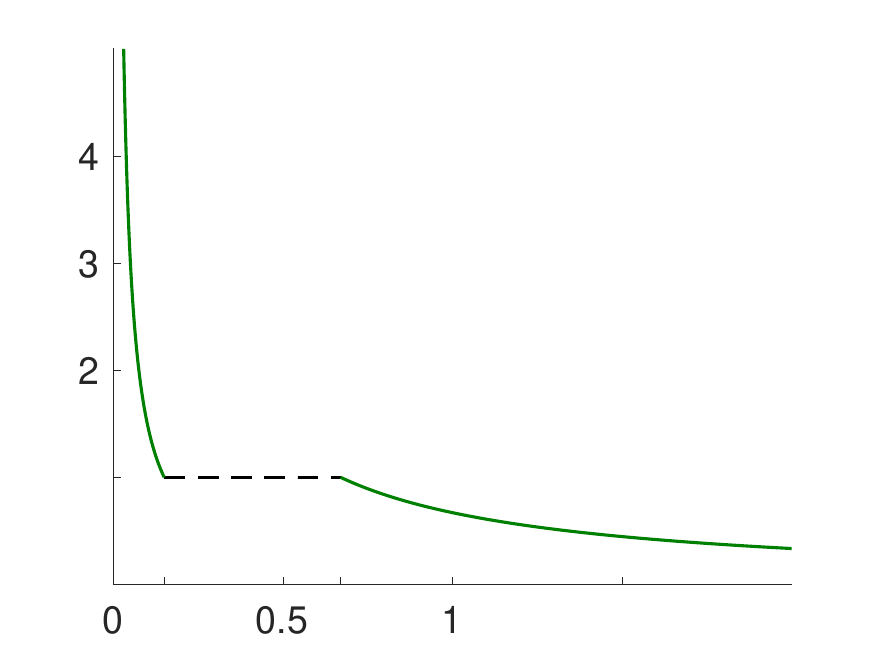}\hspace*{0.5em}\includegraphics[scale=0.5]{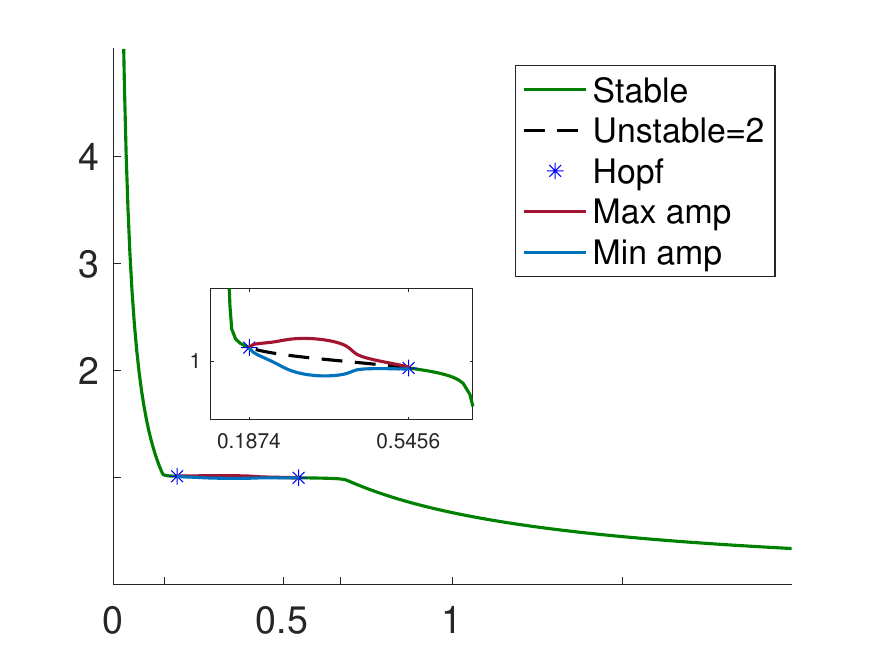}
	\put(-406,140){\rotatebox{90}{$x$}}
	\put(-412,40){$\theta_v$}
	\put(-240,10){$\gamma$}
	\put(-392,7){$\gamma_3$}
	\put(-349,7){$\gamma_4$}
	\put(-280,7){$\mu$}
    \put(-333,135){$(a)$}
	\put(-140,135){$(b)$}
	\put(-190,140){\rotatebox{90}{$x$}}
	\put(-196,40){$\theta_v$}
	\put(-24,10){$\gamma$}
	\put(-176,7){$\gamma_3$}
	\put(-132,7){$\gamma_4$}
	\put(-64,7){$\mu$}\\
	\vspace*{-0.5em}
	\includegraphics[scale=0.5]{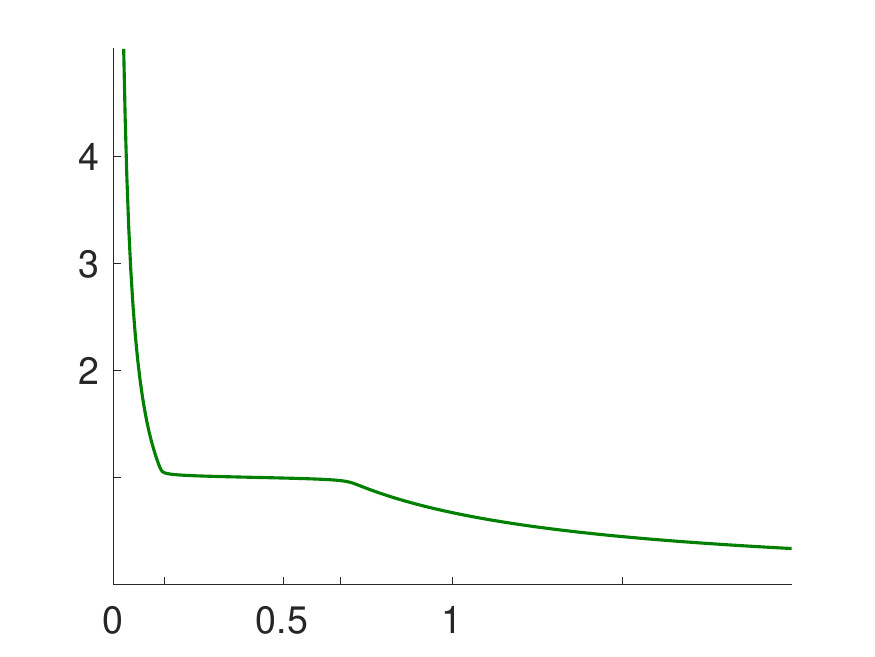}\hspace*{0.5em}\includegraphics[scale=0.5]{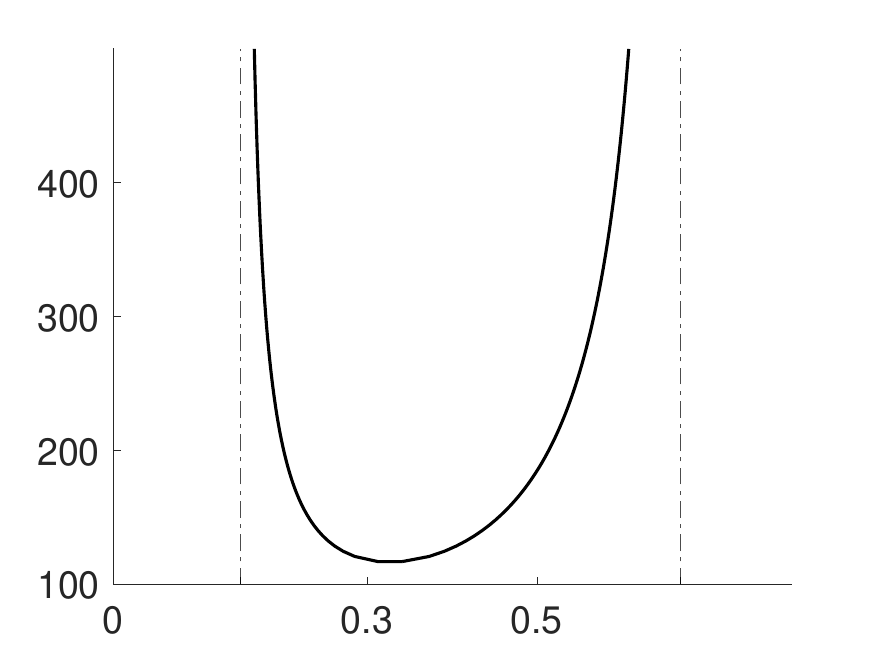}
	\put(-406,140){\rotatebox{90}{$x$}}
	\put(-412,40){$\theta_v$}
	\put(-240,10){$\gamma$}
	\put(-392,7){$\gamma_3$}
	\put(-349,7){$\gamma_4$}
	\put(-280,7){$\mu$}
    \put(-333,135){$(c)$}
	\put(-125,135){$(d)$}
	\put(-190,140){\rotatebox{90}{\footnotesize$m$}}
	\put(-24,10){$\gamma$}
	\put(-157,7){$\gamma_3$}
	\put(-50,7){$\gamma_4$}\\
	\vspace*{-0.5em}
	\includegraphics[scale=0.5]{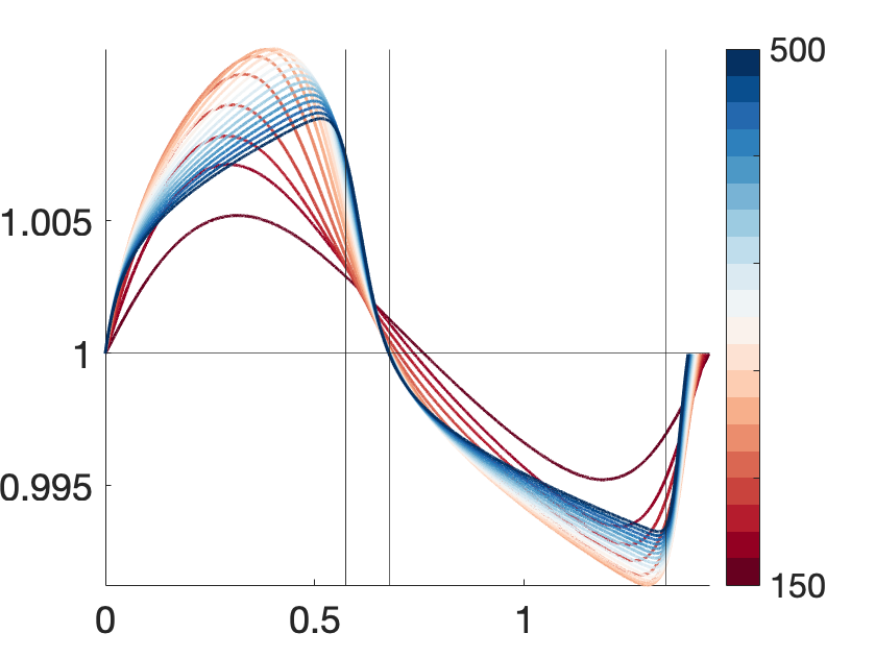}\hspace*{0.5em}\includegraphics[scale=0.5]{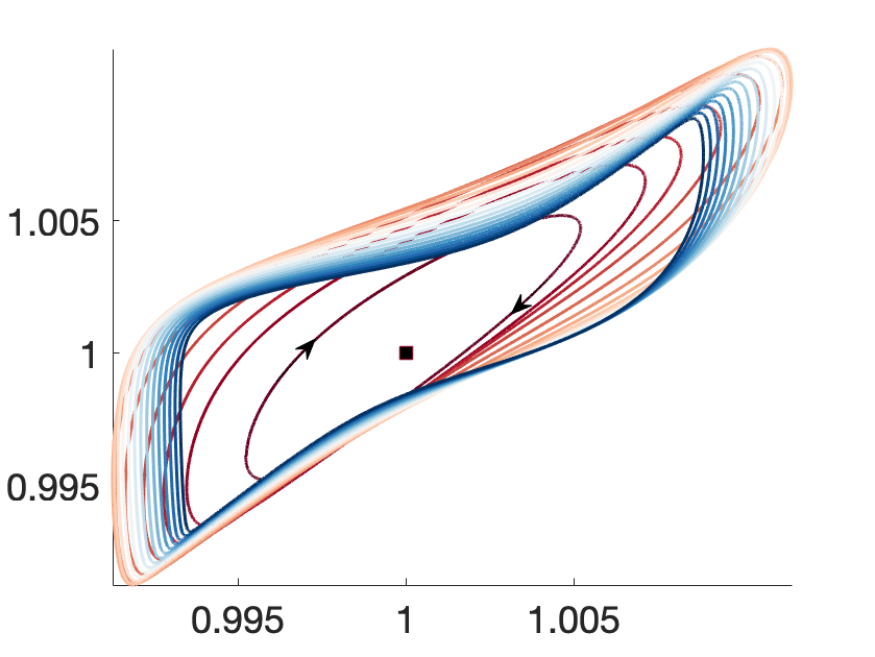}
	\put(-408,140){\rotatebox{90}{$x$}}
	\put(-260,10){$t$}
   \put(-240,80){$m$}
    \put(-315,135){$(e)$}
	\put(-140,135){$(f)$}
	\put(-195,110){\rotatebox{90}{$x(t-\hat\tau)$}}
	\put(-34,8){$x(t)$}
	\caption{Bifurcation diagram of \eqref{eq:basic}-\eqref{eq:thres}
		with $(g \leftrightarrow, v \downarrow)$ and parameters $\beta=3, \mu=1.5, g^-=g^+=1, \gamma=1, \theta_v=1, a=2, v^-=2$ and $v^+=1$.
		(a) The limiting case with $v$ defined by \eqref{eq:vpwconst}. The stable steady state is shown as a green solid line, and the singular steady state as a black dashed line.
		(b) With a smooth velocity nonlinearity $v$ defined by \eqref{eq:vghill} with $m=250$.
		Solid lines represent stable objects including the stable steady state (in green) and a stable limit cycle (represented by maximum and minimum of $x(t)$ on the periodic solution). Dashed lines represent unstable steady states which have two eigenvalues with positive real part (in black).
		(c) As in (b) but with $m=100$.		
		(d) Two-parameter continuations in $m$ and $\gamma$ of the Hopf bifurcations seen in (b). The dash-dotted lines denote $\gamma=\gamma_3=0.1494$ and $\gamma=\gamma_4=0.6694$, the location of the Hopf bifurcations in the limiting case as $m\to\infty$.
		(e) Profile of the stable periodic orbits from the Hopf bifurcations in (b) at $\gamma=0.4$.  The color map indicates values of the continuation parameter $m$.
		(f) Projection of the phase space dynamics into the $(x(t),  x(t-\hat\tau))$ plane at $\gamma=0.4$ where $\hat\tau=\tau(\theta_v)=4/3$. The arrow indicates the direction of the flow. The square marks the unstable steady state in the limiting case at the threshold. }
	\label{fig:vdown_ex1}
\end{figure}

{\bf Case 3:}
The final case to consider for  \eqref{eq:basic}-\eqref{eq:thres} with $(g \leftrightarrow, v \downarrow)$
is when $\gamma_3 < \gamma_4 <\mu$, and this is illustrated in Figure~\ref{fig:vdown_ex1}.
Panels (a)-(c) show the dynamics for the limiting case with the piecewise constant velocity nonlinearity \eqref{eq:vpwconst} and for the smooth case with the Hill function from \eqref{eq:vghill} with $m=250$ and $m=100$ respectively. For the limiting case in (a) the singular steady state with $x=\theta_v$ appears as a straight line segment for $\gamma\in(\gamma_3,\gamma_4)$. The singular steady states become unstable steady states in the smooth case as shown in (b), and there is a bubble of stable periodic orbits between the pair of Hopf bifurcations.

Figure~\ref{fig:vdown_ex1}(d) shows the locus of the principal Hopf bifurcation for a two-parameter continuation in $m$ and $\gamma$. Below and outside the curve of Hopf bifurcations the steady state is stable, while between the Hopf bifurcations it is unstable. We see that the Hopf bifurcation curves approach $\gamma_3$ and $\gamma_4$ as $m\to\infty$, as follows from Theorem~\ref{thm:vdown}. Since $\gamma_4<\mu$ and necessarily $\gamma(\theta_v)\in(\gamma_3,\gamma_4)$ we are guaranteed that
$\gamma(\theta_v)<\mu$, so for all such examples Theorem~\ref{thm:vdown} guarantees that there are infinitely many Hopf bifurcations. However, they may occur for very large values of $m$. For the example depicted in
Figure~\ref{fig:vdown_ex1}, only one Hopf bifurcation is observed with $m\geq117$. For the parameter values of this example the sufficient condition \eqref{eq:msuffvdown} (with $k=0,1$) ensures a first Hopf bifurcation for $m>152$, while a second
one is ensured for $m>602$, which is outside the parameter range considered in Figure~\ref{fig:vdown_ex1}.

If the value of the Hill coefficient $m$ is reduced sufficiently then there are no longer any Hopf bifurcations and the steady state remains stable, as seen in Figure~\ref{fig:vdown_ex1}(c). Interestingly, even though there are no longer any Hopf bifurcations in this case, the dependence of the steady state $x$ on $\gamma$ is still very similar to (a) and (b), with an obvious plateau visible in the graph of $x$ in Figure~\ref{fig:vdown_ex1}(c) for $\gamma\in(\gamma_3,\gamma_4)$.

Figure~\ref{fig:vdown_ex1}(e) and (f) respectively show the profiles and the $(x(t), x(t-\hat\tau))$-space projections of the stable periodic orbits at $\gamma=0.4$ as $m$ increases, where we choose $\hat\tau=\tau(\theta_v)$.
The periodic orbits approach a certain structure as $m\to\infty$.
Note that the direction of rotation in both Figure~\ref{fig:vdown_ex1}(f) and  Figure~\ref{fig:vdown_ex2}(f) is clockwise, while in Figure~\ref{fig:gdown_ex1}(f) it is counterclockwise.

\subsection{State-dependent delay with $v$ increasing and $g$ constant $(g\leftrightarrow,v\uparrow)$}
\label{sec:gconstvup}

Here we again consider the model \eqref{eq:basic}-\eqref{eq:thres} with a constant function $g(\xi)=g^\pm$ and
state-dependent delay, but now we consider the case of increasing $v$, so $v^-<v^+$ in \eqref{eq:vghill} and \eqref{eq:vpwconst}:  $(g \leftrightarrow, v \uparrow)$.

\begin{figure}[htbp!]
	\centering
	\begin{tabular}{cc}
		\begin{tikzpicture}[scale = 1.5]
		\tikzstyle{line} = [-,very thick]
		\tikzstyle{dotted line}=[.]
		\tikzstyle{arrow} = [->,line width = .4mm]
		\tikzstyle{unstable} = [red]
		\tikzstyle{stable} = [blue]
		\tikzstyle{pt} = [circle,draw=black,fill = black,minimum size = 1pt];
		\tikzstyle{bifpt} = [circle,draw=red,fill = red,minimum size = 1pt];
		\tikzstyle{upt} = [circle,draw=black,minimum size = 1pt];
		\def\arrlen{.3}

		\draw[line] (0,1) to (1.2,1);
		\draw[line] (1.2,1) to (1.2,2);
		\draw[line] (1.2,2) to (3,2);
		\draw[dotted line] (0,2) to (1.2,2);
		\draw[dotted line] (1.2,0) to (1.2,2);
		
		\draw[line] (0,0.125) to (1,2.5);
		\draw[line] (0,0) to (1.5,2.5);
		\draw[line] (0,-0.125) to (2,2.5);
		\draw[line] (0,-0.25) to (2.5,2.5);
		\draw[line] (0.1,-0.25) to (3,2.4);
		
		\node at (1.2,-.2) {$\theta_g$};

		\node at (-.6,2) {$\beta e^{-\mu\tau^+}g$};
		\node at (-.6,1) {$\beta e^{-\mu\tau^-}g$};
		
		\node[bifpt] at (1.2,1) (8) [] {};
		\node[bifpt] at (1.2,2) (8) [] {};
		\node[upt] at (1.2,1.45) (8) [] {};
		\node[pt] at (1.62,2) (8) [] {};
		\node[pt] at (2.05,2) (8) [] {};
		\node[pt] at (2.55,2) (8) [] {};
		\node[pt] at (0.88,1) (12) [] {};
		\node[pt] at (0.62,1) (12) [] {};
		\node[pt] at (0.36,1) (12) [] {};
\end{tikzpicture}		
\end{tabular}
\caption{Illustration of how the number of steady states of \eqref{eq:basic}, which are given by \eqref{eq:h}, changes with the intersections of $\xi\mapsto\beta e^{-\mu\tau(\xi)}g$ and $\xi\mapsto\gamma\xi$.  These are shown in the limiting case of \eqref{eq:gpwconst} and \eqref{eq:vpwconst} with 
$g^\pm=  g^-=g^+$, so $g(\xi)=g^\pm$ is a constant function, 
and $v^-<v^+$ so $v$ is 
piecewise constant 
and monotonically increasing: $(g\leftrightarrow,v\uparrow)$. Then $\tau(\xi)=a/v(\xi)$
is state-dependent; $\tau(\xi)=\tau^-=a/v^-$ for $\xi<\theta_v$, $\tau(\xi)=\tau^+=a/v^+$ for $\xi>\theta_v$ and $\tau(\xi)$ is set-valued when $\xi=\theta_v$.}
	\label{fig:4}
\end{figure}

As illustrated in Figure~\ref{fig:4}, it is possible for up to three steady states to coexist due to the fold bifurcations at the corners associated with $\gamma=\gamma_3$ and $\gamma=\gamma_4$.
Note that in Section~\ref{sec:gconstvdown} we  had $\gamma_3 < \gamma_4$. Since the function $v$ is increasing in this section, it follows from the  definition \eqref{eq:corners_gconst} that now $\gamma_4 < \gamma_3$.
In the limiting case where $v$ is defined in \eqref{eq:vpwconst},
at the non-singular steady states
the characteristic equation is of the form \eqref{eq:char_A0g'0}, and hence
these steady states are stable. For $\gamma\in(\gamma_4,\gamma_3)$ the two stable steady states co-exist with a singular steady state, at which the characteristic function is not defined. This leads us to consider the smooth velocity nonlinearity \eqref{eq:vghill}.

Because the smooth velocity nonlinearity approaches the step function shown in Figure~\ref{fig:4} as $m\to\infty$, it follows that for all $m$ sufficiently large there will be a pair of fold bifurcations at
$(\xi,\gamma)=(\xi^-,\gamma(\xi^-))$ and $(\xi,\gamma)=(\xi^+,\gamma(\xi^+))$,
with $(\xi^-,\gamma(\xi^-))\to(\theta_v,\gamma_4)$ and
$(\xi^+,\gamma(\xi^+))\to(\theta_v,\gamma_3)$ as $m\to\infty$.
To study the associated bifurcations consider
the characteristic equation which has the form \eqref{eq:char_Ag'0}.
Writing the characteristic function $\Delta(\lambda)$ as
\begin{equation} \label{eq:charfn54}
	\Delta(\lambda) = \lambda+\gamma-A+Ae^{-\lambda\tau(\xi)}-\mu A\int_{-a/v(\xi)}^{0} e^{\lambda s} ds,
\end{equation}
it follows from \eqref{eq:h} that at a steady state
\begin{align}  \notag
	\Delta(0)  & = \gamma-\mu A\int_{-a/v(\xi)}^{0} 1 ds = \gamma - \mu\tau(\xi) A
= \gamma - \mu\tau(\xi)\gamma\frac{\xi v'(\xi)}{v(\xi)} \\
& = \gamma\mu\tau(\xi)\Big(\frac{1}{\mu\tau(\xi)}-\frac{\xi v'(\xi)}{v(\xi)}\Big).
\label{eq:charfn540}
\end{align}
Hence, $\lambda=0$ is a characteristic value if and only if
\be  \label{eq:vuplam0}
	\frac{\xi v'(\xi)}{v(\xi)} - \frac{1}{\mu\tau(\xi)} = 0.
\ee
In addition, note that
\be \label{eq:vupposlam}
\frac{\xi v'(\xi)}{v(\xi)} > \frac{1}{\mu\tau(\xi)}
\ee
implies that $\Delta(0)<0$. At the same time,  it is easily seen that for real $\lambda$ the characteristic function satisfies $\Delta(\lambda)\to+\infty$ as $\lambda\to+\infty$. We conclude that  when \eqref{eq:vupposlam} is satisfied there is a real positive characteristic value and hence the steady state is unstable.

We now show that fold bifurcations occur when \eqref{eq:vuplam0} is satisfied. To that end, consider the curve of
solutions $(\gamma(\xi),\xi)$ where $\gamma(\xi)$ is defined by \eqref{eq:gamv}, which by \eqref{eq:h} is  the locus of steady states in $(\gamma,\xi)$-plane. Differentiating this
relationship with respect to $\xi$ we find
$$\gamma'(\xi)
= -\frac{\gamma(\xi)}{\xi}\Big(1+\mu\xi\tau'(\xi)\Big).
$$
However, $\tau(\xi)=a/v(\xi)$ implies that
\be \label{eq:taudashxi}
\tau'(\xi) = -\frac{av'(\xi)}{v(\xi)^2}=-\frac{v'(\xi)\tau(\xi)}{v(\xi)}.
\ee
Thus
\be \label{eq:gammadashvup}
\gamma'(\xi)
 = -\frac{\gamma(\xi)}{\xi}\Big(1-\mu\xi\frac{v'(\xi)\tau(\xi)}{v(\xi)}\Big)
 = \frac{\mu\tau(\xi)\gamma(\xi)}{\xi}\Big(\frac{\xi v'(\xi)}{v(\xi)} - \frac{1}{\mu\tau(\xi)}\Big).
\ee
Hence $\sgn(\gamma'(\xi))=\sgn\Big(\frac{\xi v'(\xi)}{v(\xi)} - \frac{1}{\mu\tau(\xi)}\Big)$,
and $\gamma'(\xi)$ changes sign when the left-hand side of \eqref{eq:vuplam0} changes sign, which ensures
that fold bifurcations occur at these points in the $(\gamma,\xi)$-plane.

While the above analysis appears similar to the constant delay case considered in Section~\ref{sec:gupvconst}, the state-dependency of the delay introduces two significant differences.

First, with a state-dependent delay and increasing $v$, fold bifurcations occur when \eqref{eq:vuplam0} is satisfied, whereas in Section~\ref{sec:gupvconst} with increasing $g$ we found fold bifurcations when \eqref{eq:xigdashgfold} is satisfied.
Thus for the constant delay case the fold bifurcations occur when the unimodal function $\xi\mapsto f(\xi/\theta_g,n,g^-/g^+)=\xi g'(\xi)/g(\xi)$ is equal to $1$, and so there is at most one pair of fold bifurcations. In contrast, for the state-dependent delay case, fold bifurcations occur when the unimodal function \eqref{eq:xivdashvf} is equal to $1/\mu\tau(\xi)$. For increasing $v$, the function $\xi\mapsto 1/\mu\tau(\xi)$ is also an increasing function, thus the fold bifurcation occurs at the intersections of a
unimodal function with an increasing function. Clearly, such functions may have multiple pairs of intersections and hence multiple pairs of fold bifurcations. From Proposition~\ref{prop:fxpr}, for the Hill function $v(\xi)$ defined by \eqref{eq:vghill} we have that 
$\xi\mapsto\xi v'(\xi)/v(\xi)$ is monotonically decreasing for $\xi>\theta_v(v^-/v^+)^{1/2m}$, so there will be at most one intersection with the increasing function $\xi\mapsto1/\mu\tau(\xi)$ on this interval. Thus if there are additional fold bifurcations they must occur for $\xi<\theta_v(v^-/v^+)^{1/2m}<\theta_v$. In the current work, we will not look for these additional fold bifurcations, nor will we show that they cannot exist. Even if it were possible to show that they cannot exist when $v$ is defined by \eqref{eq:vghill}, we would expect that additional fold bifurcations could be induced by suitable modifications to the velocity nonlinearity.

A second difference between the state-dependent and constant delay cases concerns the stability of the steady state near  the fold bifurcation.
Note that the argument after \eqref{eq:xivdashvmutau} in Section~\ref{sec:gconstvdown} still applies to show that the steady state is stable if $\gamma\leq\mu$ and \eqref{eq:xivdashvmutau} holds. Consequently, a fold bifurcation which occurs for
$\gamma<\mu$ must involve a stable steady state which loses stability at the fold bifurcation. However, this argument does not apply when $\gamma>\mu$, and we will see examples where stable steady states lose stability in a Hopf bifurcation with
$\gamma>\mu$, and also fold bifurcations where the steady state is unstable  on  both sides of the bifurcation (with different numbers of unstable characteristic values). This contrasts with the constant delay case in Section~\ref{sec:gupvconst} where we
showed that a stable steady state can only lose stability at a fold bifurcation.

As was the case when $v$ was decreasing, Hopf bifurcations are governed by equations \eqref{eq:sin_cancelled} and \eqref{eq:cos_cancelled}. However, since $v$  is increasing we now have $A>0$.
At a Hopf bifurcation the right-hand side of \eqref{eq:cos_cancelled} is strictly positive, and hence
$\sin(\omega\tau)>0$.
Since $\sin(\omega\tau)>0$ implies $\cos(\omega\tau)<1$, it then follows that
the left-hand side of \eqref{eq:sin_cancelled} is also strictly positive,
and a Hopf bifurcation is only possible if the
right-hand side is also strictly positive, that is if
\be \label{eq:gam>mu}
\gamma>\mu.
\ee

As in Section~\ref{sec:gconstvdown}, we find Hopf bifurcations by sequentially
solving \eqref{eq:gamv}, \eqref{eq:tanU} and \eqref{eq:AU}.
Equation \eqref{eq:tanU} has  infinitely many solutions $\omega_k(\gamma(\xi))$
for $k=0,1,\ldots$ with $\omega_k\tau\in(2k\pi,(2k+1)\pi)$.
Finally, to solve \eqref{eq:AU} with $v'(\xi)>0$, we need to determine $\xi$ that satisfies
$f(\xi/\theta_v,m,v^-/v^+) = r_k(\gamma(\xi))>0$ (defined in \eqref{eq:rk}).
Usually we will take $\gamma=\gamma(\xi)$, defined by \eqref{eq:gamv}, but in the proposition below we consider
$r_k(\gamma)$ for general $\gamma$.

\begin{prop} \label{prop:rgammaup}
Let $r_k(\gamma)$ be defined by \eqref{eq:rk} and $\omega_k$ satisfy
$\omega_k\tau\in(2k\pi,(2k+1)\pi)$ for integer $k\geq0$. Then for $\gamma>\mu$ it holds that $r_k(\gamma)>0$ and
\begin{enumerate}
\item
$r_k(\gamma)$ is monotonically decreasing,
\item
$r_k(\gamma)\to+\infty$ as $\gamma\to\mu$ and $r_k(\gamma)\to1/2$ as $\gamma\to+\infty$,
\item
$$r_k(\gamma) \leq \frac{(2k+1)^2(\pi/\tau^+)^2+\gamma^2}{2\gamma(\gamma-\mu)}.$$
\end{enumerate}
\end{prop}

\begin{proof}
The proof is elementary.
\end{proof}

The following lemma will be needed in the proof of Theorem~\ref{thm:vup}.

\begin{lem} \label{lem:lambertw}
For $x>0$ let $f(x)=\tfrac1x(1-e^{-x})$ then $f(x)$ is strictly monotonically decreasing for $x>0$
with
\be \label{eq:flamw}
1=\lim_{x\to0}f(x) > \lim_{x\to+\infty}f(x)=0.
\ee
\end{lem}

\begin{proof}
The two limits in \eqref{eq:flamw} are easily verified. To show the monotonicity of $f(x)$ first differentiate to obtain 
$$f'(x)=-\frac{1}{x^2}\Bigl[1 -(x+1)e^{-x}\Bigr]$$
The assumption $f'(x)=0$ for some $x>0$ yields $e^x=1+x$ which is impossible for $x\neq0$. Using continuity we infer that on $(0,\infty)$ there is no sign change of $f'(x)$, and that $f'(x)$ is strictly monotonic for $x>0$. 
Because of \eqref{eq:flamw} $f$ is decreasing. 
\end{proof}

The following theorem collects our results for the case of increasing $v$
with constant $g$.

\begin{thm} \label{thm:vup}
Let $\xi$ be a steady state of the DDE \eqref{eq:basic},\eqref{eq:thres} with $g(\xi)=g^\pm$ constant,
and $v$ monotonically increasing, so $v^-<v^+$, and the state-dependent delay $\tau$ defined by \eqref{eq:vghill} evaluated at the steady state is $\tau(\xi)=a/v(\xi)$.
Then
\begin{enumerate}
\item
If $\frac{\xi v'(\xi)}{v(\xi)}<\frac{1}{\tau(\xi)\max\{\mu,\gamma\}}$
then the steady state  $\xi$ is asymptotically stable,
while if $\frac{\xi v'(\xi)}{v(\xi)}>\frac{1}{\mu\tau(\xi)}$ it is unstable.
\item
For any fixed $\xi\ne\theta_v$ let $\gamma= \gamma(m,\xi)$ be the value of $\gamma$ such that \eqref{eq:h} is satisfied and hence $\xi$ is a steady state.
Or, for any fixed
$\gamma$ with $0<\gamma<\gamma_4$ or $\gamma>\gamma_3$ let $\xi=\xi(m,\gamma)$ satisfy  \eqref{eq:h} and hence
be a steady state.
Then $\xi$ is asymptotically stable for all $m$ sufficiently large.
\item
The steady state $\xi$ is asymptotically stable
if $m\leq1/(\tau(\xi)\max\{\mu,\gamma\})$, or both  $m>1/(\tau(\xi)\max\{\mu,\gamma\})$ and
$1\geq \frac{v^-}{v^+}>\Bigl(1-\frac{2}{m\tau(\xi)\max\{\mu,\gamma\}+1}\Bigr)^2$.
\item
For any fixed $m>0$, and for $\xi\gg\theta_v$, or equivalently for $0<\gamma\ll\gamma_4$, the steady state is asymptotically stable.
For any fixed $m>1$, and for $0<\xi\ll\theta_v$ or equivalently for $\gamma\gg\gamma_3$ the steady state is asymptotically stable.
\item
If $m>1/(\mu\tau^+)$
and $\frac{v^-}{v^+}<\Bigl(1-\frac{2}{1+m\mu\tau^+}\Bigr)^2$ then
there exist $\xi^-<\theta_v(v^-/v^+)^{1/2m}<\xi^+$ with
$\gamma(\xi^-)<\gamma(\xi^+)$
such that as $\gamma$ is varied there is a branch of steady states
with $\xi<\xi^-$ and $\gamma>\gamma(\xi^-)$ and another branch
with $\xi>\xi^+$ and $\gamma<\gamma(\xi^+)$.
For $\gamma\in(\gamma(\xi^-),\gamma(\xi^+))$ these two branches of steady states co-exist with a branch of unstable steady states which exists between fold bifurcations at $(\xi,\gamma)=(\xi^-,\gamma(\xi^-))$ and $(\xi,\gamma)=(\xi^+,\gamma(\xi^+))$.
\item
Let $\mu<\gamma_3$.
\begin{enumerate}
\item
For any $k>0$, for all $m=m(k)$ sufficiently large there are
two families of (at least) $k+1$ Hopf bifurcations as $\gamma$ is varied. In the first family, the characteristic values
$\lambda_j=\alpha_j\pm i\omega_j$ with $\omega_j\tau\in(2j\pi,(2j+1)\pi)$ for $j=0,1,\ldots,k$
cross the imaginary axis from left to right as $\gamma$ is increased, while in the second family they cross the imaginary axis from right to left.
For all $j$ sufficiently large these bifurcations occur on the branch of unstable steady states between the fold bifurcations identified in (5).
\item
Let $\gamma\in(\max\{\gamma_4,\mu\},\gamma_3)$ be fixed. Then as $m$ is increased there is an infinite
sequence of Hopf bifurcations where the real part of $\lambda_k=\alpha_k\pm i\omega_k$ becomes positive
with $\omega_k\tau\in((2k\pi,(2k+1)\pi)$. All but finitely many of these bifurcations are located on the branch of unstable steady states between the fold bifurcations.
\end{enumerate}
\end{enumerate}
\end{thm}

\begin{proof}
The last part of (1) is shown after \eqref{eq:vupposlam}.
To establish asymptotic stability we show that $Re(\lambda)<0$ for all of the characteristic values that solve \eqref{eq:char_Ag'0}. From $\xi v'(\xi)/v(\xi)<1/\mu\tau(\xi)$ and \eqref{eq:charfn540} it follows immediately that $\lambda=0$ is not a characteristic value. To show that $\lambda>0$ is not a characteristic value, evaluating the integral in \eqref{eq:charfn54} we obtain
$$\Delta(\lambda)=\lambda+\gamma-(1-e^{-\lambda\tau(\xi)})\frac{A}{\lambda}(\lambda+\mu).$$
But
$$A=\gamma\frac{\xi v'(\xi)}{v(\xi)}<\frac{\gamma}{\tau(\xi)\max\{\mu,\gamma\}},$$
hence
$$\Delta(\lambda)
>\lambda+\gamma-\frac1{\lambda\tau(\xi)}(1-e^{-\lambda\tau(\xi)})\frac{\gamma}{\max\{\mu,\gamma\}}(\lambda+\mu)
>\lambda+\gamma-\frac{\gamma}{\max\{\mu,\gamma\}}(\lambda+\mu),$$
where the last inequality follows from Lemma~\ref{lem:lambertw}. Now there are two cases to consider.
If $\gamma>\mu$ then
$$\Delta(\lambda)>\lambda+\gamma-(\lambda+\mu)=\gamma-\mu>0.$$
On the other hand,   
if $\gamma\leq\mu$ then
$$\Delta(\lambda)
>\lambda+\gamma-\frac{\gamma}{\mu}(\lambda+\mu)=\lambda\Bigl(1-\frac{\gamma}{\mu}\Bigr)\geq0.$$
In both cases the characteristic function satisfies $\Delta(\lambda)>0$ for all $\lambda>0$ so there are no real positive characteristic values.

To complete the proof of (1) it remains to show that there are no complex characteristic values
$\lambda=\alpha+i\omega$ with $\alpha\geq 0$ and $\omega>0$. For the case $\gamma\leq\mu$
the proof is the same as in Section~\ref{sec:gconstvdown} (where \eqref{eq:xivdashvmutau} holds independent of the sign of $v'(\xi)$).
For the remaining case, if
$\gamma>\mu$ the assumption of the theorem reads $\frac{\xi v'(\xi)}{v(\xi)}<\frac{1}{\gamma\tau(\xi)}$. Let
 $\lambda=\alpha+i\omega$ be  a characteristic value with $\alpha\geq 0$ and $\omega>0$ then
\begin{align*}
|Ae^{-\alpha\tau(\xi)}\sin{\omega\tau(\xi)}\big[(\alpha+\mu)^2+\omega^2\big]|
& \leq A\omega\tau(\xi)\big[(\alpha+\mu)^2+\omega^2\big]\\
& = \gamma\frac{\xi v'(\xi)}{v(\xi)}\omega\tau(\xi)\big[(\alpha+\mu)^2+\omega^2\big]\\
& < \omega\big[(\alpha+\mu)^2+\omega^2\big]\\
& < \omega\Big((2\alpha+\gamma)\mu+\alpha^2+\omega^2\Big).
\end{align*}
Consequently \eqref{eq:sin} is violated, and so there is no such characteristic value. Thus $\alpha<0$ for all characteristic values and the steady state is asymptotically stable.

Statement (2) is derived similarly to Theorem~\ref{thm:gdown}(4), by first using \eqref{eq:gamv} and
\eqref{eq:corners_gconst} to show that
$e^{-\mu\tau^-}<e^{-\mu\tau(\xi)}<e^{-\mu\tau^+}$
 implies that
$\gamma_4/\gamma<\xi/\theta_v<\gamma_3/\gamma$, then applying Proposition~\ref{prop:fxpr}
and (1).

Statement (3) follows from (1) using Proposition~\ref{prop:fxpr} (point 4), since
\begin{displaymath}
\frac{\xi v'(\xi)}{v(\xi)}\leq f(r_v^{1/2m},m,r_v)=\frac{m(1-r_v^{1/2})}{1+r_v^{1/2}},
\end{displaymath}
where $r_v=v^-/v^+\in(0,1)$.
Statement (4) also follows from (1), using Proposition~\ref{prop:fxpr} (point 1), similar to 
Theorem~\ref{thm:gdown} and Theorem~\ref{thm:vdown},
where for the second part of (4),
$m>1$ ensures that
$$f(\xi/\theta_v,m,v^-/v^+)\gamma(\xi)
=\frac{m(1-v^-/v^+)(\xi/\theta_v)^m}{\big(1+(\xi/\theta_v)^m\big)\big(v^-/v^+ +(\xi/\theta_v)^m\big)   }\,\frac{\beta g^\pm}{\xi}e^{-\mu\tau(\xi)}\to0\;\text{as}\;\xi\to0.$$

To show (5), note that the parameter constraints ensure that
$f(\xi/\theta_v,m,v^-/v^+)=\frac{\xi v'(\xi)}{v(\xi)}>\frac{1}{\mu\tau(\xi)}$ when
$\xi=\theta_v(v^-/v^+)^{1/2m}$. Then since $f(\xi/\theta_v,m,v^-/v^+)\to0$
as $\xi\to0$ and as $\xi\to\infty$ it follows that there exists
$\xi^+>\theta_v(v^-/v^+)^{1/2m}$ and $\xi^-<\theta_v(v^-/v^+)^{1/2m}$
which both solve \eqref{eq:vuplam0}. We take $\xi^-$ to be the largest $\xi$ which solves \eqref{eq:vuplam0} with $\xi^-<\theta_v(v^-/v^+)^{1/2m}$, while as discussed after
\eqref{eq:gammadashvup}, $\xi^+$ is unique.
Then it follows from
\eqref{eq:vuplam0}, \eqref{eq:gammadashvup} and the adjacent arguments
that there is a pair of fold bifurcations at
$(\xi,\gamma)=(\xi^-,\gamma(\xi^-))$ and $(\xi,\gamma)=(\xi^+,\gamma(\xi^+))$, which are connected by a branch of unstable steady states, and no other fold bifurcation (besides $\xi=\xi^+$) with $\xi>\xi^-$.
For $\xi>\xi^+$,  it follows from \eqref{eq:gammadashvup} that $\gamma'(\xi)<0$
and hence this branch exists for $\xi\in(\xi^+,\infty)$ and $\gamma\in(0,\gamma(\xi^+))$. If $\gamma(\xi^+)<\mu$ then the whole of this branch of steady states is stable, otherwise by (1) it is stable for all $\xi$ sufficiently large,
given by condition $\gamma(\xi)<\mu$.

For the branch which exists for $\xi\in(0,\xi^-)$,
 it follows  from \eqref{eq:gamv} that $\gamma(\xi)\to+\infty$ as $\xi\to0$.
Hence $\gamma(\xi)$ takes all values in $[\gamma(\xi^-),+\infty)$  for $\xi\in(0,\xi^-)$. However, as discussed after \eqref{eq:gammadashvup}, it is possible for this branch to have additional fold bifurcations. If there are no additional fold bifurcations and $\gamma(\xi^-)<\mu$ then the branch of steady states is stable for $\gamma\in(\gamma(\xi^-),\mu)$.
Also, if $m>1$ by (4) it is stable for all $\xi$ sufficiently small, or, equivalently, for all $\gamma(\xi)$ sufficiently large.

To show (6), first consider the case when $\gamma(\theta_v)$ defined by
\eqref{eq:gamthetav}
satisfies $\gamma(\theta_v)>\mu$. If
\be \label{eq:msuffvup}
m > 2\frac{1+(v^-/v^+)}{1-(v^-/v^+)}
\left( \frac{(2k+1)^2(\pi/\tau(\theta_v))^2+\gamma(\theta_v)^2}{2\gamma(\theta_v)(\gamma(\theta_v)-\mu)}\right),
\ee
then
\begin{align*}
f(1,m,v^-/v^+)=\frac{m(1-v^-/v^+)}{2(1+v^-/v^+)}
 & > \left( \frac{(2k+1)^2(\pi/\tau(\theta_v))^2+\gamma(\theta_v)^2}{2\gamma(\theta_v)(\gamma(\theta_v)-\mu)}\right)\\
& \geq r_j(\gamma(\theta_v)), \quad j=0,1,\ldots,k.
\end{align*}
Here the equality comes from the definition of $f$, the strict inequality from \eqref{eq:msuffvup} and the last inequality follows from a similar argument
as in the proof of Proposition~\ref{prop:rgammaup} (point 3), but using the actual value of $\tau(\theta_v)$ in the inequality, rather than the bound $\tau^+$.

\sloppy{With this inequality as the starting point, increasing $\xi$ away from $\theta_v$
we have
$\lim_{\xi\to\infty}\gamma(\xi)=0$. But $r_j(\gamma(\xi))$ and $f(\xi/\theta_v,m,v^-/v^+)$ are both
continuous functions of $\xi$ with
$\lim_{\xi\to\infty}r_j(\gamma(\xi))=\infty$ and
$\lim_{\xi\to\infty}f(\xi/\theta_v,m,v^-/v^+)=0$.
Consequently for each $j=0,1,\ldots,k$ there exists a $\xi$ such that
$f(\xi/\theta_v,m,v^-/v^+)=r_j(\gamma(\xi))$.}

Similarly, for $\xi$ decreasing from $\theta_v$ (see also the proof of Theorem~\ref{thm:vdown})
it follows for each $j=0,1,\ldots,k$ that there exists a $\xi$ such that
$f(\xi/\theta_v,m,v^-/v^+)=r_j(\gamma(\xi))$. As in the proof of
Theorem~\ref{thm:vdown} this defines the required Hopf points for
the two families of $k+1$ Hopf bifurcations.

This argument fails in the more delicate case where $\gamma_4<\gamma(\theta_v)\leq \mu <\gamma_3$. In this
case fix $\gamma\in(\mu,\gamma_3)$, and consider the behaviour as $m\to\infty$.
Since $\gamma(\xi^+)\to\gamma_4$ and $\gamma(\xi^-)\to\gamma_3$, it follows that
$\gamma\in(\gamma(\xi^+),\gamma(\xi^-))$ for all $m$ sufficiently large.
Then since $\gamma(\xi)$ defined by \eqref{eq:gamv} is monotonically increasing
for $\xi\in(\xi^-,\xi^+)$, this function is locally invertible and we can instead consider $\xi = \xi(\gamma) $ as a function of $\gamma$. From here the argument proceeds as in the proof of  Theorem~\ref{thm:vdown}(6).
\end{proof}

We now present several examples to illustrate the complex dynamics allowed by Theorem~\ref{thm:vup}. There are essentially three main cases to consider depending on
whether $\mu$ falls above, below or within the interval $(\gamma_4,\gamma_3)$. As in Section~\ref{sec:gconstvdown} we start with the simplest case where there are no Hopf bifurcations.
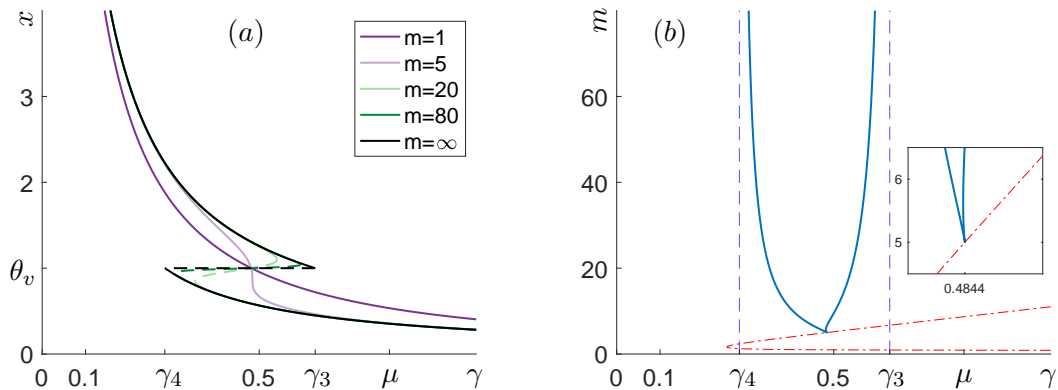
\begin{figure}[thp!]
	\centering

\includegraphics[scale=0.5]{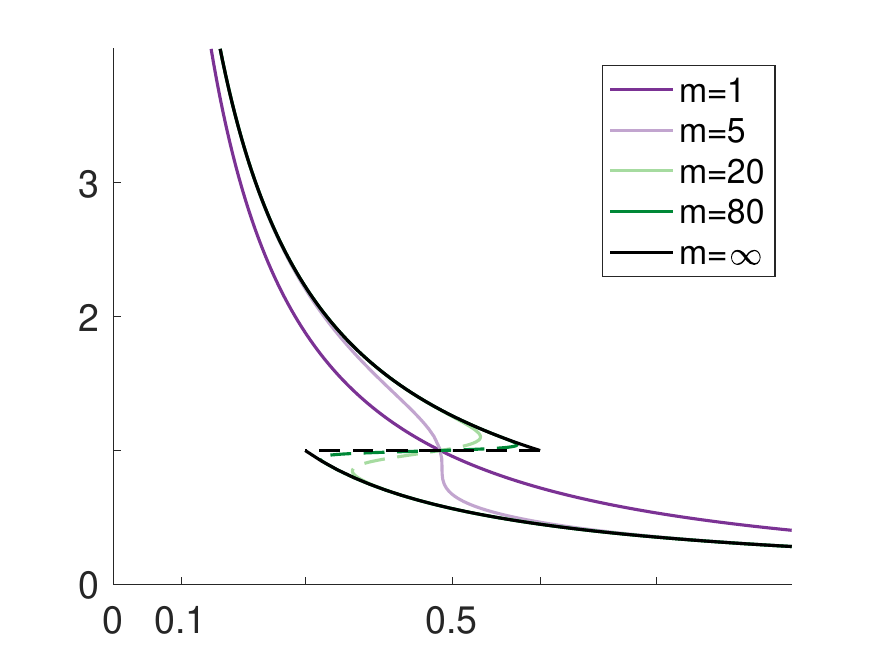}\hspace*{0.5em}\includegraphics[scale=0.5]{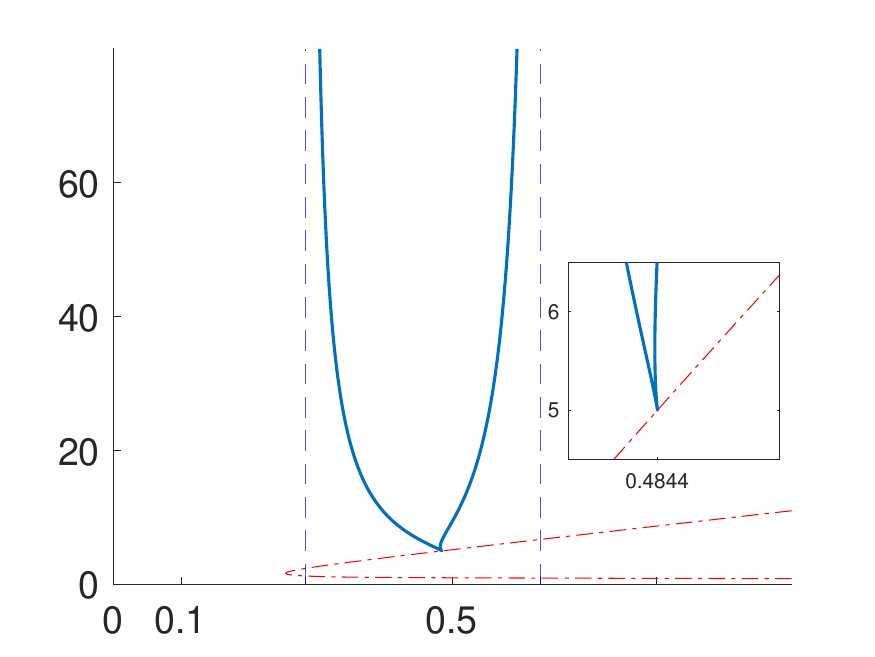}
	\put(-408,140){\rotatebox{90}{$x$}}
	\put(-412,45){$\theta_v$}
	\put(-240,7){$\gamma$}
	\put(-357,7){$\gamma_4$}
	\put(-301,7){$\gamma_3$}
	\put(-272,7){$\mu$}
	\put(-330,135){$(a)$}
	\put(-170,135){$(b)$}
	\put(-192,138){\rotatebox{90}{$m$}}
	\put(-24,7){$\gamma$}
	\put(-141,7){$\gamma_4$}
	\put(-85,7){$\gamma_3$}
	\put(-56,7){$\mu$}
	\caption{Bifurcation diagram of \eqref{eq:basic}-\eqref{eq:thres}
	with $(g \leftrightarrow, v \uparrow)$ and parameters $\beta=1.4$, $\mu=0.8$, $g^-=g^+=1$, $\gamma=\theta_v=a=1$, $v^-=0.5$ and $v^+=1$.
By \eqref{eq:corners_gconst} this implies
$\gamma_4 = 0.2827 < \gamma_3 = 0.6291 < \mu$ (Case 1).
(a) One parameter continuation of the steady state in $\gamma$ for the velocity functions in \eqref{eq:vpwconst} and in \eqref{eq:vghill} with different fixed values of $m$. For $m<5$ the steady state is always stable. For $m>5$ it is unstable between two fold bifurcations, and for large $m$ approaches the limiting case \eqref{eq:vpwconst} (denoted by $m=\infty$ in the figure).
(b) Two-parameter continuation of the fold bifurcations in the $(\gamma,m)$-plane,
with a
cusp point at $(\gamma, m)=(0.4844, 5.0002)$. The red dash-dotted curve denotes the bound on the fold bifurcations given by \eqref{eq:foldboundvup}.}
\label{fig:vup_ex3}
\end{figure}

{\bf Case 1:}  $\gamma_4 < \gamma_3 < \mu$.  This is illustrated in Figure~\ref{fig:vup_ex3}. Panel (a) shows one-parameter continuations of the steady-states in $\gamma$ for several fixed values of $m$, along with the limiting case \eqref{eq:vpwconst},
while panel (b) shows the two-parameter continuation of fold bifurcations of steady-states in the
$(\gamma,m)$ plane.
We see that for all $m$ sufficiently large there is a pair of fold bifurcations
at $(\xi,\gamma)=(\xi^-,\gamma(\xi^-))$ and $(\xi,\gamma)=(\xi^+,\gamma(\xi^+))$,
with $(\xi^-,\gamma(\xi^-))\to(\theta_v,\gamma_4)$ and
$(\xi^+,\gamma(\xi^+))\to(\theta_v,\gamma_3)$ as $m\to\infty$.
As required by Theorem~\ref{thm:vup}, the steady state is always unstable between the fold bifurcations, while we observe it to always be stable otherwise.

Similar to the analysis in Section~\ref{sec:gupvconst} that led to \eqref{eq:foldbound},
imposing $\max\{h'(x)\}\geq0$
leads to the necessary condition
\be \label{eq:foldboundvup}
\gamma \leq -\beta\mu\tau'(\xi)e^{-\mu\tau(\xi)}g^\pm
\ee
for the coexistence of three steady states. However, the algebra to turn this into an explicit condition (compare \eqref{eq:foldboundvup} with \eqref{eq:foldbound})
is challenging with a state-dependent delay \eqref{eq:steadydelay}, so instead we
apply this condition numerically.
The red dash-dotted curve in Figure~\ref{fig:vup_ex3}(b) depicts the condition \eqref{eq:foldboundvup}. For any fixed $m$, this curve provides an upper bound on $\gamma$ for the existence of fold bifurcations and hence multiple coexisting steady states.
By the same argument as in Section~\ref{sec:gupvconst} the cusp bifurcation must lie on this curve. This can be seen in Figure~\ref{fig:vup_ex3}(b) where the cusp point
$(\gamma, m)=(0.4844, 5.0002)$ lies on this bounding curve.

Because of \eqref{eq:gam>mu}, there can be no Hopf bifurcations for $\gamma<\mu$, and
since the fold bifurcations all occur for $\gamma<\mu$, there can be no Hopf bifurcations between the folds, as seen in Figure~\ref{fig:vup_ex3}.
For $\gamma>\mu$, Theorem~\ref{thm:vup} ensures that the steady state is stable for $\gamma$ sufficiently large or for $m$ sufficiently large. In Figure~\ref{fig:vup_ex3} we see that actually there are no Hopf bifurcations at all and the steady state is stable for all values of $\gamma>\mu$ and  for all $m>0$.

\begin{figure}[thp!]
	\centering
	\includegraphics[scale=0.5]{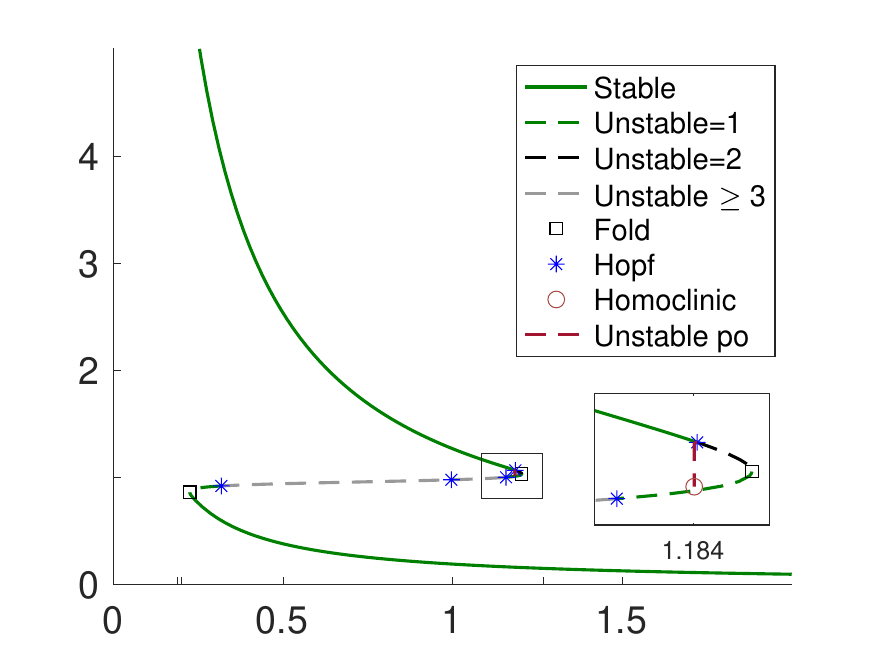}\hspace*{0.5em}\includegraphics[scale=0.5]{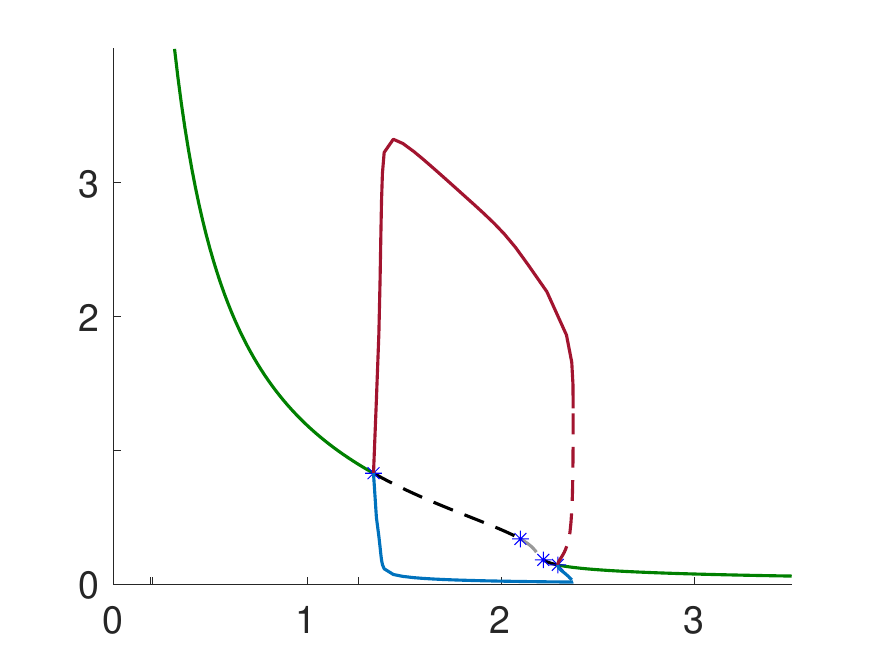}
	\put(-406,140){\rotatebox{90}{$x$}}
	\put(-412,40){$\theta_v$}
	\put(-388,7){$\gamma_4$}
	\put(-300,7){$\gamma_3$}
	\put(-250,8){$\gamma$}
	\put(-340,140){$(a)$}
	\put(-100,140){$(b)$}
	\put(-195,45){$\theta_v$}
    \put(-190,140){\rotatebox{90}{$x$}}
	\put(-25,8){$\gamma$}

	\caption{One parameter bifurcation diagrams of \eqref{eq:basic}-\eqref{eq:thres} with $(g \leftrightarrow, v \uparrow)$ for
       varying $\gamma$    and smooth increasing $v$ defined by \eqref{eq:vghill},
       with the exponent $m$ in the velocity nonlinearity defined by (a) $m=50$ and (b) $m=2$. The other parameters are taken to be
        $\beta=1.4$, $\mu=0.2$, $g^-=g^+=1$, $\theta_v=1$, $a=1$, $v^-=0.1$ and $v^+=2$.
        By \eqref{eq:corners_gconst} this implies
        $\gamma_4 = 0.1895 < \mu < \gamma_3 = 1.2668$ (Case 2). 
        Solid lines represent stable objects (with stable steady states shown in green), and
        dashed lines represent unstable objects (with unstable steady states shown in green, black or grey depending on the number of unstable characteristic values). In (a) a branch of unstable periodic orbits is represented by the 2-norm \eqref{eq:2norm} of the periodic solution (red dashed line in inset), while in (b) red and blue lines represent the maximum and minimum values of $x(t)$ on the periodic orbit. }
	\label{fig:vup_ex1a}
\end{figure}
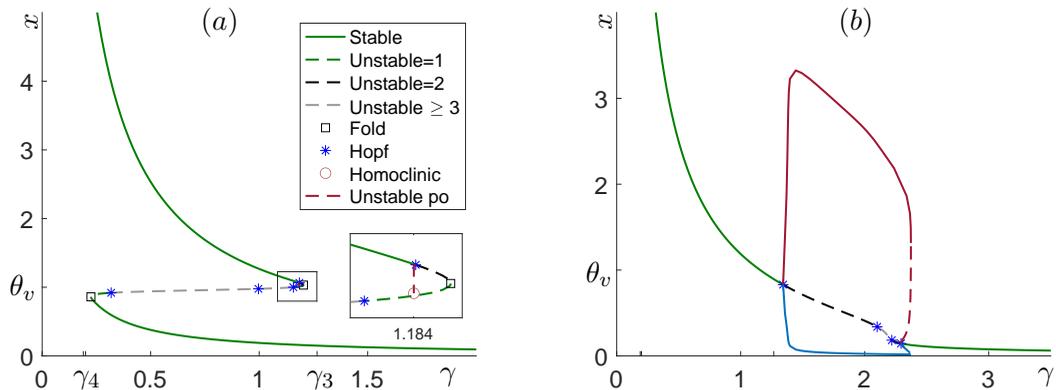

{\bf Case 2.} The case $\mu\in(\gamma_4,\gamma_3)$ is illustrated in Figures~\ref{fig:vup_ex1a}-\ref{fig:vup_ex1e}.
Figure~\ref{fig:vup_ex1a} shows two one-parameter continuations in $\gamma$ for different values of $m$. At first glance, Figure~\ref{fig:vup_ex1a}(a) for which $m\gg0$, looks much like the corresponding constant delay case with increasing $g$ illustrated in Figure~\ref{fig:gup_ex1}(b) in Section~\ref{sec:gupvconst}. In both of these examples there are two fold bifurcations, leading to an interval of bi-stability of steady states, and Hopf bifurcations between the fold bifurcations leading to unstable periodic orbits. However, as shown in the inset of Figure~\ref{fig:vup_ex1a}(a), for the state-dependent delay with $\mu<\gamma_3$ the upper branch of stable steady states, which exists for $x>\theta_v$, loses stability in a (subcritical) Hopf bifurcation, and not at the fold bifurcation. Recall that for the constant delay case considered in  Section~\ref{sec:gupvconst} we showed that the steady-state could only lose stability in a fold bifurcation.

Figure~\ref{fig:vup_ex1a}(b) illustrates the dynamics for a much smaller value of $m$, but with the other parameters unchanged. In this case $m$ is too small for fold bifurcations to occur, but we still find a pair of Hopf bifurcations, one supercritical and one subcritical, leading to a stable periodic orbit, and also a very narrow interval of bistability of this periodic orbit and a stable steady state between the subcritical Hopf bifurcation and a saddle-node of periodic orbits bifurcation. This is quite different from the dynamics seen before where, for decreasing $g$ or $v$ respectively in Sections~\ref{sec:gdownvconst} and~\ref{sec:gconstvdown}, only supercritical Hopf bifurcations were observed, and for increasing $g$ in Section~\ref{sec:gconstvdown} where Hopf bifurcations only occur between the pair of fold bifurcations.

The branches of periodic orbits in Figure~\ref{fig:vup_ex1a}(a) and (b) are represented differently. In Figure~\ref{fig:vup_ex1a}(a) we plot the $L_2$ norm of the periodic solution
of period $T$, defined as
\be \label{eq:2norm}
\|x\|_2=\left(\frac1T\int_{t=0}^{T}|x(t)|^2 dt\right)^{\!1/2}.
\ee
In contrast, in Figure~\ref{fig:vup_ex1a}(b) we display the branch of periodic orbits by
plotting both $\max x(t)$ and $\min x(t)$ over the periodic orbit, which clearly shows the amplitude of the periodic solution. Both representations can be useful on bifurcation diagrams since at a Hopf bifurcation all three expressions are equal to the steady state value and  therefore  the periodic orbits emanate from the steady states at Hopf bifurcations. However, the representation \eqref{eq:2norm} has the additional property
that $\|x\|_2\to x_s$ as the solution approaches a homoclinic orbit to a saddle steady state $x_s$. An example of this is seen in Figure~\ref{fig:vup_ex1a}(a) where the branch of unstable periodic orbits emanating from the subcritical Hopf bifurcation appears to terminate in a homoclinic bifurcation at the middle steady state. We will investigate
homoclinic bifurcations below, and so we will mainly use the $L_2$ norm \eqref{eq:2norm} to represent periodic orbits.

\begin{figure}[thp!]
	\centering
	\includegraphics[scale=0.5]{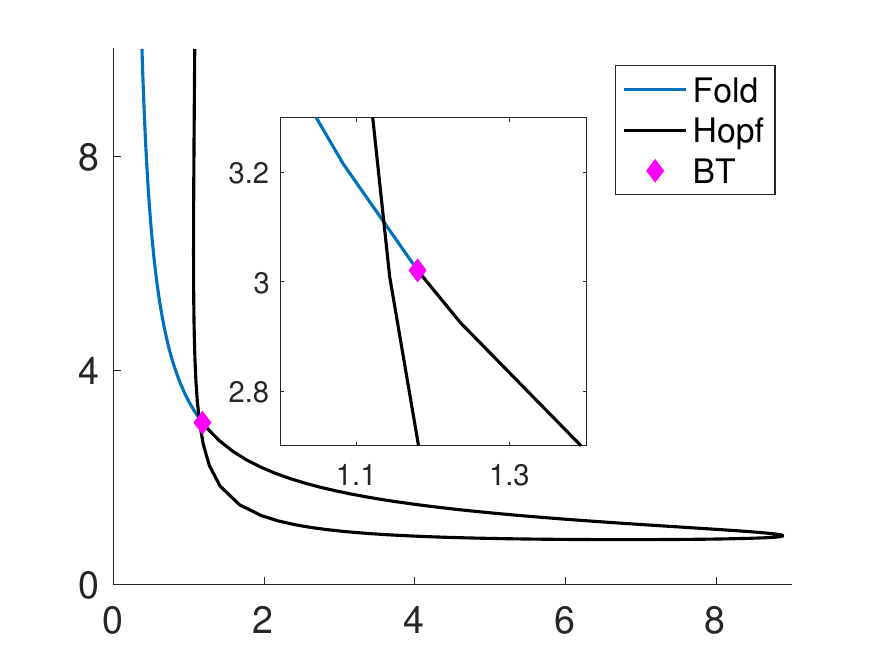}\hspace*{0.5em}\includegraphics[scale=0.5]{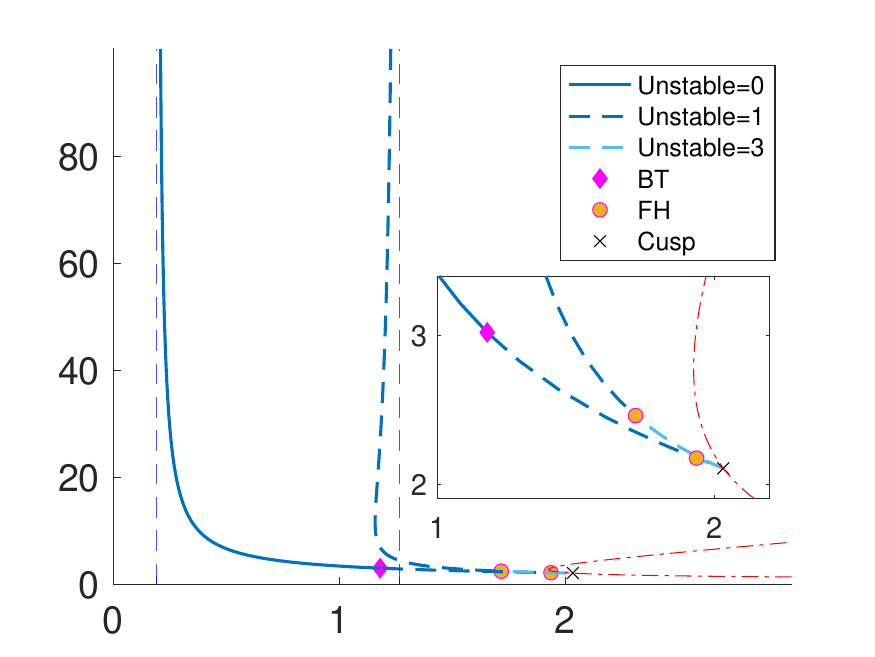}
	\put(-406,140){\rotatebox{90}{$m$}}
	\put(-247,8){$\gamma$}
	\put(-340,140){$(a)$}
	\put(-363,35){$0$}
	\put(-260,40){$1$}
	\put(-390,30){$1$}
	\put(-389,130){$2$}
	\put(-150,140){$(b)$}
	\put(-190,140){\rotatebox{90}{$m$}}
	\put(-40,8){$\gamma$}
	\put(-178,7){$\gamma_4$}
	\put(-120,7){$\gamma_3$}\\

\includegraphics[scale=0.5]{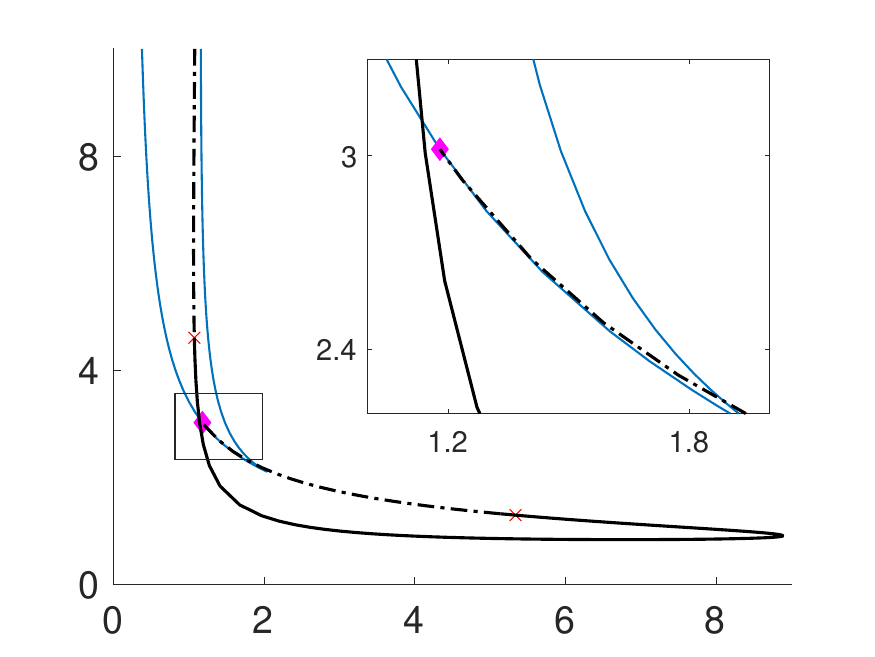}\hspace*{0.5em}\includegraphics[scale=0.5]{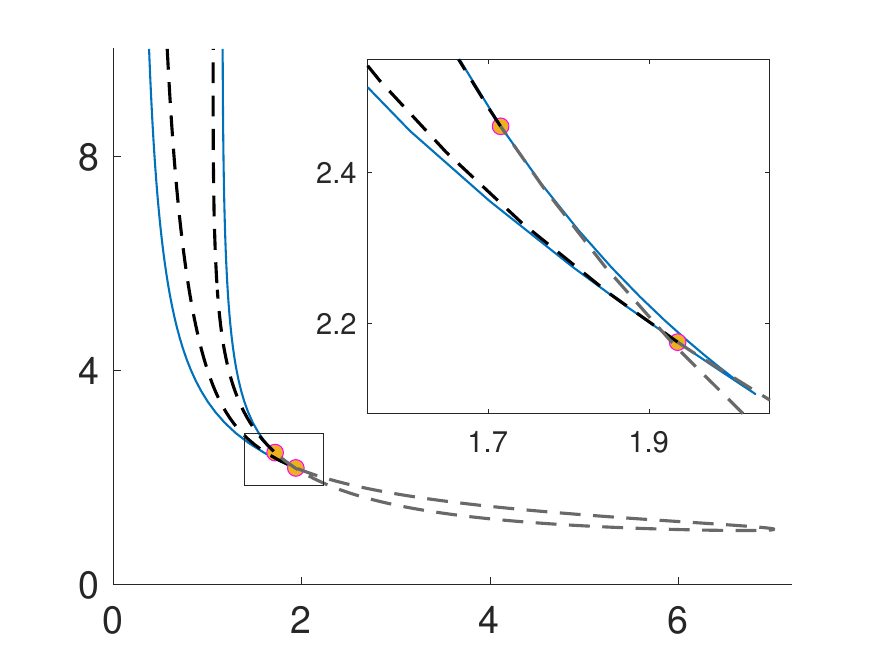}
	\put(-396,22){$(c)$}
	\put(-406,140){\rotatebox{90}{$m$}}
	\put(-244,8){$\gamma$}
	\put(-180,22){$(d)$}
	\put(-190,140){\rotatebox{90}{$m$}}
	\put(-30,8){$\gamma$}

	\caption{Two parameter (Case 2) bifurcation diagrams of \eqref{eq:basic}-\eqref{eq:thres}
      with $(g \leftrightarrow, v \uparrow)$ and smooth increasing $v$ defined by \eqref{eq:vghill}.
      The bifurcation parameters are $\gamma$ and $m$, with the other parameters for all the panels
      taken as stated in the caption to Figure~\ref{fig:vup_ex1a}.
       (a) The parameter space is divided into regions where there are $0$, $1$ or $2$ stable steady states as indicated. These regions are partly bounded by a curve of Hopf bifurcations (black),
       and partly by a curve of fold bifurcations (blue). These two curves meet in a Bogdanov-Takens (BT) bifurcation point at $(\gamma, m)=(1.1795, 3.0209)$. 
      (b) Two-parameter continuation of the fold bifurcations only. The fold curve is drawn according to the number of characteristic values with positive real part; solid blue for zero, dashed blue for one, and dashed light blue for three. The dashed vertical blue lines denote $\gamma=\gamma_3$ and $\gamma=\gamma_4$, given by \eqref{eq:corners_gconst}, the location of the fold bifurcations in the limiting case as $m\to\infty$. The red dash-dotted curve denotes the bound on the fold bifurcations given by \eqref{eq:foldboundvup}. There is a cusp point at $(\gamma, m)=(2.0321, 2.1058)$.
The same BT point is detected again as well as two fold-Hopf points at $(\gamma, m)=(1.7153, 2.4612)$ and $(\gamma, m)=(1.9354, 2.1748)$.  
(c) Continuation of the branch of Hopf bifurcations emanating from the BT point (in black), with the branch of fold bifurcations shown in blue for scale. These Hopf bifurcations always involve a stable steady state that loses stability, either in a supercritical Hopf bifurcation to a stable periodic orbit (the solid black curve), or in a subcritical Hopf bifurcation to an unstable periodic orbit
(the dotted-dashed black curve). Two Bautin bifurcations at $(\gamma, m)=(1.0721, 4.6069)$
and $(5.3352, 1.2929)$ where the criticality of the Hopf bifurcation changes are denoted by red stars.
(d) Continuation of the branch of Hopf bifurcations which passes through both fold-Hopf points. This bifurcation curve only involves unstable steady states.
The dashed black curve represents Hopf bifurcations with one unstable eigenvalue, and the dashed gray curve with two unstable eigenvalues. The change in number of eigenvalues with positive real part occurs at the fold-Hopf points.}
\label{fig:vup_ex1b}
\end{figure}
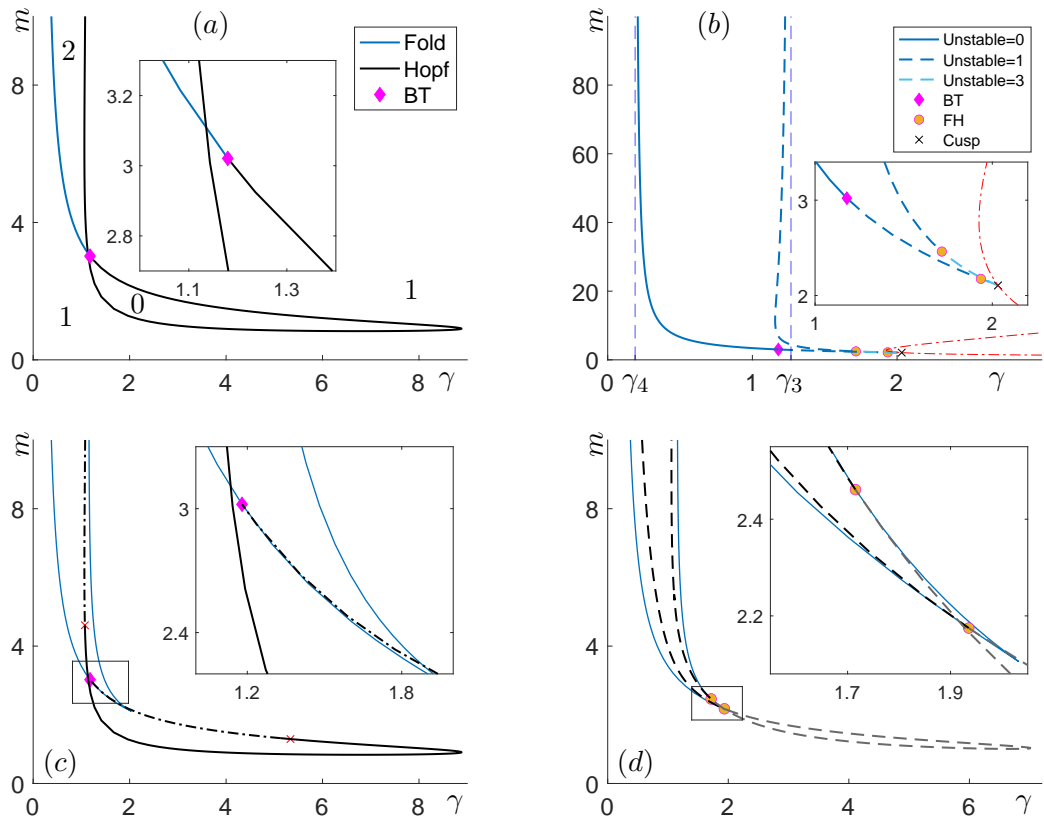

To investigate this example further in Figure~\ref{fig:vup_ex1b} we present two-parameter continuations of the bifurcations in $\gamma$ and $m$, with all the other parameter values the same as in Figure~\ref{fig:vup_ex1a}.

Figure~\ref{fig:vup_ex1b}(a) shows how the $\gamma$-$m$ parameter plane is divided into regions where there are $0$, $1$ or $2$ stable steady states. These regions are partly bounded by a curve of fold bifurcations, but mainly by a curve of Hopf bifurcations, with the two bifurcation curves meeting at a
Bogdanov-Takens (BT) bifurcation point, at which the characteristic equation has a double zero characteristic value.

Figure~\ref{fig:vup_ex1b}(b) shows the locus of the fold bifurcations, along with codimension-two bifurcations which occur on this branch. Different line types/colours are used to indicate the number of unstable characteristic values (not counting the zero characteristic value associated with the fold bifurcation itself).
Since $(\xi^+,\gamma(\xi^+))\to(\theta_v,\gamma_4)$ as $m\to\infty$, and $\mu>\gamma_4$,
it follows that  $\gamma(\xi^+) < \mu$ for all $m$ sufficiently large.  As
explained before 
Proposition~\ref{prop:rgammaup}, these fold bifurcations which occur for
$\gamma<\mu$ always involve a stable steady state.
The fold bifurcations where a stable steady state
is created/destroyed are shown as the solid blue segment of the curve of fold bifurcations which is asymptotic to $\gamma=\gamma_4$ in Figure~\ref{fig:vup_ex1b}(b). In contrast, the part of the branch asymptotic to $\gamma=\gamma_3$ as $m\to\infty$ consists of fold bifurcations of two unstable steady states (with different numbers of unstable characteristic values).
The BT point, already seen in Figure~\ref{fig:vup_ex1b}(a), separates these two parts of the branch of fold bifurcations. Consequently, only part of the branch of fold bifurcations delineates the boundary between regions of parameter space with different numbers of stable steady states. This contrasts starkly with the constant delay case considered in Section~\ref{sec:gupvconst} and the state-dependent case with $\mu>\gamma_3$ (shown in Figure~\ref{fig:vup_ex3}),
for both of which steady states only lose stability in a fold bifurcation.

Figure~\ref{fig:vup_ex1b}(c) shows the Hopf curve which terminates at the BT point. We verified numerically that at all points on this curve the Hopf bifurcation is from a stable steady state that loses stability. Close to the BT point, and at the other end of this branch as $m\to\infty$, the Hopf bifurcation is subcritical leading to an unstable periodic orbit. But on a large segment of this curve (and in particular for the smallest values of $m$ on the curve) the Hopf bifurcation is supercritical leading to a stable periodic orbit. Two Bautin bifurcation points, where the Hopf bifurcation changes criticality, separate the sub and supercritical segments of the curve. Figure~\ref{fig:vup_ex1b}(a) shows the number of stable steady states in the $\gamma$-$m$ parameter plane, and just the parts of the bifurcation curves that delineate these regions. The region with no stable steady states contains a stable periodic orbit (because of the supercritical Hopf bifurcation on the left-side of this region).

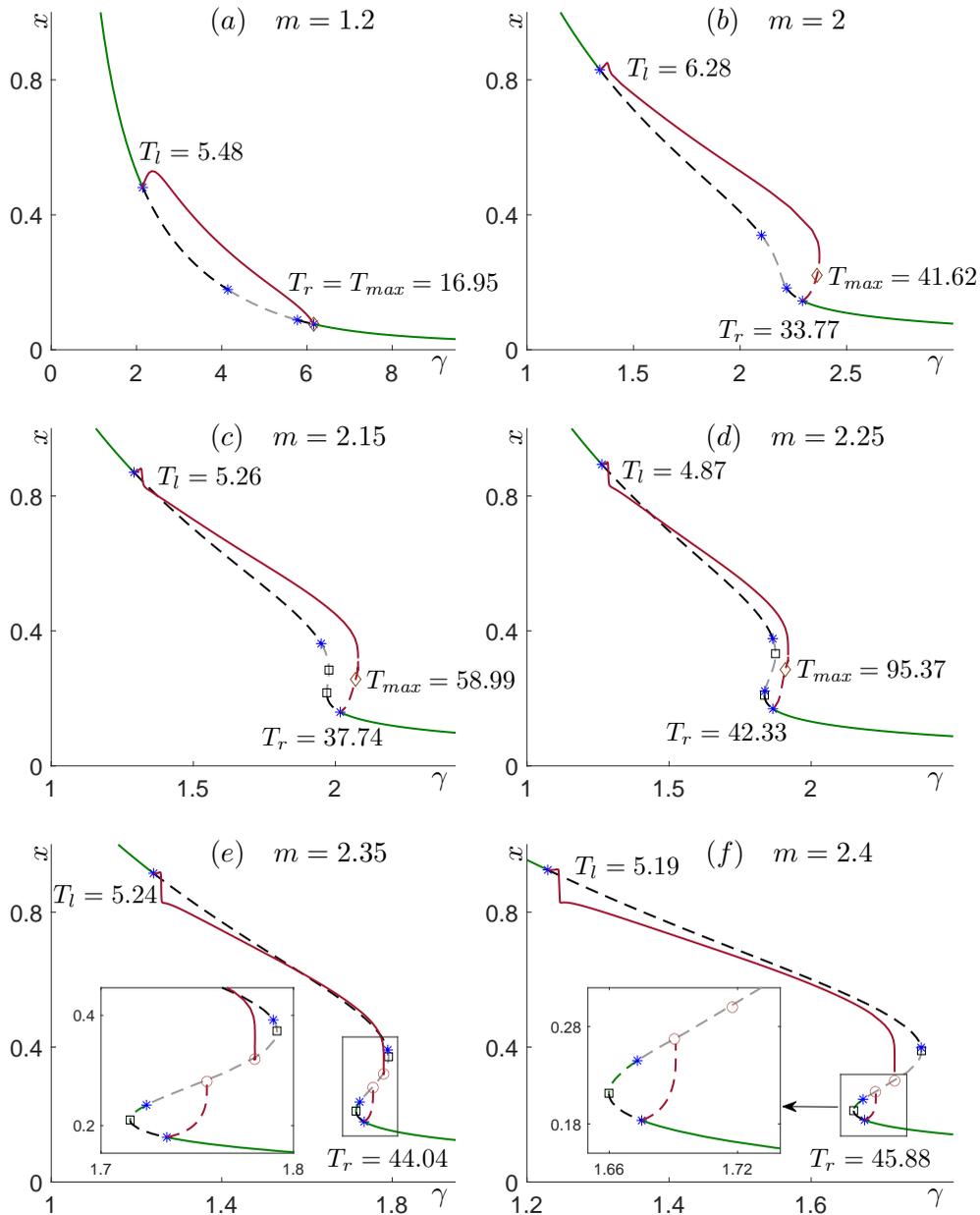
\begin{figure}[thp!]
	\centering	
	\includegraphics[scale=0.5]{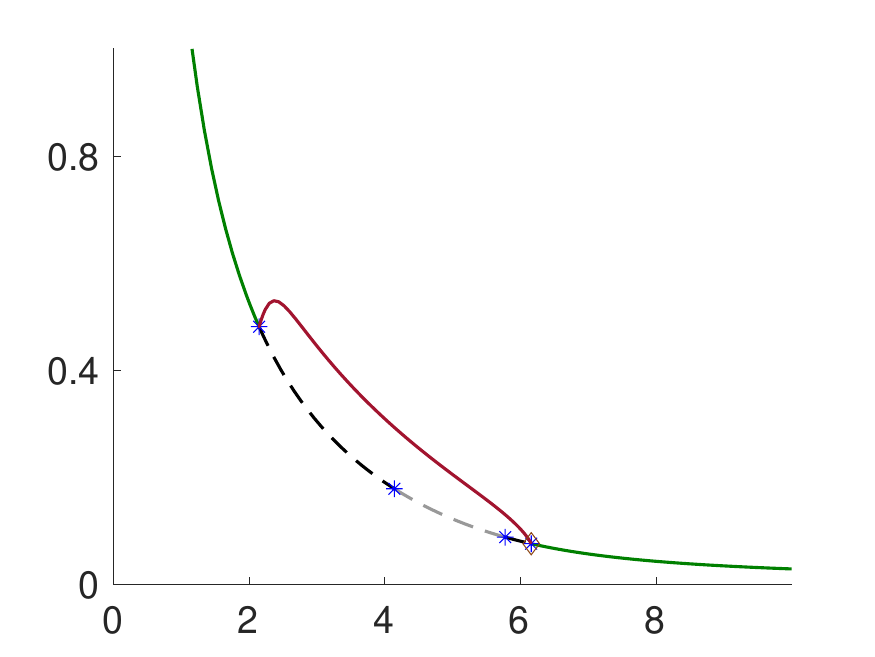}\hspace*{-1cm}\includegraphics[scale=0.5]{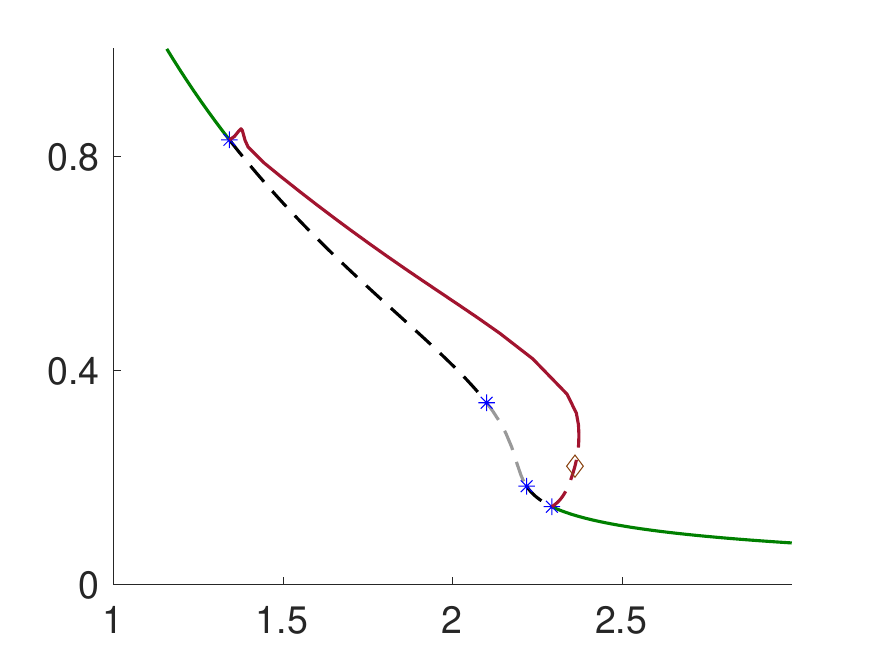}
	\put(-372,140){\rotatebox{90}{$x$}}
	\put(-220,10){$\gamma$}
	\put(-190,140){\rotatebox{90}{$x$}}
	\put(-30,10){$\gamma$}
    \put(-332,90){\small$T_l=5.48$}
    \put(-275,40){\small$T_r=T_{\textit{max}}=16.95$}
    \put(-145,122){\small$T_l=6.28$}
    \put(-110,22){\small$T_r=33.77$}
    \put(-68,42){\small$T_{\textit{max}}=41.62$}
	\put(-305,140){$(a) \quad m=1.2$}
    \put(-115,140){$(b) \quad m=2$}\\ 

\includegraphics[scale=0.5]{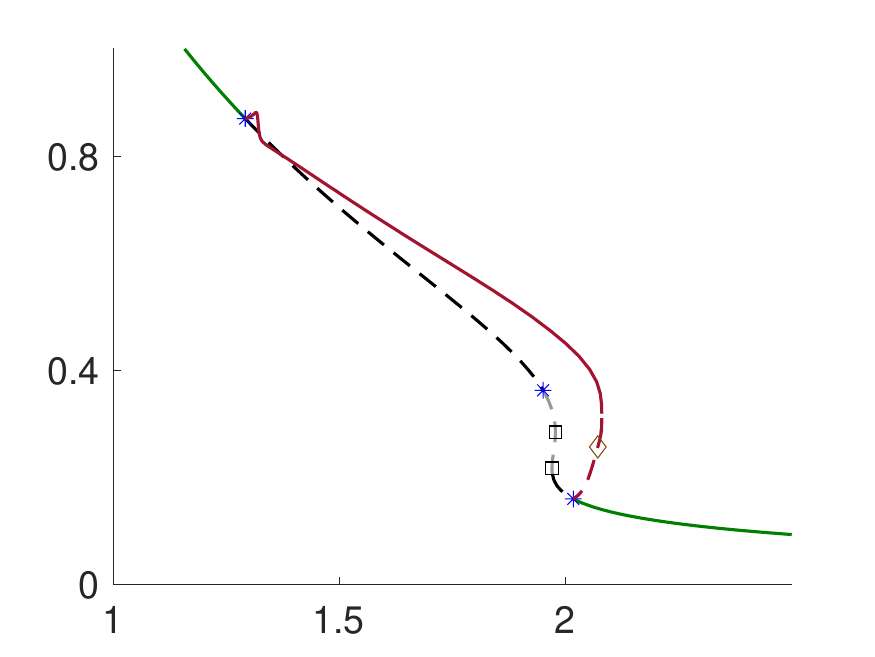}\hspace*{-1cm}\includegraphics[scale=0.5]{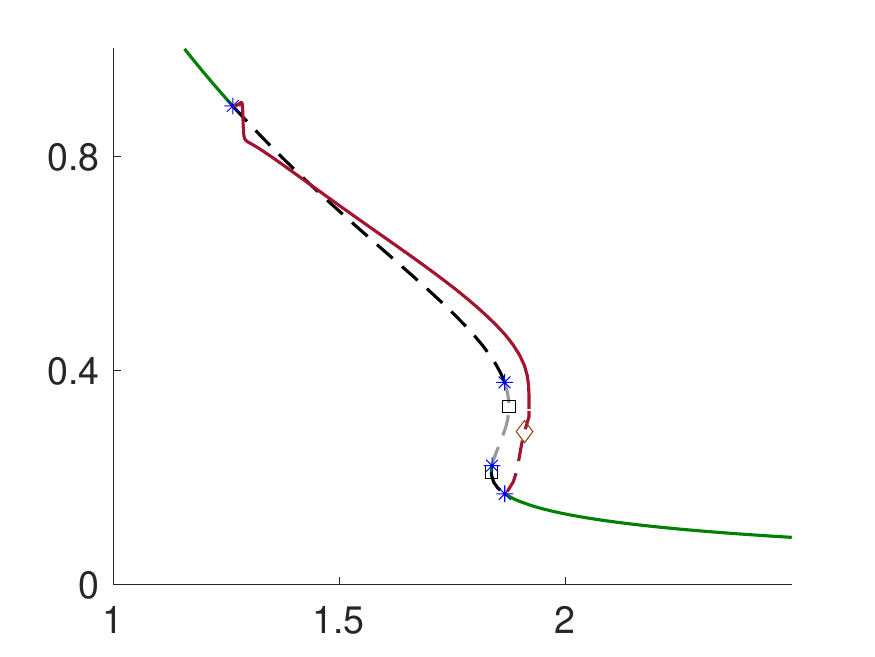}
	\put(-372,140){\rotatebox{90}{$x$}}
	\put(-220,10){$\gamma$}
	\put(-190,140){\rotatebox{90}{$x$}}
	\put(-30,10){$\gamma$}
    \put(-325,125){\small$T_l=5.26$}
    \put(-285,25){\small$T_r=37.74$}
    \put(-245,47){\small$T_{\textit{max}}=58.99$}
    \put(-147,127){\small$T_l=4.87$}
    \put(-130,27){\small$T_r=42.33$}
    \put(-80,51){\small$T_{\textit{max}}=95.37$}
	\put(-305,140){$(c) \quad m=2.15$} 
	\put(-115,140){$(d) \quad m=2.25$}\\ 
	\includegraphics[scale=0.5]{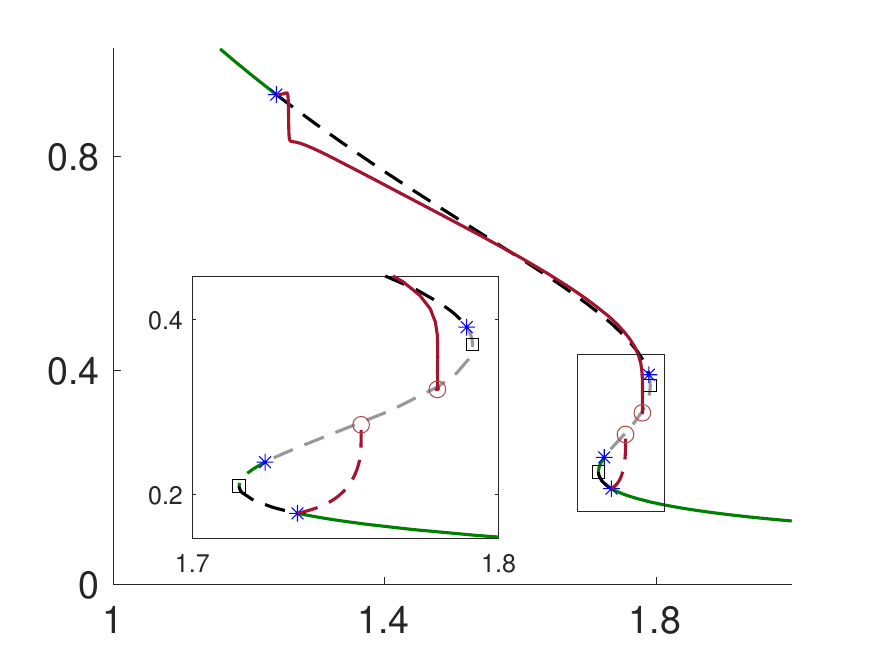}\hspace*{-1cm}\includegraphics[scale=0.5]{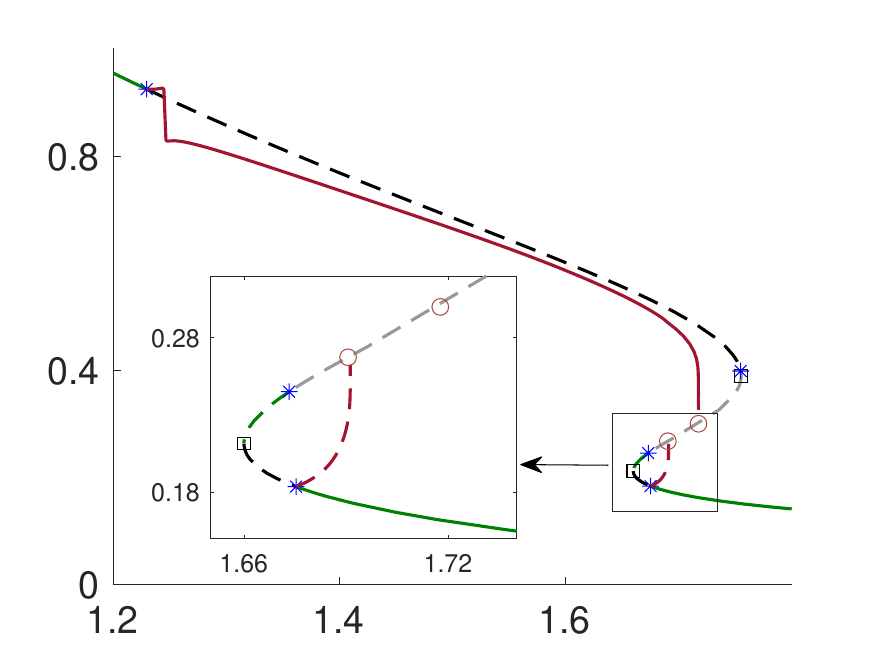}
	\put(-372,140){\rotatebox{90}{$x$}}
	\put(-220,10){$\gamma$}
	\put(-190,140){\rotatebox{90}{$x$}}
	\put(-30,10){$\gamma$}
   \put(-365,124){\small$T_l=5.24$}
    \put(-260,23){\small$T_r=44.04$}
    \put(-165,135){\small$T_l=5.19$}
    \put(-75,23){\small$T_r=45.88$}
 	\put(-305,140){$(e) \quad m=2.35$}
	\put(-115,140){$(f) \quad m=2.4$}\\
	\caption{Case 2 one-parameter bifurcation diagrams for $(g \leftrightarrow, v \uparrow)$ with varying $\gamma$ and different fixed values of $m$, with the other parameters the same as in Figures~\ref{fig:vup_ex1a} and~\ref{fig:vup_ex1b}.
Solid lines represent stable objects including stable steady states (in green) and stable periodic orbits (represented by the 2-norm \eqref{eq:2norm}). Dashed lines represent unstable objects including unstable steady states (depending on the number of eigenvalues with positive real part, green for one, black for two and gray for three and more) and unstable periodic orbits (represented by 2-norm \eqref{eq:2norm}). Hopf bifurcations are marked by blue stars, fold bifurcations of steady states by black squares, and homoclinic bifurcations by red circles.
For the branch of periodic orbits, $T_l$ is the period at the left Hopf bifurcation, and $T_r$ at the right Hopf bifurcation. The point on the branch where the largest period, $T_{\textit{max}}$, occurs is marked by
a brown diamond. }
\label{fig:vup_ex1c}
\end{figure}

Figure~\ref{fig:vup_ex1b}(d) shows another curve of Hopf bifurcations for the same problem. This curve passes through the two fold-Hopf bifurcation points seen on the curve of fold bifurcations in Figure~\ref{fig:vup_ex1b}(b). We notice that while these points are close together on the branch of fold bifurcations (they are both close to, but on different sides of, the cusp point), as the inset shows, they are on different legs of the branch of Hopf points, and hence far from each other on this branch of Hopf bifurcations.
The codimension-two fold-Hopf bifurcation points are of interest, as such bifurcations are impossible for the constant delay model of Section~\ref{sec:gupvconst}. However the Hopf bifurcations seen in this example all generate unstable periodic orbits bifurcating from unstable steady states, and we will not study them further.


In Figure~\ref{fig:vup_ex1c} we continue our study of the dynamics seen in Figures~\ref{fig:vup_ex1a} and~\ref{fig:vup_ex1b} by showing one-parameter continuations of the dynamics and bifurcations as $\gamma$ is varied for different fixed values of the steepness parameter $m$ in the Hill function \eqref{eq:vghill}.
For $m$ small there is a unique stable steady state for each value of $\gamma$ and no bifurcations occur.
For $m>0.8355$ 
there is still a unique steady state, but it is unstable between a pair of super-critical Hopf bifurcations, where a branch of stable periodic orbits exists.
This is illustrated in Figure~\ref{fig:vup_ex1c}(a) for $m=1.2$,
where we see that the period of the periodic orbits is monotonically increasing from the left Hopf point to the right Hopf point.

Increasing $m$ further for $m>1.2929$ we pass the lower Bautin bifurcation seen in Figure~\ref{fig:vup_ex1b}(c)
and the right Hopf point becomes subcritical. The branch of periodic orbits is still contiguous, but a fold bifurcation of periodic orbits born in the Bautin bifurcation divides the branch of periodic orbits into an upper stable branch and a lower unstable branch. The period is no longer monotonically increasing on the branch of periodic orbits, but attains a maximum period near to the fold bifurcation of periodic orbits, but on the
unstable branch of periodic orbits (as illustrated for $m=2$ in Figure~\ref{fig:vup_ex1c}(b)).

For $m>2.1058$ the cusp bifurcation seen in  Figure~\ref{fig:vup_ex1b}(b) introduces two fold bifurcations of steady states (seen for $m=2.15$ in Figure~\ref{fig:vup_ex1c}(c)). As $m$ is increased further the maximum period seen on the branch of periodic orbits increases dramatically and the point where the maximum period occurs approaches the intermediate steady state (see Figure~\ref{fig:vup_ex1c}(d)).

Comparing panels (d) and (e) of Figure~\ref{fig:vup_ex1c} suggests that
\eqref{eq:basic} undergoes a codimension-two global bifurcation between these panels.
The branches of periodic orbits emanating from the Hopf bifurcations where the steady states lose stability no longer join up in panels (e) and (f).  Instead, each branch ends at a homoclinic bifurcation to the intermediate steady state (at different values of the parameter $\gamma$). The   panels shown suggest, but do not prove, that as $m$ is increased the branch of periodic orbits
approaches 
the fold point with $T_{\textit{max}}\to\infty$ at the first point where a homoclinic bifurcation occurs. Further evidence is presented in the two-parameter continuation of Figure~\ref{fig:vup_ex1e}(a) where the homoclinic curves seem to approach the curve of folds. For larger values of $m>2.3$, codimension-one homoclinic bifurcations to the intermediate saddle steady state are found. Examples of dynamical systems with a single branch of homoclinic bifurcations terminating at a fold bifurcation of steady states can be found in \cite{AlDarabsahCampbell21,Kuznetsov4ed}. Our example is somewhat different from those, as we have two branches of homoclinic bifurcations.

\begin{figure}[tbp!]
	\includegraphics[scale=0.5]{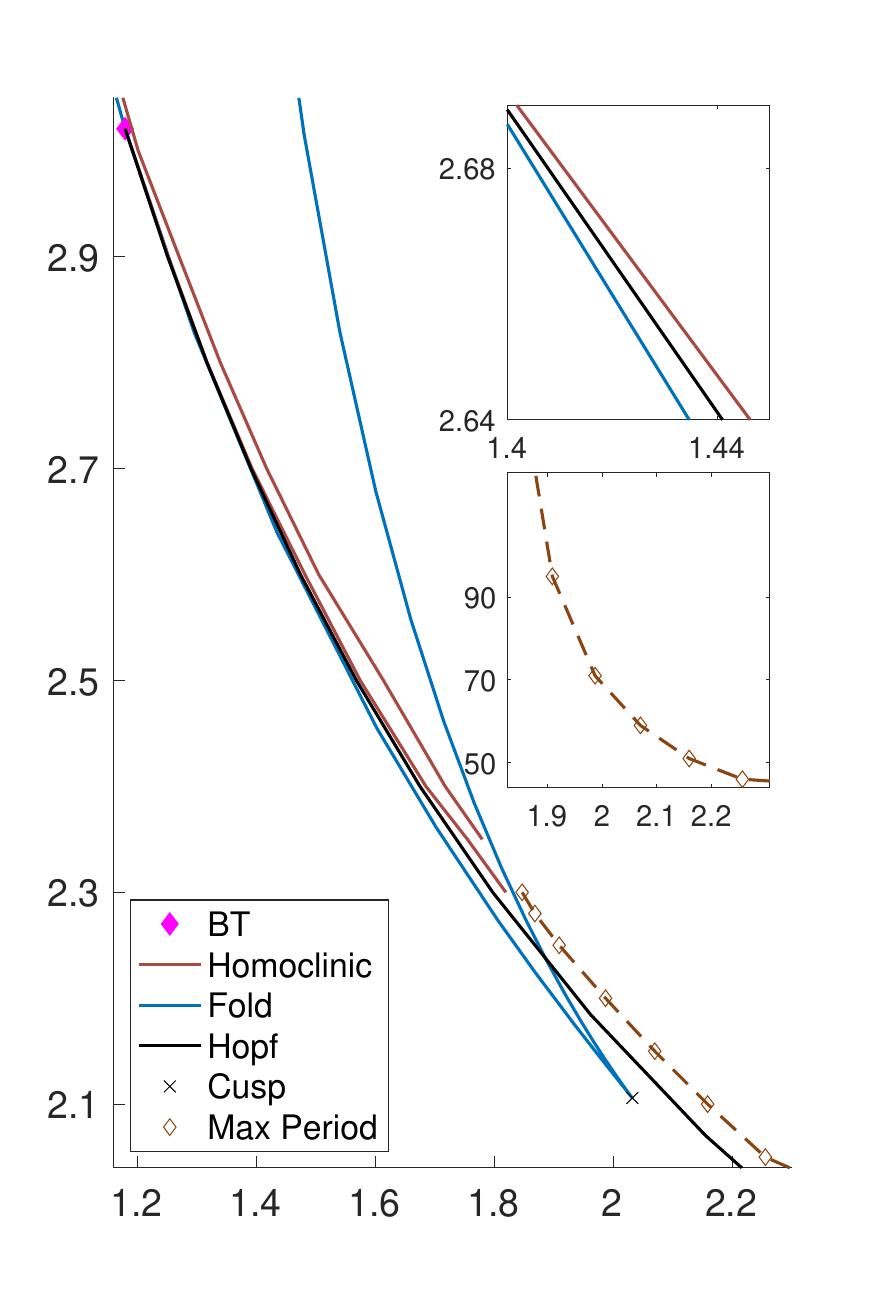}\hspace*{0.1em}
	\put(-192,285){\rotatebox{90}{$m$}}
	\put(-20,25){$\gamma$}
	\put(-165,285){$(a)$}
	\put(-100,181){\small\rotatebox{90}{$T_{\textit{max}}$}}
	\put(-30,117){\small$\gamma$}
    \put(-30,206){\small$\gamma$}
    \put(-98,284){\rotatebox{90}{\small$m$}}
    \parbox[b]{8cm}{\includegraphics[scale=0.5]{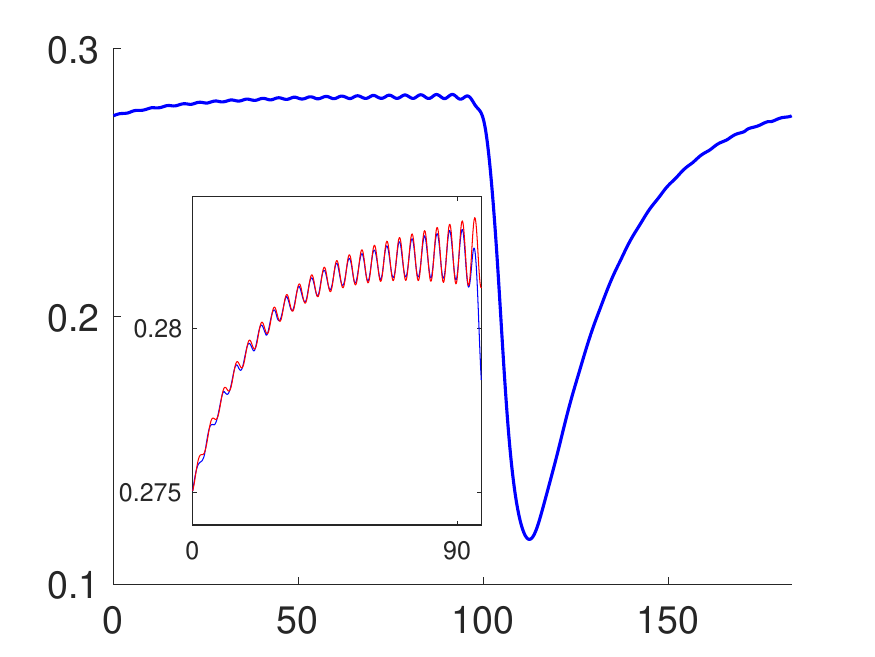}
  	\put(-80,135){$(b)$}
    \put(-23,6){$t$}
	\put(-203,127){$x(t)$}
    \\
    \includegraphics[scale=0.5]{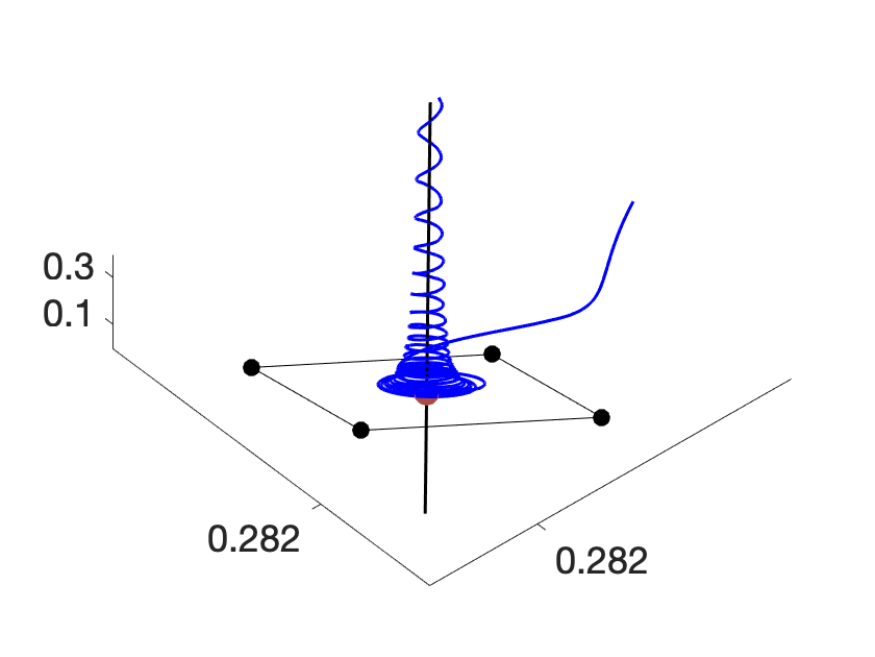}
	\put(-150,127){(c)}
	\put(-180,48){$x(t)$}
	\put(-40,46){$x(t-\tau/2)$}
	\put(-215,102){$x(t-\tau)$}}

	\caption{(a) Bifurcation diagram near the Bogdanov-Takens point shown in Figure~\ref{fig:vup_ex1b} for $(g \leftrightarrow, v \uparrow)$ shows a branch of Hopf bifurcations and a branch of homoclinic bifurcations emerging from the BT point. The dashed brown line denotes the point in parameter space where the largest period was encountered on the branches of periodic orbits shown in Figure~\ref{fig:vup_ex1c}, with the period shown in the lower inset. The upper inset shows the ordering of the curves emerging from the BT point.
(b) Profile of the unstable periodic orbit with parameters $\beta=1.4$, $\mu=0.2$, $g^-=g^+=1$, $m=2.35$, $\gamma=1.7550$, $\theta_v=1$, $a=1$, $v^-=0.1$, and $v^+=2$. The blue curve shows the periodic orbit with period $T=183$.
In the inset the part of the periodic orbit (in blue) near to the intermediate steady state is overlayed by a linear approximation to the dynamics in red (see text for details).
(c) Projection of the periodic orbit from (b) onto the space $(x(t), x(t-\tau/2), x(t-\tau))$ where $\tau=5.1926$ is chosen to be the delay at the intermediate steady state represented by the orange dot. The solid black line represents the dominant one-dimensional linear stable manifold associated with the characteristic value $-0.045$. The parallelogram denotes the plane associated with
the dominant unstable characteristic values $\lambda = 0.0184 \pm 1.4738i$.}
\label{fig:vup_ex1e}
\end{figure}
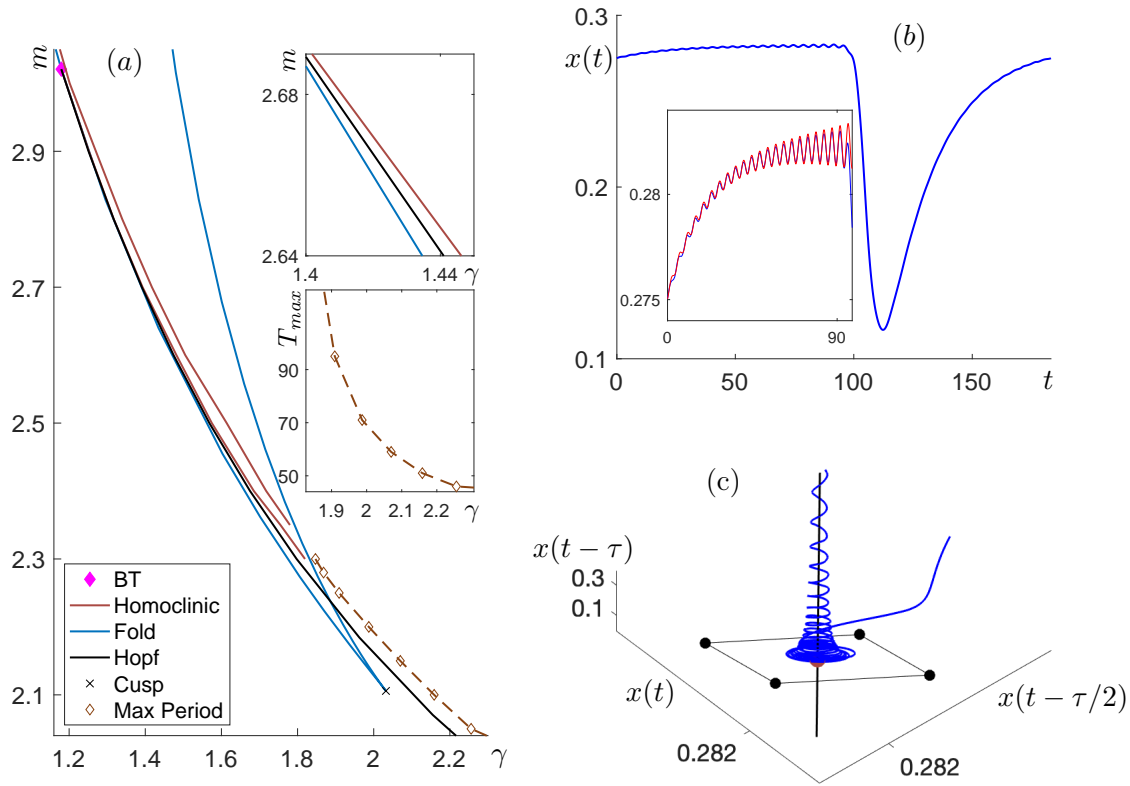

The homoclinic bifurcation on the branch emanating from the lower Hopf point
in Figures~\ref{fig:vup_ex1c}(e) and (f) does not persist for large $m$, but
instead the homoclinic bifurcation and the Hopf bifurcation itself move towards the fold bifurcation and terminate at the BT-point found previously. This is illustrated in a two-parameter continuation in Figure~\ref{fig:vup_ex1e}(a), which shows the branch of fold bifurcations of steady states passing through the BT point, next to a branch of Hopf bifurcations, then a branch of homoclinic bifurcations, the latter two branches both terminating at the BT point tangential to the branch of fold bifurcations at that point.
This is well-known behaviour for Bogdanov-Takens bifurcations, and can be seen in the normal form diagram for this bifurcation in Section 8.4 of \cite{Kuznetsov4ed}.
Bogdanov-Takens bifurcations have recently been analysed for constant delay DDEs in
\cite{BK_BT2022} (where Figures 5 and S8 resemble the part of Figure~\ref{fig:vup_ex1e}(a) close to the BT point),
but we are not aware of any systematic study of them in state-dependent DDE problems.

While in the classical unfolding the bifurcation curves extend arbitrarily far from the BT point, in our example in Figure~\ref{fig:vup_ex1e}(a) both the branch of homoclinic orbits and the branch of fold bifurcations terminate.  The branch of fold bifurcations terminates at a cusp bifurcation with the other branch of fold bifurcations. The proximity of the cusp point to the BT point suggests that our system may be close to a codimension-three Bogdanov-Takens-cusp (BTC) point. While we are not aware of a systematic study of this bifurcation,  they have been observed in a neuron model in \cite{AlDarabsahCampbell21}, and some of the bifurcation structures that we find resemble those in \cite{AlDarabsahCampbell21}.

The homoclinic bifurcation on the branch emanating from the upper Hopf point
in Figures~\ref{fig:vup_ex1c}(e) and (f) is also shown in Figure~\ref{fig:vup_ex1e}(a) and
persists for arbitrary large values of $m$ (it was already seen with $m=50$ in  Figure~\ref{fig:vup_ex1a}(a)), but there is a change in stability on this branch for $m>4.6069$ due to the Bautin bifurcation seen in Figure~\ref{fig:vup_ex1b}(c).

The maximum period of orbits from the one-parameter continuations described in Figure~\ref{fig:vup_ex1c}
are also shown as a curve in two-parameter space in Figure~\ref{fig:vup_ex1b}(a). This curve approaches the right most curve of fold bifurcations with the period becoming unbounded as it does so. Our computations of the two branches of homoclinic bifurcations also terminate close to this point. We conjecture that the co-dimension two of the homoclinic orbits already described in Figure~\ref{fig:vup_ex1c} that occurs where the branch of maximum period orbits terminates
will occur on the branch of fold bifurcations. This would be consistent with
the previously mentioned examples
of curves of homoclinic bifurcations which terminate at fold bifurcations
\cite{AlDarabsahCampbell21,Kuznetsov4ed}.

In Figure~\ref{fig:vup_ex1e}(b) and (c) we display one of the periodic orbits for $m=2.35$ from the lower branch of
periodic orbits shown in Figure~\ref{fig:vup_ex1c}(e). 
The orbit is shown close to the homoclinic bifurcation at the end of the branch, for which the period is large.
Figure~\ref{fig:vup_ex1e}(c) shows a phase space projection of the part of the orbit close to the intermediate steady
state. In this projection the orbit approaches the saddle steady state from above close to the dominant stable direction (associated with the characteristic value with negative real part closest to zero which is $\lambda=-0.045$), but with a growing oscillation about this point in the plane defined by the dominant unstable direction (associated with the
characteristic value with positive real part closest to zero which is $\lambda=0.0184 \pm 1.4738i$). To further confirm
that this linear behaviour is determining the dynamics near to the intermediate steady state $x^*$, in the inset of
Figure~\ref{fig:vup_ex1e}(b) we plot in red the curve
$$x(t)=x^*+k_1 e^{-0.045t}+k_2 e^{0.0184t}\cos(1.4738t),$$
for suitably chosen coefficients $k_j$, and observe that it overlays the periodic solution over most of the time interval shown.
This behavior is similar to that associated with Shilnikov  type complex dynamics, but in contrast with the construction of a chaotic attractor of Shilnikov type, the periodic orbit  pictured in  Figure~\ref{fig:vup_ex1e}(b) and (c) is unstable.

\begin{figure}[thp!]
	\centering
	\includegraphics[scale=0.5]{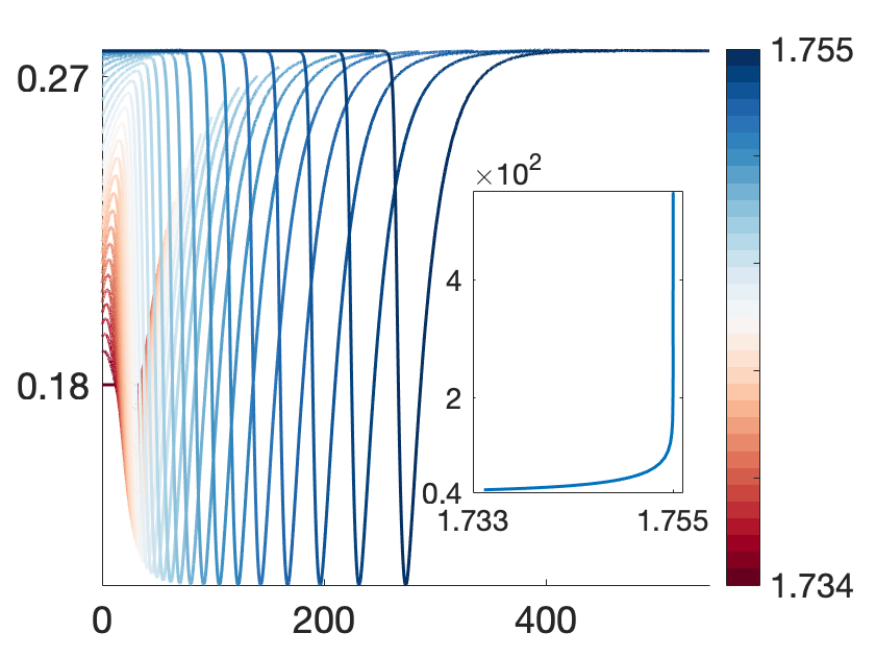}\hspace*{0.5em}\includegraphics[scale=0.5]{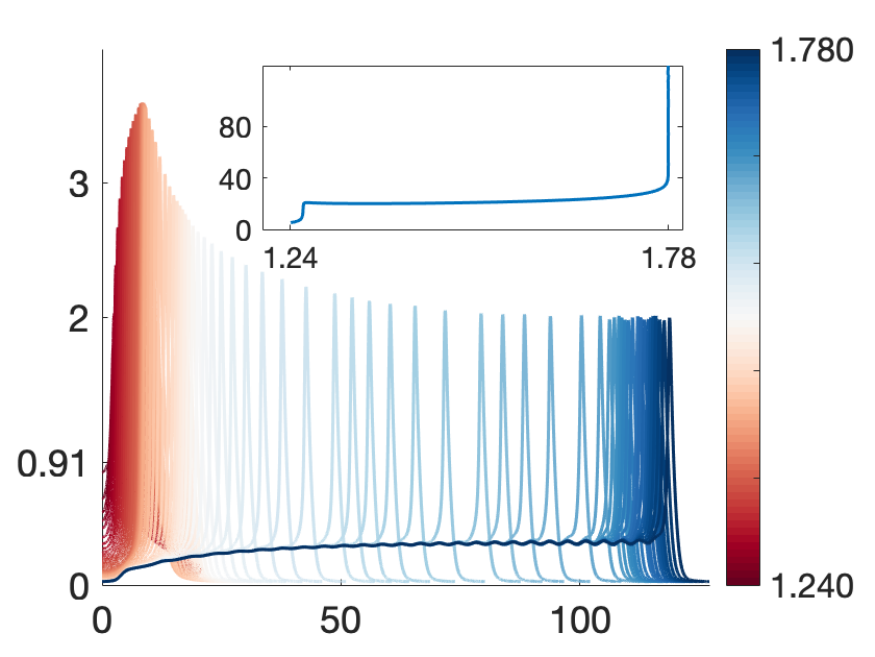}
	\put(-408,145){\rotatebox{90}{$x$}}
	\put(-261,8){$t$}
	\put(-242,80){$\gamma$}
	\put(-320,110){\scriptsize\rotatebox{90}{$T$}}
	\put(-288,33){\scriptsize$\gamma$}
	\put(-300,135){$(a)$}
	\put(-178,140){$(b)$}
	\put(-192,140){\rotatebox{90}{$x$}}
	\put(-45,8){$t$}
	\put(-26,80){$\gamma$}
	\put(-154,135){\scriptsize\rotatebox{90}{$T$}}
	\put(-100,95){\scriptsize$\gamma$} 	
\caption{Profiles and periods of periodic orbits on the branches that approach homoclinic bifurcations in Figure~\ref{fig:vup_ex1c}(e) with $m=2.35$ for
(a) the unstable periodic orbits on the branch emanating from the subcritical Hopf bifurcation at $\gamma=1.734$, and (b) the stable periodic orbits on the branch emanating from the supercritical Hopf bifurcation at $\gamma=1.240$.
 }
	\label{fig:vup_ex1d}
\end{figure}
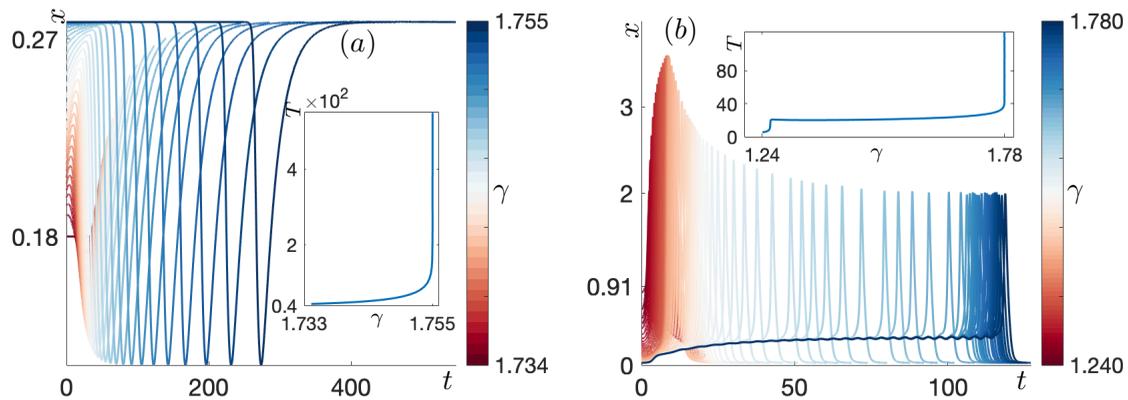

Figure~\ref{fig:vup_ex1d}(a) and (b) shows the profiles and periods on the two branches of periodic solutions seen in Figure~\ref{fig:vup_ex1c}(e). Comparing these two panels we see that the periodic orbits on the two branches are very different, with the unstable orbit in Figure~\ref{fig:vup_ex1d}(a)  having a small amplitude and with $x(t)$ below the steady state value along the whole orbit on the entire branch. In contrast stable periodic orbits shown in panel (b) have much larger amplitude.


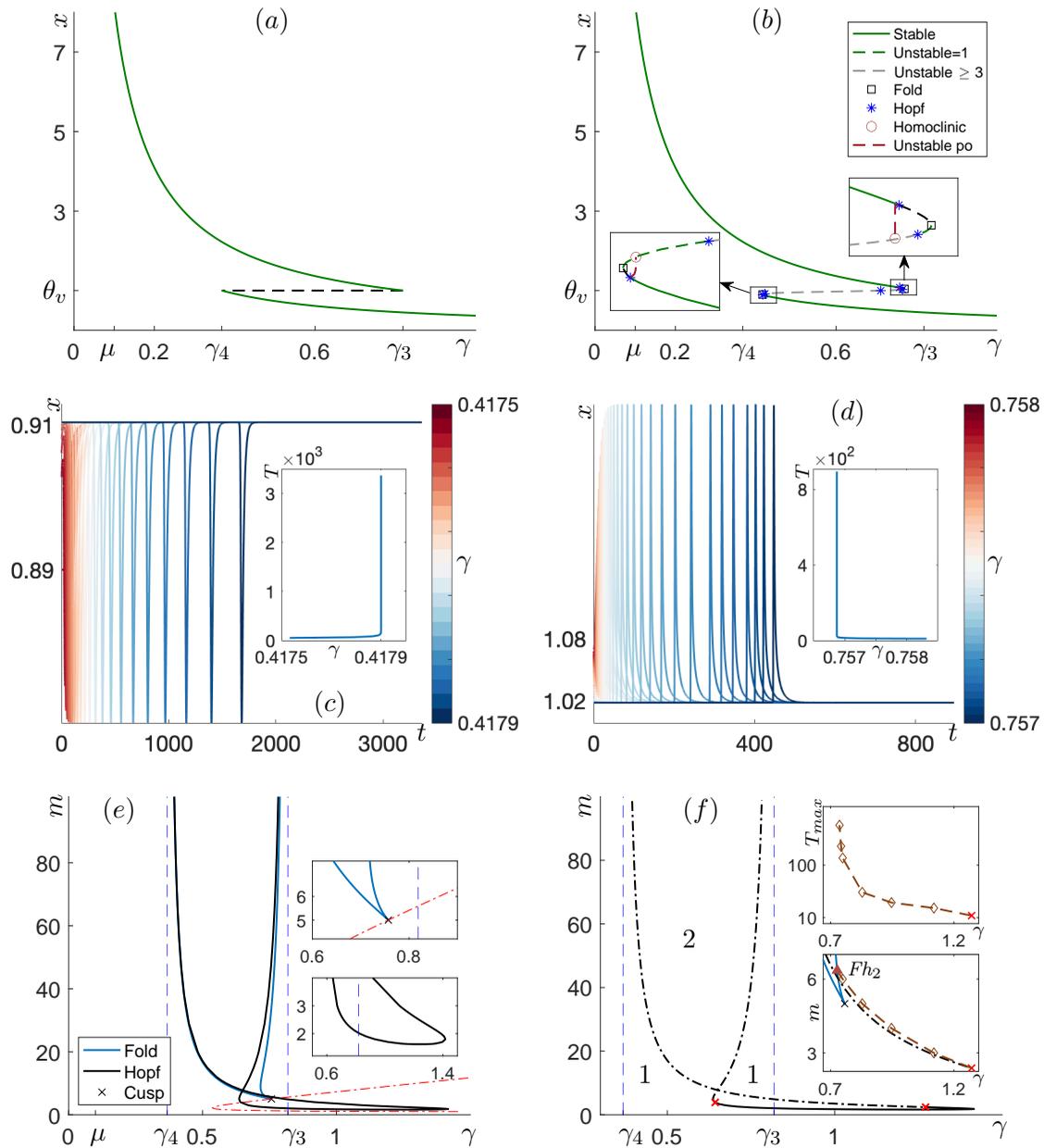
\begin{figure}[thp!]
	\centering
	\vspacefig
	\includegraphics[scale=0.5]{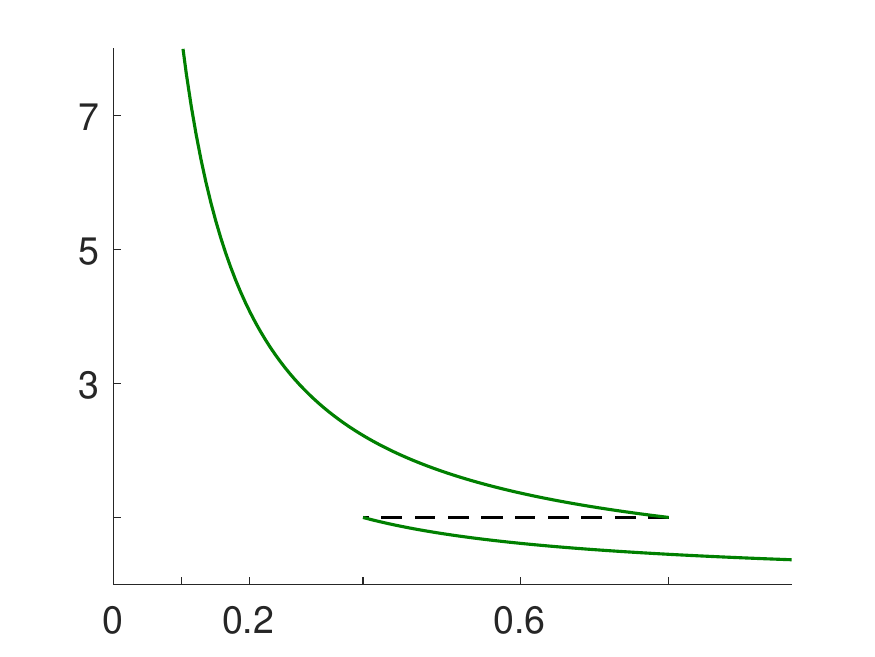}\hspace*{0.1em}\includegraphics[scale=0.5]{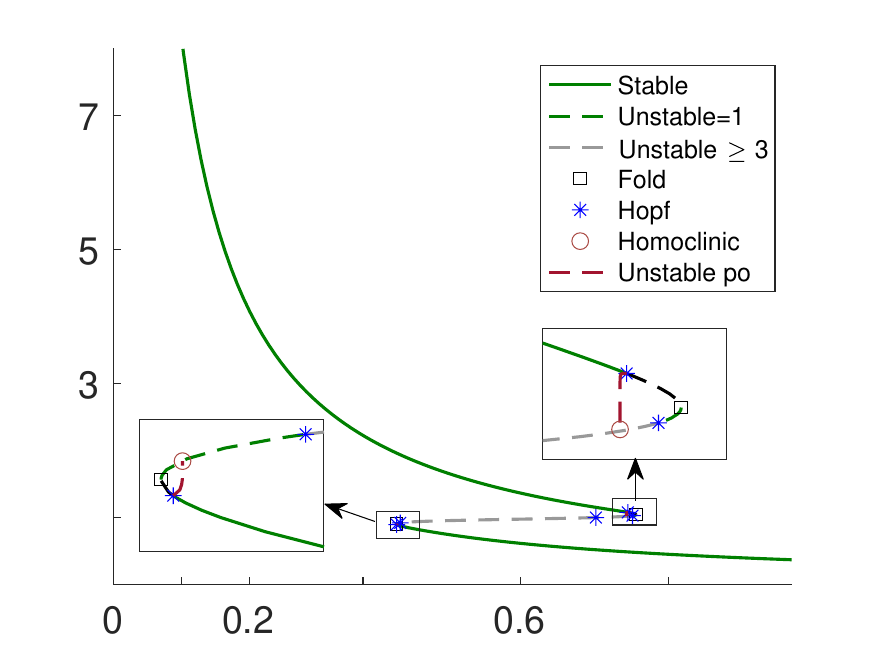}
	\put(-403,140){\rotatebox{90}{$x$}}
	\put(-408,30){$\theta_v$}
	\put(-240,9){$\gamma$}
	\put(-385,7){$\mu$}
	\put(-342,7){$\gamma_4$}
	\put(-268,7){$\gamma_3$}
	\put(-322,140){$(a)$}
	\put(-120,140){$(b)$}
	\put(-190,140){\rotatebox{90}{$x$}}
	\put(-196,30){$\theta_v$}
	\put(-24,9){$\gamma$}
	\put(-171,7){$\mu$}
	\put(-128,7){$\gamma_4$}
	\put(-54,7){$\gamma_3$}\\
	\vspace*{-0.5em}
	\includegraphics[scale=0.5]{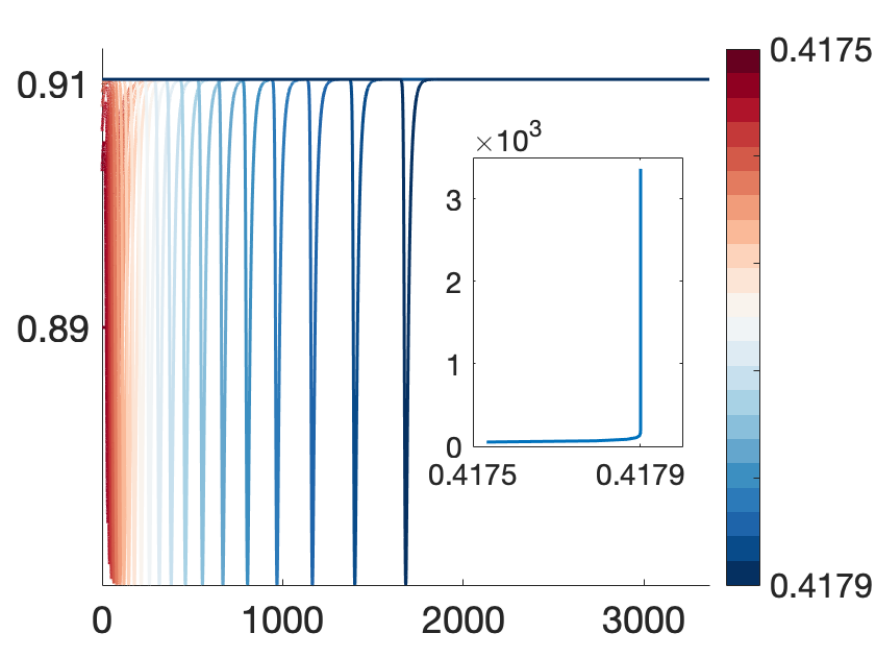}\hspace*{0.5em}\includegraphics[scale=0.5]{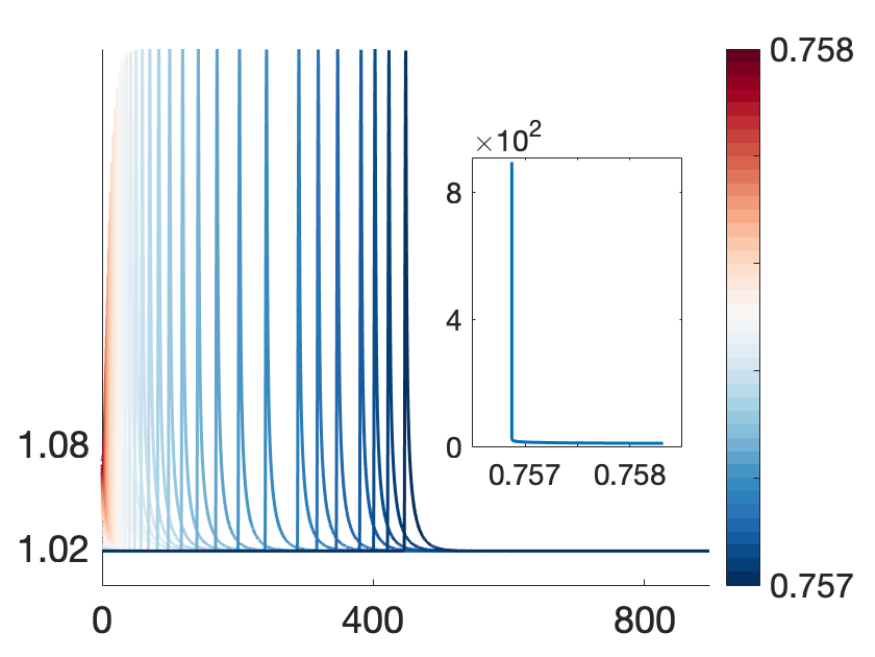}
	\put(-408,143){\rotatebox{90}{$x$}}
	\put(-258,10){$t$}
	\put(-242,80){$\gamma$}
	\put(-320,116){\scriptsize\rotatebox{90}{$T$}}
	\put(-294,45){\scriptsize$\gamma$}
	\put(-300,22){$(c)$}
	\put(-90,140){$(d)$}
	\put(-192,140){\rotatebox{90}{$x$}}
	\put(-42,10){$t$}
	\put(-26,80){$\gamma$}
	\put(-104,116){\scriptsize\rotatebox{90}{$T$}}
	\put(-73,45){\scriptsize$\gamma$}\\
	\vspace*{-0.5em} \includegraphics[scale=0.5]{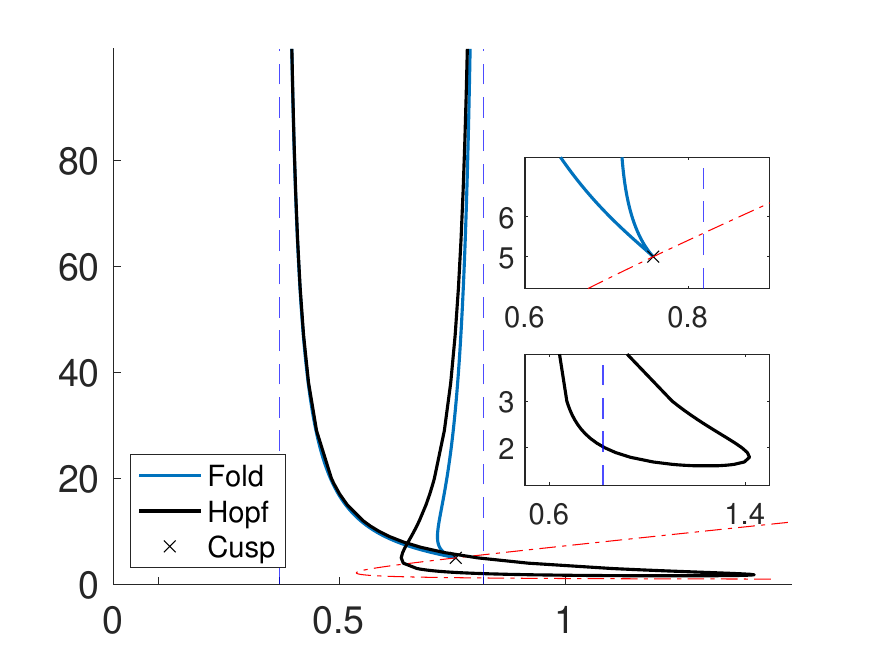}\hspace*{0.5em}\includegraphics[scale=0.5]{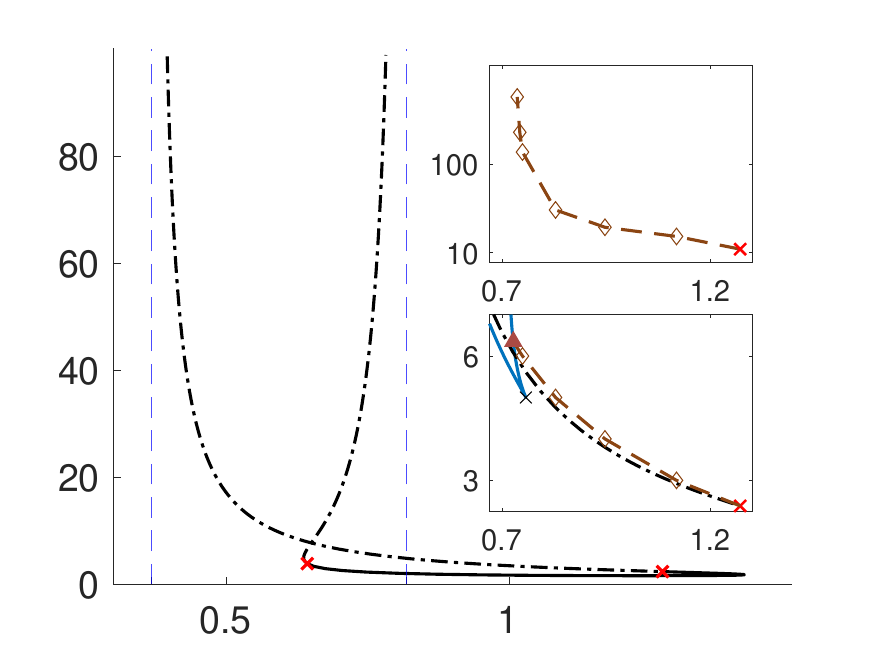}
	\put(-386,138){$(e)$}
	\put(-150,138){$(f)$}
	\put(-192,138){\rotatebox{90}{$m$}}
	\put(-407,138){\rotatebox{90}{$m$}}
	\put(-100,126){\scriptsize\rotatebox{90}{$T_{max}$}}
	\put(-32,90){\scriptsize{$\gamma$}}
	\put(-32,30){\scriptsize{$\gamma$}}
    \put(-98,55){\scriptsize\rotatebox{90}{$m$}}
	\put(-83,74){\scriptsize$\textit{Fh}_2$}
	\put(-24,10){$\gamma$}
	\put(-176,7){$\gamma_4$}
	\put(-120,7){$\gamma_3$}
    \put(-240,7){$\gamma$}
	\put(-313,7){$\gamma_3$}
	\put(-365,7){$\gamma_4$}
	\put(-391,7){$\mu$}
	\put(-124,30){$1$}
	\put(-150,85){$2$}
	\put(-168,30){$1$}

	\caption{Bifurcation diagrams of \eqref{eq:basic}-\eqref{eq:thres} with $(g \leftrightarrow, v \uparrow)$ and parameters $\beta=1$, $\mu=0.1$, $g^-=g^+=1$, $\theta_v=1$, $a=2$, $v^-=0.2$ and $v^+=1$.
By \eqref{eq:corners_gconst} this implies  $\mu<\gamma_4 = 0.3679<\gamma_3 = 0.8187$ (Case 3).
(a) The limiting case with $v$ defined by \eqref{eq:vpwconst}. The stable steady state is shown as a green solid line, and the singular steady state as a black dashed line.
(b) With smooth velocity nonlinearity $v$ defined by \eqref{eq:vghill} with $m=50$. Solid lines represent stable steady states (in green). Dashed lines represent unstable objects including unstable steady states (depending on the number of eigenvalues with positive real part, green for one, and gray for three and more) and unstable periodic orbits (represented by 2-norm \eqref{eq:2norm}).
(c) and (d) Profiles and periods of the periodic orbits on the branches shown in (b) that
terminate at homoclinic bifurcations and emanate from the subcritical Hopf bifurcations
at (c) $\gamma=0.4175$ and (d) $\gamma=0.758$.
(e) Two-parameter continuations in $m$ and $\gamma$ of the fold and the Hopf bifurcations with other parameters as above. The dashed vertical blue lines denote $\gamma=\gamma_3$ and $\gamma=\gamma_4$ the location of the fold bifurcations in the limiting case as $m\to\infty$. The red dash-dotted curve denotes the bound on the fold bifurcations.
(f) Two parameter continuation of the curve of Hopf bifurcations, with the two Bautin bifurcation points and the criticality of the Hopf bifurcation indicated. The Hopf bifurcations delineate the parameter space into regions where there is $0$, $1$ or two stable steady states, as indicated. Also shown in the insets is the maximum period curve (similar to Figures~\ref{fig:vup_ex1c} and~\ref{fig:vup_ex1e}), which terminates at the point marked
$\textit{Fh}_2$ with infinite period.
}
\label{fig:vup_ex2}
\end{figure}

{\bf Case 3:}
It remains to consider the case where $\mu < \gamma_4 < \gamma_3$. Making a small change to the parameters considered in the example in Figures~\ref{fig:vup_ex1a}
to~\ref{fig:vup_ex1d} by changing the value of $\beta$ from $\beta=1.4$ to $\beta=2$
but leaving the values of the other parameters unchanged, results in
$\mu=0.2<\gamma_4 = 0.2707<\gamma_3 = 1.8097$. However, although $\mu<\gamma_4$ in this case $\mu$ is close to $\gamma_4$ and the dynamics and bifurcations are very similar to those shown in Figure~\ref{fig:vup_ex1a}, so we  do not show them here. Instead in Figure~\ref{fig:vup_ex2} we consider an example with $\mu\ll\gamma_4<\gamma_3$.

Theorem~\ref{thm:vup} suggests that stability may be lost in either a Hopf or a fold bifurcation, but for the example in Figure~\ref{fig:vup_ex2} with $\mu\ll\gamma_4$ we see that the steady state always loses stability at a Hopf bifurcation. Unlike the previous example, there is no longer a BT point, and for all $m$ sufficiently large the steady state loses stability in a subcritical Hopf bifurcation close to the fold bifurcation, resulting in a short branch of periodic orbits that terminates in a homoclinic bifurcation to the intermediate steady state created at the fold bifurcation, as illustrated for $m=50$ in  Figure~\ref{fig:vup_ex2}(b)-(d). For $m$ small, below the cusp bifurcation at $m=5.0007, \gamma=0.7572$, the Hopf bifurcations are supercritical leading to a branch of stable periodic orbits between the Hopf bifurcations, similar to Figure~\ref{fig:vup_ex1c}(a).
As seen in Figure~\ref{fig:vup_ex2}(b) the Hopf bifurcations and homoclinic points are very close to the fold bifurcations. This makes numerical continuation of the branches of homoclinic bifurcations very delicate, and so we do not show the curves of homoclinic bifurcations in the two parameter continuations in Figure~\ref{fig:vup_ex2}(e) and (f). Instead in the insets in Figure~\ref{fig:vup_ex2}(f) we show the curve formed by the periodic orbits of maximum periods from the one-parameter continuations in $\gamma$; as was done in case (b) in Figures~\ref{fig:vup_ex1c} and~\ref{fig:vup_ex1e}.
This curve is born from the Bautin bifurcation point and terminates at the point labelled
$\textit{Fh}_2$ on the curve of fold bifurcations, where the period becomes infinite. We expect two curves of homoclinic bifurcations to be born at this point, and to persist for arbitrarily large $m$ (as seen in (b) for $m=50$).

\subsection{Summary}\label{summary one hill}

We briefly summarize the results of Section~\ref{sec:onehill}. We have separately studied cases with constant delay and decreasing (Section~\ref{sec:gdownvconst}) and increasing (Section~\ref{sec:gupvconst}) function $g$, and then the cases with constant $g$ but with state dependent delay with velocity $v$ decreasing (Section~\ref{sec:gconstvdown}) or increasing (Section~\ref{sec:gconstvup}).
These results can be compared between constant and state-dependent cases but with the same
type of nonlinearity (Section~\ref{sec:gdownvconst} and~\ref{sec:gconstvdown}, and Section~\ref{sec:gupvconst}
and~\ref{sec:gconstvup}),
or between the two types of nonlinearity with the same type of delay
(Section~\ref{sec:gdownvconst} and~\ref{sec:gupvconst}, and Section~\ref{sec:gconstvdown}
and~\ref{sec:gconstvup}), as we will do below.

In general, the dynamics with state-dependent delay is significantly richer than that of the corresponding constant delay case.  For decreasing non-linearity (Section~\ref{sec:gdownvconst} vs Section~\ref{sec:gconstvdown}) in the constant delay case there is a range $(\gamma_1,\gamma_2)$ of the parameter $\gamma$ where the unique equilibrium undergoes a series of Hopf bifurcations as the Hill coefficient $n \to \infty$ (Figure~\ref{fig:gdown_ex1}).
In contrast, in the state dependent case a new constraint becomes important as  we show that Hopf bifurcations can only occur for $\gamma < \mu$. Therefore, if the $\mu$ is below the range  $(\gamma_3,\gamma_4)$ the unique  equilibrium is always stable (Figure \ref{fig:vdown_ex3}), when it is above $(\gamma_3,\gamma_4)$ we recover the result from
Section~\ref{sec:gdownvconst} (Figure \ref{fig:vdown_ex1}), and when $\mu \in (\gamma_3,\gamma_4)$, the result of Section~\ref{sec:gdownvconst} is confined to  values of $\gamma \in (\gamma_3,\gamma_4)\cap \{\gamma <\mu\}=(\gamma_3,\mu)$; see Figure~\ref{fig:vdown_ex2}.

When the feedback nonlinearity is increasing with constant delay (Section~\ref{sec:gupvconst}) the system has multiple equilibria for $\gamma$ in an interval which approaches $(\gamma_2,\gamma_1)$ in the limit as the feedback nonlinearity approaches a piecewise constant function. The stable equilibria always lose stability at fold bifurcations at the ends of the interval, leading to bistability of steady states for this interval of $\gamma$ values. Hopf bifurcations can only occur on the middle branch of unstable equilibria.

As documented in Section~\ref{sec:gconstvup}, this scenario changes dramatically for state dependent delays.
For increasing $v$ there is again a constraint
and we show that Hopf bifurcations can only occur for $\gamma > \mu$. Consequently if $\mu>\gamma_3$ there are no Hopf bifurcations. On the other hand, when $\mu<\gamma_3$ new dynamics are observed.
In this case stable equilibria can now lose stability in a Hopf bifurcation and the interaction of the Hopf and fold bifurcation curves in two-parameter continuation in $\gamma$ and the Hill coefficient leads to existence of Bogdanov-Takens points, fold-Hopf points and a possible double homoclinic bifurcation.  This last scenario, documented in Figure~\ref{fig:vup_ex1c}, suggests that a curve of periodic orbits connecting subcritical and supercritical Hopf bifurcations of the stable equilibria, collides with a fold bifurcation where the equilibrium also admits a homoclinic orbit. Continuation of the curve of periodic orbits in two-parameter space suggests that the  curve splits into two curves of periodic orbits each terminating at a homoclinic bifurcation at the middle equilibrium branch, but at different values of parameter $\gamma$. A less dramatic result, but perhaps more important for applications, is the observation of  bistability between a stable fixed point and a stable periodic orbit in Figure~\ref{fig:vup_ex1c}(b), a scenario  that is impossible 
in the constant delay case of the Goodwin operon model considered in Sections~\ref{sec:gdownvconst}
and~\ref{sec:gupvconst}.

\section{Dynamics with two Hill functions} \label{sec:twohill}

We now consider the model \eqref{eq:basic},\eqref{eq:thres} with both the functions $g$ and $v$ non-constant at the same time. In the limiting case these functions are defined by
\eqref{eq:gpwconst} and \eqref{eq:vpwconst}, and when $\theta_v\ne\theta_g$, the steps in the functions occur for different values of $x$, so they reduce to the cases considered in
Section~\ref{sec:onehill} with $v$  constant for $x$ in a neighbourhood of $\theta_g$,
and $g$  constant for $x$ in a neighbourhood of $\theta_v$. With smooth nonlinearities
defined by \eqref{eq:vghill} and large exponents $m\gg0$, $n\gg0$, when
$\theta_v\ne\theta_g$ we
will see dynamics and bifurcations near $x=\theta_g$ similar to those
observed in Sections~\ref{sec:gdownvconst} and~\ref{sec:gupvconst}, and similar to
Sections~\ref{sec:gconstvdown} and~\ref{sec:gconstvup} near $x=\theta_v$.
On the other hand, when $\theta_v=\theta_g$ new dynamics are observed.

Recall the notation $(g\downarrow,v\downarrow)$ that indicates the monotonicity type of functions $g$ and $v$, introduced in Section~\ref{sec:intro}. Note that if $v^->v^+$ then $v(x)$ is a decreasing function of $x$, then  $\tau(x)$ given by
\eqref{eq:steadydelay} is an increasing function of $x$, and therefore  $e^{-\mu\tau(x)}$ is a decreasing function of $x$.
It follows immediately that for $(g\downarrow,v\downarrow)$ the function $h(x)$ (see \eqref{eq:h}) is strictly monotonically decreasing and hence for $g$ and $v$ both monotonically decreasing there is exactly one steady state.

If either or both of $g$ or $v$ is increasing then it is possible to have fold bifurcations and additional steady states. To see this, consider a steady state $\xi$ with both $v$ and $g$ given by
\eqref{eq:vghill}, in which case the characteristic equation
has the form \eqref{eq:char}, which can be written as $\Delta(\lambda)=0$ with
\be \label{eq:charvg}
\Delta(\lambda) = \lambda+\gamma(\xi)-A(\xi)-(Q(\xi)-A(\xi))e^{-\lambda\tau(\xi)}-\mu A(\xi)\int_{-\tau(\xi)}^{0} e^{\lambda s} ds
\ee
where (recalling \eqref{eq:A} and \eqref{eq:gam})
\begin{displaymath}
Q(\xi):= \beta e^{-\mu\tau(\xi)} g'(\xi)=\gamma\frac{\xi g'(\xi)}{g(\xi)}, \qquad A(\xi)=\beta\frac{v'(\xi)}{v(\xi)}e^{-\mu\tau(\xi)} g(\xi)=\gamma\frac{\xi v'(\xi)}{v(\xi)}.
\end{displaymath}
The properties of $A(\xi)$ and $Q(\xi)$ were already described in Proposition~\ref{prop:fxpr}.
Now
\begin{align*}
\Delta(0) & = \gamma(\xi)-Q(\xi)-\mu A(\xi)\int_{-\tau(\xi)}^{0}\hspace*{-1em} 1 ds
= \gamma(\xi)-Q(\xi)-\mu\tau(\xi) A(\xi)\\
& =\gamma(\xi)\Bigl(1-\frac{\xi g'(\xi)}{g(\xi)}-\mu\tau(\xi)\frac{\xi v'(\xi)}{v(\xi)}\Bigr).
\end{align*}
Hence, $\lambda=0$ is a characteristic value if and only if
\be  \label{eq:vglam0}
	M(\xi):=\frac{\xi g'(\xi)}{g(\xi)} + \mu\tau(\xi)\frac{\xi v'(\xi)}{v(\xi)}-1=0.
\ee
In addition, we note that
$M(\xi)>0$
implies that $\Delta(0)<0$. At the same time, for real $\lambda$ the characteristic function satisfies $\Delta(\lambda)\to+\infty$ as $\lambda\to+\infty$. We conclude that when $M(\xi)>0$ 
 there is a real positive characteristic value and hence the steady state is unstable.

We now show that fold bifurcations occur when \eqref{eq:vglam0} is satisfied. To that end, consider the curve of
solutions $(\gamma(\xi),\xi)$ where $\gamma(\xi)$ is defined by
\be \label{eq:gamvg}
\gamma(\xi) = \frac{\beta g(\xi)}{\xi} e^{-\mu\tau(\xi)},
\ee
which by \eqref{eq:h} is the locus of steady states in the $(\gamma,\xi)$-plane. Differentiating this
relationship with respect to $\xi$ we find
$$\gamma'(\xi)=-\frac{\gamma(\xi)}{\xi}+\frac{\gamma(\xi)g'(\xi)}{g(\xi)}-\gamma(\xi)\mu\tau'(\xi).$$
Using \eqref{eq:taudashxi}
we obtain
\be \label{eq:xigammadashvg}
\xi\gamma'(\xi)=\gamma(\xi)\Big(\frac{\xi g'(\xi)}{g(\xi)} + \mu\tau(\xi)\frac{\xi v'(\xi)}{v(\xi)} -1\Big)=\gamma(\xi)M(\xi).
\ee
Hence $\sgn(\gamma'(\xi))=\sgn(M(\xi))$,
and $\gamma'(\xi)$ changes sign when $M(\xi)$ does, which ensures that
fold bifurcations occur at these points in the $(\gamma,\xi)$-plane.

Note that this is a direct generalisation of the fold bifurcations found in Section~\ref{sec:onehill}
as \eqref{eq:vglam0} reduces to \eqref{eq:xigdashgfold} when $v$ is a constant function, and to \eqref{eq:vuplam0} when $g$ is constant.

Next we determine when Hopf bifurcations may arise
for the general case where both $g$ and $v$ are non-constant. For $\lambda\ne0$, using \eqref{eq:charvg}
we can rewrite the characteristic equation $\Delta(\lambda)=0$ as
\begin{equation} \label{eq:char2}
\lambda = -\gamma + A(\xi)(1- e^{-\lambda\tau(\xi)})( 1+ \frac{\mu}{\lambda}) + Q(\xi)e^{-\lambda\tau(\xi)}.
\end{equation}
At a Hopf bifurcation we set $\lambda = i\omega$, with $\omega \neq 0$,
and from \eqref{eq:char2} obtain
\begin{align*}
i\omega &= -\gamma + A(\xi)(1- e^{-i\omega \tau})( 1+ \frac{\mu}{i\omega }) + Q(\xi)e^{-i\omega  \tau}  \\
&= -\gamma + A(\xi)(1- \cos \omega \tau + i \sin \omega \tau)( \frac{\omega^2 - i\omega\mu}{\omega^2 }) + Q(\xi) ( \cos \omega \tau - i \sin \omega \tau).
\end{align*}
Taking real and imaginary parts, this reduces to
\begin{align} \label{eq:char_vg1x}
\gamma &= A(\xi) ( 1-\cos \omega \tau + \frac{\mu}{\omega} \sin \omega \tau) + Q(\xi) \cos \omega \tau \\
\omega &= A(\xi) ( \sin \omega \tau - \frac{\mu}{\omega} ( 1-\cos \omega \tau) ) - Q(\xi) \sin \omega \tau.
\label{eq:char_vg2x}
\end{align}
Note that when $\omega\tau=2k\pi$ for $k\in\mathbb{Z}$ the last equation reads $\omega=0$, which contradicts the requirement that $\omega \neq 0$ for a Hopf bifurcation. It follows that at a Hopf bifurcation
\begin{equation} \label{eq:cos-tomas}
|\cos(\omega\tau)| <1 .
\end{equation}

The equations \eqref{eq:char_vg1x} and \eqref{eq:char_vg2x} can each be rearranged in two ways: first, as $A(\xi)(1-\cos \omega \tau)=$ \emph{other terms}, or, second, as
$A(\xi) \sin \omega \tau=$ \emph{other terms}. Doing so, then equating the same  expressions from the two equations  gives
\begin{align}
A(\xi) \sin  \omega \tau (\omega^2+\mu^2) &= \omega(\gamma \mu + \omega^2) + Q(\xi)(\omega \sin \omega \tau - \mu \cos \omega \tau)  \label{eq:char_vg1}\\
A(\xi) (1-\cos  \omega \tau) (\omega^2+\mu^2) &= \omega^2(\gamma-\mu) - Q(\xi)(\omega^2 \cos \omega \tau +\mu\omega  \sin \omega \tau).  \label{eq:char_vg2}
\end{align}
Note that when $g^+=g^-$ so $g$ is constant and $Q(\xi)=0$ these equations reduce to
\eqref{eq:sin_cancelled} and \eqref{eq:cos_cancelled}.
On the other hand when $v^+=v^-$ so $v'(\xi)=0$ and hence $A(\xi)=0$, equations
\eqref{eq:char_vg1} and \eqref{eq:char_vg2} can be shown to reduce to
\eqref{eq:realhopf} and \eqref{eq:imaghopf}.

Using \eqref{eq:cos-tomas}, the left-hand side of \eqref{eq:char_vg2} has the sign of $A(\xi)$
(recall also \eqref{eq:sin_cancelled} which has an identical left-hand side).

When $Q(\xi)=0$, the right-hand side of \eqref{eq:char_vg2}
has the sign of $\gamma-\mu$.
For the general case with $Q(\xi)\ne0$,
since  $|\cos\omega\tau|\leq1$ and $|\sin\omega\tau|\leq\omega\tau$, we obtain
$$
|Q(\xi)(\omega^2\cos\omega\tau+\mu\omega\sin\omega\tau)|
 \leq |Q(\xi)|(\omega^2|\cos\omega\tau|+\mu\omega|\sin\omega\tau|) \leq
 |Q(\xi)|\omega^2(1+\mu\tau).
$$
Then,  when $A(\xi)<0$, and hence $v$ is a decreasing function, the left-hand side of \eqref{eq:char_vg2} is negative, while the right-hand side satisfies
\begin{align} \notag
\omega^2(\gamma-\mu) - Q(\xi)(\omega^2\cos\omega\tau+\mu\omega\sin\omega\tau)
& \geq \omega^2(\gamma-\mu) - |Q(\xi)|\omega^2(1+\mu\tau)\\
& = \omega^2(\gamma-\mu- |Q(\xi)|(1+\mu\tau)). \label{eq:gam>mumodboth}
\end{align}
Consequently, there can be no Hopf bifurcations if
$\gamma>\mu+|Q(\xi)|(1+\mu\tau)$.

Similarly for the case that $A(\xi)>0$ (and hence $v$ is an increasing function), the left-hand side of \eqref{eq:char_vg2} is positive, while the right-hand side satisfies
\begin{align} \notag
\omega^2(\gamma-\mu) - Q(\xi)(\omega^2\cos\omega\tau+\mu\omega\sin\omega\tau)
& \leq \omega^2(\gamma-\mu) + |Q(\xi)|\omega^2(1+\mu\tau) \\
& = \omega^2(\gamma-\mu+|Q(\xi)|(1+\mu\tau)). \label{eq:gam<mumodboth}
\end{align}
Thus, when $A(\xi)>0$ there can be no Hopf bifurcations if
$\gamma<\mu-|Q(\xi)|(1+\mu\tau)$.

Note that both the conditions here reduce to the previous ones in the limit as $Q(\xi)\to0$, or equivalently as $g'(\xi)\to0$.
As in Section~\ref{sec:onehill}, we will deal with this by considering the different combinations of increasing and decreasing functions $v$ and $g$  separately.

\subsection{Two Hill functions with $\theta_g \neq \theta_v$}
\label{sec:twothetas}

With $g$ and $v$ both varying and $\theta_g \neq \theta_v$, the dynamics will be comparable to a combination of the four simplified cases as discussed in
Section~\ref{sec:gdownvconst}-\ref{sec:gconstvup}. In particular, if both $g$ and $v$ are decreasing functions, there is always exactly one steady state. The dynamics in the case
$(g \downarrow, v \downarrow)$ will be similar to the case in Section~\ref{sec:gdownvconst} near to $x=\theta_g$ and similar to
the case in Section~\ref{sec:gconstvdown} near to $x=\theta_v$.
The singular steady states in the limiting case give rise to steady states in the smooth case which may undergo Hopf bifurcations
where the stability of steady state changes, as illustrated in Sections~\ref{sec:gdownvconst} and~\ref{sec:gconstvdown}.

For the cases when $g$ and $v$ have opposite monotonicity, there are up to three coexisting steady states. The bifurcations occurring near the two respective thresholds agree with the examples in Section~\ref{sec:onehill}. Since the dynamics can be explained using the results of
Section~\ref{sec:onehill} we only briefly consider the dynamics with $\theta_g\ne\theta_v$
and present two examples,  one with $(g \downarrow, v \uparrow)$ in Section~\ref{sec:gdownvup}, and another for  $(g \uparrow, v \uparrow)$ in
Section~\ref{sec:gupvup}.
In this last example, it is possible for up to five steady states to coexist with sufficiently
steep  nonlinearity.

The constants $\gamma_j$ for $j=1,2,3,4$ introduced in (\ref{eq:corners_vconst}) and (\ref{eq:corners_gconst}) in Section~\ref{sec:onehill} will also play a role here. When $\theta_v\ne\theta_g$ we again define these using
\eqref{eq:corners_vconst} and \eqref{eq:corners_gconst}. However, since both the functions $g$ and $v$ are non-constant we need to define a value for $\tau$ in \eqref{eq:corners_vconst}
and for $g$ in \eqref{eq:corners_gconst}. We proceed as follows. For
$\gamma_1$ and $\gamma_2$ given by \eqref{eq:corners_vconst}, the required value of $\tau$ is $\tau(\theta_g)$ and hence we take $\tau=\tau^-$ if $\theta_v>\theta_g$ and
$\tau=\tau^+$ if $\theta_v<\theta_g$. Similarly in \eqref{eq:corners_gconst} we take
$g=g^-$ if $\theta_g>\theta_v$ and
$g=g^+$ if $\theta_g<\theta_v$.

\subsubsection{Decreasing $g$ and increasing $v$ $(g \downarrow, v \uparrow, \theta_g \neq \theta_v)$} \label{sec:gdownvup}
When $g$ and $v$ have opposite monotonicity, either $(g \downarrow, v \uparrow)$ or $(g \uparrow, v \downarrow)$,  and $\theta_g \neq \theta_v$, there are up to three coexisting steady states in the limiting case. Figure~\ref{fig:5} shows such an example in the case $(g \downarrow, v \uparrow,\theta_v < \theta_g)$.

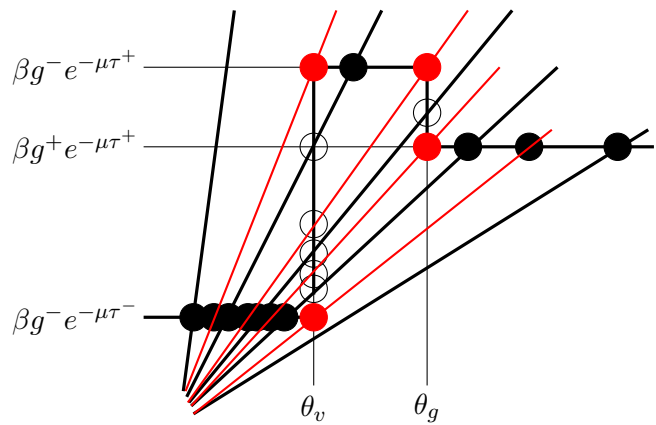
\begin{figure}[thp!]
	\centering
\begin{tikzpicture}[scale = 1.5]
	\tikzstyle{line} = [-,very thick]
	\tikzstyle{upt} = [circle,draw=black,minimum size = 1pt];
	\tikzstyle{dotted line}=[.]
	\tikzstyle{uline}=[.,draw=red,thick]
	\tikzstyle{arrow} = [->,line width = .4mm]
	\tikzstyle{unstable} = [red]
	\tikzstyle{stable} = [blue]
	\tikzstyle{pt} = [circle,draw=black,fill = black,minimum size = 1pt];
	\tikzstyle{bifpt} = [circle,draw=red,fill = red,minimum size = 1pt];
	\def\arrlen{.3}
	\draw[line] (-0.5,-0.2) to (1,-0.2);
	\draw[line] (1,0) to (1,2);
	\draw[line] (1,2) to (2,2);
	\draw[line] (2,2) to (2,1.2);
	\draw[line] (2,1.3) to (4,1.3);
	\draw[dotted line] (-0.5,2) to (1,2);
	\draw[dotted line] (1,-0.8) to (1,2);
	\draw[dotted line] (2,-0.8) to (2,2);
	\draw[dotted line] (-0.5,1.3) to (2,1.3);
	
	\node at (1,-1) {$\theta_v$};
	\node at (2,-1) {$\theta_g$};
	\node at (-1.1,2) {$\beta g^- e^{-\mu\tau^+}$};
	\node at (-1.1,1.3) {$\beta g^+ e^{-\mu\tau^+}$};
	\node at (-1.1,-0.2) {$\beta g^- e^{-\mu\tau^-}$};
	
	\node[upt] at (1,1.3) (12) [] {};
	\node[upt] at (1,0.36) (12) [] {};
	\node[upt] at (1,0.05) (12) [] {};
	\node[upt] at (2,1.6) (12) [] {};
	
	\node[upt] at (1,0.62) (12) [] {};
	\node[upt] at (1,0.18) (12) [] {};
	
	\node[pt] at (-0.06,-0.2) (12) [] {};
	\node[pt] at (0.25,-0.2) (12) [] {};
	\node[pt] at (0.5,-0.2) (12) [] {};
	\node[pt] at (0.74,-0.2) (12) [] {};
	\node[pt] at (1.35,2) (12) [] {};
	\node[pt] at (2.36,1.3) (12) [] {};
	\node[pt] at (3.68,1.3) (12) [] {};
	
	\node[pt] at (0.12,-0.2) (12) [] {};
	\node[pt] at (0.42,-0.2) (12) [] {};
	\node[pt] at (0.62,-0.2) (12) [] {};
	\node[pt] at (2.9,1.3) (12) [] {};
	
	\node[bifpt] at (1,-0.2) (8) [] {};
	\node[bifpt] at (1,2) (8) [] {};
	\node[bifpt] at (2,1.3) (8) [] {};
	\node[bifpt] at (2,2) (8) [] {};
	
	\draw[line] (-0.15,-0.85) to (0.3,2.5);
	\draw[line] (-0.12,-0.9) to (1.6,2.5);
	\draw[line] (-0.1,-0.95) to (2.75,2.5);
	\draw[line] (-0.08,-0.98) to (3.15,2);
	\draw[line] (-0.06,-1.05) to (3.95,1.45);
	
	\draw[uline] (-0.13,-0.85) to (1.2,2.5);
	\draw[uline] (-0.11,-0.95) to (2.36,2.5);
	\draw[uline] (-0.08,-0.98) to (2.64,2);
	\draw[uline] (-0.06,-1.05) to (3.1,1.45);
\end{tikzpicture}
\caption{
Steady states of \eqref{eq:basic}, given by \eqref{eq:h}, occur at the intersections
of $\xi \mapsto\beta e^{-\mu\tau(\xi)}g(\xi)$ and $\xi \mapsto\gamma\xi$.  These are illustrated
for various $\gamma$ in the limiting case of \eqref{eq:gpwconst} and \eqref{eq:vpwconst} for $(g \downarrow, v \uparrow,\theta_v < \theta_g)$. The red lines denote the slopes $\gamma_j$, illustrated here with
$\gamma_4<\gamma_1<\gamma_2<\gamma_3$.
There is a singular steady state at $\xi=\theta_v$ for $\gamma\in(\gamma_4,\gamma_3)$
and at $\xi=\theta_g$ for $\gamma\in(\gamma_1,\gamma_2)$.}
\label{fig:5}
\end{figure}

\begin{figure}[htp!]
	\centering \includegraphics[scale=0.5]{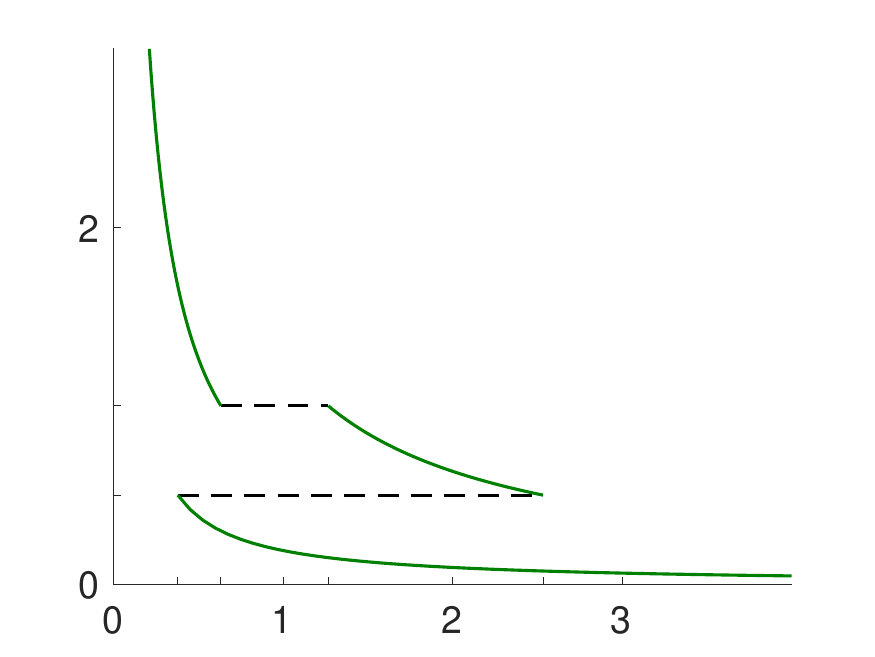}\hspace*{0.5em}\includegraphics[scale=0.5]{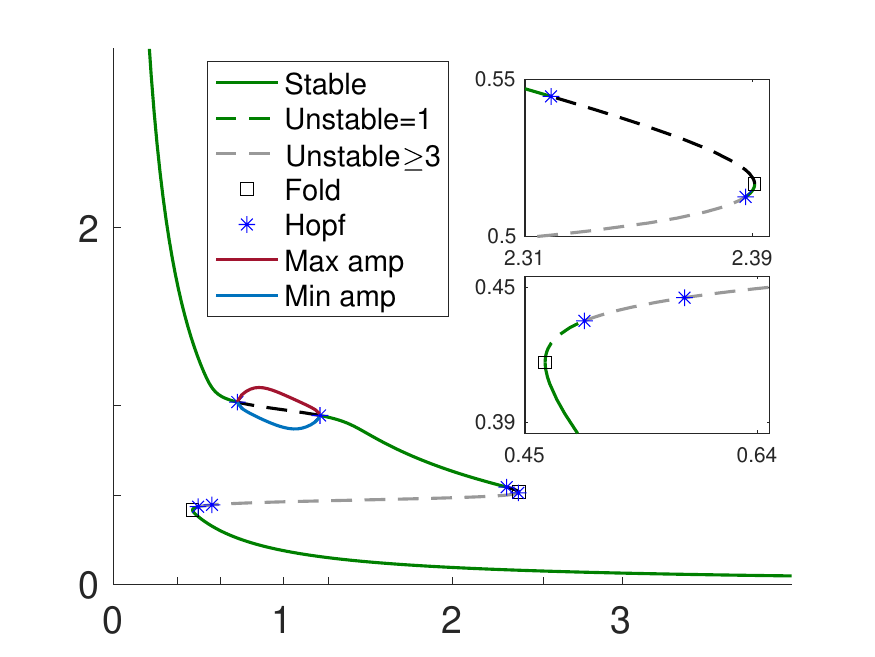}
	\put(-406,140){\rotatebox{90}{$x$}}
	\put(-240,10){$\gamma$}
	\put(-391,7){$\gamma_4$}
	\put(-378,7){$\gamma_1$}
	\put(-352,7){$\gamma_2$}
	\put(-300,7){$\gamma_3$}
	\put(-412,57){$\theta_g$}
	\put(-412,35){$\theta_v$}
	\put(-330,135){$(a)$}
	\put(-60,30){$(b)$}
	\put(-190,140){\rotatebox{90}{$x$}}
	\put(-24,10){$\gamma$}
	\put(-175,7){$\gamma_4$}
	\put(-162,7){$\gamma_1$}
	\put(-136,7){$\gamma_2$}
	\put(-84,7){$\gamma_3$}
	\put(-196,57){$\theta_g$}
	\put(-196,35){$\theta_v$}\\
	\includegraphics[scale=0.5]{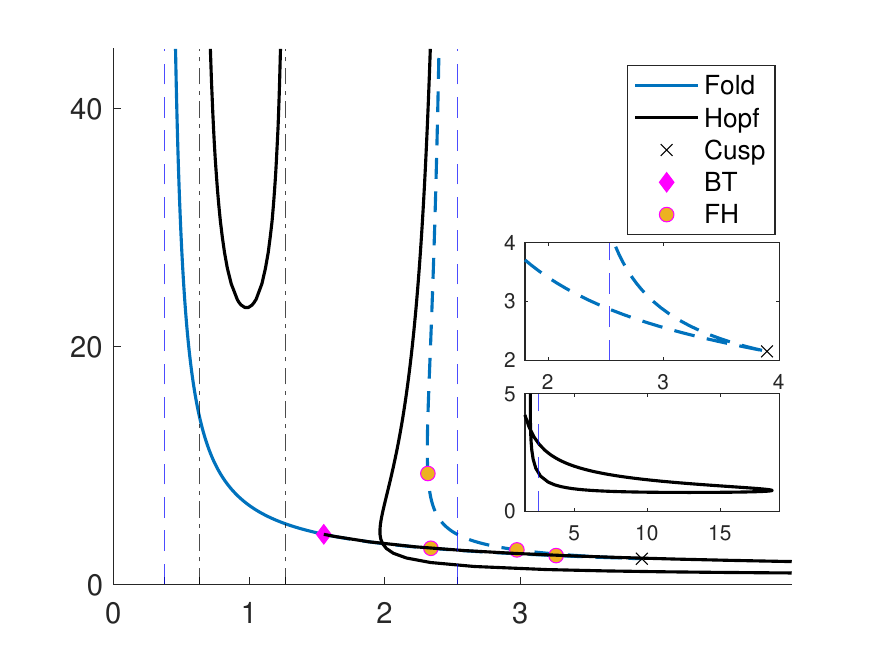}\hspace*{0.5em}\includegraphics[scale=0.5]{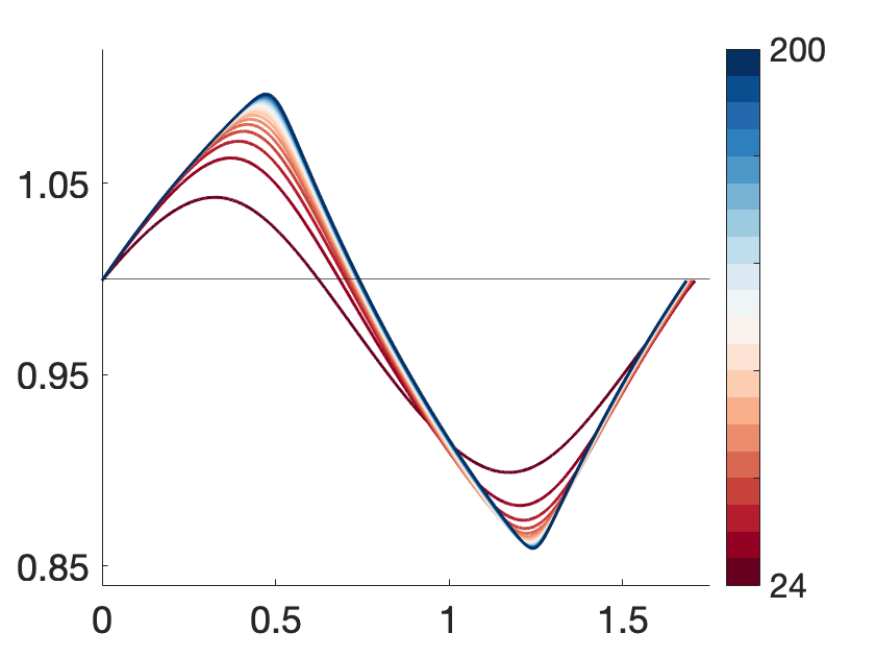}
	\put(-406,100){\rotatebox{90}{\scriptsize$m=n$}}
	\put(-394,7){$\gamma_4$}
	\put(-382,7){$\gamma_1$}
	\put(-362,7){$\gamma_2$}
	\put(-320,7){$\gamma_3$}
	\put(-240,10){$\gamma$}
	\put(-305, 140){$(c)$}
	\put(-110, 140){$(d)$}
	\put(-192,140){\rotatebox{90}{$x$}}
	\put(-198,86){$\theta_g$}
	\put(-42,8){$t$}
	\put(-25,80){$m$}
	\caption{Bifurcations of \eqref{eq:basic}-\eqref{eq:thres}
in the case 
$(g \downarrow, v \uparrow,\theta_v < \theta_g)$
with parameters $\beta=1.4$, $\mu=0.2$, $g^-=1>g^+=0.54$, $\theta_v=0.5<\theta_g=1$, $a=1$, and $v^-=0.1<v^+=2$.
With these parameters $\mu<\gamma_4=0.3789<\gamma_1=0.6334<\gamma_2=1.2668<\gamma_3=2.5335$.
(a) The limiting case with $g$ defined by \eqref{eq:gpwconst} and $v$ defined by \eqref{eq:vpwconst}. The stable steady state is shown as a green solid line, and the singular steady state as a black dashed line.
(b) With smooth nonlinearity $g$ and $v$ defined by \eqref{eq:vghill} with $m=n=40$. Periodic solutions are
represented by maximum (in red) and minimum (in blue) of x(t) on the solution.
(c) Two-parameter continuations in $m=n$ and $\gamma$ of the fold and the Hopf bifurcations with the other parameters as above. There is a cusp point at $(\gamma, m)=(3.8948, 2.1535)$, a BT point at $(\gamma, m)=(1.5502, 4.2073)$ and fold-Hopf points at $(\gamma, m)=(2.3399, 3.0346), (3.2632, 2.4198), (2.9741, 2.8976)$ and $(2.3178, 9.3191)$. (d) Profile of the stable periodic orbits at $\gamma=1$ emanating from the Hopf bifurcations in (c).
The color map indicates values of the continuation parameter $m=n$.} \label{fig:gdownvup}
\end{figure}
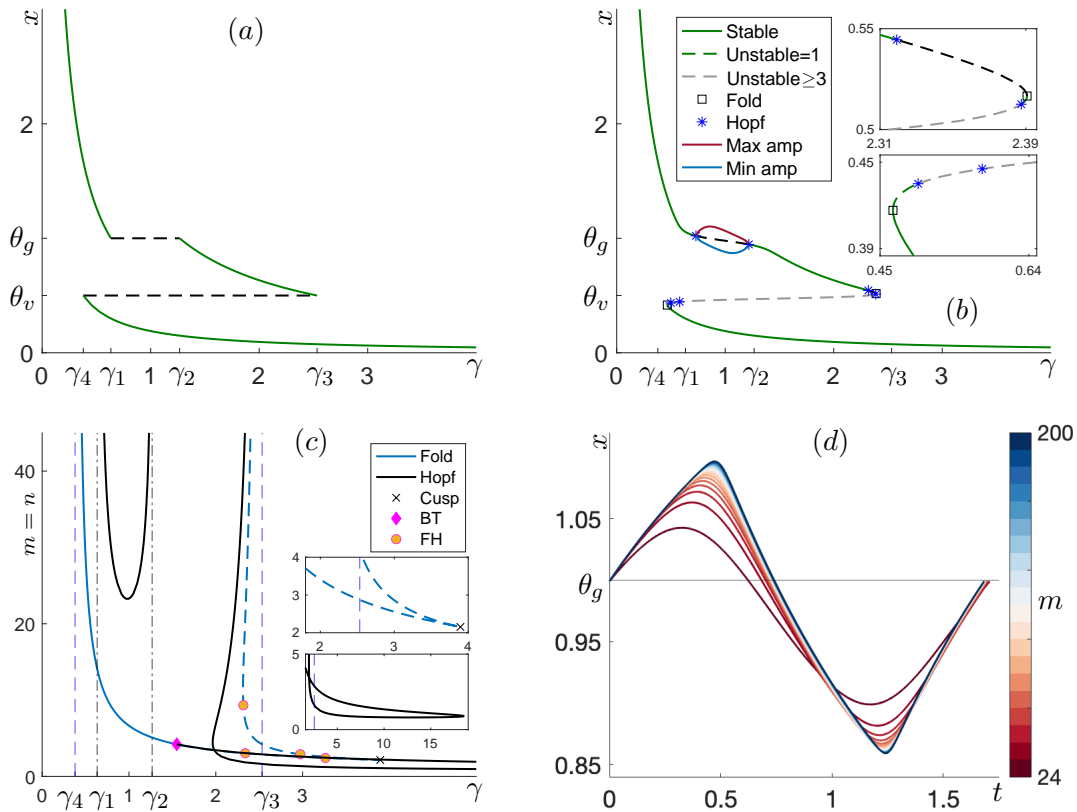

Figure~\ref{fig:5} shows that if $\gamma\in(\gamma_4,\gamma_3)$ then there are three steady states in the limiting case. At most two of these steady states will be stable, with the intermediate steady state at $x=\theta_v$ singular. If the intervals $(\gamma_1,\gamma_2)$
and $(\gamma_4,\gamma_3)$ overlap then it is possible to have one stable steady state
coexisting with two singular steady states. Such a scenario is illustrated in Figure~\ref{fig:gdownvup}(a) which has two singular steady states and one stable steady state for $\gamma\in(\gamma_1,\gamma_2)$.
The number of steady states changes when $\gamma\xi$ intersects $\beta e^{-\mu\tau(\xi)}g(\xi)$ at the corners along $\theta_v$ at parameter values $\gamma_3$ and $\gamma_4$.
The singular `fold' bifurcations give rise to classical fold bifurcations in the case
of smooth nonlinearities defined by \eqref{eq:vghill}, as seen in
Figure~\ref{fig:gdownvup}(b) for the case $m=n=40$.
We note from \eqref{eq:vglam0} that since $g'$ and $v'$ have opposite signs, we require $\xi v'(\xi)/v(\xi)>1/(\mu\tau(\xi))$ for a fold bifurcation to occur, and so
for $m=n$ a somewhat larger value of the steepness parameter $m$ is required to obtain a fold bifurcation than would have been the case with constant $g$ (for which, by \eqref{eq:vuplam0}, fold bifurcations occur when $\xi v'(\xi)/v(\xi)=1/(\mu\tau(\xi))$ is satisfied).

In the example presented in Figure~\ref{fig:gdownvup} the function $v$ is increasing with
$\mu<\gamma_4$, and so we expect to see similar behaviour for $x\approx\theta_v$ as in Case 3 of
Section~\ref{sec:gconstvup}, while since $g$ is decreasing the behaviour
for $x\approx\theta_g$ should be similar to that seen in
Section~\ref{sec:gdownvconst}. This is indeed what is observed.

Since $v$ is increasing, from \eqref{eq:char_vg2} and \eqref{eq:gam<mumodboth} it is necessary that $\gamma>\mu-|Q(\xi)|(1+\mu\tau)$ for Hopf bifurcations to occur.
However, this does not impose any additional constraints 
as in this example $\mu-|Q(\xi)|(1+\mu\tau) < \mu < \min_{j=\{1,2,3,4\}}\gamma_j$,
and as we already saw in Section~\ref{sec:onehill}, Hopf bifurcations typically occur for $\gamma> \min_{j=\{1,2,3,4\}}\gamma_j$.
So the steady state can lose stability in a Hopf bifurcation close to the fold
point. This is seen in Figure~\ref{fig:gdownvup}(b) close to $(\gamma,x)=(\gamma_3,\theta_v)$ and is
similar to the behaviour observed in Cases 2 and 3 in Section~\ref{sec:gconstvup}.
Figure~\ref{fig:gdownvup}(c) shows two-parameter continuations of these bifurcations.
Apart from the Hopf bifurcations giving rise to a stable periodic orbit for $\gamma\in(\gamma_1,\gamma_2)$ the rest of the dynamics observed are remarkably similar to those seen in
Figure~\ref{fig:vup_ex1b} with $v$ increasing and constant $g$. In particular,
as the
steepness of the nonlinearities given by $m=n$ increases, the fold bifurcations approach $\gamma=\gamma_4$ and $\gamma=\gamma_3$ with the steady states at which they happen approaching $\theta_v$, the fold bifurcations disappear at a cusp point (at $\gamma=2.1535$), while the
Hopf bifurcations and stable periodic orbits can occur for smaller values of $m>0.8079$.
There is also a Bogdanov-Takens bifurcation, fold-Hopf bifurcations at which Hopf bifurcations cross the fold points, and for the fold close to $\gamma=\gamma_3$ the steady state loses stability in a subcritical Hopf bifurcation for arbitrary large $m=n$, and not at the fold bifurcation.

The most significant difference between the example shown in Figure~\ref{fig:gdownvup}
and Cases 2 and 3 from Section~\ref{sec:gconstvup} is the pair of Hopf bifurcations seen for $x\approx\theta_g$
and $\gamma\in(\gamma_1,\gamma_2)$. In Figure~\ref{fig:gdownvup}(c) the steady state is seen to lose stability
between a pair of Hopf bifurcations for $m>23$. Since for $m=n\gg0$ the function $v$ is essentially constant except for $\theta\approx\theta_v$, it follows that $v$ is close to constant for $\theta\approx\theta_g$ and so these Hopf bifurcations and the resulting periodic orbits follow the dynamics explored in Section~\ref{sec:gdownvconst}. Indeed, the dynamics associated with the Hopf bifurcations and
periodic orbits for $x\approx\theta_g$ seen in Figure~\ref{fig:gdownvup} is remarkably similar to that seen for constant $v$ and decreasing $g$ in Figure~\ref{fig:gdown_ex1} of Section~\ref{sec:gdownvconst}.
The Hopf bifurcations are supercritical, generating a stable periodic orbit which co-exists with the unstable steady state between the pair of Hopf bifurcations. The amplitude of the periodic orbit is shown for $m=n=40$ in Figure~\ref{fig:gdownvup}(b), from which we see that it is always small and shrinks to zero at the Hopf bifurcations at each end of the interval.

Figure~\ref{fig:gdownvup}(d) depicts the profile of such stable periodic orbits for a fixed
value of $\gamma$ in between the two Hopf bifurcations. As $m=n \to \infty$, the periodic orbits appear to approach a limiting profile. Notice that this stable periodic orbit
oscillates around $\theta_g$ but does not cross the other threshold at $\theta_v=0.5$.
Thus in the limiting case as $m=n\to\infty$ the delay will be constant and equal to $\tau^+$
on this periodic orbit. 
In \cite{sausage} we construct
slowly oscillating periodic solutions for both constant and threshold delays,
and using those constructions it can be seen that
the limiting profile in Figure~\ref{fig:gdownvup}(d) is a slowly oscillating periodic solution.

\subsubsection{Increasing $g$ and increasing $v$ $(g\uparrow,v\uparrow,\theta_g \neq \theta_v)$}
\label{sec:gupvup}

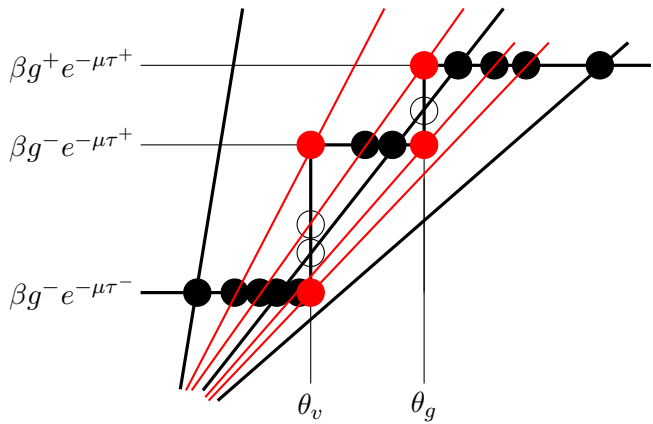
\begin{figure}[thp!]
	\centering
	\begin{tabular}{cc}
		\begin{tikzpicture}[scale = 1.5]
			\tikzstyle{line} = [-,very thick]
			\tikzstyle{dotted line}=[.]
			\tikzstyle{uline}=[.,draw=red,thick]
			\tikzstyle{arrow} = [->,line width = .4mm]
			\tikzstyle{unstable} = [red]
			\tikzstyle{stable} = [blue]
			\tikzstyle{pt} = [circle,draw=black,fill = black,minimum size = 2pt];
			\tikzstyle{upt} = [circle,draw=black,minimum size = 2pt];
			\tikzstyle{bifpt} = [circle,draw=red,fill = red,minimum size = 1pt];
			\def\arrlen{.3}
			
			\draw[line] (-0.5,0) to (1,0);
			\draw[line] (1,0) to (1,1.3);
			\draw[line] (1,1.3) to (2,1.3);
			\draw[line] (2,1.3) to (2,2);
			\draw[line] (2,2) to (4,2);
			\draw[dotted line] (2,0) to (2,1);
			\draw[dotted line] (-0.5,1.3) to (1,1.3);
			\draw[dotted line] (-0.5,2) to (2,2);
			\draw[dotted line] (1,-0.8) to (1,1);
			\draw[dotted line] (2,-0.8) to (2,1.3);
			
			\node at (1,-1) {$\theta_v$};
			\node at (2,-1) {$\theta_g$};
			\node at (-1.1,2) {$\beta g^+ e^{-\mu\tau^+}$};
			\node at (-1.1,0) {$\beta g^- e^{-\mu\tau^-}$};
			\node at (-1.1,1.3) {$\beta g^- e^{-\mu\tau^+}$};
			
			\node[pt] at (0,0) (11) [] {};
			\node[pt] at (0.7,0) (11) [] {};
			\node[pt] at (1.72,1.3) (12) [] {};
			\node[pt] at (2.3,2) (11) [] {};
			\node[pt] at (3.55,2) (11) [] {};
			
			\node[upt] at (1,0.35) (12) [] {};
			\node[upt] at (1,0.6) (12) [] {};
			\node[upt] at (2,1.6) (11) [] {};
			
			\node[pt] at (0.33,0) (11) [] {};
			\node[pt] at (0.55,0) (11) [] {};
			\node[pt] at (0.9,0) (11) [] {};
			\node[pt] at (1.48,1.3) (11) [] {};
			\node[pt] at (2.62,2) (11) [] {};
			\node[pt] at (2.9,2) (11) [] {};
			
			\node[bifpt] at (1,0) (8) [] {};
			\node[bifpt] at (1,1.3) (8) [] {};
			\node[bifpt] at (2,1.3) (8) [] {};
			\node[bifpt] at (2,2) (8) [] {};
			
			\draw[line] (-0.15,-0.85) to (0.4,2.5);
			\draw[line] (0.05,-0.86) to (2.7,2.5);
			\draw[line] (0.18,-0.95) to (3.8,2.2);
			
			\draw[uline] (-0.1,-0.86) to (1.65,2.5);
			\draw[uline] (-0.05,-0.86) to (2.35,2.5);
			\draw[uline] (0.07,-0.92) to (2.8,2.2);
			\draw[uline] (0.1,-0.95) to (3.1,2.2);
		\end{tikzpicture}
	\end{tabular}
\caption{Steady states of \eqref{eq:basic}, given by \eqref{eq:h}, occur at the intersections of $\xi \mapsto\beta e^{-\mu\tau(\xi)}g(\xi)$ and $\xi \mapsto\gamma\xi$.  These are illustrated for various $\gamma$ in the limiting case of \eqref{eq:gpwconst} and \eqref{eq:vpwconst} for $(g \uparrow, v \uparrow,\theta_v < \theta_g)$. The red lines denote $\gamma=\gamma_4, \gamma_2, \gamma_1, \gamma_3$ in ascending order.
}
\label{fig:6}
\end{figure}

\begin{figure}[thp!]
	\centering
	\includegraphics[scale=0.5]{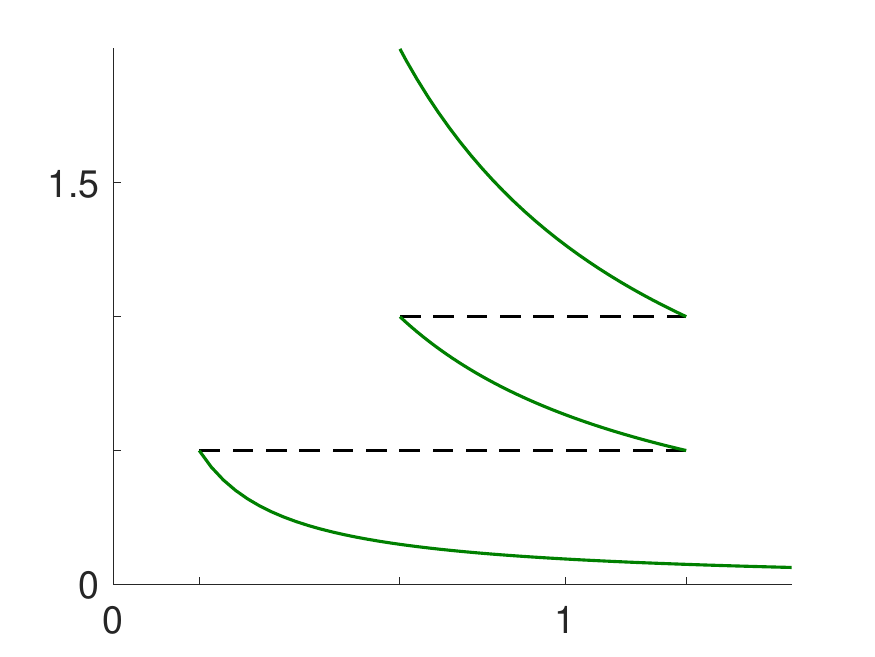}\hspace*{0.5em}\includegraphics[scale=0.5]{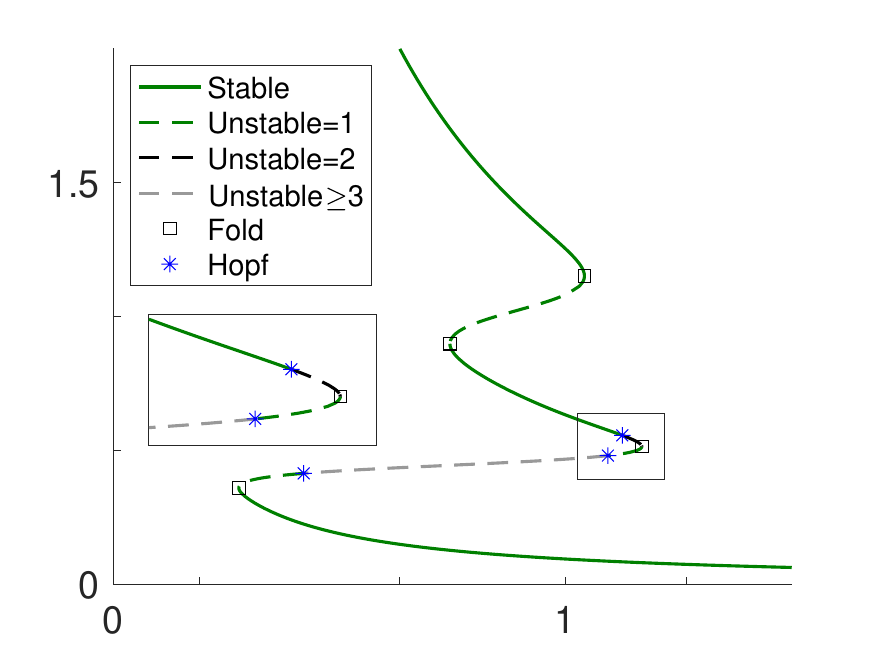}
	\put(-406,140){\rotatebox{90}{$x$}}
	\put(-240,10){$\gamma$}
	\put(-383,7){$\gamma_4$}
	\put(-335,7){$\gamma_2$}
	\put(-279,7){$\gamma_1=\gamma_3$}
	\put(-412,78){$\theta_g$}
	\put(-412,45){$\theta_v$}
	\put(-360,135){$(a)$}
	\put(-90,135){$(b)$}
	\put(-190,140){\rotatebox{90}{$x$}}
	\put(-24,10){$\gamma$}
	\put(-167,7){$\gamma_4$}
	\put(-119,7){$\gamma_2$}
	\put(-64,7){$\gamma_1=\gamma_3$}
	\put(-196,78){$\theta_g$}
	\put(-196,45){$\theta_v$}\\
	\includegraphics[scale=0.5]{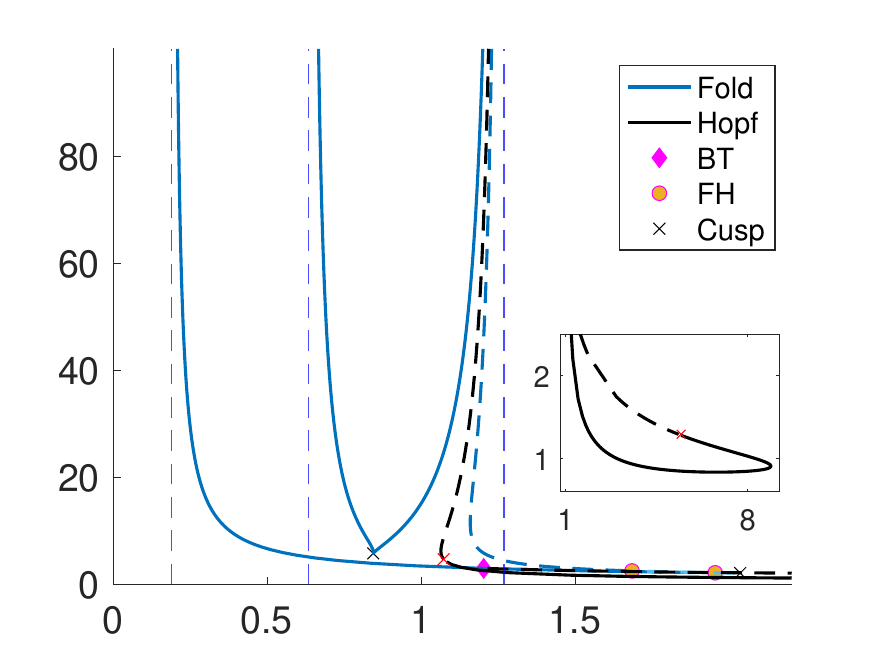}
	\put(-120,140){$(c)$}
	\put(-190,127){\rotatebox{90}{\scriptsize$m=n$}}
	\put(-24,10){$\gamma$}
	\put(-174,7){$\gamma_4$}
	\put(-140,7){$\gamma_2$}
	\put(-94,7){$\gamma_1$}
	\caption{Bifurcation diagram of \eqref{eq:basic}-\eqref{eq:thres} for $(g \uparrow, v \uparrow,\theta_v < \theta_g)$
		with parameters $\beta=1.4, \mu=0.2, g^-=0.5, g^+=1, \gamma=0.8, \theta_g=1, \theta_v=0.5, a=1, v^-=0.1$ and $v^+=2$. With these parameters
$\gamma_4=0.1895< \gamma_2=0.6334<\gamma_1=\gamma_3=1.2668$.
(a) The limiting case with $g$ defined by \eqref{eq:gpwconst} and $v$ defined by \eqref{eq:vpwconst}. The stable steady state is shown as a green solid line, and the singular steady state as a black dashed line. (b) With smooth nonlinearity $g$ and $v$ defined by \eqref{eq:vghill} with $m=n=20$. (c) Two-parameter continuations in $m=n$ and $\gamma$ of the fold and the Hopf bifurcations with the other parameters as above. There are two cusp points at $(\gamma, m)=(2.0329, 2.1053)$ and $(0.84265, 5.7871)$, a BT point at $(\gamma, m)=(1.2009, 2.9825)$, two fold-Hopf points at $(\gamma, m)=(1.6820, 2.5172)$ and $(1.9519, 2.1629)$, and two Bautin bifurcations at  $(\gamma, m)=(5.4365, 1.2928)$ and $(1.0709, 4.6890)$. }
	\label{fig:gupvup}
\end{figure}

When $v$ and $g$ are both increasing with $\theta_g \neq \theta_v$
there are up to five coexisting steady states in the limiting case as shown in Figure~\ref{fig:6}. The four corners at the respective thresholds $\theta_g$ and $\theta_v$ are associated with the emergence of fold bifurcations.
A necessary and sufficient condition to obtain five coexisting steady states is that
$(\gamma_2,\gamma_1)\cap(\gamma_4,\gamma_3)\ne\emptyset$.
Note that since $g$ is now increasing the values  $\gamma_2 < \gamma_1$ have the opposite order than in previous example, see \eqref{eq:corners_vconst}.

Figure~\ref{fig:gupvup}(a) shows such an example, with five co-existing steady states in the limiting case when
$\gamma\in(\gamma_2,\gamma_1)$.
Referring to the discussion in Sections~\ref{sec:gupvconst} and \ref{sec:gconstvup}, in the limiting case the non-singular steady states are always stable, leading to tristability of steady states in this parameter interval.
This then leads to tristability of steady states for smooth nonlinearities with $m=n$ sufficiently large, as depicted in Figure~\ref{fig:gupvup}(b).

Following the discussion in Sections~\ref{sec:gupvconst} and \ref{sec:gconstvup}, the intermediate singular steady states become unstable steady states between the fold bifurcations in the smooth case. However, the behaviour is not the same at each of the folds.
In the example presented in Figure~\ref{fig:gupvup} we have $\mu\in(\gamma_4,\gamma_3)$, and so we expect to see similar behaviour for $x\approx\theta_v$ as in Case 2 of
Section~\ref{sec:gconstvup}, while since $g$ is increasing the behaviour
for $x\approx\theta_g$ should be similar to what was observed in
Section~\ref{sec:gupvconst}, and this is what we observe.

For the fold bifurcation near  $(\gamma,\xi)=(\gamma_4,\theta_v)$,
the steady state loses stability at the fold for all $m$ sufficiently large (seen for $m=n=20$ in Figure~\ref{fig:gupvup}(b)). Since $v$ is increasing, from \eqref{eq:gam<mumodboth}
there can be no Hopf bifurcations for $\gamma<\mu-|Q(\xi)|(1+\mu\tau)$.
However,  $\lim_{m\to\infty}Q(\xi)=0$ for
$\xi\ne\theta_g$, and since $\gamma_4<\mu$ there can be no Hopf bifurcations near to the fold for $m$ sufficiently large.
The same is not true for the fold close to
$(\gamma,\xi)=(\gamma_3,\theta_v)$, where since $\gamma_3>\mu$ the steady state may lose stability in a Hopf bifurcation near to the fold point, which is indeed what is seen for all
$m$ sufficiently large. This  is similar to the behaviour seen
in the previous example in Figure~\ref{fig:gdownvup} and to Cases 2 and 3 in Section~\ref{sec:gconstvup}.

Following the theory of Sections~\ref{sec:gupvconst} and~\ref{sec:gconstvup},
as $m=n$ is increased there will also be an infinite sequence of Hopf bifurcations on the unstable branches that cross $\theta_g$ and $\theta_v$.
For the branch that crosses $x=\theta_g$ these Hopf bifurcations will be confined between the fold bifurcations as $m=n$ is increased. No Hopf bifurcations are seen on this branch in Figure~\ref{fig:gupvup}(b) because the value of $m$ is too small, with the first Hopf bifurcation only occurring for $m=n\approx61$.

\begin{figure}[thp!]
	\centering	\includegraphics[scale=0.5]{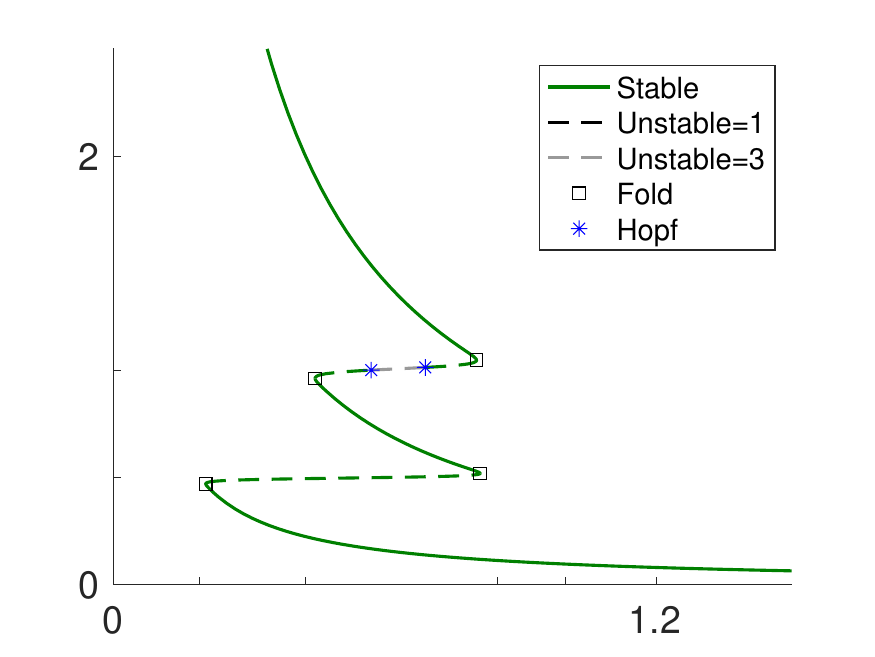}
	\put(-192,140){\rotatebox{90}{$x$}}
	\put(-195,66){$\theta_g$}
	\put(-195,39){$\theta_v$}
	\put(-30,10){$\gamma$}
	\put(-166,7){$\gamma_4$}
	\put(-142,7){$\gamma_2$}
	\put(-95,7){$\gamma_1$}
	\put(-77,7){$\mu$}
	\caption{Bifurcations of \eqref{eq:basic}-\eqref{eq:thres}
		in the case $(g \uparrow, v \uparrow,\theta_v < \theta_g)$
with smooth nonlinearities $g$ and $v$ defined by \eqref{eq:vghill} with $m=n=100$, and
other parameters $\beta=1.4$, $\mu=1$, $g^-=0.5>g^+=1$, $\theta_v=0.5<\theta_g=1$, $a=1$, and $v^-=0.5<v^+=2$. With these parameters $\gamma_4=0.1895$, $\gamma_2=0.4246$, and $\gamma_1=\gamma_3=0.8491$.}
	\label{fig:gupvup_ex2}
\end{figure}

Figure~\ref{fig:gupvup_ex2} shows another example for
$(g\uparrow,v\uparrow,\theta_g \neq \theta_v)$, but this time with $\mu>\gamma_3$ which corresponds to
Case 1 in Section~\ref{sec:gconstvup}.
In this example no Hopf bifurcations are observed on the unstable branch that crosses $\xi=\theta_v$.
This follows from \eqref{eq:gam<mumodboth} because there can be no Hopf bifurcations
for $\gamma<\mu-|Q(\xi)|(1+\mu\tau)$, and $|Q(\xi)|\to0$ as $n\to\infty$ for $\xi\ne\theta_g$. On the other hand, a pair of Hopf bifurcations is observed on the branch of unstable steady states that crosses
$\xi=\theta_g$. In this case \eqref{eq:gam<mumodboth} does not preclude Hopf bifurcations
because Proposition~\ref{prop:fxpr}(2),
with $f(x,p,r)$ given by $Q(\xi)$ with parameter $p=n$ from $g$,
shows that $|Q(\xi)|$ can be arbitrarily large for $\xi\approx\theta_g$ as $n\to\infty$.

\subsection{Two Hill functions with $\theta_g = \theta_v$}
\label{sec:onetheta}

The cases of $(g\uparrow,v\uparrow, \theta_v = \theta_g)$ and
$(g\downarrow,v\downarrow, \theta_v = \theta_g)$ where
both functions are either increasing or decreasing
are reasonably straightforward. However,  the two cases for which one function is increasing and the other decreasing,
$(g\downarrow,v\uparrow, \theta_v = \theta_g)$ and
$(g\downarrow,v\uparrow, \theta_v = \theta_g)$ are altogether more delicate and surprising.

\subsubsection{$g$ and $v$ both increasing or decreasing $(g\uparrow,v\uparrow,\theta_g=\theta_v)$
and $(g\downarrow,v\downarrow,\theta_g=\theta_v)$}
\label{sec:bothupordown}

In both these cases $\xi \mapsto\beta e^{-\mu\tau(\xi)}g(\xi)$ is monotonic, so the existence of steady states is straightforward following the theory in Section~\ref{sec:onehill}. Writing $\theta_{gv}$ for the value of $\theta$ when $\theta_g=\theta_v$, and defining
\be \label{eq:corners_onetheta}
\gamma_{13}=\frac{\beta e^{-\mu\tau^+}\! g^+}{\theta_{gv}} \quad \text{and} \quad
\gamma_{24}=\frac{\beta e^{-\mu\tau^-\!} g^-}{\theta_{gv}},
\ee
with $g$ and $v$ both decreasing, there is always a unique steady state. In the limiting case of piecewise constant nonlinearities \eqref{eq:gpwconst}-\eqref{eq:vpwconst}, this steady state is stable for
$\gamma<\gamma_{13}$ and $\gamma>\gamma_{24}$, and is singular for $\gamma\in(\gamma_{13},\gamma_{24})$.
In the smooth case, this gives rise to a unique stable steady state for very small or very large $\gamma$. The question then arises as to when the steady state may lose stability in a Hopf bifurcation for $\gamma$ in or near the interval $(\gamma_{13},\gamma_{24})$.
From \eqref{eq:gam>mumodboth} when $g$ and $v$ are both decreasing Hopf bifurcations cannot occur if
\be \label{eq:gam>muboth}
\gamma>\mu+|Q(\xi)|(1+\mu\tau)=\mu+\gamma(1+\mu\tau)\left|\frac{\xi g'(\xi)}{g(\xi)}\right|.
\ee
Thus when $g$ is a constant function, which implies that $Q=0$, there can only be Hopf bifurcations when $\gamma<\mu$ as  seen in Section~\ref{sec:gconstvdown}. In contrast,
by Proposition~\ref{prop:fxpr}, we have $|\theta_{gv} g'(\theta_{gv})/g(\theta_{gv})|\to\infty$ as $n\to\infty$
while \eqref{eq:gam>muboth} cannot be satisfied when $|\xi g'(\xi)/g(\xi)|>1$. Consequently for $n$ large there is no constraint preventing Hopf bifurcations, and we expect to see Hopf bifurcations for $\xi\approx\theta_{gv}$ for $n$ sufficiently large, just as we did in Section~\ref{sec:gdownvconst}.

\begin{figure}[thp!]
	\centering
	
\includegraphics[scale=0.5]{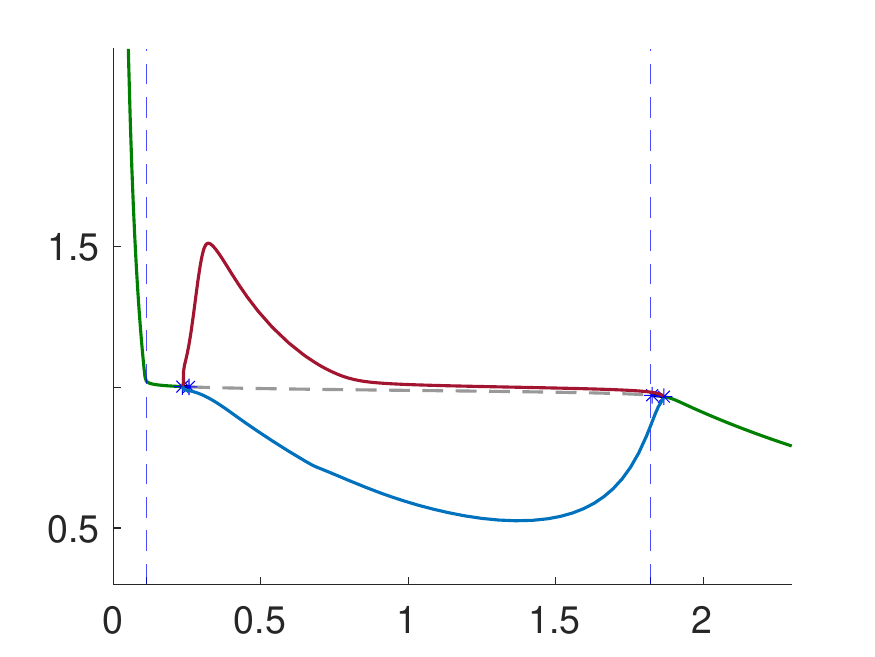}\hspace*{0.5em}\includegraphics[scale=0.5]{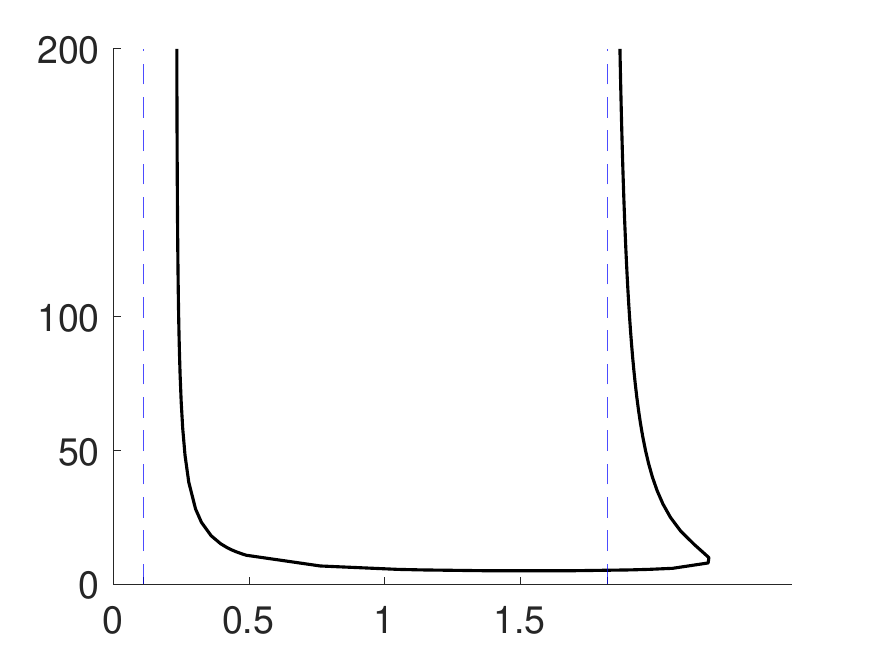}
	\put(-316,135){$(a)$}
	\put(-406,140){\rotatebox{90}{$x$}}
	\put(-396,7){$\gamma_{13}$}
	\put(-281,7){$\gamma_{24}$}
	\put(-241,10){$\gamma$}
	\put(-416,65){$\theta_{gv}$}
	\put(-114,140){$(b)$}
	\put(-190,102){\rotatebox{90}{$\scriptsize m=n$}}
    \put(-180,7){$\gamma_{13}$}
	\put(-71,7){$\gamma_{24}$}
	\put(-24,10){$\gamma$}\\
	
\includegraphics[scale=0.5]{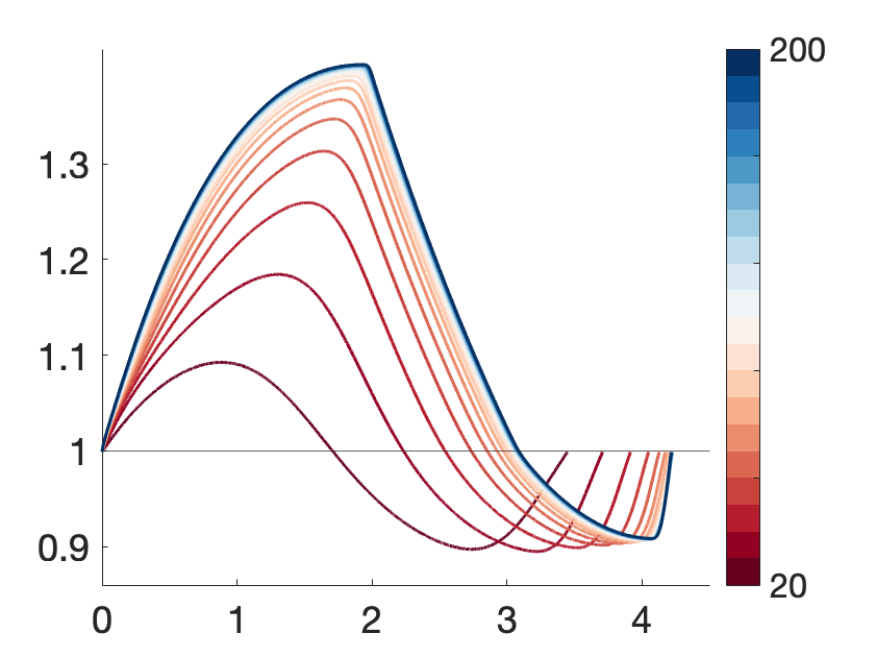}\hspace*{0.5em}\includegraphics[scale=0.5]{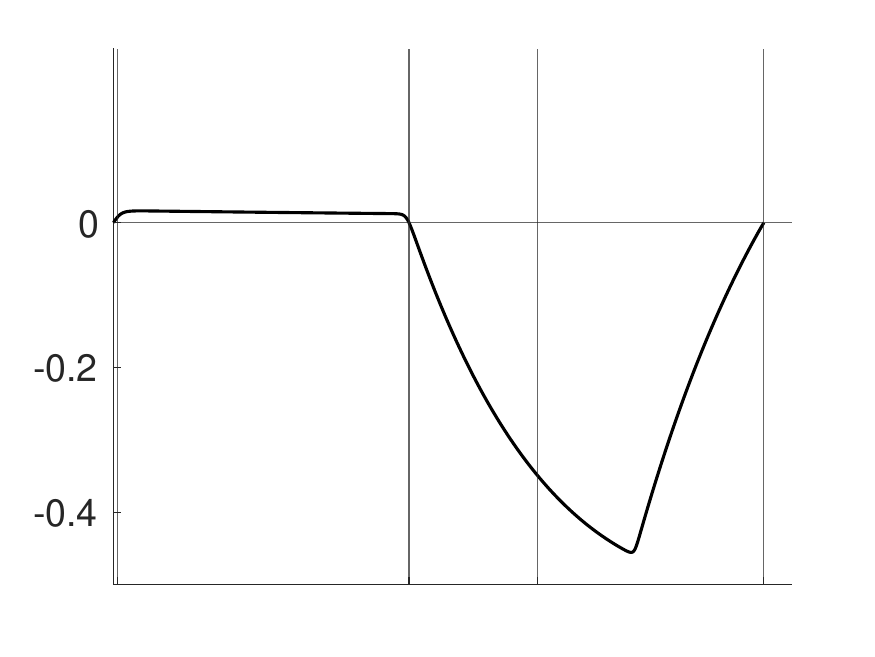}
	\put(-104,135){$(d)$}
	\put(-193,124){\rotatebox{90}{$x\!-\!x^*$}}
	\put(-192,8){$0$}
	\put(-183,8){$\scriptsize t_2\!-\!\tau(x_{t_2\!})$}
	\put(-115,8){$t_2$}
	\put(-93,8){$\scriptsize t_4\!-\!\tau(x_{t_4\!})$}
	\put(-24,8){$t_4$}
	\put(-325,140){$(c)$}
	\put(-411,140){\rotatebox{90}{$x$}}
	\put(-263,8){$t$}
	\put(-242,80){\small$m\!=\!n$}

	\caption{Bifurcation diagram of \eqref{eq:basic}-\eqref{eq:thres} for $(g\downarrow,v\downarrow,\theta_g=\theta_v)$
		with parameters $\beta=3$, $\mu=0.5$,  $\gamma=1$, $\theta_g=\theta_v=\theta_{vg}=1$, $a=1$,
$g^-=1>g^+=0.1$, and $v^-=1>v^+=0.5$. With these parameters
$\gamma_{13}=0.1104$ and $\gamma_{24}=1.8196$.
(a) With nonlinearities $g$ and $v$ defined by \eqref{eq:vghill} with $m=n=200$. The Hopf bifurcations associated with stability change occur at $(\gamma,x)=(0.2333,1.0018)$ with $\omega=2.3128$ and $(\gamma,x)=(1.8669,0.9659)$ with $\omega=2.3155$.
Stable periodic orbits are represented by maximum (in red) and minimum (in blue) of $x(t)$ on the periodic solution.
(b) Two-parameter continuations in $m=n$ and $\gamma$ of the Hopf bifurcations with the other parameters as above.
(c) Profile of the stable periodic orbits with $\gamma=0.4$. The color map indicates values of the continuation parameter $m=n$. 
(d) Profile of the stable periodic orbit with $m=200$ and $\gamma=1.3$ and the other parameters as above. The associated steady state is at $x^*=0.9848$. 
The period of the periodic orbit $T=t_4=2.8772$.}
	\label{fig:gdownvdown}
\end{figure}
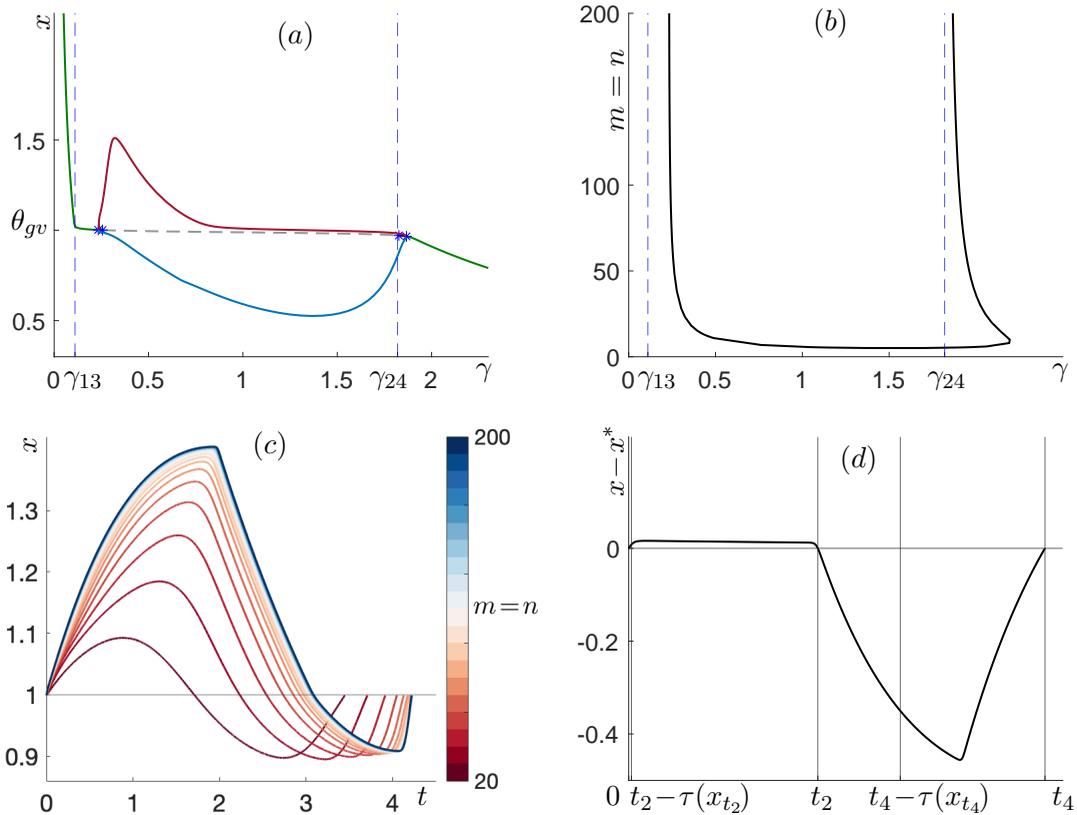

Figure~\ref{fig:gdownvdown} illustrates the case of $g$ and $v$ both decreasing with $\theta_v=\theta_g=\theta_{gv}$.
Panel (a) shows the dynamics for smooth nonlinearities $g$ defined by \eqref{eq:vghill} with $m=n=200$.
In this case the steady state is unstable between two supercritical Hopf bifurcations.
These Hopf bifurcations occur close to but slightly to the right of the limiting values $\gamma_{13}$ and $\gamma_{24}$.
The resulting periodic orbit is stable
on the whole of the interval between the Hopf bifurcations.

Figure~\ref{fig:gdownvdown}(b) shows a two-parameter continuation of the curve of Hopf bifurcations at which the stable periodic orbit is created.
The amplitude plot of the periodic orbits in panel (a) of Figure~\ref{fig:gdownvdown} suggests that different behavior is seen near the two Hopf points, and this is illustrated in
the periodic solution profiles shown in panels (c) and (d).
In (c) with $\gamma=0.4$, as $m=n$ is increased the profile of the periodic orbit converges to a limiting profile (which we construct in \cite{sausage}). The situation is different with  $\gamma>1$ as shown in panel (d) with $\gamma=1.3$ and $m=n=200$. While there is again a large period periodic orbit, the profile is quite different, with $x(t)$ never much greater than $x^*$ even though there is a large segment $t\in(t_0,t_2)$ of the periodic orbit for which $x(t)>x^*$.
This large subinterval of the periodic orbit on which the solution is nearly constant
looks quite different to the periodic orbits seen in (c).

The case of $(g\uparrow,v\uparrow, \theta_v = \theta_g)$ is analogous to
the scenarios considered in Sections~\ref{sec:gupvconst} and~\ref{sec:gconstvup} which lead to a pair of fold bifurcations when \eqref{eq:vglam0} is satisfied,
and an interval of $\gamma$ values for which two stable
steady states coexist with an unstable steady state.
Because $v$ and $g$ are both increasing, the functions
$\xi \mapsto\xi g'(\xi)/g(\xi)$ and $\xi \mapsto\mu\tau(\xi)\xi v'(\xi)/v(\xi)$ are both positive,
and moreover
$\theta_{gv} g'(\theta_{gv})/g(\theta_{gv})$ and $\mu\tau(\theta_{gv})\theta_{gv} v'(\theta_{gv})/v(\theta_{gv})$
are increasing functions of $n$ and $m$, respectively.
Since the condition \eqref{eq:vglam0} is more easily
satisfied in this case than when one of these functions is zero or decreasing, fold bifurcations can occur for smaller values of $m$ or $n$ for $(g\uparrow,v\uparrow, \theta_v = \theta_g)$
than was seen for $(g\uparrow,v\leftrightarrow)$ or $(g\leftrightarrow,v\uparrow)$.
As $m=n\to\infty$ the interval of bistability converges
to $(\gamma_{24},\gamma_{13})$.
Equation~\eqref{eq:gam<mumodboth} imposes a constraint on Hopf bifurcations to occur
when $g$ and $v$ are both increasing. However, since $Q(\theta_{g})\to\infty$ as $n\to\infty$, Hopf bifurcations
can occur for all $n$ sufficiently large.

\subsubsection{$g$ and $v$ with opposing monotonicity $(g\uparrow,v\downarrow,\theta_g=\theta_v)$
and $(g\downarrow,v\uparrow,\theta_g=\theta_v)$}
\label{sec:gvopposite}

These cases  are altogether more delicate and surprising,
and need explaining in some detail.
For the smooth nonlinearities when the power $m$ or $n$ is large, and the function $g(x)$ or $v(x)$ is close to a step function, we will informally refer to the part of the function with large gradient as the interface. When $m\gg n$ the $v(x)$ function will have a much narrower interface than the function $g(x)$, while $g(x)$ will have a narrower interface when $n\gg m$. These interfaces are centred at
$x=\theta_v$ and $x=\theta_g$, so when $\theta_g=\theta_v$ they will overlap or if $m\gg n$ or $m\ll n$ the interval on which one interface occurs will be inside the interval on which the other interface occurs. When $g$ and $v$ have opposing monotonicity
the dynamics of
\eqref{eq:basic}-\eqref{eq:thres} will be very different depending on which interface is narrower.

From \eqref{eq:vglam0} there will be a fold bifurcation when
$$0=M(\xi)=
\frac{\xi g'(\xi)}{g(\xi)} + \mu\tau(\xi)\frac{\xi v'(\xi)}{v(\xi)}-1.
$$
When $v$ is increasing by Proposition~\ref{prop:fxpr}
\be \label{argument}
\lim_{m\to\infty}\frac{\theta_v v'(\theta_v)}{v(\theta_v)} =+\infty, \qquad
\mbox{ and } \qquad
\lim_{\xi\to0}\frac{\xi v'(\xi)}{v(\xi)} =\lim_{\xi\to+\infty}\frac{\xi v'(\xi)}{v(\xi)} =0,
\ee
which ensures that a fold bifurcation must occur for $m$ sufficiently large when $n$ is held constant. If $g$ is decreasing then $\xi g'(\xi)/g(\xi)<0$ and the value of $m$ required for a fold bifurcation to occur will be larger than when this term is non-negative, that is, when $g$ is constant or increasing. Those cases were considered
in Sections~\ref{sec:gconstvup} and~\ref{sec:bothupordown}.
An analogous argument to \eqref{argument} shows that when $g$ is increasing a fold bifurcation will occur for $n$ sufficiently large when $m$ is held constant.

A more delicate question is, what happens when $m\to\infty$ and $n\to\infty$ with $g$ and $v$ having opposite monotonicities. In that case one could expect that whichever of the terms $\xi g'(\xi)/g(\xi)$ or $\xi v'(\xi)/v(\xi)$ grows faster as $m,n\to\infty$ will determine where and whether fold bifurcations occur, however, as we will see, the behaviour is more nuanced than that.

We will now consider  the case $(g\downarrow,v\uparrow,\theta_g=\theta_v=\theta_{gv})$
for which $0<r_v<1<r_g$, where $r_g=g^-/g^+$ and $r_v=v^-/v^+$. Evaluating $M(\theta_{gv})$
using Proposition~\ref{prop:fxpr}(2) we obtain
\be \label{eq:Mthetagv}
M(\theta_{gv})=\frac{\theta_{gv} g'(\theta_{gv})}{g(\theta_{gv})} + \mu\tau(\theta_{gv})\frac{\theta_{gv} v'(\theta_{gv})}{v(\theta_{gv})} -1
 =
\frac{n(1-r_g)}{2(1+r_g)} + \mu\tau(\theta_{gv})\frac{m(1-r_v)}{2(1+r_v)}-1.
\ee
Then if
\be \label{eq:bothbif}
\mu\tau(\theta_{gv})\frac{(1-r_v)}{(1+r_v)}>\frac{(r_g-1)}{(1+r_g)},
\ee
it follows that $M(\theta_{gv})\to+\infty$ as $m=n\to\infty$, and in particular $M(\theta_{gv})>0$ for all $m=n$ sufficiently large. This ensures that there is a pair of fold bifurcations as $\gamma$ is varied with $m=n$ sufficiently large.

Of course, since $m$ and $n$ appear in the different nonlinearities $v$ and $g$ modelling different processes, there is no reason beyond mathematical convenience to assume that $m=n$. So, it is interesting to consider whether fold bifurcations still occur as $m\to\infty$ and $n\to\infty$ independently.

The condition \eqref{eq:bothbif} was arrived at by evaluating $M(\theta_{gv})$, which has a convenient form because of Proposition~\ref{prop:fxpr}(2). However, we will show below that it is not a necessary condition  for existence of fold bifurcations. To see this, note that for  the case we consider
we have $0<r_v<1<r_g$, and hence
\be \label{eq:thetarvghalf}
r_v^{1/2m}\theta_{gv} < \theta_{gv} < r_g^{1/2n}\theta_{gv}.
\ee
From Proposition~\ref{prop:fxpr}(4) the positive function $\xi v'(\xi)/v(\xi)$
achieves its maximum at $\xi=r_v^{1/2m}\theta_{gv} < \theta_{gv}$ and the negative function
$\xi g'(\xi)/g(\xi)$ achieves its minimum at $\xi=r_g^{1/2n}\theta_{gv}>\theta_{gv}$.
Thus, for $\xi \in  (r_v^{1/2m}\theta_{gv},r_g^{1/2n}\theta_{gv})$ we have
$\xi v'(\xi)/v(\xi)$ and $\xi g'(\xi)/g(\xi)$ as well as $\tau(\xi)$ all decreasing functions of $\xi$ and it follows that $M(\xi)$
is decreasing for $\xi\in (r_v^{1/2m}\theta_{gv},r_g^{1/2n}\theta_{gv})$
as well as satisfying $\lim_{\xi\to\pm\infty}M(\xi)=-1$.
Consequently, we expect $M(\xi)$ to obtain its maximum value for $\xi\lessapprox r_v^{1/2m}\theta_{gv}$. If this maximum value is positive, then \eqref{eq:vglam0} will have at least two solutions, indicating two fold bifurcations for different $\xi$ values, one on each side of this maximum.
While we will not consider the other case of $(g\uparrow,v\downarrow,\theta_g=\theta_v=\theta_{gv})$ in detail,
we note  that in that case  $g$ would be the increasing function,
with $r_g=g^-/g^+<1<r_v$, and for
$\xi\in (r_g^{1/2n}\theta_{gv},r_v^{1/2m}\theta_{gv})$
we would have $\xi g'(\xi)/g(\xi)$ positive and decreasing, $\xi v'(\xi)/v(\xi)$ negative and decreasing, and
$\tau(\xi)$ increasing, from which it follows that $M(\xi)$
is decreasing for $\xi\in(r_g^{1/2n}\theta_{gv},r_v^{1/2m}\theta_{gv})$, so
$M(\xi)$ would obtain its maximum value for $\xi<\theta_{gv}$.

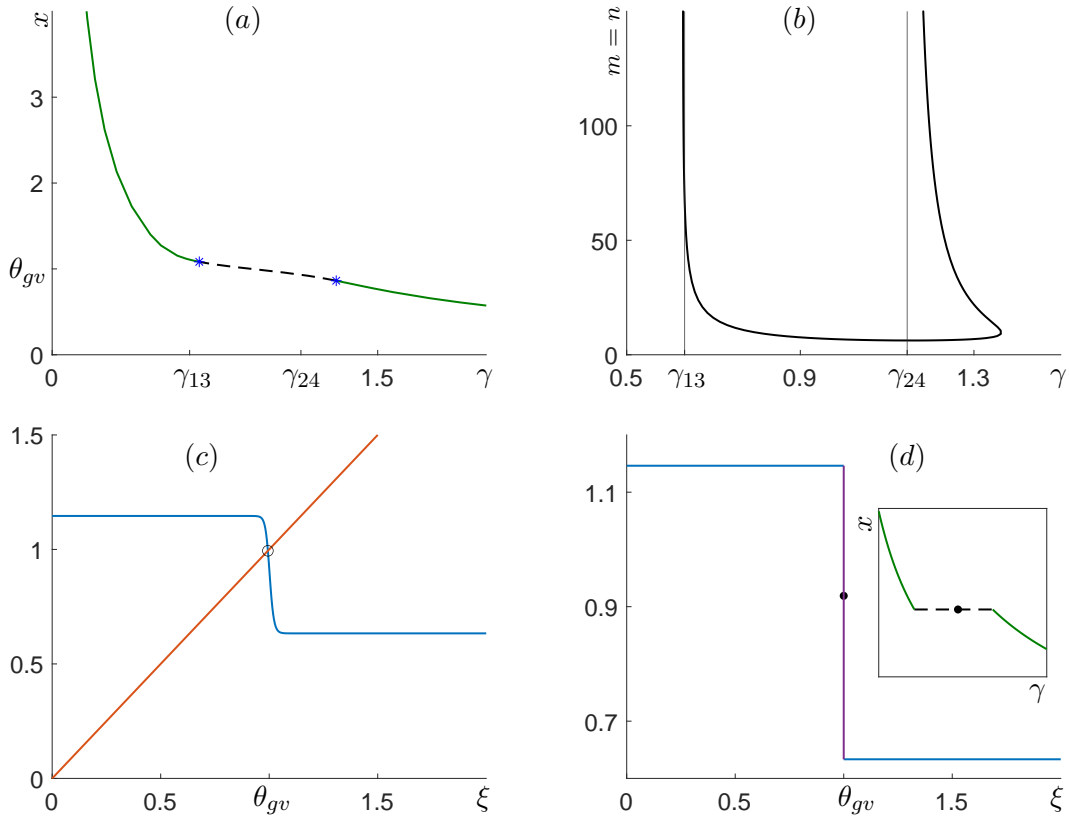
\begin{figure}[thp!]
	\centering
	\includegraphics[scale=0.5]{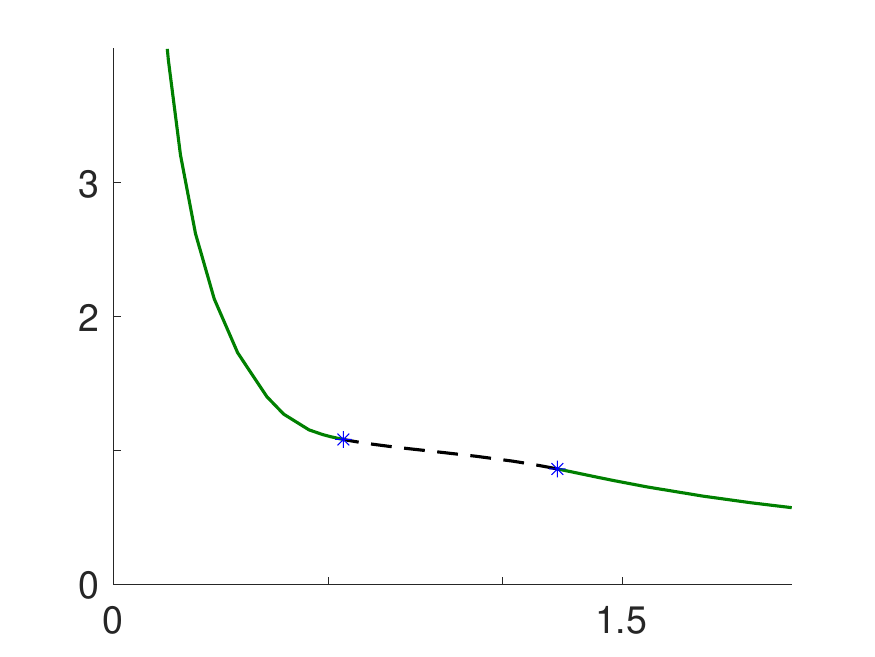}\hspace*{0.5em}\includegraphics[scale=0.5]{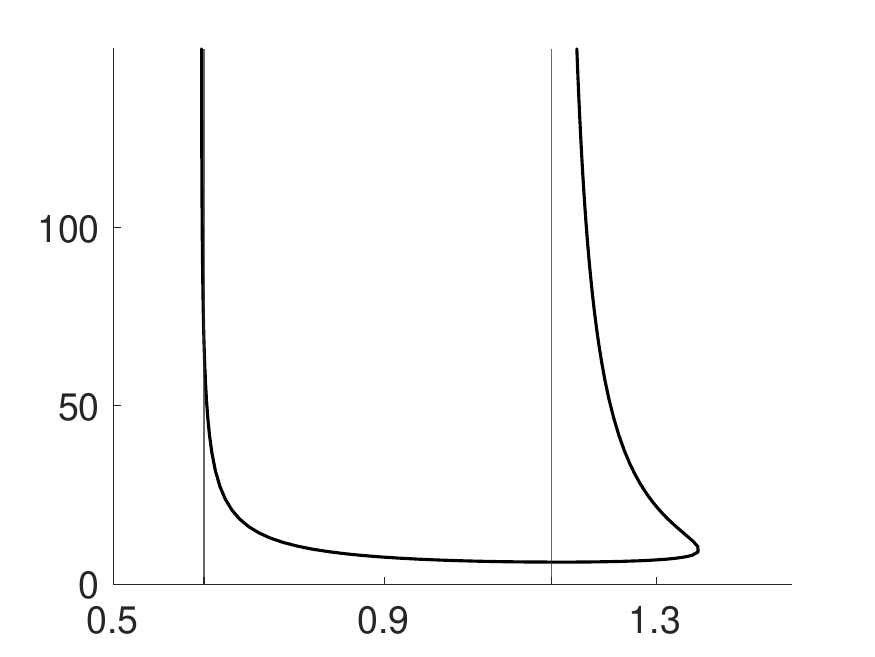}
	\put(-406,140){\rotatebox{90}{$x$}}
	\put(-416,46){$\theta_{gv}$}
	\put(-354,7){$\gamma_{13}$}
	\put(-313,7){$\gamma_{24}$}
	\put(-240,7){$\gamma$}
	\put(-335,140){$(a)$}
	\put(-125,140){$(b)$}
	\put(-190,122){\rotatebox{90}{\footnotesize$m=n$}}
	\put(-168,7){$\gamma_{13}$}
	\put(-85,7){$\gamma_{24}$}
	\put(-24,7){$\gamma$}\\
	\includegraphics[scale=0.5]{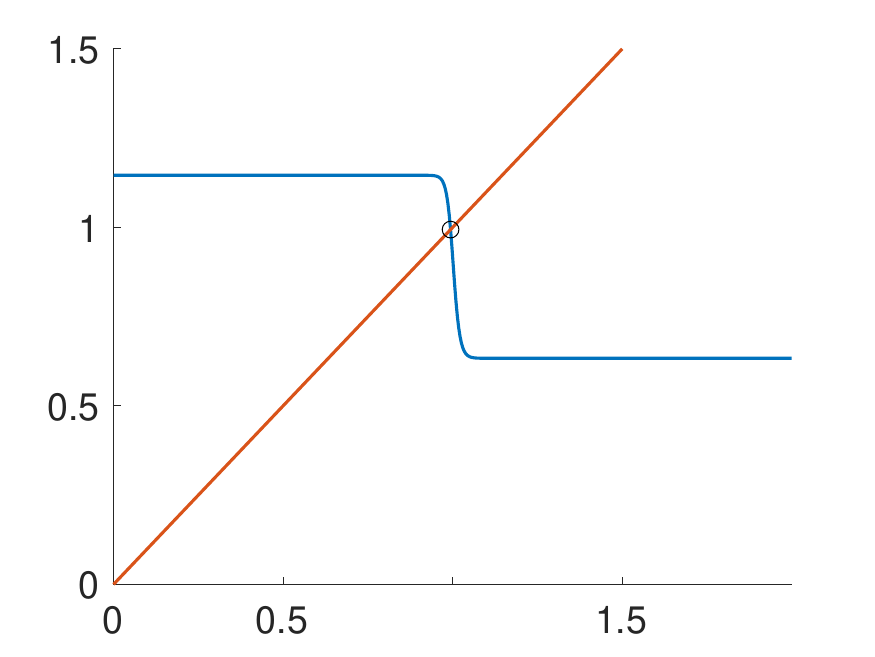}\hspace*{0.5em}\includegraphics[scale=0.5]{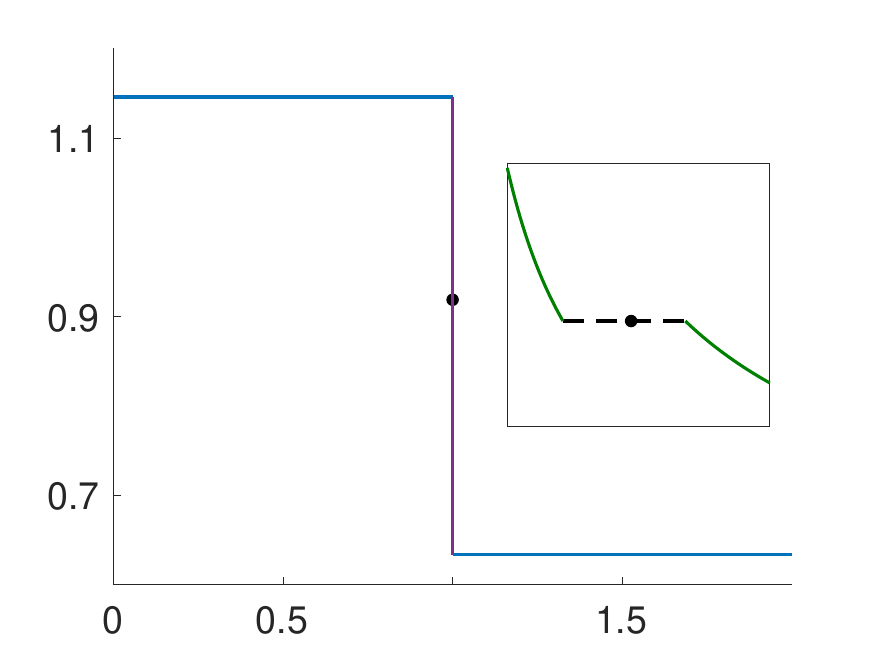}
	\put(-240,7){$\xi$}
	\put(-324,7){$\theta_{gv}$}
	\put(-350,135){$(c)$}
	\put(-85,135){$(d)$}
	\put(-24,7){$\xi$}
	\put(-104,7){$\theta_{gv}$}
	\put(-32,48){$\gamma$}
	\put(-96,110){\rotatebox{90}{$x$}}
	\caption{Bifurcations of \eqref{eq:basic}-\eqref{eq:thres} for $(g\downarrow,v\uparrow,\theta_g=\theta_v=\theta_{gv})$ with $\beta=1.4$, $\mu=0.2$, $g^-=1$, $g^+=0.5$, $\theta_g=\theta_v=\theta_{gv} =1$, $a=1$, $v^-=1$ and $v^+=2$. With these parameters
$\beta e^{-\mu\tau^-\!} g^- >  \beta e^{-\mu\tau(\theta_{gv})}\! g(\theta_{gv}) >\beta e^{-\mu\tau^+}\! g^+$.
(a) With smooth nonlinearity $g$ and $v$ defined by \eqref{eq:vghill} with $m=n=100$.
(b) Two-parameter continuations in $m=n$ and $\gamma$ of the Hopf bifurcation at which the steady state loses stability with the other parameters as above. The vertical lines at $\gamma=\gamma_{13}=0.6334$
and $\gamma=\gamma_{24}=1.1462$ denote the location of the Hopf bifurcations in the limiting case as $m=n\to\infty$.
(c) Graphs of $\xi \mapsto\beta e^{-\mu\tau(\xi)} g(\xi)$ and $\xi \mapsto\gamma\xi$  with $m=n=100$ and $\gamma=1$. Since $\xi \mapsto\beta e^{-\mu\tau(\xi)} g(\xi)$ is monotonically decreasing there is a unique intersection and hence a unique steady state for and $\gamma>0$.
(d) Graph of $\xi \mapsto\beta e^{-\mu\tau(\xi)} g(\xi)$ in the limiting case, with the black dot denoting
the value of $\beta e^{-\mu\tau(\theta_{gv})} g(\theta_{gv})$. The inset depicts the continuation of the steady states in the limiting case with $g$ and $v$ defined by \eqref{eq:gpwconst} and \eqref{eq:vpwconst}. Stable steady states are shown as green solid lines, and the singular steady state as a black dashed line. }
	\label{fig:gdownvup_ex1}
\end{figure}

To study the steady states and fold bifurcations of \eqref{eq:basic}-\eqref{eq:thres}
for $(g\downarrow,v\uparrow,\theta_g=\theta_v=\theta_{gv})$ we will  work directly with the
function $h(\xi)$ defined in \eqref{eq:h}. Suppose the parameters are chosen so that
$\beta e^{-\mu\tau^-}g^- > \beta e^{-\mu\tau^+}g^+$. Then
\be \label{eq:decreasing}
\beta e^{-\mu\tau^-}g^- = \lim_{\xi\to0}\beta e^{-\mu\tau(\xi)}g(\xi)>
\lim_{\xi\to+\infty}\beta e^{-\mu\tau(\xi)}g(\xi)=\beta e^{-\mu\tau^+}g^+.
\ee

Figure~\ref{fig:gdownvup_ex1} depicts such an example for which
$\beta e^{-\mu\tau(\xi)}g(\xi)$ is a
monotonically decreasing function of $\xi$ in both the smooth case and the limiting case as $m=n\to\infty$, so that there is always a unique steady state.
One could easily and incorrectly assume that $\xi\mapsto\beta e^{-\mu\tau(\xi)}g(\xi)$ will be
monotonically decreasing whenever \eqref{eq:decreasing} holds.
However, this is not always true, and indeed is never true when \eqref{eq:bothbif} holds. To show this, recalling
the definitions of $g$ and $v$ given by \eqref{eq:vghill} in the smooth case, notice
that at the threshold, the function values are independent of the nonlinearity $m$ or $n$, with
\begin{equation}
g(\theta_g)=\frac{g^-+g^+}{2} \quad \text{and} \quad v(\theta_v)=\frac{v^-+v^+}{2}.
\label{eq:thresvalue}
\end{equation}
Thus for $(g\downarrow,v\uparrow,\theta_g=\theta_v=\theta_{gv})$ and  $0<r_v<1<r_g$
it holds that
\begin{equation} \label{eq:gs}
g^-> g(\theta_{gv}) > g^+,
\end{equation}
while $v^+ > v(\theta_{gv}) > v^-$ implies
\begin{equation} \label{eq:vs}
\tau^- > \tau(\theta_{gv}) > \tau^+.
\end{equation}
Then \eqref{eq:bothbif} implies that
\be \label{eq:thetagvarg}
\mu\tau^-\frac{(1-r_v)}{(1+r_v)}
 > \mu\tau(\theta_{gv})\frac{(1-r_v)}{(1+r_v)}
 > \frac{r_g-1}{1+r_g}
 > \ln\Bigl(1+\frac{r_g-1}{1+r_g}\Bigr)
 = \ln\Bigl(\frac{2r_g}{1+r_g}\Bigr).
\ee
But
$$
\frac{2r_g}{1+r_g}=\frac{2g^-/g^+}{1+g^-/g^+}=\frac{2g^-}{g^++g^-}=\frac{g^-}{g(\theta_{gv})},$$
and
$$\tau^-\frac{(1-r_v)}{(1+r_v)}=\frac{a}{v^-}\Bigl(\frac{v^+-v^-}{v^-+v^+}\Bigr)
=\frac{a}{v^-}-\frac{2a}{v^-+v^+}=\tau^- - \tau(\theta_{gv}).$$
Thus
\begin{align*}
\mu\tau^-\frac{(1-r_v)}{(1+r_v)} > \ln\Bigl(\frac{2r_g}{1+r_g}\Bigr)
& \quad \Longleftrightarrow\quad
\mu(\tau^- - \tau(\theta_{gv})) > \ln(g^-/g(\theta_{gv}))\\
& \quad\Longleftrightarrow\quad
e^{-\mu(\tau(\theta_{gv})-\tau^-)} > g^-/g(\theta_{gv})\\
& \quad\Longleftrightarrow\quad
\beta e^{-\mu\tau(\theta_{gv})}g(\theta_{gv}) > \beta e^{-\mu\tau^-}g^-.
\end{align*}
Thus if the parameters are chosen so that \eqref{eq:bothbif} and \eqref{eq:decreasing} both hold then
\be \label{eq:betaexpabove}
\beta e^{-\mu\tau(\theta_{gv})}g(\theta_{gv})>\beta e^{-\mu\tau^-}g^- > \beta e^{-\mu\tau^+}g^+,
\ee
and the function $\xi \mapsto\beta e^{-\mu\tau(\xi)}g(\xi)$ is not monotonic. This was not the case for the example in
Figure~\ref{fig:gdownvup_ex1}, but below
we will demonstrate examples
where both \eqref{eq:bothbif} and \eqref{eq:betaexpabove} hold, and more interesting dynamics arise. Moreover,
the inequalities in \eqref{eq:thetagvarg} are not all tight, and we will also show that it is possible to
obtain \eqref{eq:betaexpabove} for parameters for which \eqref{eq:bothbif} does not hold.

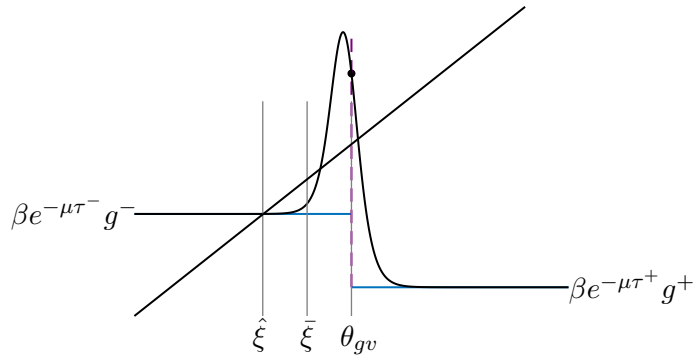
\begin{figure}[thp!]
	\centering
	\includegraphics[scale=0.5]{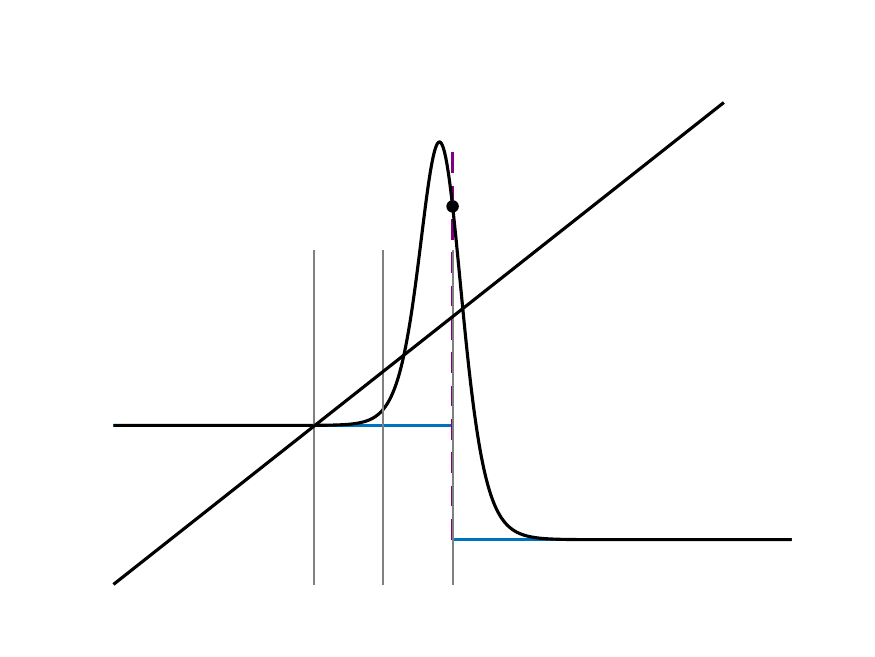}
	\put(-139,6){$\hat{\xi}$}
	\put(-121,6){$\bar{\xi}$}
	\put(-230,52){$\beta e^{-\mu\tau^-}g^-$}
	\put(-20,24){$\beta e^{-\mu\tau^+}g^+$}
	\put(-106,6){$\theta_{gv}$}
	\caption{An illustration of $\beta e^{-\mu\tau(\xi)}g(\xi)$ and $\hat{\gamma}\xi$ in the case $(g\downarrow, v\uparrow, \theta_g=\theta_v=\theta_{gv})$ for both smooth nonlinearities and in the limiting case $m=n\to\infty$ when \eqref{eq:betaexpabove} holds.}
\label{fig:thetagvlimits}
\end{figure}

Now suppose that \eqref{eq:betaexpabove} holds and consider $h(\xi)$ defined in \eqref{eq:h}. This situation is illustrated in Figure~\ref{fig:thetagvlimits}. Recall \eqref{eq:corners_onetheta}
and also let
\be \label{eq:gammathetagv}
\gamma_{gv}:=\frac{\beta e^{-\mu\tau(\theta_{gv})}g(\theta_{gv})}{\theta_{gv}}.
\ee
We will now show that when $\gamma\in(\gamma_{24},\gamma_{gv})$ for all $m$ sufficiently large there are three steady states. From \eqref{eq:h}, we have $h(0)=\beta e^{-\mu\tau^-}g^->0$ and $h(\xi)\to-\infty$ as $\xi\to+\infty$. Moreover, \eqref{eq:gammathetagv} and $\gamma<\gamma_{gv}$ implies
$$h(\theta_{gv})=\beta e^{-\mu\tau(\theta_{gv})}g(\theta_{gv})-\gamma \theta_{gv}=\gamma_{gv}\theta_{gv}-\gamma \theta_{gv}=(\gamma_{gv}-\gamma)\theta_{gv}>0.$$
From
$h(\theta_{gv})>0$
the intermediate value theorem immediately yields the existence of one steady state $\xi>\theta_{gv}$;  to show that there are three steady states it suffices to show that there exists
$\bar{\xi}\in(0,\theta_{gv})$ with $h(\bar{\xi})<0$.

For any $\hat{\gamma}\in(\gamma_{24},\gamma_{gv})$ we have
$\hat{\gamma}\theta_{gv} \in (\beta e^{-\mu\tau^-}g^-, \beta e^{-\mu\tau(\theta_{gv})}g(\theta_{gv}))$
so for some $\varepsilon > 0$ it holds that
$\hat{\gamma}\theta_{gv} = \beta e^{-\mu\tau^-}g^-+2\hat{\gamma}\varepsilon$.
This implies that the point $\hat{\xi}$ (shown in Figure~\ref{fig:thetagvlimits}) where the line $\hat{\gamma}\xi$
intersects $\beta e^{-\mu\tau^-}g^-$ satisfies
$\hat{\xi}=\beta e^{-\mu\tau^-}g^-/\hat{\gamma}=\theta_{gv}-2\varepsilon$.
Let
$$\bar{\xi}=\theta_{gv}-\varepsilon=\frac12(\hat{\xi}+\theta_{gv})\in(\hat{\xi},\theta_{gv}).$$
Then
$$
\hat{\gamma}\bar{\xi} = \hat{\gamma}(\theta_{gv}-\varepsilon)
	= \beta e^{-\mu\tau^-}g^-+2\hat{\gamma}\varepsilon-\hat{\gamma}\varepsilon
= \beta e^{-\mu\tau^-}g^-+\hat{\gamma}\varepsilon.
$$
We have $g(\bar{\xi})<g^-$ since $g$ is strictly decreasing, and then
\begin{align*}
h(\bar{\xi})& = \beta e^{-\mu\tau(\bar{\xi})}g(\bar{\xi})-\hat{\gamma}\bar{\xi}
< \beta e^{-\mu\tau(\bar{\xi})}g^- - (\beta e^{-\mu\tau^-}g^-+\hat{\gamma}\varepsilon)
 = \beta(e^{-\mu\tau(\bar{\xi})}-e^{-\mu\tau^-})g^--\hat{\gamma}\varepsilon.
\end{align*}
Since $\bar{\xi} < \theta_{gv}$, using the pointwise limit of the Hill function to its piecewise constant limit we have $\lim_{m\to\infty}\tau(\bar{\xi})=\tau^-$. It then follows from the previous equation that there exists $m>0$ such that $h(\bar{\xi})<0$,
which shows that
$h(\xi)$ has at least three sign changes, and so there are three co-existing steady states.

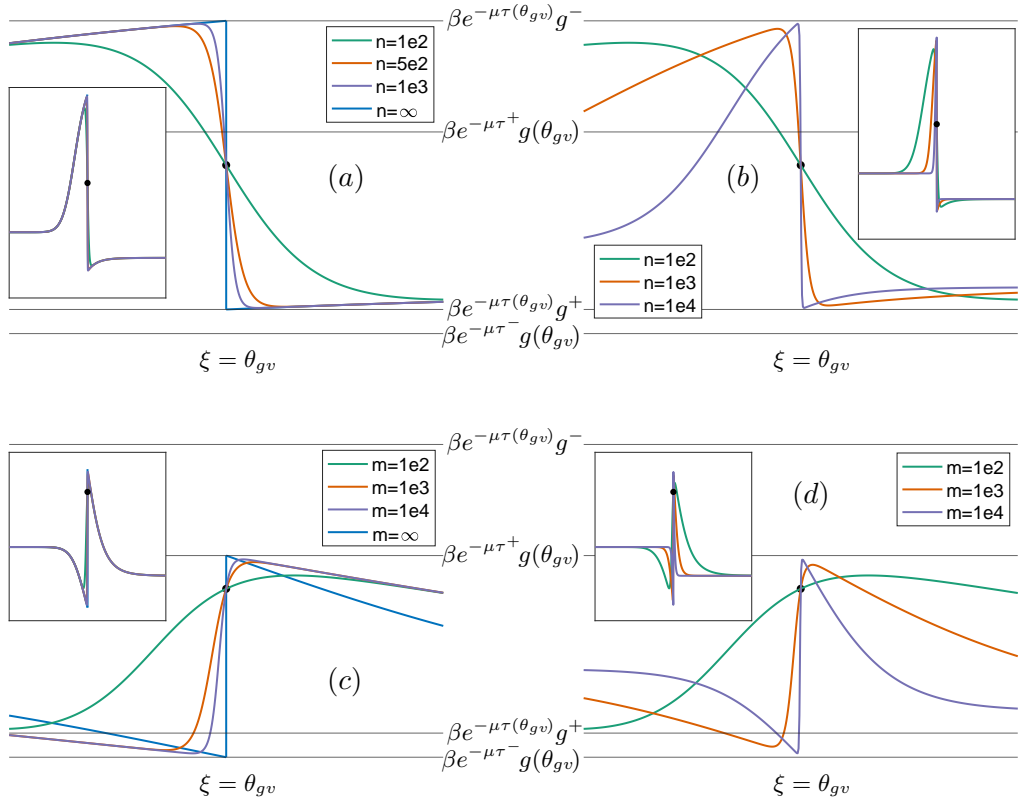
\begin{figure}[thp!]
	\centering
	\includegraphics[scale=0.5]{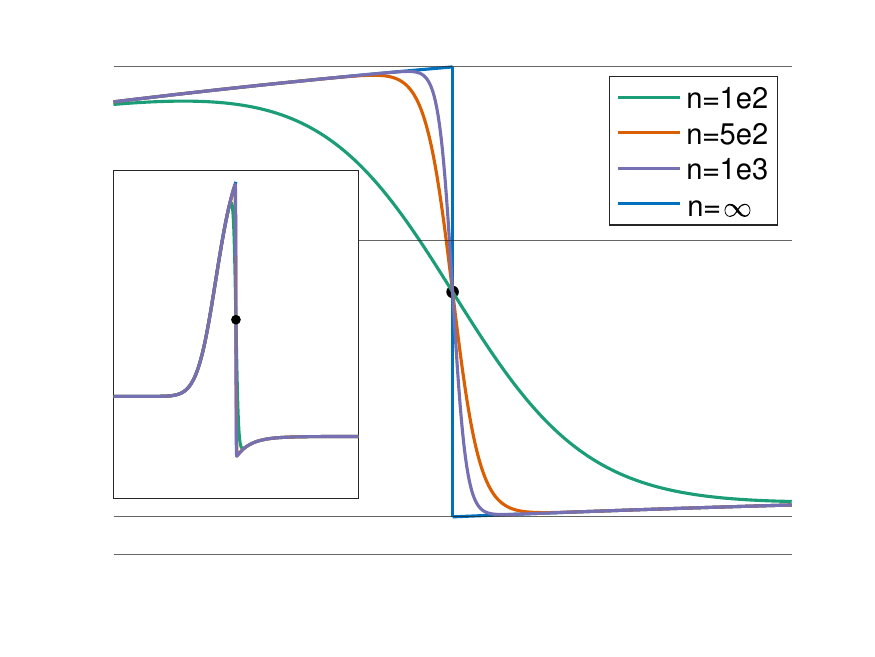}\hspace*{0.5em}\includegraphics[scale=0.5]{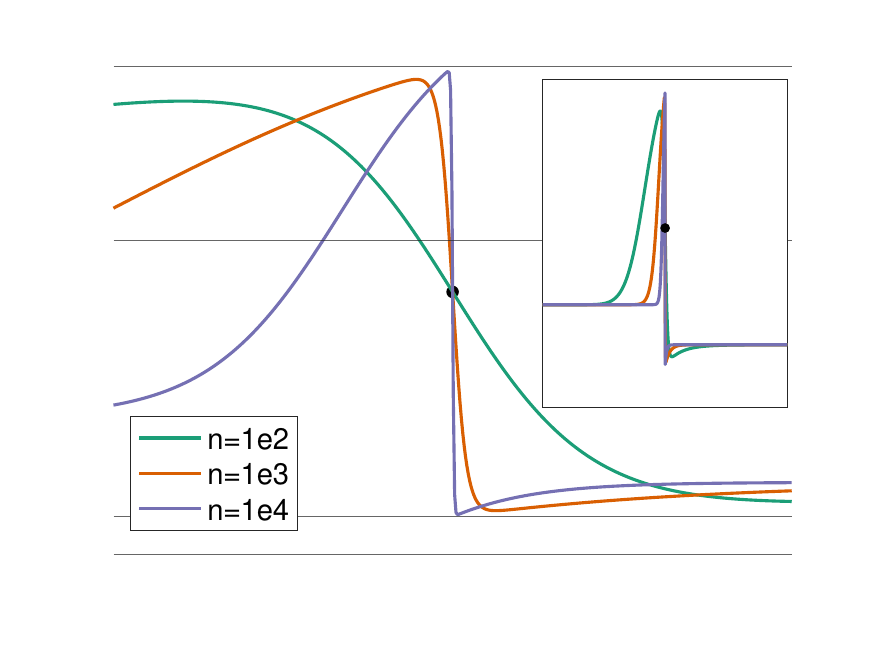}
	\put(-280,80){$(a)$}
	\put(-130,80){$(b)$}
	\put(-328,12){\footnotesize$\xi=\theta_{gv}$}
	\put(-112,12){\footnotesize$\xi=\theta_{gv}$}
	\put(-235,140){\footnotesize$\beta e^{-\mu\tau(\theta_{gv\!})}g^-$}
	\put(-237,98){\footnotesize$\beta e^{-\mu\tau^+}\!g(\theta_{gv\!})$}
	\put(-235,31){\footnotesize$\beta e^{-\mu\tau(\theta_{gv\!})}g^+$}
	\put(-237,21){\footnotesize$\beta e^{-\mu\tau^-}\!g(\theta_{gv\!})$}\\
	\includegraphics[scale=0.5]{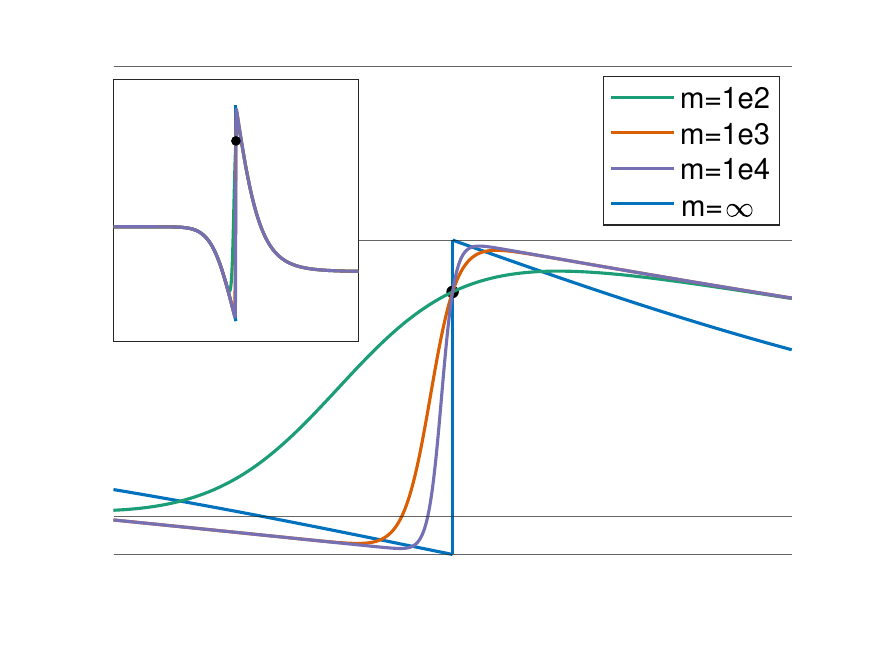}\hspace*{0.5em}\includegraphics[scale=0.5]{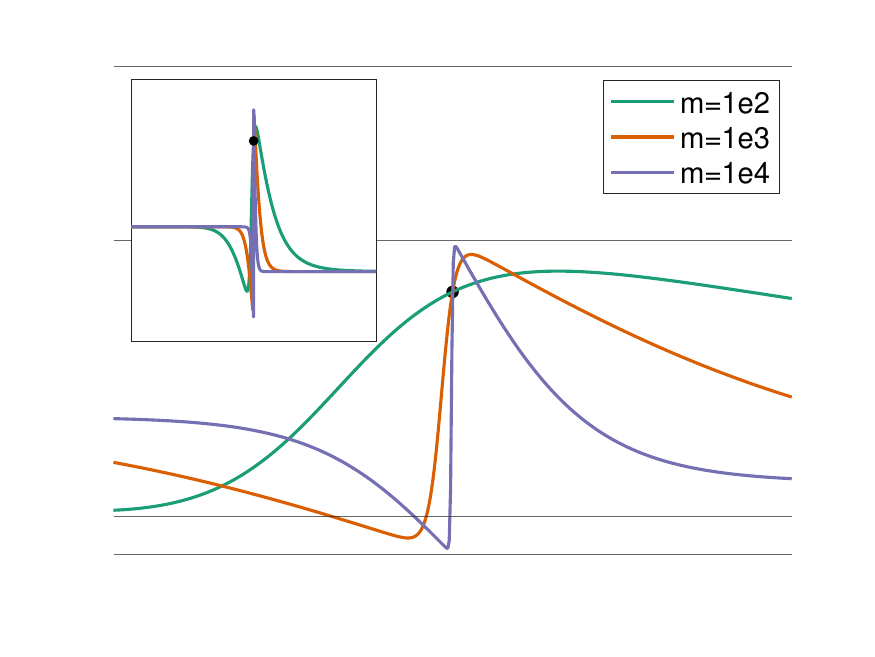}
	\put(-280,50){$(c)$}
	\put(-105,120){$(d)$}
	\put(-328,12){\footnotesize$\xi=\theta_{gv}$}
	\put(-112,12){\footnotesize$\xi=\theta_{gv}$}
	\put(-235,140){\footnotesize$\beta e^{-\mu\tau(\theta_{gv\!})}g^-$}
	\put(-237,98){\footnotesize$\beta e^{-\mu\tau^+}\!g(\theta_{gv\!})$}
	\put(-235,31){\footnotesize$\beta e^{-\mu\tau(\theta_{gv\!})}g^+$}
	\put(-237,21){\footnotesize$\beta e^{-\mu\tau^-}\!g(\theta_{gv\!})$}\\
	\caption{Illustrations of the behaviour of the function $\xi \mapsto\beta e^{-\mu\tau(\xi)} g(\xi)$
in the case $(g\downarrow, v\uparrow, \theta_g=\theta_v=\theta_{gv})$ with $\beta e^{-\mu\tau(\theta_{gv})}g(\theta_{gv})>\beta e^{-\mu\tau^-}g^->\beta e^{-\mu\tau^+}g^+$
as $m$ and/or $n$ tend to infinity, showing that different maxima and minima are possible depending on the limits taken. The insets show the function on the interval $\xi\in[0,2\theta_{gv}]$ while the main panels show the behaviour for $\xi\approx\theta_{gv}$.
The values of $m$ and $n$ are taken to be
(a) $m=10$ fixed and $n\to\infty$,
(b) $m^2=n \to \infty$,
(c) $n=10$ fixed and $m \to \infty$, and,
(d) $n=\sqrt{m} \to \infty$. The values of the other parameters are
$\beta=1.4$, $\mu=0.2$, $g^-=1$, $g^+=0.5$, $\theta_g=\theta_v =\theta_{gv} =1$, $a=1$, $v^-=0.3$ and $v^+=2$.
}
	\label{fig:thetagvonelimit}
\end{figure}

When $m>1$ and $n>1$ it is easy to verify that
$$\lim_{\xi\to0}\frac{\phantom{\xi}d}{d\xi}\left(\beta e^{-\mu\tau(\xi)}g(\xi)\right)
=
\lim_{\xi\to+\infty}\frac{\phantom{\xi}d}{d\xi}\left(\beta e^{-\mu\tau(\xi)}g(\xi)\right)=0,$$
from which it follows that there is a unique steady state for all $\gamma$ sufficiently small or large. Bifurcations from this steady state will depend on the nonlinearities $g$ and $v$.

\begin{figure}[thp!]
	\centering
	\mbox{}\hspace*{-4.5em}\includegraphics[scale=0.5]{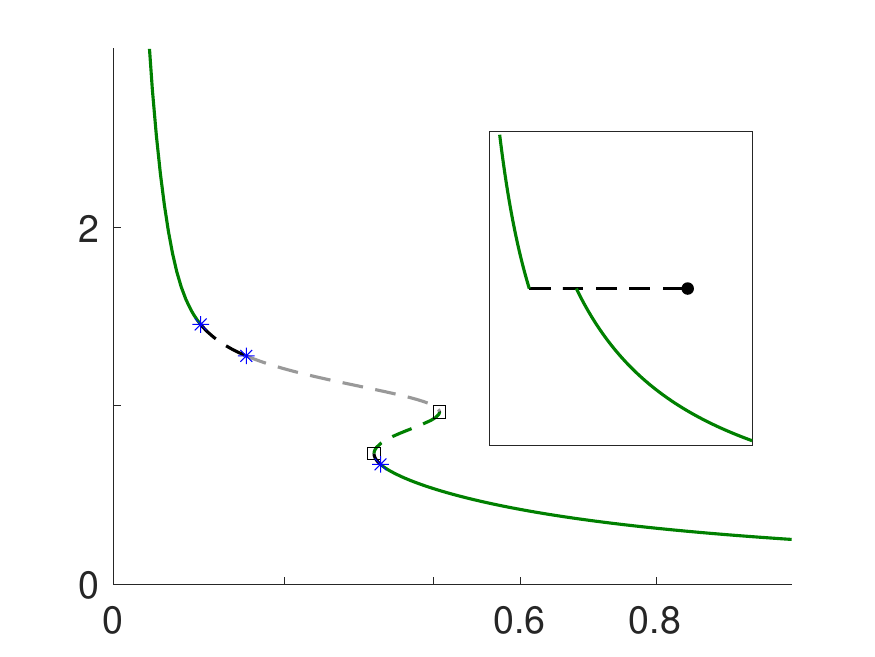}\hspace*{-0.3em}\includegraphics[scale=0.5]{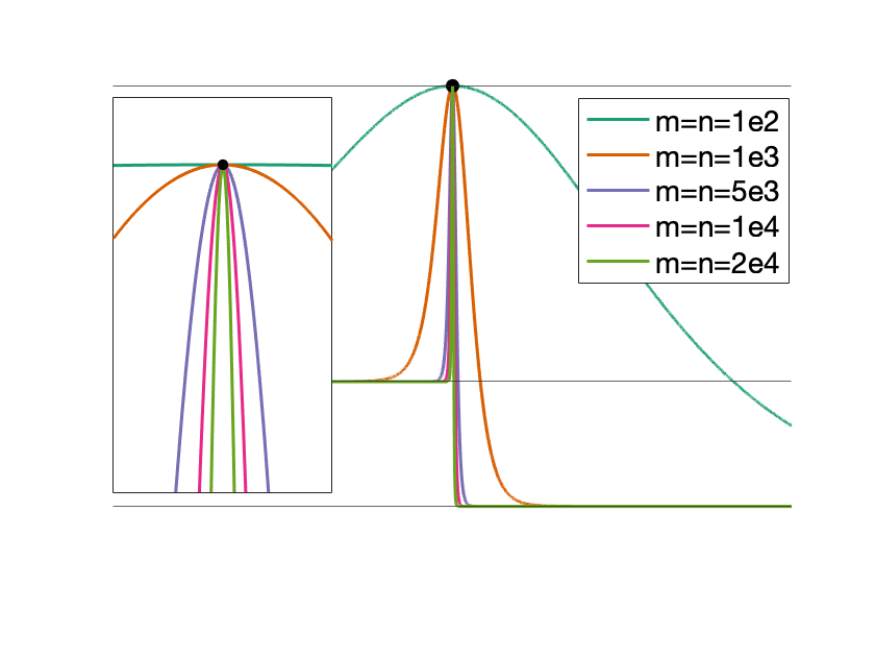}
	\put(-325,6){$\gamma_{24}$}
	\put(-290,6){$\gamma_{gv}$}
	\put(-213,44){$\gamma$}
	\put(-276,119){\rotatebox{90}{$x$}}
	\put(-310,135){$(a)$}
	\put(-120,135){$(b)$}
	\put(-100,15){\footnotesize$\xi=\theta_{gv}$}
	\put(-6,126){\scriptsize$\beta e^{-\mu\tau^{\!}(\theta_{gv\!})^{\!}}g(^{\!}\theta_{gv\!})$}
	\put(-6,53){\scriptsize$\beta e^{-\mu\tau^-}\!g^-$}
	\put(-6,25){\scriptsize$\beta e^{-\mu\tau^+}\!g^+$}
	\put(-367,140){\rotatebox{90}{$x$}}
	\put(-202,10){$\gamma$}
	\put(-377,57){$\theta_{gv}$}\\
	\includegraphics[scale=0.5]{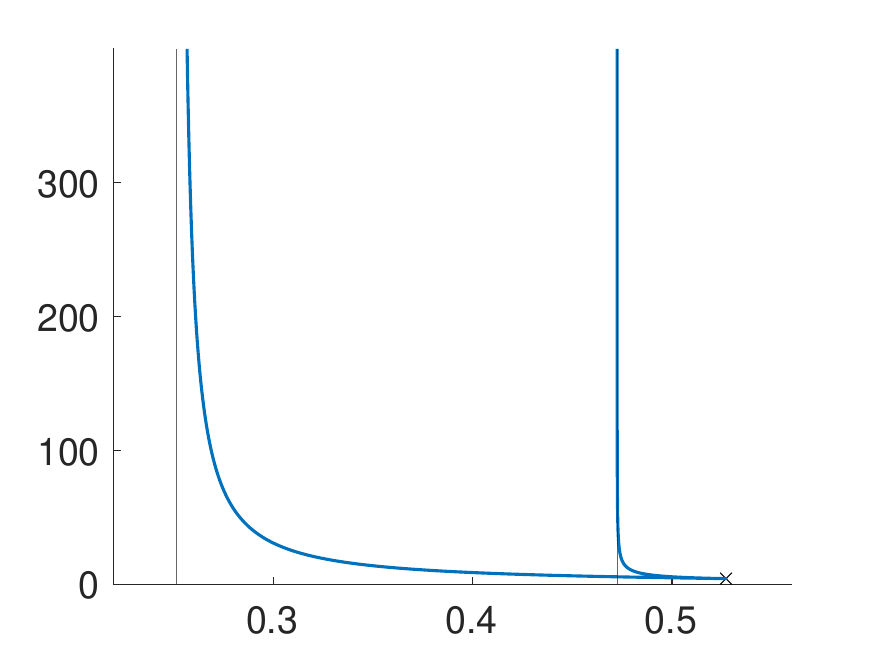}\hspace*{0em}\includegraphics[scale=0.5]{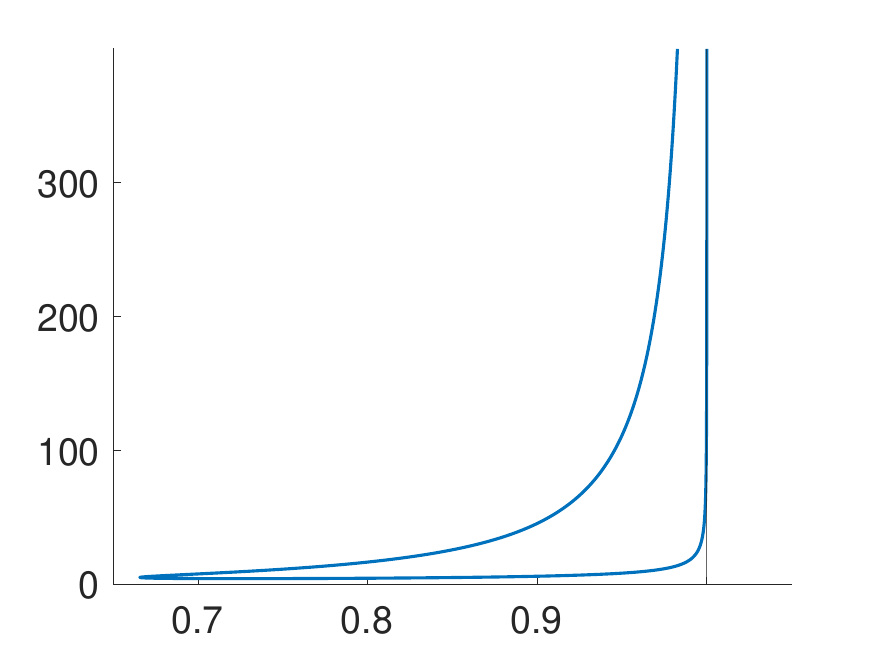}
	\put(-340,135){$(c)$}
	\put(-120,135){$(d)$}
	\put(-400,122){\rotatebox{90}{\scriptsize$m=n$}}
	\put(-240,10){$\gamma$}
	\put(-24,8){$\xi$}
    \put(-282,6){$\gamma_{gv}$}
   \put(-384,6){$\gamma_{24}$}
	\put(-190,122){\rotatebox{90}{\scriptsize$m=n$}}
	\put(-48,6){$\theta_{gv}$}
	\caption{Bifurcations of \eqref{eq:basic}-\eqref{eq:thres} in the case $(g\downarrow, v\uparrow, \theta_g=\theta_v=\theta_{gv})$ with $\beta e^{-\mu\tau(\theta_{gv})}g(\theta_{gv})>\beta e^{-\mu\tau^-\!}g^->\beta e^{-\mu\tau^+}\!g^+$ and $m=n$. Parameter values are
$\beta=10$, $\mu=0.2$, $g^-=1$, $g^+=1/10$, $\theta_{gv}=1$, $v^-=1$, $v^+=2$ and $a=18.4091$, for which $M(\theta_{gv})=-1$ when $m=n$.
(a) With smooth nonlinearity $g$ and $v$ defined by \eqref{eq:vghill} with $m=n=10$, and inset showing the limiting case as $m=n\to\infty$.
(b) The behaviour of $\xi \mapsto\beta e^{-\mu\tau(\xi)} g(\xi)$ as $m=n\to\infty$ with the inset showing that the maximum of the function is $\beta e^{-\mu\tau(\theta_{gv})}g(\theta_{gv})$ in the limiting case.
(c) Two-parameter continuation in $m=n$ and $\gamma$ of the fold bifurcation with the other parameters as above. The fold bifurcations are always associated with unstable steady states, and there is a cusp point at $(\gamma, m)=(0.5271, 4.2830)$. In the limiting case as $m=n\to\infty$ the folds occur at
$\gamma=\gamma_{24}=0.2518$ (recall \eqref{eq:corners_onetheta}) and $\gamma=\gamma_{gv}=0.4725$.
(d) Two-parameter continuation in $m=n$ and $\xi$ of the fold bifurcations with the other parameters as above.}
\label{fig:FoldsMZero}
\end{figure}

In the case of smooth nonlinearities when one of $m$ or $n$ is held fixed while the other one is increased to infinity the arguments of Sections~\ref{sec:onehill} can be adapted to show that
in the limiting case  $m\to\infty$ the fold bifurcations occur at
\begin{equation} \label{thetagvonelimitm}
\gamma_4=\frac{\beta e^{-\mu\tau^-}g(\theta_{gv})}{\theta_{gv}}
<
\gamma_3=\frac{\beta e^{-\mu\tau^+}g(\theta_{gv})}{\theta_{gv}}.
\end{equation}
with singular solutions for $\gamma\in(\gamma_4,\gamma_3)$.
Similarly, with $m$ fixed in the limiting case  $n\to\infty$, singular solutions occur for $\gamma\in(\gamma_1,\gamma_2)$
where
\begin{equation} \label{thetagvonelimitn}
\gamma_1=\frac{\beta e^{-\mu\tau(\theta_{gv})} g^+}{\theta_{gv}}
<
\gamma_2=\frac{\beta e^{-\mu\tau(\theta_{gv})} g^-}{\theta_{gv}}.
\end{equation}
With $g$ decreasing, whether fold bifurcations occur at the ends of this interval will depend also on the steepness of $v$.
These limiting cases correspond to solutions of $h(\xi)=0$ where the line $\gamma\xi$ intersects
the curve $\beta e^{-\mu\tau(\xi)}g(\xi)$ at its maxima and minima which are
illustrated in panels (a) and (c) of Figure~\ref{fig:thetagvonelimit} when one of the values $m$, $n$ is fixed and the other tends to $\infty$.

Panels (b) and (d) of Figure~\ref{fig:thetagvonelimit} illustrate what happens when $m$ and $n$ both tend to infinity, but one is increased much faster than the other, showing that the same extrema are observed as when
the more slowly increasing exponent is held fixed.
Since the maxima and minima of the nonlinearities converge to different values depending on whether $m$ or $n$ is increasing faster, and the limiting locations of the fold bifurcations are determined by the
extrema, the location of the fold bifurcations
as $m\to\infty$ and $n\to\infty$ will depend on how these limits are approached. We will investigate the dynamics for different relationships between $m$ and $n$ below.

\begin{figure}[thp!]
	\centering
	\includegraphics[scale=0.5]{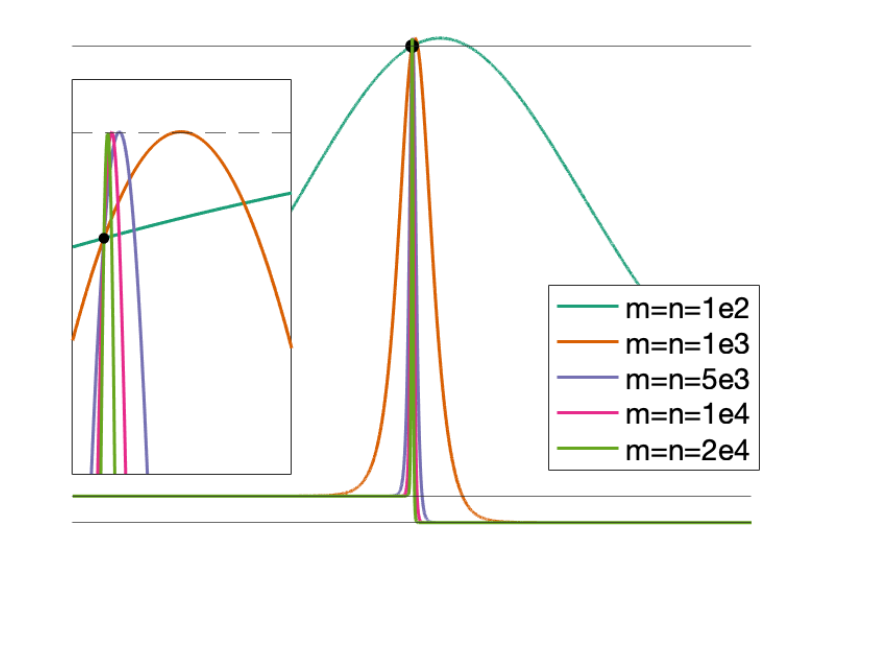}\hspace*{-2em}\includegraphics[scale=0.5]{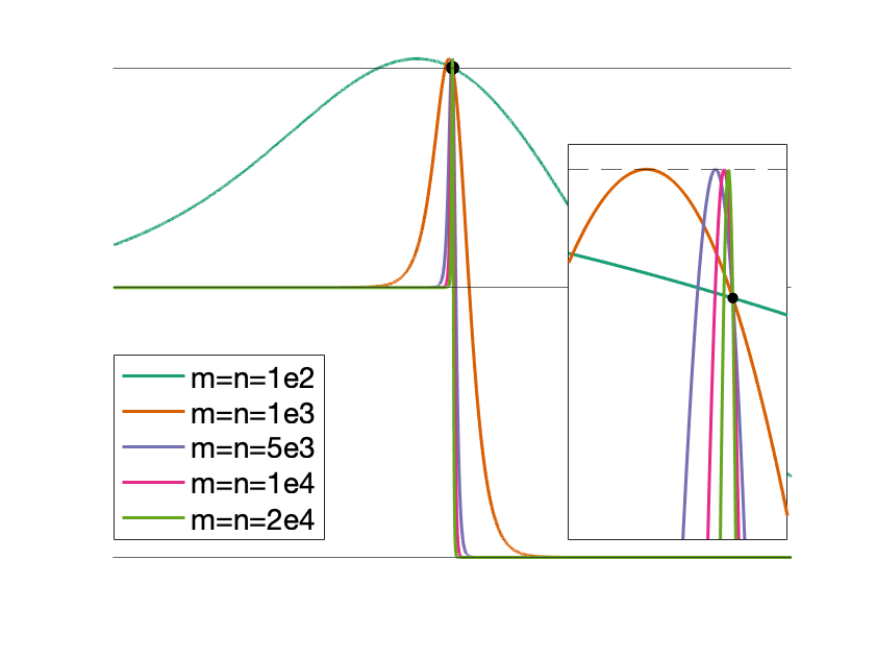}
	\put(-255,130){$(a)$}
	\put(-180,125){$(b)$}
	\put(-310,19){\footnotesize$\xi=\theta_{gv}$}
	\put(-112,12){\footnotesize$\xi=\theta_{gv}$}
	\put(-82,144){\scriptsize$\beta e^{-\mu\tau^{\!}(\theta_{gv\!})^{\!}}g(^{\!}\theta_{gv\!})$}
	\put(-182,80){\scriptsize$\beta e^{-\mu\tau^-}\!g^-$}
	\put(-20,21){\scriptsize$\beta e^{-\mu\tau^+}\!g^+$}
	\put(-378,149){\scriptsize$\beta e^{-\mu\tau^{\!}(\theta_{gv\!})^{\!}}g(^{\!}\theta_{gv\!})$}
	\put(-421,37){\scriptsize$\beta e^{-\mu\tau^-}\!g^-$}
	\put(-421,25){\scriptsize$\beta e^{-\mu\tau^+}\!g^+$}\\
	\includegraphics[scale=0.5]{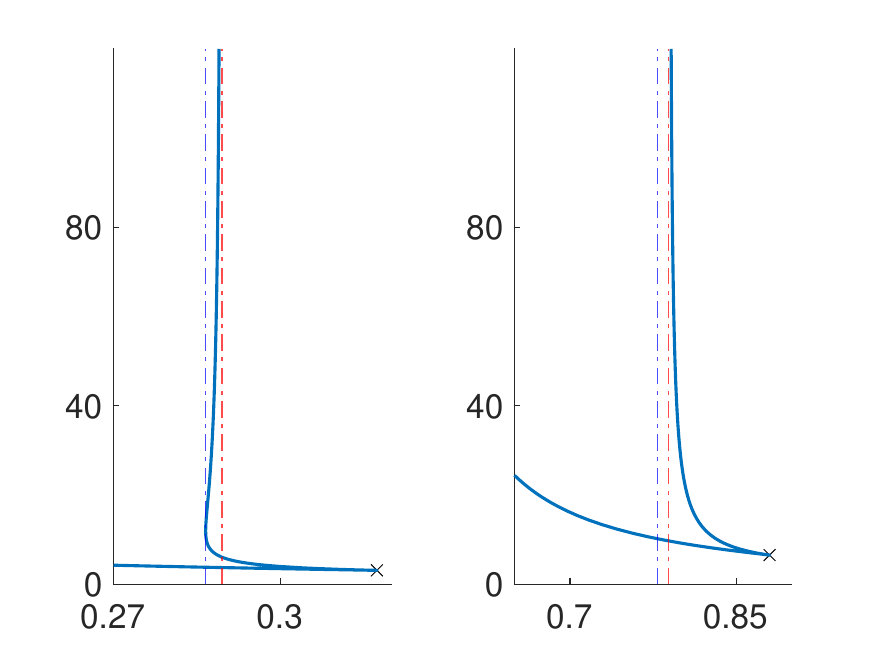}\hspace*{0.5em}\includegraphics[scale=0.5]{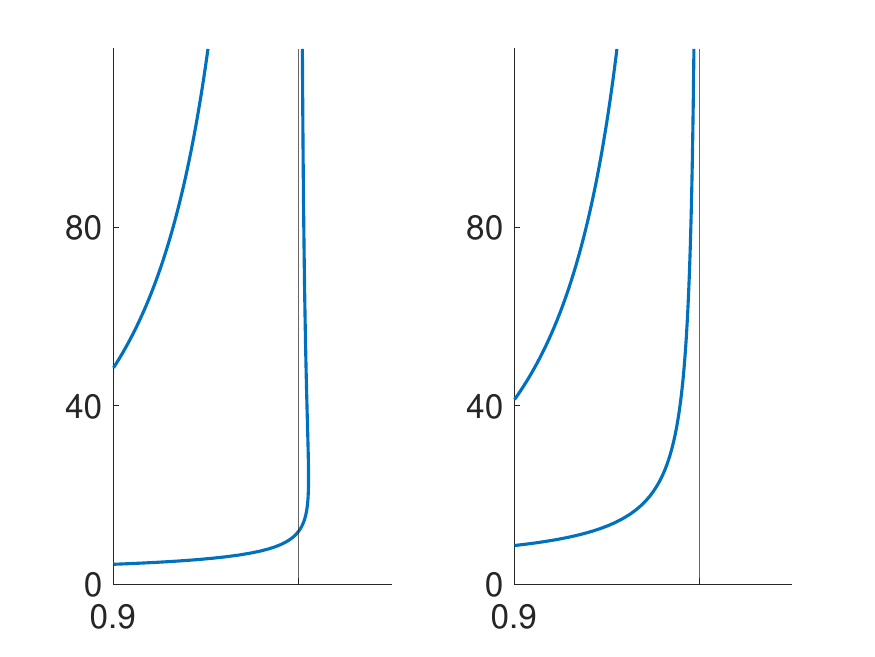}
	\put(-345,140){$(c)$}
	\put(-120,140){$(d)$}
	\put(-406,120){\footnotesize\rotatebox{90}{$m=n$}}
	\put(-309,120){\footnotesize\rotatebox{90}{$m=n$}}
	\put(-388,8){\footnotesize$\gamma_{\!gv}$}
	\put(-375,8){\footnotesize$\gamma^{*}$}
	\put(-280,8){\footnotesize$\gamma_{\!gv}$}
	\put(-267,8){\footnotesize$\gamma^*$}
	\put(-336,10){\footnotesize$\gamma$}
	\put(-240,10){\footnotesize$\gamma$}
	\put(-190,120){\rotatebox{90}{\footnotesize$m=n$}}
	\put(-93,120){\rotatebox{90}{\footnotesize$m=n$}}
	\put(-143,6){\footnotesize$\theta_{gv}$}
	\put(-47,6){\footnotesize$\theta_{gv}$}
	\put(-119,10){\footnotesize$\xi$}
	\put(-24,10){\footnotesize$\xi$}
	\caption{(a) and (b) The behaviour of $\xi \mapsto\beta e^{-\mu\tau(\xi)} g(\xi)$ as $m=n\to\infty$
in the case $(g\downarrow, v\uparrow, \theta_g=\theta_v=\theta_{gv})$ with $\beta e^{-\mu\tau(\theta_{gv})}g(\theta_{gv})>\beta e^{-\mu\tau^-}g^->\beta e^{-\mu\tau^+}g^+$ with the same parameter values as in Figure~\ref{fig:FoldsMZero} except for the value of $a$. In (a) $a=22.1591$
and in (b) $a=14.6591$. The insets show that $\lim_{m=n\to\infty}\max_\xi\{\beta e^{-\mu\tau(\xi)}g(\xi)\}>
\beta e^{-\mu\tau(\theta_{gv})}g(\theta_{gv})$, with the maximum indicated by the dashed line in the inset, and the value at $\theta_{gv}$ by the black dot. (c) and (d) show the corresponding two-parameter continuations in $m=n$ and either (c) $\gamma$ or (d) $\xi$ of the fold bifurcations with the left panel showing the case of
$a=22.1591$ and the right panel for $a=14.6591$.
The fold bifurcations are always associated with unstable steady states. In (c) the asymptote denoted $\gamma^*$ for the fold bifurcation is obtained dividing the maximum value of $\beta e^{-\mu\tau(\xi)}g(\xi)$ from (a) and (b) by $\theta_{gv}$ and is seen to be strictly larger than $\gamma_{gv}$ in both cases. }
\label{fig:FoldsMnotZero}
\end{figure}

Since the value of $\beta e^{-\mu\tau(\theta_{gv})}g(\theta_{gv})$ is independent of $m$ and $n$, the maximum of $\xi \mapsto\beta e^{-\mu\tau(\xi)}g(\xi)$ cannot be smaller than this value, which imposes the bound that the rightmost fold bifurcation must occur for $\gamma\geq\gamma_{gv}$. Using \eqref{eq:h} and \eqref{eq:vglam0} it is easy to show that $h'(\xi)=\gamma M(\xi)$ when $h(\xi)=0$, that is  at a steady state. Consequently, if $M(\theta_{gv})=0$ the right fold bifurcation will occur at $\xi=\theta_{gv}$ with $\gamma=\gamma_{gv}$, while if $M(\theta_{gv})<0$ then $h'(\theta_{gv})<0$ so the fold bifurcation occurs for $\xi<\theta_{gv}$ and $\gamma>\gamma_{gv}$. Similarly, if $M(\theta_{gv})>0$ the fold bifurcation will occur for $\xi>\theta_{gv}$ and $\gamma>\gamma_{gv}$. These situations are illustrated in Figures~\ref{fig:FoldsMZero} and~\ref{fig:FoldsMnotZero} which we now describe in more detail.

From \eqref{eq:Mthetagv}, for any parameter set with $0<r_v<1<r_g$ taking
$$m=\frac{1}{\mu\tau(\theta_{gv})}\frac{1+r_v}{1-r_v}\Bigl[2+n\frac{r_g-1}{r_g+1}\Bigr],$$
ensures that $M(\theta_{gv})=0$. However, for the example in Figure~\ref{fig:FoldsMZero}, we prefer to set $m=n$ and choose the other parameters to enforce equality in \eqref{eq:bothbif}, which ensures that $M(\theta_{gv})=-1$ for all $m=n>0$. This results in a fold bifurcation at $\gamma=\gamma_{gv}$ in the limiting case as $m=n\to\infty$, since as already noted $M(\xi)$ is decreasing at $\xi=\theta_{gv}$
and also both $\xi v'(\xi)/v(\xi)$ and $\xi g'(\xi)/g(\xi)$ become increasingly steep as $m=n\to\infty$ so the point $\xi$ at which $M(\xi)=0$ approaches $\theta_{gv}$ as  $m=n\to\infty$. This is illustrated in
Figure~\ref{fig:FoldsMZero} where panel (b) shows that the maximum of
$\xi \mapsto\beta e^{-\mu\tau(\xi)}g(\xi)$ tends to $\beta e^{-\mu\tau(\theta_{gv})}g(\theta_{gv})$ as $m=n\to\infty$ while panels (c) and (d) show that the corresponding fold bifurcation asymptotes to $\gamma=\gamma_{gv}$ and $\xi=\theta_{gv}$ as $m=n\to\infty$.

Figure~\ref{fig:FoldsMnotZero} illustrates two examples with the same parameters as in
Figure~\ref{fig:FoldsMZero}, except for the length parameter $a$.
Changing just the length parameter $a$ leaves the ratio $r_v$ unchanged, but changes the value of $\tau(\theta_{gv})$ as well as $\tau^-$ and $\tau^+$. With $a=22.1591$ and $m=n$ it follows that $M(\theta_{gv})=n/12-1$ and $\lim_{m=n\to\infty}M(\theta_{gv})=+\infty$, while
for $a=14.6591$ we have $M(\theta_{gv})=-n/12-1$ with $\lim_{m=n\to\infty}M(\theta_{gv})=-\infty$.
As Figure~\ref{fig:FoldsMnotZero}(a) and (b) show, for both of these parameter sets,
\eqref{eq:betaexpabove} is satisfied and $\xi \mapsto\beta e^{-\mu\tau(\xi)}g(\xi)$ is non-monotone.
It is also clear from the insets in these panels that in contrast to Figure~\ref{fig:FoldsMZero}, the maximum of $\xi \mapsto\beta e^{-\mu\tau(\xi)}g(\xi)$ is now strictly larger than
$\beta e^{-\mu\tau(\theta_{gv})}g(\theta_{gv})$ for all $m=n\gg0$. Consequently,
as seen in Figure~\ref{fig:FoldsMnotZero}(c),
the rightmost curve of fold bifurcations asymptotes to a value of $\gamma$ which is strictly larger than $\gamma_{gv}$ as $m=n\to\infty$. As Figure~\ref{fig:FoldsMnotZero}(d) shows in both cases the location of this fold bifurcation converges to $\theta_{gv}$ as $m=n\to\infty$. When $M(\theta_{gv})>0$ the maximum of
$\beta e^{-\mu\tau(\xi)}g(\xi)$ occurs for $\xi>\theta_{gv}$
and consequently for the fold bifurcation as $m=n\to\infty$,
$\xi\to\theta_{gv}$ from above and $\gamma\to\gamma^*$
from below for some $\gamma^*$ strictly larger than $\gamma_{gv}$.
For $M(\theta_{gv})<0$ the convergence is from the other side, but again
$\gamma\to\gamma^*$ for some $\gamma^*$ strictly larger than $\gamma_{gv}$.

We remark that for the example with $a=14.6591$ in Figure~\ref{fig:FoldsMnotZero}, because $M(\theta_{gv})<-1$ the condition \eqref{eq:bothbif}
is violated, even though the condition  \eqref{eq:betaexpabove} is satisfied, demonstrating as we claimed earlier that
\eqref{eq:bothbif} is not a necessary condition to obtain \eqref{eq:betaexpabove}.
This also shows that it is insufficient
to just evaluate the value of $M(\theta_{gv})$ to determine whether fold bifurcations occur.

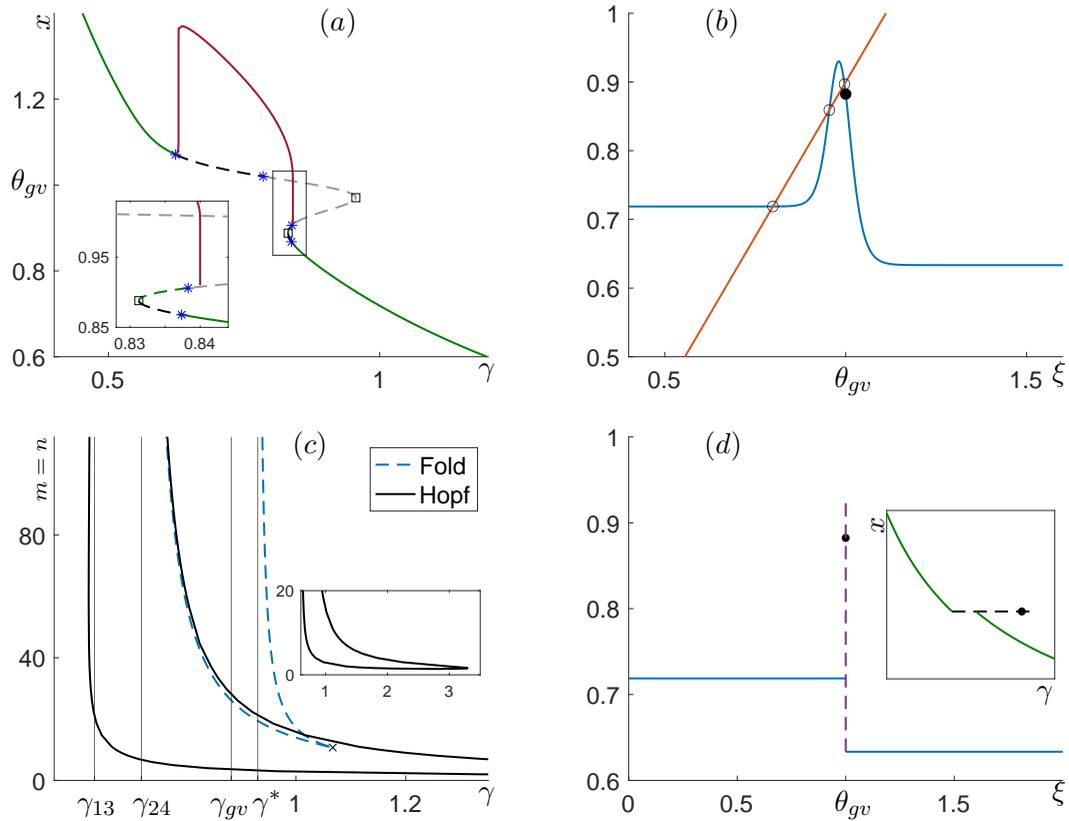
\begin{figure}[thp!]
	\centering
	\includegraphics[scale=0.5]{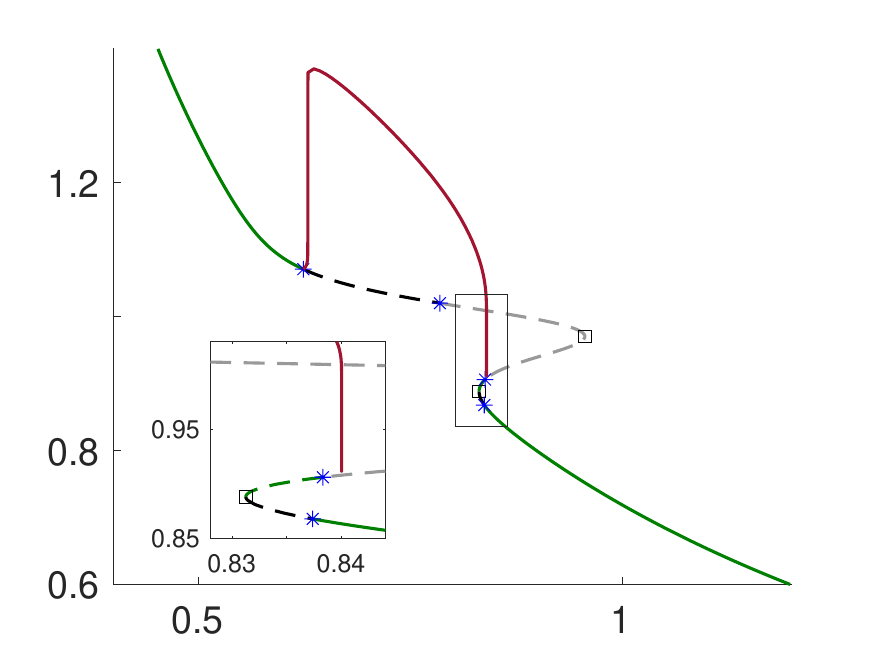}\hspace*{0.5em}\includegraphics[scale=0.5]{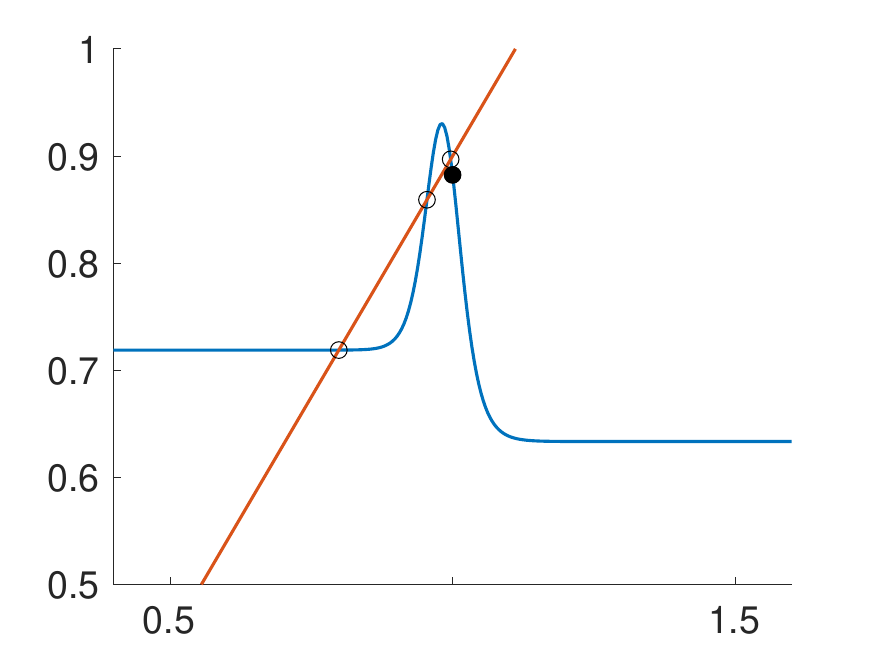}
	\put(-406,140){\rotatebox{90}{$x$}}
	\put(-240,10){$\gamma$}
	\put(-416,80){$\theta_{gv}$}
	\put(-24,8){$\xi$}
	\put(-300, 140){$(a)$}
	\put(-155, 140){$(b)$}
	\put(-106,6){$\theta_{gv}$}\\
	\includegraphics[scale=0.5]{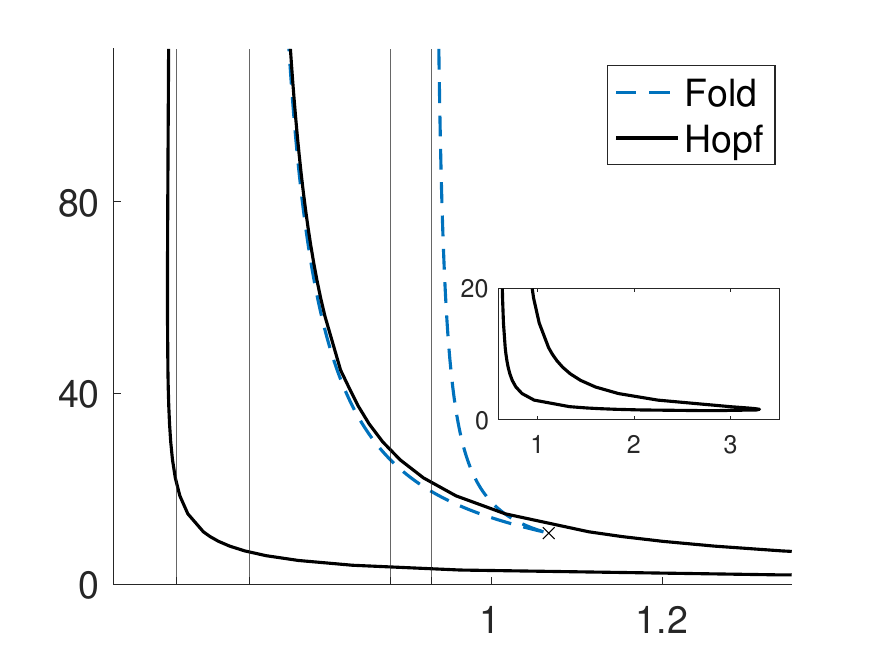}\hspace*{0.5em}\includegraphics[scale=0.5]{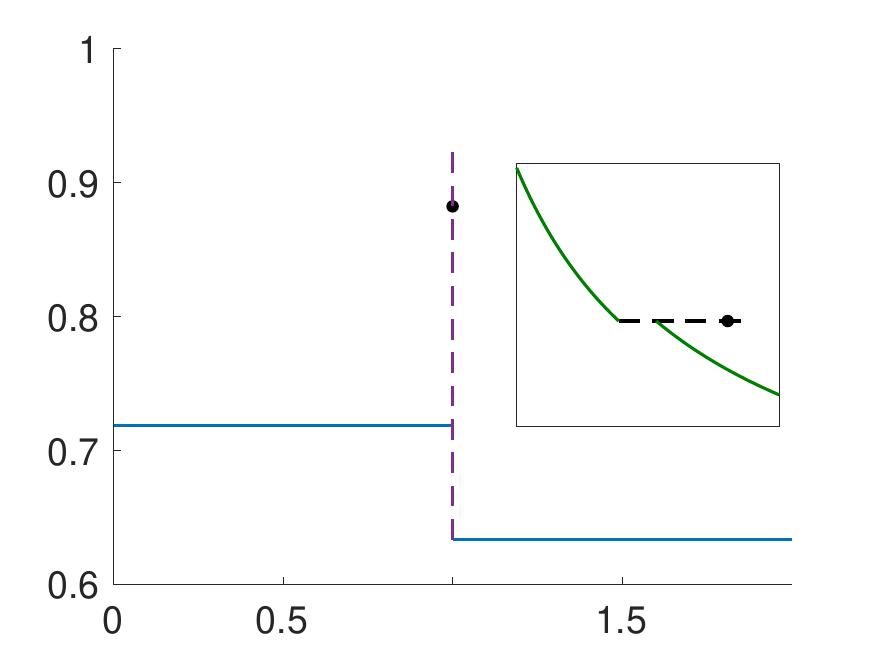}
	\put(-406,122){\rotatebox{90}{\scriptsize $m=n$}}
	\put(-391,6){$\gamma_{13}$}
	\put(-370,6){$\gamma_{24}$}
	\put(-341,6){$\gamma_{gv}$}
	\put(-240,10){$\gamma$}
	\put(-325,6){$\gamma^*$}
	\put(-310, 140){$(c)$}
	\put(-155, 140){$(d)$}
	\put(-24,10){$\xi$}
	\put(-106,6){$\theta_{gv}$}
	\put(-30,48){$\gamma$}
	\put(-92,110){\rotatebox{90}{$x$}}
	\caption{Bifurcations of \eqref{eq:basic}-\eqref{eq:thres} when $(g \downarrow, v \uparrow, \theta_g = \theta_v = \theta_{gv})$ with the same parameters as
in Figure~\ref{fig:thetagvonelimit}.
(a) With smooth nonlinearity $g$ and $v$ defined by \eqref{eq:vghill} with $m=n=40$. The stable periodic orbits are represented by the 2-norm.
(b) Steady states occur at the intersections of $\xi \mapsto\beta e^{-\mu\tau(\xi)} g(\xi)$ and
$\xi \mapsto\gamma\xi$, with $\gamma\xi$ illustrated here with $\gamma=0.9$. The black dot marks the value of
$\beta e^{-\mu\tau(\theta_{gv})} g(\theta_{gv})$.
(c) Two-parameter continuation in $m=n$ and $\gamma$ of the fold bifurcations (in blue) and Hopf bifurcations (black) with the other parameters as above.
Solid curves indicate the parts of the bifurcation branch where there are no characteristic values with positive real part (and hence a stability change at the bifurcation), and dashed lines indicate where the parts of the branch where there is at least one unstable characteristic value. In this example the steady state always loses stability in a Hopf bifurcation, and not at the fold.
The four vertical lines denote
$\gamma_{13}$ and $\gamma_{24}$ (recall \eqref{eq:corners_onetheta})
and $\gamma_{gv}$ (recall \eqref{eq:gammathetagv}), along with
$\gamma^*=0.9305$.
(d) Plot of $\xi \mapsto\beta e^{-\mu\tau(\xi)} g(\xi)$ in the limiting case
with $g$ and $v$ defined by \eqref{eq:gpwconst} and \eqref{eq:vpwconst}.
The black dot marks the value of
$\beta e^{-\mu\tau(\theta_{gv})} g(\theta_{gv})$. The dashed line is obtained by taking the Hausdorff limit of the smooth function, and extends above the point-wise limit at $\theta_{gv}$ because
$\lim_{m=n\to\infty}\max_\xi \beta e^{-\mu\tau(\xi)} g(\xi) > \beta e^{-\mu\tau(\theta_{gv})} g(\theta_{gv})$.
The inset panel depicts the continuation of steady states in the limiting case. Stable steady states are shown as green solid lines, and the singular steady state as a black dashed line. }
	\label{fig:gdownvup_ex2}
\end{figure}

In Figures~\ref{fig:gdownvup_ex2} and~\ref{fig:gdownvup_ex2_po} we return to the parameter set first considered
in Figure~\ref{fig:thetagvonelimit} for which $\beta e^{-\mu\tau(\theta)}g(\theta) > \beta e^{-\mu\tau^-}g^- > \beta e^{-\mu\tau^+}g^+$. Now, we also consider the Hopf bifurcations and resulting periodic orbits.
With these parameters $M(\theta_{gv})<0$ when $m=n>0$ with $M(\theta_{gv})\to-\infty$ as $m=n\to\infty$
and so the upper fold bifurcation approaches $\gamma^*$ for some $\gamma^*>\gamma_{gv}$ from above as
$m=n\to\infty$, as seen in Figure~\ref{fig:gdownvup_ex2}(c), similar to the left panel of
Figure~\ref{fig:FoldsMnotZero}(c). The other fold bifurcation is seen in Figure~\ref{fig:gdownvup_ex2}(c) to approach $\gamma=\gamma_{24}$ as $m=n\to\infty$, as $\xi \mapsto\beta e^{-\mu\tau(\xi)} g(\xi)$ evolves from the smooth
function shown in Figure~\ref{fig:gdownvup_ex2}(b) to its limiting form as in
Figure~\ref{fig:gdownvup_ex2}(d). This agrees  with the arguments under Figure~\ref{fig:thetagvlimits} above.

For the one parameter continuation shown in Figure~\ref{fig:gdownvup_ex2}(a) the steady state is stable for all $\gamma$ sufficiently large or sufficiently small, and loses stability in a Hopf bifurcation (and not at the fold bifurcations). A two parameter continuation of these Hopf bifurcations in Figure~\ref{fig:gdownvup_ex2}(c)
shows that they approach $\gamma_{13}$ and $\gamma_{24}$ as $m=n\to\infty$, similarly to how we saw the Hopf bifurcations were forced to the ends of the interval of singular steady state solutions in
Section~\ref{sec:gdownvconst}.

\begin{figure}[thp!]
	\centering
	\includegraphics[scale=0.5]{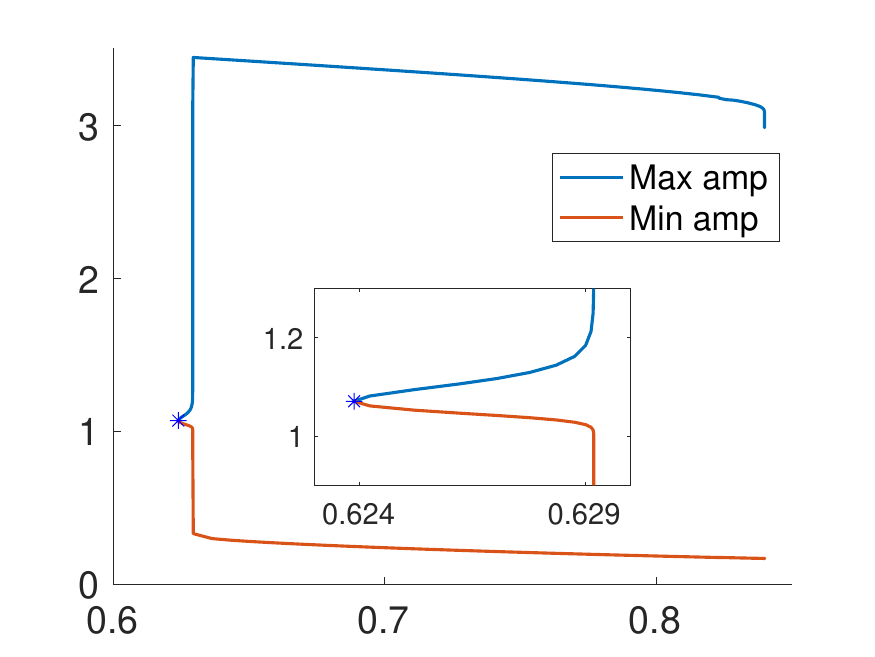}\hspace*{0.5em}\includegraphics[scale=0.5]{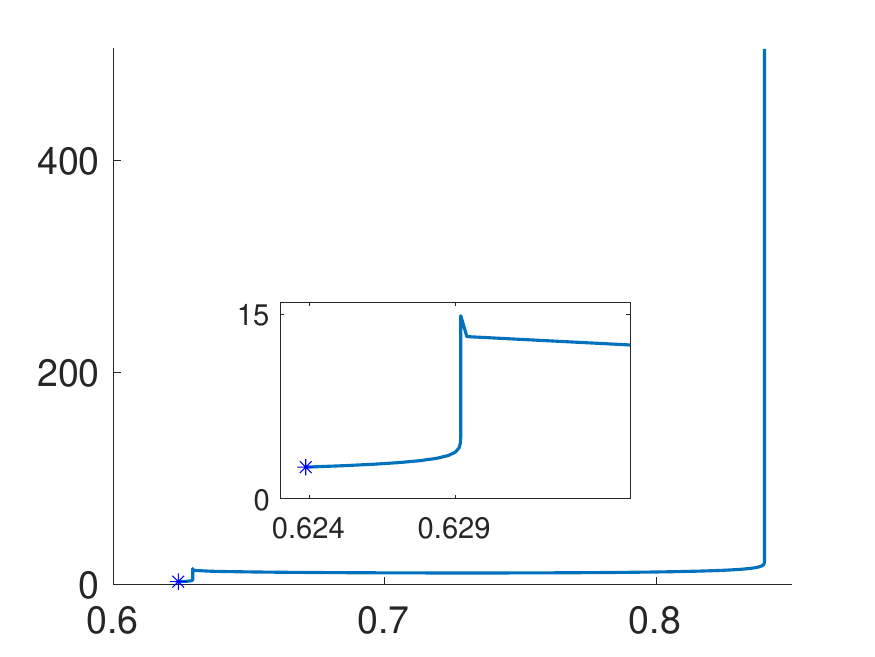}
	\put(-406,140){\rotatebox{90}{$x$}}
	\put(-240,10){$\gamma$}
	\put(-350,115){$(a)$}
	\put(-120,140){$(b)$}
	\put(-192,140){\rotatebox{90}{$T$}}
	\put(-24,10){$\gamma$}\\
	\includegraphics[scale=0.5]{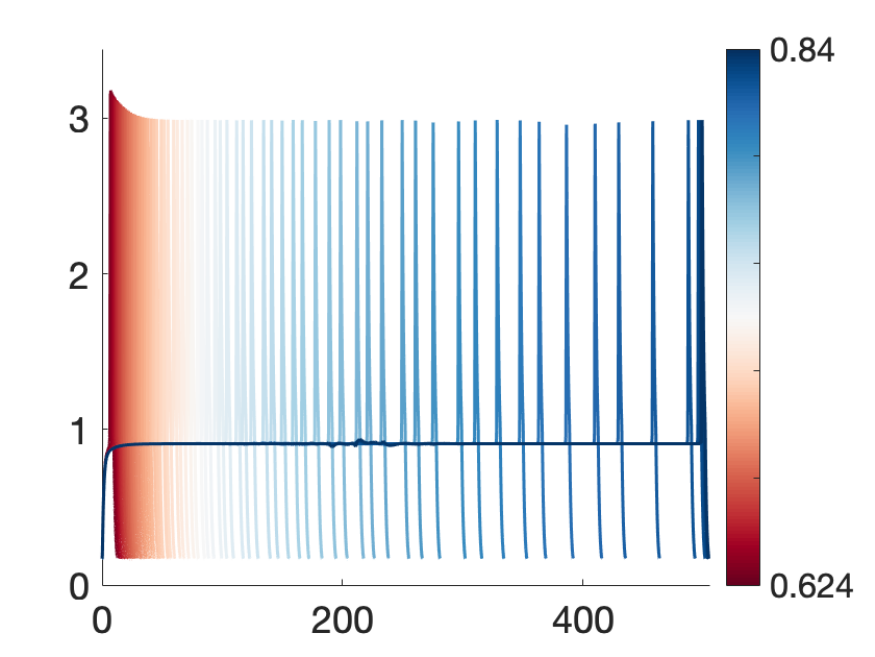}\hspace*{0.5em}\includegraphics[scale=0.5]{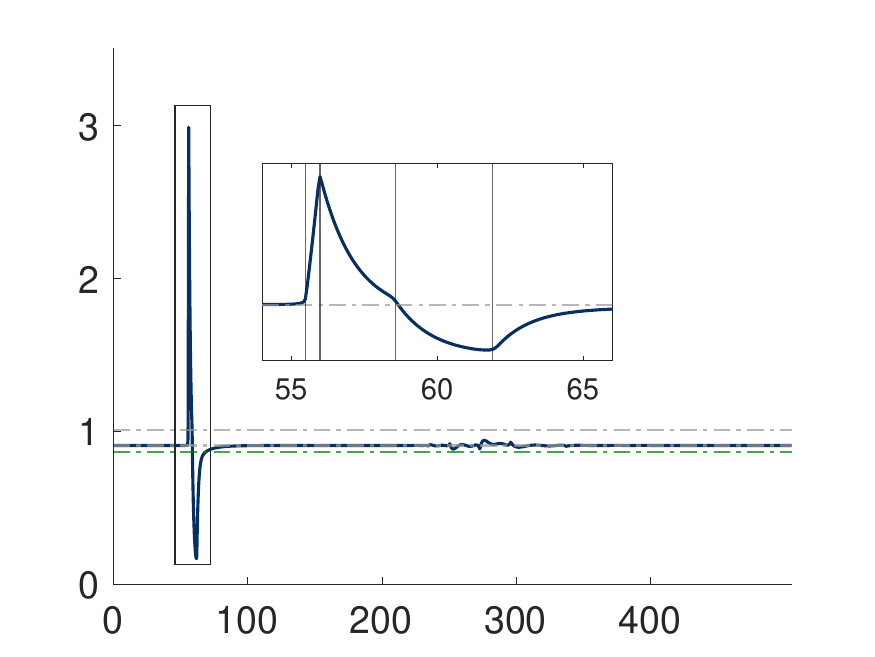}
	\put(-408,140){\rotatebox{90}{$x$}}
	\put(-258,10){$t$}
	\put(-242,80){$\gamma$}
	\put(-350,140){$(c)$}
	\put(-100,140){$(d)$}
	\put(-190,140){\rotatebox{90}{$x$}}
	\put(-24,10){$t$}
	\caption{The branch of stable periodic orbits for $(g \downarrow, v \uparrow, \theta_g = \theta_v = \theta_{gv})$ when $\gamma \in [0.624, 0.84]$ from Figure~\ref{fig:gdownvup_ex2}(a). Panel (a) shows the amplitude of the periodic orbits along the branch.
(b) The period of the orbits. (c) Profiles of the periodic solutions show that the large period orbits at the end of the branch are of relaxation oscillator type.
Panel (d) depicts the last periodic orbit found by \texttt{ddebiftool},
with the inset panel showing a zoom of the dynamics near the spike.}
	\label{fig:gdownvup_ex2_po}
\end{figure}

In Figure \ref{fig:gdownvup_ex2}(a) there are no stable steady states for
$\gamma\in(0.624, 0.837)$, so we explored the stable invariant objects in this case using both
numerical bifurcation detection and continuation 
and numerical simulation of the DDE.
Using \texttt{ddebiftool} we find that the Hopf bifurcation at $\gamma=0.624$
is supercritical and generates a branch of stable periodic orbits, which is illustrated in Figure~\ref{fig:gdownvup_ex2_po}.
There is an apparent canard explosion  \cite{krupa2016canard} 
on this branch at $\gamma\approx0.629$ where the amplitude of the periodic orbit increases dramatically. The continuation of the branch fails in \texttt{ddebiftool} at this point, but since the periodic orbits are stable we are able to continue them through the canard explosion using simulation with \texttt{ddesd}, and then continue again with
\texttt{ddebiftool} after the canard explosion, to obtain
the continuous branch of periodic solutions for $m=n=40$ shown in Figure~\ref{fig:gdownvup_ex2}(a) and Figure~\ref{fig:gdownvup_ex2_po}. After the canard explosion for $\gamma\in(0.629,0.84)$  the periodic orbit is stable with large amplitude and moderate period ($T\approx13$) and evolves slowly as $\gamma$ is varied. Then at $\gamma\approx0.84$ the period of the orbit grows dramatically while the amplitude is roughly constant. In Figure~\ref{fig:gdownvup_ex2}(a) this branch is represented by its 2-norm as  the red curve. This curve appears to terminate on the middle unstable branch of steady states when $\gamma=0.84$. However, as seen in Figure~\ref{fig:gdownvup_ex2_po}(c) and (d) the amplitude of these orbits remains large, so there is not a Hopf bifurcation at this end of the branch.

The large period orbits shown in Figure~\ref{fig:gdownvup_ex2}(c) and (d)
spend a lot of time close to the unstable steady state on the middle branch between the folds, which is itself very close to the threshold $x=\theta_{gv}$, and resemble relaxation oscillations.
This branch of periodic orbits most likely ends in a homoclinic orbit to the unstable steady state with an infinite period. There is also bistability between the stable periodic orbit and a stable steady state for a very small range of parameters  $\gamma \in [0.837, 0.84]$.

\begin{figure}[thp!]
	\centering
	\includegraphics[scale=0.5]{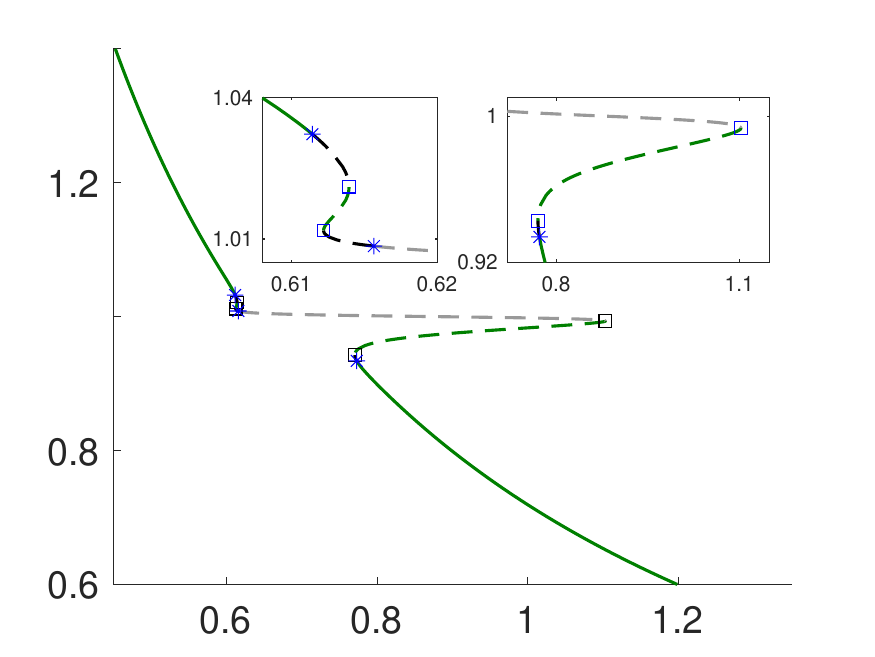}\hspace*{0.5em}\includegraphics[scale=0.5]{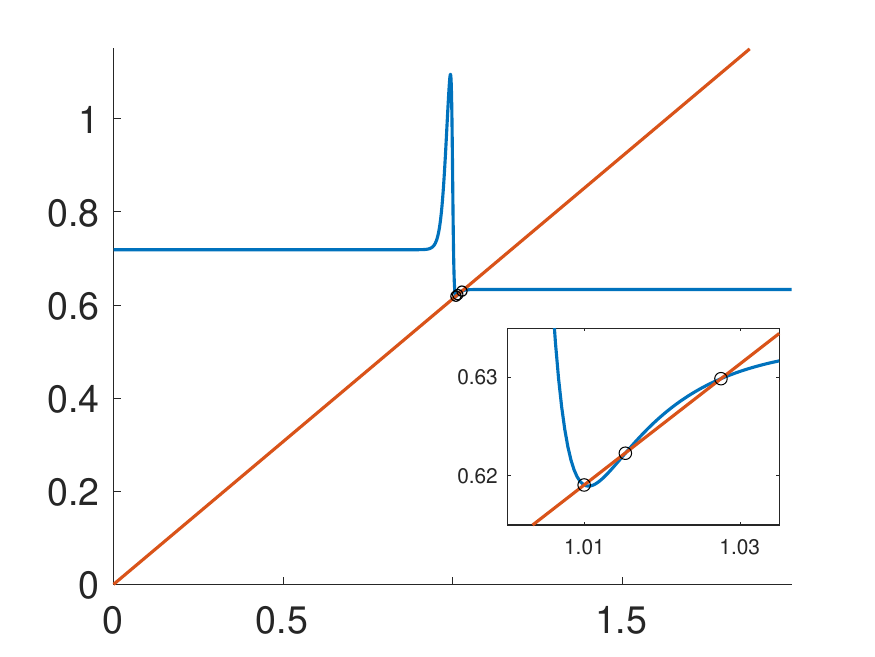}
	\put(-406,140){\rotatebox{90}{$x$}}
	\put(-240,10){$\gamma$}
	\put(-416,80){$\theta_{gv}$}
	\put(-24,10){$\xi$}
	\put(-393, 140){$(a)$}
	\put(-160, 140){$(b)$}
	\put(-106,6){$\theta_{gv}$}
\caption{(a) One parameter continuation of the steady state when $(g \downarrow, v \uparrow, \theta_g = \theta_v = \theta_{gv})$ for varying $\gamma$ with the other parameters the same as in Figure~\ref{fig:gdownvup_ex2} except $m=100$ and $n=500$ reveals four fold bifurcations.
(b) Three steady states are indicated for $\gamma=0.6130$
at the intersections of $\xi \mapsto\beta e^{-\mu\tau(\xi)} g(\xi)$ and $\gamma\xi$.}
	\label{fig:gdownvup_down3}
\end{figure}
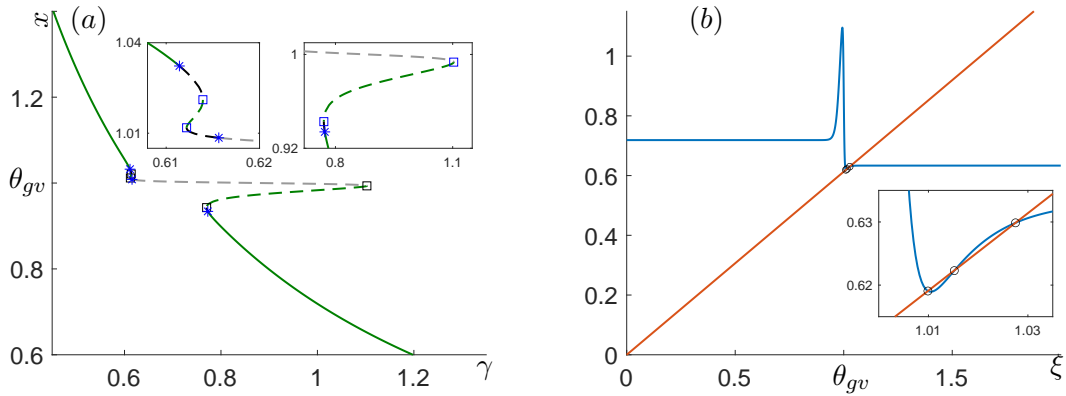

In Figure~\ref{fig:gdownvup_ex2} there are two fold bifurcations, as we have seen in many examples when one of $g$ and $v$ is increasing. However,  as Figure~\ref{fig:gdownvup_down3} shows, it is possible to obtain four fold bifurcations by just changing the steepness of the nonlinearities and resulting width of the interfaces of the nonlinearities.
In Figure~\ref{fig:gdownvup_down3} we do this with $n\gg m\gg0$. Since $n$ is very large the function $g$ has a very sharp interface, and drops from very close to its maximum value $g^+$ to its minimum value $g^-$ while $v$
is essentially unchanged and equal to $v(\theta_{gv})$. This results in the function
$\xi \mapsto\beta e^{-\mu\tau(\xi)}g(\xi)$
having extrema close to
$\beta e^{-\mu\tau(\theta_{gv})}g^-$ and $\beta e^{-\mu\tau(\theta_{gv})}g^+$, as shown previously in Figure~\ref{fig:thetagvonelimit}(a) and (b). Outside of this narrow interface, $g$ is essentially constant,
and therefore  $\beta e^{-\mu\tau(\xi)}g(\xi)$ is increasing, since $v(\xi)$ is an increasing function.
Thus the function $\xi \mapsto\beta e^{-\mu\tau(\xi)}g(\xi)$ is increasing, then decreasing, then increasing again.
For $m$ sufficiently large (so that the derivative of $\beta e^{-\mu\tau(\xi)}g(\xi)$ is larger than $\gamma$) this creates two fold bifurcations close to the ends of the inner interface for $g$ in addition to the other two fold bifurcations close to the edges of the outer interface of $v$.

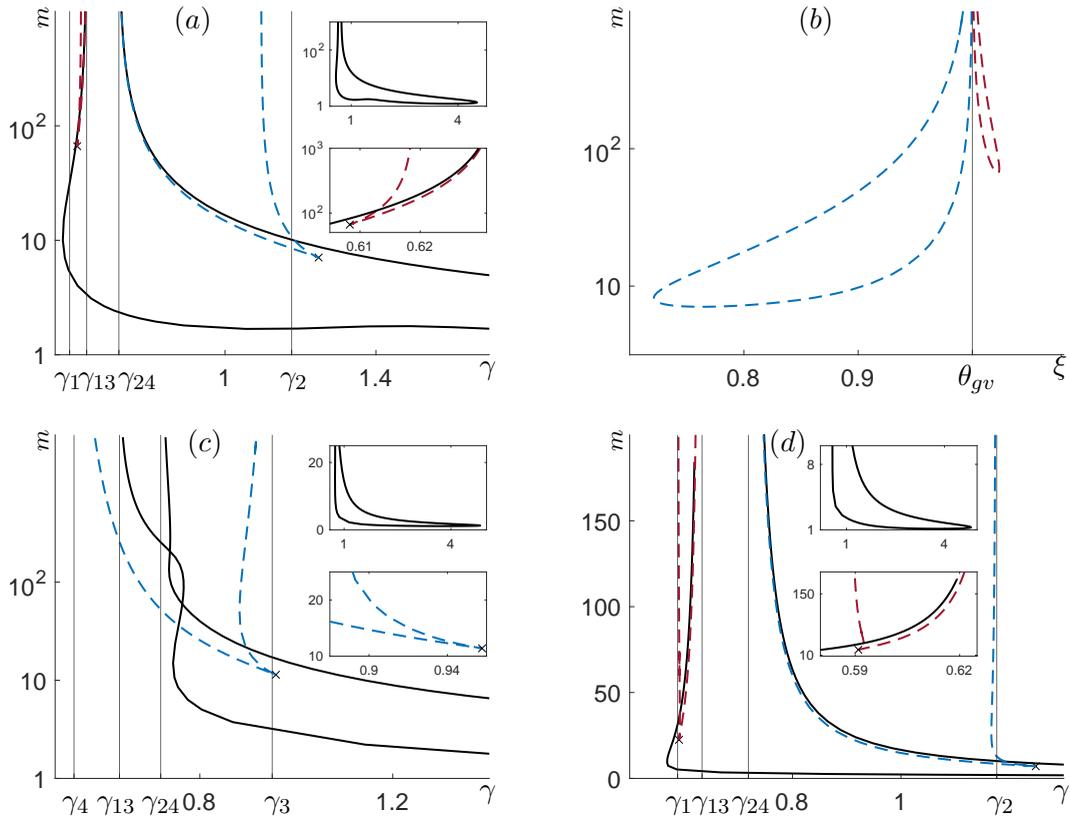
\begin{figure}[htp!]
	\centering	
\includegraphics[scale=0.5]{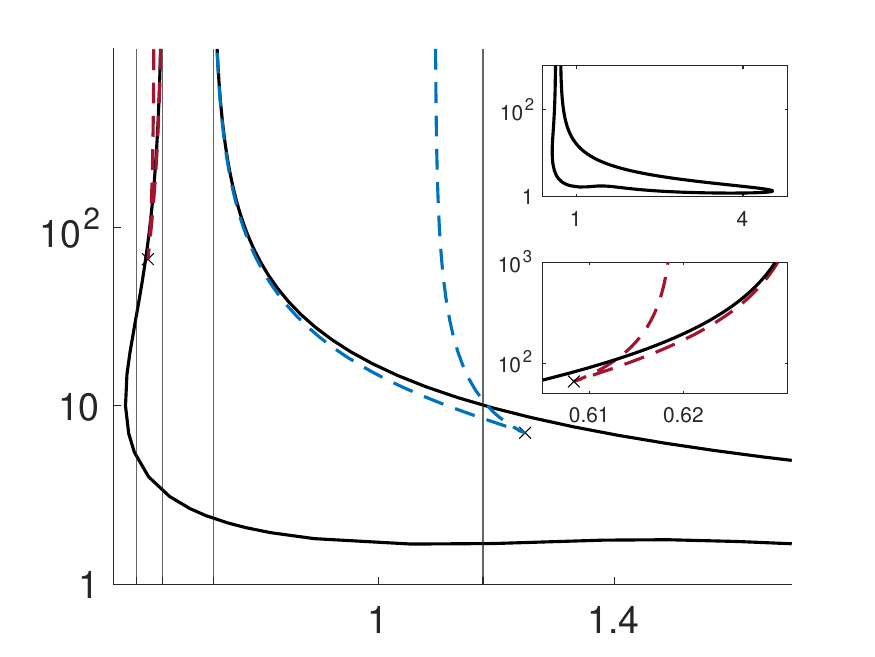}\hspace*{0.5em}\includegraphics[scale=0.5]{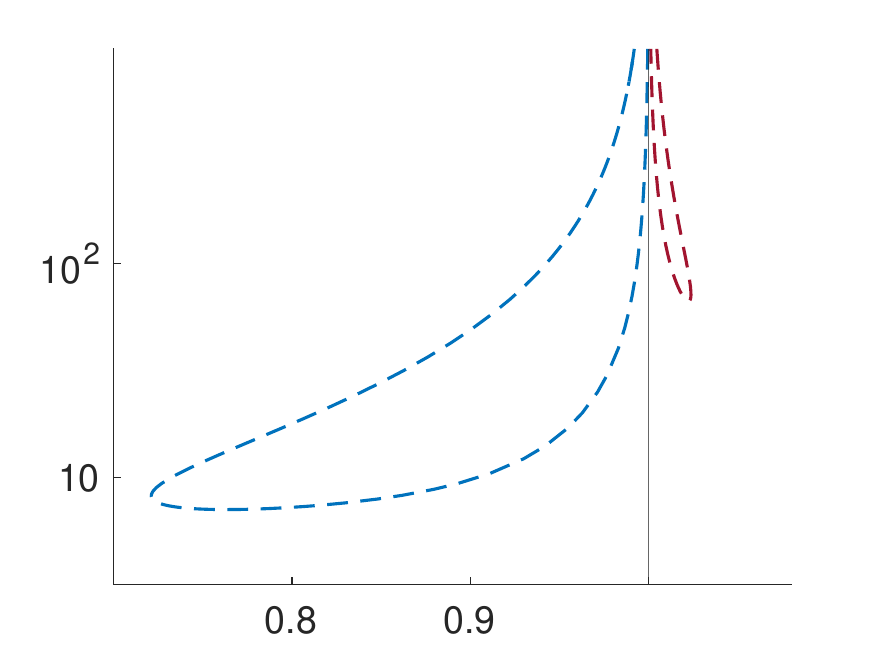}
	\put(-406,140){\rotatebox{90}{\footnotesize$m$}}
	\put(-240,10){$\gamma$}
	\put(-400,6){$\gamma_1$}
	\put(-391,6){$\gamma_{13}$}
	\put(-376,6){$\gamma_{24}$}
	\put(-315,6){$\gamma_2$}
	\put(-355, 140){$(a)$}
	\put(-120, 140){$(b)$}
	\put(-190,140){\rotatebox{90}{\footnotesize$m$}}
	\put(-24,10){$\xi$}
	\put(-60,6){$\theta_{gv}$}\\

\includegraphics[scale=0.5]{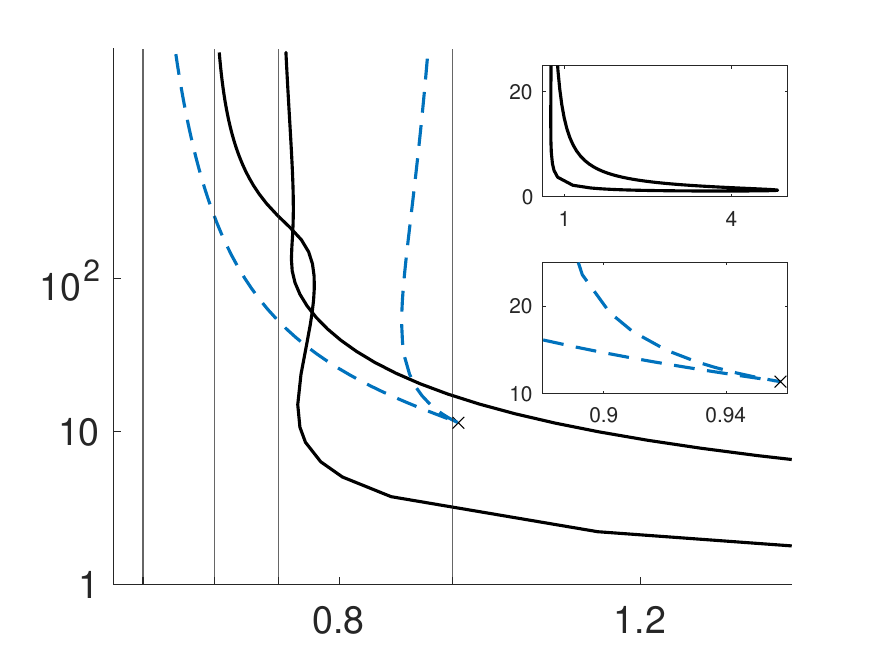}\hspace*{0.5em}\includegraphics[scale=0.5]{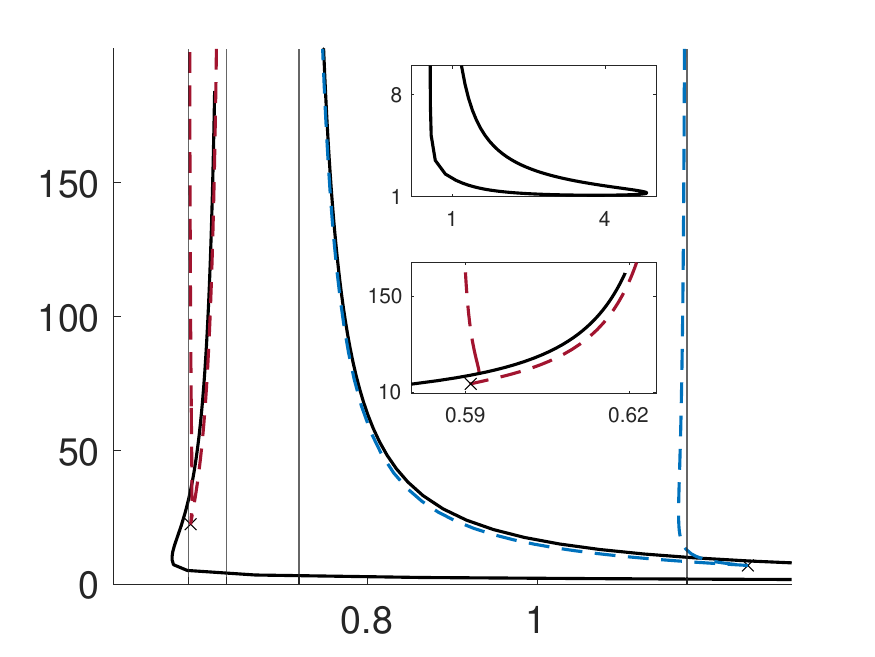}
	\put(-406,140){\rotatebox{90}{\footnotesize$m$}}
	\put(-240,10){$\gamma$}
	\put(-397,6){$\gamma_4$}
	\put(-384,6){$\gamma_{13}$}
	\put(-367,6){$\gamma_{24}$}
	\put(-321,6){$\gamma_3$}
	\put(-350, 140){$(c)$}
	\put(-131, 140){$(d)$}
	\put(-190,140){\rotatebox{90}{\footnotesize$m$}}
	\put(-24,10){$\gamma$}
    \put(-170,6){$\gamma_1$}
	\put(-160,6){$\gamma_{13}$}
	\put(-145,6){$\gamma_{24}$}
	\put(-49,6){$\gamma_2$}\\
	\caption{Two parameter continuations of the bifurcations of \eqref{eq:basic}-\eqref{eq:thres} when $(g \downarrow, v \uparrow, \theta_g = \theta_v = \theta_{gv})$ with the same parameters as in Figures~\ref{fig:thetagvonelimit},~\ref{fig:gdownvup_ex2} and~\ref{fig:gdownvup_down3}, except for $m$ and $n$. The branches of fold bifurcations are shown as blue and red curves and the Hopf bifurcations in black.
Solid curves indicate the parts of the bifurcation branch where there are no characteristic values with positive real part (and hence a stability change at the bifurcation), and dashed lines indicate where the parts of the branch where there is at least one unstable characteristic value. In all four examples the steady state always loses stability at a Hopf bifurcation, and not at the fold bifurcation.
(a) and (b) Continuation of the fold bifurcations in $\gamma$ and $m$ with $n=5m$, showing in (a) the $\gamma$ values of the fold bifurcations and (b) the value of the steady state $\xi$ at the fold bifurcation.
Cusp bifurcations occur at $(\gamma, m)=(0.6083, 66.3853)$ and $(\gamma, m)=(1.2480, 7.0681)$.
(c) and (d) are similar to (a), except in (c) $n=m^{2/3}$ and in (d) $n=m^2$. Cusp bifurcations occur in (c) at $(\gamma, m)=(0.9577, 11.3879)$, and in (d) at $(\gamma, m)=(0.5909, 22.4752)$ and $(\gamma, m)=(1.2484, 7.0591)$.}
	\label{fig:gdownvup_down4}
\end{figure}

Figure~\ref{fig:gdownvup_down4}(a) shows that with $n=5m$ the second fold bifurcation exists for all $m\ge 66.39$. As shown in Figure~\ref{fig:gdownvup_down4}(b) one pair of fold bifurcations exists for
$\xi<\theta_{gv}$ and the other for $\xi>\theta_{gv}$. This happens because $\theta_{gv}$ is in the middle of the sharp interface of $g$, the only (very short) interval on which
$\beta e^{-\mu\tau(\xi)}g(\xi)$ is decreasing, with two of the corners giving rise to the fold bifurcations on each side of $\theta_{gv}$. Figure~\ref{fig:gdownvup_down4}(b) shows that as $n=5m\to\infty$ all of the fold bifurcations are squeezed into $\xi=\theta_{gv}$, but in Figure~\ref{fig:gdownvup_down4}(a) we see that the limiting $\gamma$ values of the folds all appear to be different. Because
$\beta e^{-\mu\tau(\xi)}g(\xi)\to \beta e^{-\mu\tau^-}g^-$ as $m,n\to\infty$ for all
$\xi<\theta_{gv}$ (and to $\beta e^{-\mu\tau^+}g^+$ for $\xi>\theta_{gv}$), two of the fold bifurcations asymptote to $\gamma=\gamma_{24}$ and $\gamma=\gamma_{13}$ as $m,n\to\infty$, as shown in
Figure~\ref{fig:gdownvup_down4}(a). These are the outer two fold bifurcations for the smallest and largest $\xi$ values in Figure~\ref{fig:gdownvup_down4}(b). The asymptotic $\gamma$ values for the inner pair of fold bifurcations depends on the relative widths of the interfaces of the $g$ and $v$ functions. With $n=5m$ it appears
that the limiting $\gamma$ values are strictly inside the interval $(\gamma_1,\gamma_2)$.  Different behaviour is observed
if $m$, $n$ approach infinity with different relationships.
For example, if $n=m^2$ (as considered first in Figure~\ref{fig:thetagvonelimit}(b)) then the $g$ interface is much narrower than for $v$
and the inner folds asymptote to $\gamma=\gamma_2$ and  $\gamma=\gamma_1$ while the outer folds still asymptote to $\gamma=\gamma_{24}$ and $\gamma=\gamma_{13}$ as seen in Figure~\ref{fig:gdownvup_down4}(d).

If $n=m^{r}\to\infty$ for $r\in(0,1)$ then the behaviour of $\xi \mapsto\beta e^{-\mu\tau(\xi)}g(\xi)$ and the bifurcations are completely different, as shown in Figure~\ref{fig:thetagvonelimit}(d) and Figure~\ref{fig:gdownvup_down4}(c). In this case
$m\gg n$ so the $v$ interface is much narrower than the $g$ interface. As $g$ is decreasing this results in the function $\xi \mapsto\beta e^{-\mu\tau(\xi)}g(\xi)$ being decreasing for $\xi$ small and large, and increasing only on a small
interval about $\xi=\theta_{gv}$ corresponding to the narrow interface of $v$. In this scenario the limiting form of the function $\beta e^{-\mu\tau(\xi)}g(\xi)$ is still piecewise constant for $\xi<\theta_{gv}$ and  $\xi>\theta_{gv}$, but now the function is decreasing rather than increasing at the corners corresponding to
$\gamma_{24}$ and $\gamma_{13}$, so these corners no longer result in fold bifurcations. Instead the folds occur close to the local extrema of $\beta e^{-\mu\tau(\xi)}g(\xi)$ seen in Figure~\ref{fig:thetagvonelimit}(d),
so that, for instance, when $n=m^{2/3}\to\infty$ the fold bifurcations in Figure~\ref{fig:gdownvup_down4}(c) approach $\gamma_4$ and $\gamma_3$ defined by \eqref{thetagvonelimitm}.

Figure~\ref{fig:gdownvup_down4}(b) shows that the locations of the fold bifurcations are squeezed into $\xi=\theta_{gv}$ as $n=5m\to\infty$ and the nonlinearities in
\eqref{eq:basic}-\eqref{eq:thres} become piecewise constant for $\xi\ne\theta_{gv}$. Similar behaviour was observed when $n=m^r\to\infty$ for $r>0$ but we omit the figures. Notice also that as already observed in the previous example, the outer Hopf bifurcations at which the steady state changes stability approach $\gamma_{24}$ and $\gamma_{13}$ as in the limiting case as $n=m^r\to\infty$ in all of these examples.

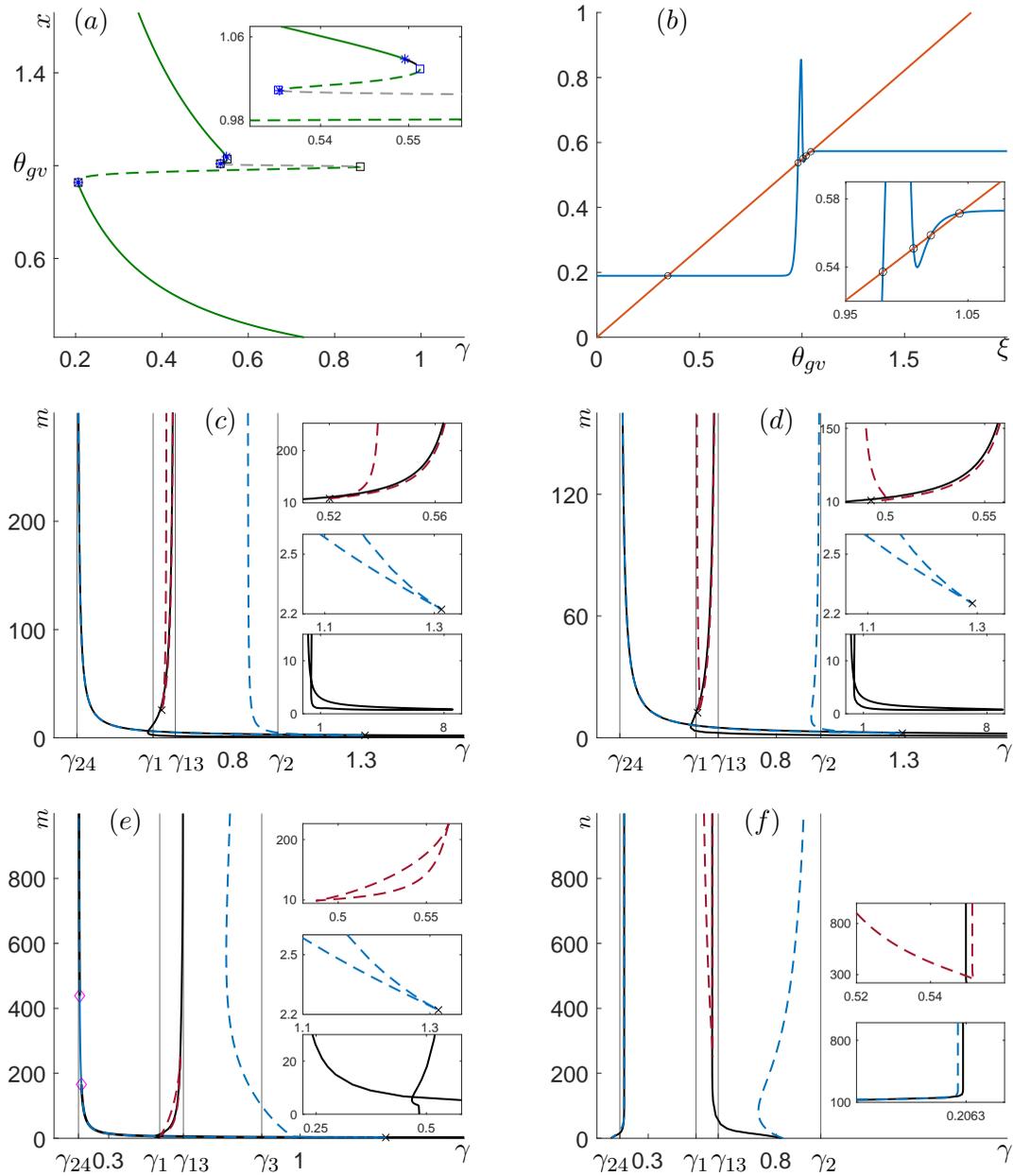
\begin{figure}[thp!]
	\centering
	\includegraphics[scale=0.5]{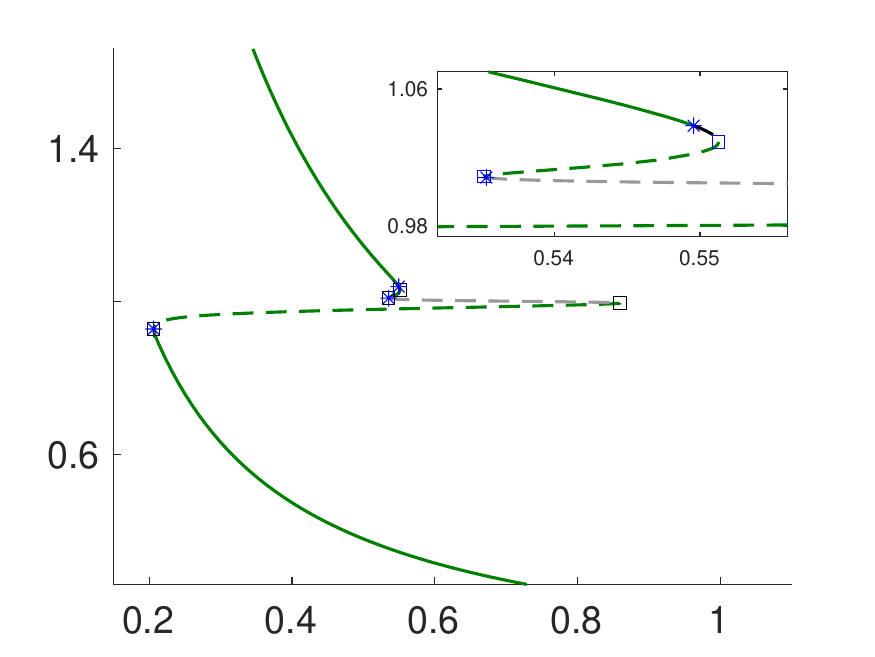}\hspace*{0.5em}\includegraphics[scale=0.5]{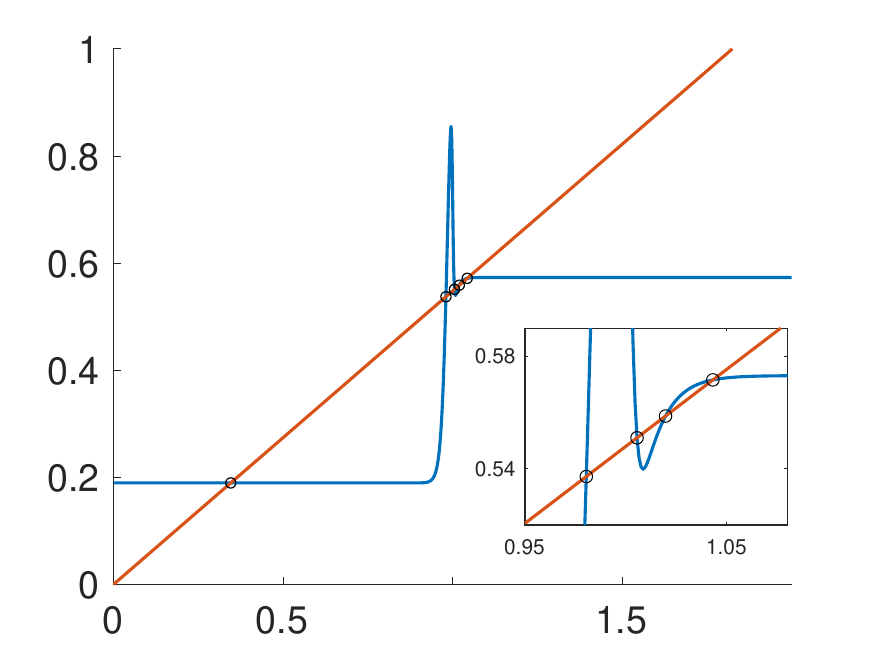}
	\put(-406,140){\rotatebox{90}{$x$}}
	\put(-240,10){$\gamma$}
	\put(-416,83){$\theta_{gv}$}
	\put(-24,10){$\xi$}
	\put(-392, 140){$(a)$}
	\put(-160, 140){$(b)$}
	\put(-106,6){$\theta_{gv}$}\\

\includegraphics[scale=0.5]{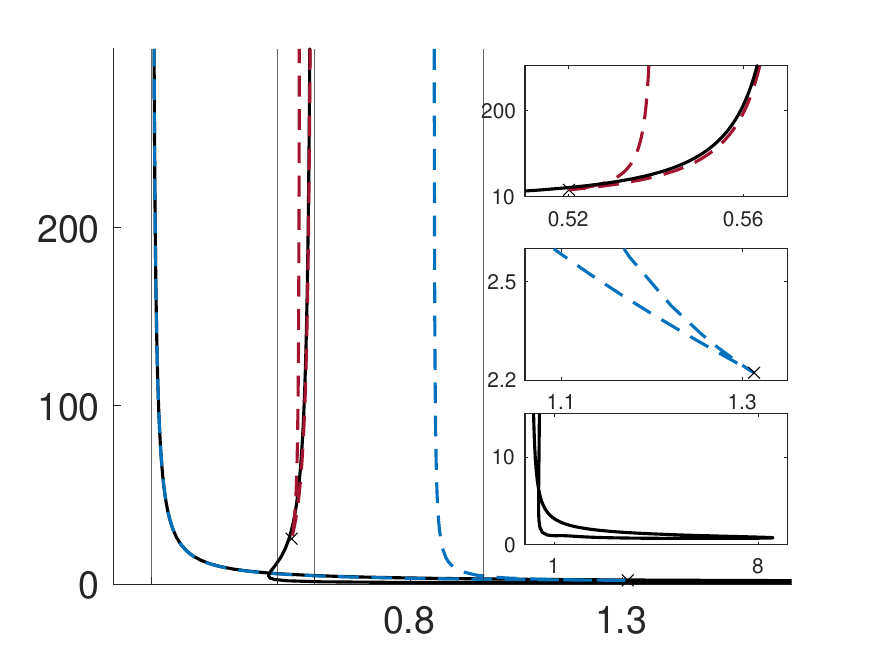}\hspace*{0.5em}\includegraphics[scale=0.5]{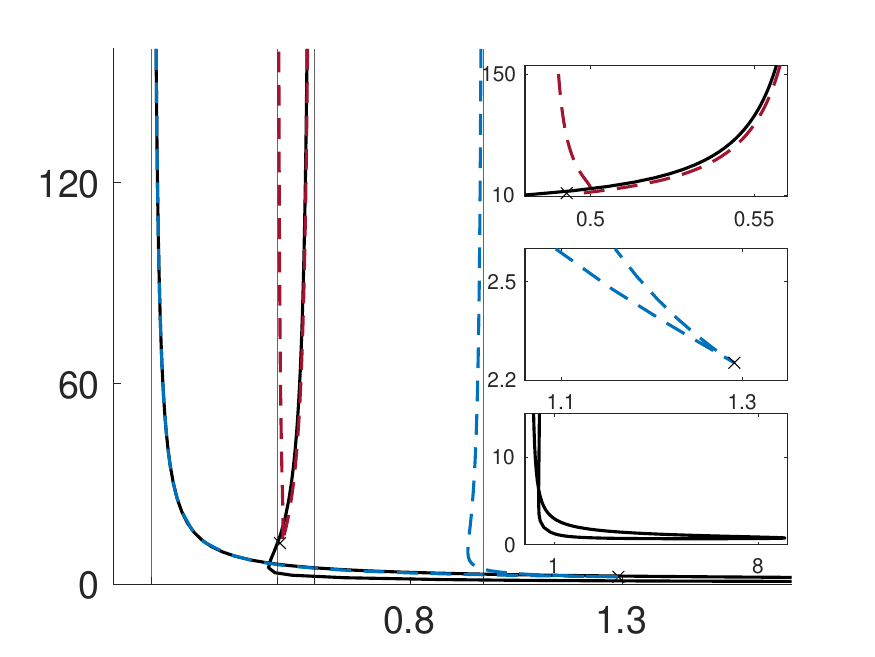}
	\put(-406,140){\rotatebox{90}{\footnotesize$m$}}
	\put(-240,10){$\gamma$}
	\put(-397,6){$\gamma_{24}$}
	\put(-366,6){$\gamma_1$}
	\put(-354,6){$\gamma_{13}$}
	\put(-314,6){$\gamma_2$}
	\put(-340, 140){$(c)$}
	\put(-190,140){\rotatebox{90}{\footnotesize$m$}}
	\put(-24,10){$\gamma$}
	\put(-179,6){$\gamma_{24}$}
	\put(-149,6){$\gamma_1$}
	\put(-138,6){$\gamma_{13}$}
	\put(-98,6){$\gamma_2$}
	\put(-119, 140){$(d)$}\\

\includegraphics[scale=0.5]{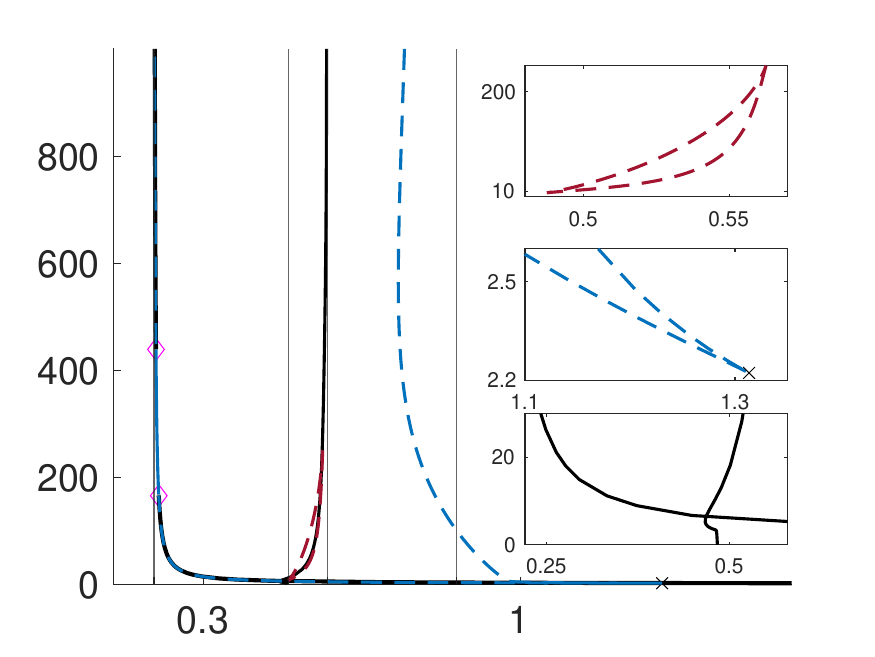}\hspace*{0.5em}\includegraphics[scale=0.5]{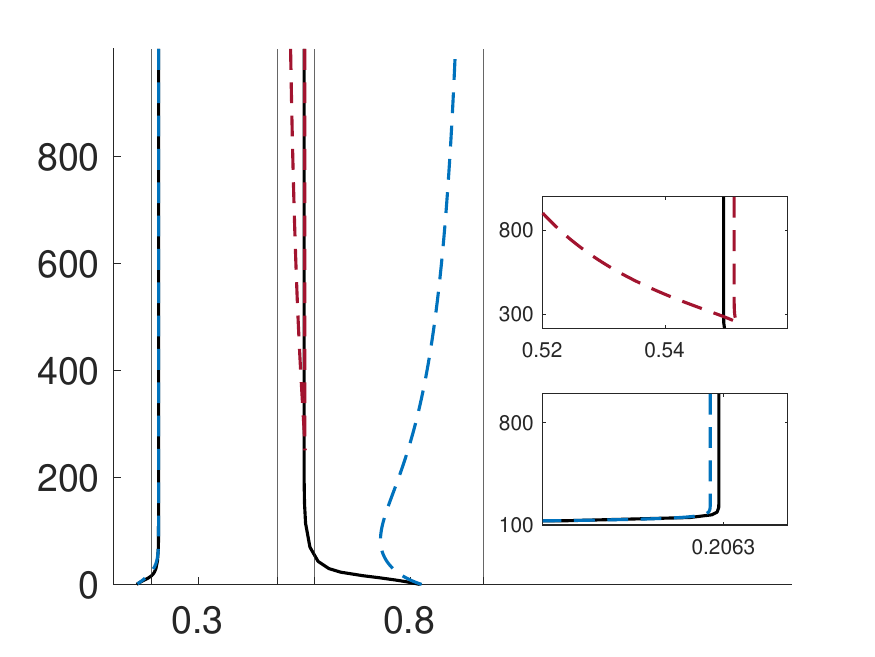}
	\put(-406,140){\rotatebox{90}{\footnotesize$m$}}
	\put(-240,10){$\gamma$}
	\put(-399,6){$\gamma_{24}$}
	\put(-364,6){$\gamma_1$}
	\put(-352,6){$\gamma_{13}$}
	\put(-320,6){$\gamma_3$}
	\put(-378, 140){$(e)$}
	\put(-125, 140){$(f)$}
	\put(-190,140){\rotatebox{90}{\footnotesize$n$}}
	\put(-24,10){$\gamma$}
	\put(-185,6){$\gamma_{24}$}
	\put(-149,6){$\gamma_1$}
	\put(-138,6){$\gamma_{13}$}
	\put(-98,6){$\gamma_2$}

\caption{Bifurcations of \eqref{eq:basic}-\eqref{eq:thres} when $(g \downarrow, v \uparrow, \theta_g = \theta_v = \theta_{gv})$ with $\beta=1.4$, $\mu=0.2$, $g^-=1$, $g^+=0.5$, $\theta_g=\theta_v=\theta_{gv}=1$, $a=1$, $v^-=0.1$ and $v^+=1$. Apart from $m$, $n$, $v^-$ and $v^+$,
these are the same parameters as in Figures~\ref{fig:thetagvonelimit}
and~\ref{fig:gdownvup_ex2}-\ref{fig:gdownvup_down4}.
(a) One parameter continuation of steady states in $\gamma$ with smooth nonlinearities $g$ and $v$ defined by \eqref{eq:vghill} with $m=100$ and $n=500$. This reveals four fold bifurcations and up to five coexisting steady states (at most two of which are stable).
(b) Illustration of the five steady states occurring in (a) at the intersections of $\xi \mapsto\beta e^{-\mu\tau(\xi)} g(\xi)$ and $\xi \mapsto\gamma\xi$ with $\gamma=0.5480$.
(c) to (f) Two parameter continuations of the fold bifurcations (red and blue) and Hopf bifurcations (black) with (c) $n=5m$ and (d) $n=m^2$ (e) $m$ varying with $n=500$ fixed and
(f) $n$ varying with $m=100$ fixed. The bifurcation curves are drawn according to the number of characteristic values with positive real part; solid for zero and dashed for one or more. Cusp points are denoted by crosses.
In (e) there is a change of stability at the fold bifurcation between two Bogdanov-Takens bifurcations (denoted by pink diamonds) at $(\gamma, m)=(0.2, 165.62)$ and $(\gamma, m)=(0.1938,438.82)$, otherwise in these examples the stability change is always at the Hopf bifurcation (as was the case in
Figure~\ref{fig:gdownvup_down4}).}
\label{fig:gdownvup_up6}
\end{figure}

Although there are four fold bifurcations in the example in Figures~\ref{fig:gdownvup_down3} and~\ref{fig:gdownvup_down4}, because of the location of the folds  it is not possible to obtain five co-existing steady states.

As a final example we change just the parameters $v^\pm$ in the $v$ function
from the previous example, to obtain a new example with
$\beta e^{-\mu\tau(\theta)}g(\theta) > \beta e^{-\mu\tau^+}g^+ > \beta e^{-\mu\tau^-}g^-$, whose dynamics are explored in Figure~\ref{fig:gdownvup_up6}.
With $m=n$ the
dynamics with one and two parameter continuations is very similar to the earlier case shown in Figure~\ref{fig:gdownvup_ex2} with a single pair of fold bifurcations and so we omit the figures here. But different choices of $m$ and $n$ lead to more interesting dynamics shown in Figure~\ref{fig:gdownvup_up6}.
One-parameter continuation in $\gamma$ with $n$ and $m$ fixed with $n\gg m\gg0$, as shown in
Figure~\ref{fig:gdownvup_up6}(a) reveals four fold bifurcations, a large interval of $\gamma$ values on which there are three coexisting steady states with two stable, and a small interval for $\gamma\in(0.5352,0.5513)$ for which there are five co-existing steady states of which one or two are stable. Figure~\ref{fig:gdownvup_up6}(b) illustrates how these five steady states arise because for $n\gg m$ the downward $g$ interface is much narrower than the interface for the increasing $v$ function (compare with Figures~\ref{fig:2} and~\ref{fig:gdownvup_down3}).
Panels (c) and (d) of Figure~\ref{fig:gdownvup_up6} show two-parameter continuations with $n=5m$ and $n=m^2$ which reveal that with these parameter constraints all four fold bifurcations persist to arbitrary large values of $m$
with one pair of fold bifurcations contained in the interval $(\gamma_1,\gamma_{13})$ and the other
in the interval $(\gamma_{24},\gamma_2)$ for all $m$ sufficiently large, just as was seen in
Figure~\ref{fig:gdownvup_down4}. But for the parameters used in Figure~\ref{fig:gdownvup_up6}
we have $(\gamma_1,\gamma_{13})\subset(\gamma_{24},\gamma_2)$ so there are five coexisting steady states between the red curves in Figure~\ref{fig:gdownvup_up6}(c) and (d) for arbitrary large $m$.

In the examples shown in Figure~\ref{fig:gdownvup_up6}(c) and (d) the stable steady states loses stability in a Hopf bifurcation just before the fold bifurcation. As in the previous examples in this section these Hopf bifurcations asymptote to $\gamma_{24}$ and $\gamma_{13}$ as $m$, $n\to\infty$.
In Figure~\ref{fig:gdownvup_up6}(e) and (f) we explore a different scenario, where we keep
one of $m$ or $n$ fixed but much larger than zero, and allow the other one to vary on the positive real line.

The value $m=0$ in Figure~\ref{fig:gdownvup_up6}(e) corresponds to the $v$ function being constant, and  since $g$ is decreasing, this corresponds to the scenario considered in Section~\ref{sec:gdownvconst}. This is why the curves of Hopf bifurcations, but not the folds, extend all the way down to $m=0$.
For $0<m\ll n$ the $g$ function has a much narrower interface than the $v$ function, so $\beta e^{-\mu\tau(\xi)}g(\xi)$ will be first increasing then decreasing and finally increasing again,
similar to many of the examples in this section.
However for
$m\gg n\gg0$ the $v$ function will have a narrower interface, so $\beta e^{-\mu\tau(\xi)}g(\xi)$
will be first decreasing, like in Figures~\ref{fig:thetagvonelimit}(d) and ~\ref{fig:gdownvup_down4}(c).
Interesting dynamics is observed with $n\gg m>0$. In particular, in Figure~\ref{fig:gdownvup_up6}(e) there is a pair of Bogdanov-Takens bifurcations,
and on the curve of fold bifurcations between these points a steady state loses stability at the fold
bifurcation, unlike the previous examples in this section, but similar to examples in Section~\ref{sec:gconstvup}.
In Figure~\ref{fig:gdownvup_up6}(f) we explore varying $n$ with $m\gg0$ fixed. Thus when $n=0$ we are in the scenario of $g$ fixed and $v$ increasing explored in Section~\ref{sec:gconstvup} which is why fold bifurcations  persist down to $n=0$ in Figure~\ref{fig:gdownvup_up6}(f). As $n$ increases, once the interface of the $g$ function becomes narrow enough there is a cusp bifurcation which leads to a second pair of fold bifurcations and five coexisting steady states for all $n$ sufficiently large. For $n$ large the folds are contained in the intervals $(\gamma_{24},\gamma_2)$ and $(\gamma_1,\gamma_{13})$, but do not asymptote to the ends of these intervals as $n\to\infty$ because $m$ is held fixed and finite.

\subsection{Summary}\label{summary two hill} 

In Section~\ref{sec:twohill} we have considered the situation where both functions $g$ and $v$ are non-constant.
We first derived formulae for the existence of both fold bifurcations \eqref{eq:vglam0} and Hopf bifurcations \eqref{eq:char_vg1}-\eqref{eq:char_vg2}, which reduce to formulae derived in Section~\ref{sec:onehill} if either $v$ or $g$ is a constant function.

In Section~\ref{sec:twothetas} we study the case when the thresholds are different; $\theta_g \not= \theta_v$.  Not surprisingly, we recover the same dynamics and bifurcation structures as in Section~\ref{sec:onehill} near each threshold when $m\gg0$, or $n\gg0$ since in these limits the behavior localizes in the area where the other function is essentially constant. In particular, in the case $(g\downarrow, v\downarrow)$ the dynamics is similar to that observed in Section~\ref{sec:gdownvconst} near $x = \theta_g$ and to that in Section~\ref{sec:gconstvdown} near $x=\theta_v$.  In Section~\ref{sec:gdownvup} we consider the example
 $(g\downarrow, v\uparrow)$ which supports bistability between an equilibrium and periodic orbit and in  Section~\ref{sec:gupvup}  we examine $(g\uparrow, v\uparrow)$  which can exhibit up to five equilibria and tristability.

 In Section~\ref{sec:onetheta} we look at the situation when $\theta_v=\theta_g$.  First in Section~\ref{sec:bothupordown} we consider a relatively straightforward situation where both functions $g$ and $v$ are either increasing $(g\uparrow, v\uparrow)$ or both decreasing $(g\downarrow, v\downarrow)$. Both of these cases are similar to their counterparts discussed in Section~\ref{sec:onehill}.

The most interesting case is when $\theta_v=\theta_g$, but  the monotonicity of $g$ and $v$ do not agree i.e.  $(g\uparrow, v\downarrow)$ and $(g\downarrow, v\uparrow)$.
This is discussed in Section~\ref{sec:gvopposite}. The main result is that which behavior dominates (i.e. the behavior observed for increasing $g$ or $v$ or behavior for decreasing $g$ or $v$) depends sensitively on the 
way in which $m$ and $n$ approach infinity. We capture this informally in a concept of ``interface width''; if $m\gg n$ then the interface of $v$ is much narrower than the interface of $g$ and vice versa.

Figure~\ref{fig:thetagvonelimit} illustrates that the function $\xi\mapsto\beta e^{-\mu \tau(\xi)}g(\xi)$ converges to different limits as $m$ and $n$ go to infinity at different rates, indicating that the fold bifurcations will occur at different places in these limits.
We also observe that additional equilibria may appear when $m$ and $n$ approach infinity at different rates (Figures~\ref{fig:gdownvup_down4} and \ref{fig:gdownvup_up6}).
Finally, we observe a canard-like explosion in amplitude of a periodic orbit that emanates from a Hopf bifurcation, see Figure~\ref{fig:gdownvup_ex2_po}.

\section {Conclusions}
\label{sec:summary}

This paper is long and complicated, and hence difficult to summarize succinctly.  Suffice to say that we have introduced a generalization of a prokaryotic  gene regulatory model, in \eqref{eq:basic}-\eqref{eq:thres}, originally developed in \cite{ghmww2020}.  The generalization has both nonlinearities in terms of feedback but also nonlinearities appearing in state dependent delays. Moreover the delays are generated by a threshold condition. It is thus somewhat novel for both the mathematical and modeling literature.  The first three sections of the paper are concerned with establishing important and relevant mathematical properties of the basic system \eqref{eq:basic}-\eqref{eq:thres}.  Section \ref{sec:semiflow} establishes properties of the semiflow generated by \eqref{eq:basic}-\eqref{eq:thres}, while Section \ref{sec:positivity} establishes conditions for  the existence of a global attractor of the semiflow.  Section \ref{sec:linearization} deals with the linearization of \eqref{eq:basic}-\eqref{eq:thres} which is of significant utility in the following Sections \ref{sec:onehill} and \ref{sec:twohill} where we examine the behaviour of \eqref{eq:basic}-\eqref{eq:thres} both numerically and analytically.

At the ends of both Sections~\ref{sec:onehill} and~\ref{sec:twohill}  we have offered general synopses of the results of each, which we will not repeat here.
As an aid to the reader,  in Table~\ref{table:figure-summary} we have listed all of the figures in this paper illustrating dynamics for various combinations of $g$ and $v$, along with the types of bifurcations that we observed numerically in each case.

\begin{table}
\begin{centering}
\begin{tabular}{|c|>{\raggedright\arraybackslash}p{5in}|}
\hline
$(g  ,v  )$ & Corresponding figure number and bifurcations \\
\hline \hline
$(g \downarrow , v\leftrightarrow)$ & \ref{fig:1}($\beta\gamma$), \ref{fig:gdown_ex1}(1,2,H,S) \\
\hline
$(g\uparrow,v \leftrightarrow)$ & \ref{fig:2}($\beta\gamma$), \ref{fig:gup_ex1}(1,2,F,H,Cu)  \\
\hline
$(g \leftrightarrow,v \downarrow)$ & \ref{fig:3}($\beta\gamma$), \ref{fig:vdown_ex3}(1,S), \ref{fig:vdown_ex2}(1,2,H,S), \ref{fig:vdown_ex1}(1,2,H,S)  \\
\hline
$(g \leftrightarrow,v \uparrow)$ & \ref{fig:4}($\beta\gamma$), \ref{fig:vup_ex3}(1,2,F,S,Cu), \ref{fig:vup_ex1a}(1,F,h,H,F$_\textrm{P}$), \ref{fig:vup_ex1b}(1,2,Ba,BT,FH,Cu), \ref{fig:vup_ex1c}(1,F,F$_\textrm{P}$,H,h), \ref{fig:vup_ex1e}(2,F,H,h,BT,Cu), \ref{fig:vup_ex1d}(1,H,h), \ref{fig:vup_ex2}(1,2,F,H,h,Cu,Ba,Fh$_2$)   \\
 \hline
\hline
$(g\downarrow ,v \uparrow)$ & \ref{fig:5}($\beta\gamma$), \ref{fig:gdownvup}(1,2,F,H,Cu,BT,FH),
\ref{fig:gdownvup_ex1}($\beta\gamma$,1,2,H),
\ref{fig:thetagvlimits}($\beta\gamma$), \ref{fig:thetagvonelimit}($\beta$),
\ref{fig:FoldsMZero}($\beta$,1,2,F,Cu),
\ref{fig:FoldsMnotZero}($\beta$,2,F,Cu),
\ref{fig:gdownvup_ex2}($\beta\gamma$,1,2,F,H,Cu),
\ref{fig:gdownvup_ex2_po}(1,H,h,Ca),
\ref{fig:gdownvup_down3}($\beta\gamma$,1,4F,H),
\ref{fig:gdownvup_down4}(2,4F,H,Cu),
\ref{fig:gdownvup_up6}($\beta\gamma$,1,2,4F,5SS,H,Cu,BT)   \\
\hline
$(g \uparrow,v \uparrow)$ & \ref{fig:6}($\beta\gamma)$, \ref{fig:gupvup}(1,2,4F,5SS,H,Ba,BT,FH),
\ref{fig:gupvup_ex2}(1,4F,5SS,H)    \\
\hline
$(g \downarrow,v \downarrow)$ & \ref{fig:gdownvdown}(1,2,H)  \\
\hline
$(g\uparrow,v \downarrow)$ &  \\
\hline
\end{tabular}
\medskip
\caption{Summary of various dynamic behaviours seen in Sections \ref{sec:onehill} and \ref{sec:twohill}.
The notation of 'n.m' refers to figure 'm' of Section 'n'  and the symbols after indicate the type of
continuation explored and the bifurcations observed. $\beta\gamma$ and $\beta$
refers to a figure in which $\beta e^{-\mu\tau}g(\xi)$ (with or without $\gamma\xi$) is plotted to show qualitatively the number of steady states of the system \eqref{eq:basic}-\eqref{eq:thres}.
1 and 2 indicate that 1 and/or 2 parameter numerical continuation was performed. The bifurcations observed are denoted by
F (Fold), F$_P$ (Fold of Periodic Orbits), H (Hopf), S (Stable-no bifurcation), and h (homoclinic) and for higher co-dimension bifurcations:
Ba (Bautin \emph{or} generalized Hopf), BT (Bogdanov-Takens), Cu (cusp), FH (fold-Hopf), Fh$_2$ (fold-homoclinic). 4F denotes examples with 4 fold bifurcations, 5SS denotes 5 co-existing steady states, and Ca a canard.
Note that 'left clicking' on a number will take you to that figure, while 'Alt $\leftarrow$' will bring you back to this table.}
\end{centering}
\label{table:figure-summary}
\end{table}

We have studied the threshold model \eqref{eq:basic},\eqref{eq:thres} directly without applying the time transformation of Smith \cite{Smith91,Smith93} to convert the equations to a distributed delay DDE with constant delay. We showed in Sections~\ref{sec:gconstvdown}-\ref{sec:gvopposite} that the dynamics of the model \eqref{eq:basic} with a threshold delay is altogether richer and more surprising than the dynamics of the corresponding constant delay model considered in Sections~\ref{sec:gdownvconst}-\ref{sec:gdownvconst}.
In so doing we demonstrated that problems with threshold delays can be analysed and studied numerically without time transforming the problems to constant delay problems. Our methods are applicable to problems with multiple delays including one or more threshold delays.

Much of the algebra in this paper is rooted in the use of monotone Hill functions for $g(\xi)$ and $v(\xi)$ which is both relevant and quite justifiable for modeling in the biological realm.  Our algebra is helped by the use of the function
$\xi\mapsto\xi f'(\xi)/f(\xi)$ which has a tractable form \eqref{eq:fxpr} when $f$ is a Hill function. Notice that if $f(\xi)=C\xi^p$
then
$$\frac{\xi f'(\xi)}{f(\xi)}=p,$$
so we can think of $\xi f'(\xi)/f(\xi)$ as a description of   a local power of a non-polynomial function.  While a similar approach could be adopted for other nonlinearities, the results would likely be quantitatively different, but we expect that for monotone functions they would not show qualitative alterations.  However, if non-monotone nonlinearities were considered then we expect that there would be substantial qualitative as well as quantitative differences.

{\bf Future directions.} 
In \cite{goldbeter2022multi} the authors examine the various dynamic patterns (what they term `multi-synchronization' or `multi-rhythmicity') that emerge from models of biochemical and cellular oscillators coupled to each other, and illustrate their considerations with examples of the circadian oscillator and the cell cycle.  While certainly of interest biologically,  perhaps  the more interesting question is whether or not one could construct a catalog, or dictionary, of possible dynamical behaviours arising from the coupling of $N$ different dynamics  generated by DDEs  as we have considered here.

In particular, consider 
\be \label{eq:basic-1}
x_i'(t)  = \beta_i e^{-\mu_i\tau_i (t )}
 \frac{v_i(x_i(t )))}{v_i(x_i (t- \tau_i (t )))}
 g_i(x_i (t -\tau_i (t ))) - \gamma_i x_i (t)
\ee
where for $t\ge 0$ the delay $\tau_i(t)$ is defined by the threshold condition
\be \label{eq:thres-1}
a_i =\int^0_{-\tau_i (t)}v_i(x_i(t+s))ds=\int^t_{t-\tau_i (t)}v_i(x_i(s))ds
\ee
for $i = 1, \ldots, N$. Assuming  that these systems are coupled through one or more of the parameters $\beta_i, \mu_i,\gamma_i, a_i$ or through the functions $\tau_i, v_i, g_i$, then what can be said about the resulting dynamics?
The paper \cite{goldbeter2022multi} considered both `series' and `parallel' coupling between their two dynamics, but one would have to precisely define what was meant by this within the context of \eqref{eq:basic-1}-\eqref{eq:thres-1}.

For $N=2$ one might consider a system
\begin{align*}
x_1'(t)  &= \beta_1 x_2 e^{-\mu_1\tau_1 (t )} \frac{v_1(x_1(t)))}{v_1(x_1 (t- \tau_1 (t )))}  g_1(x_1 (t -\tau_1 (t ))) - \gamma_1 x_1 (t)\\
x_2'(t)  &= \beta_2 x_1e^{-\mu_i\tau_2 (t )} \frac{v_2(x_2(t )))}{v_2(x_2 (t- \tau_2(t )))}  g_2(x_2 (t -\tau_2 (t ))) - \gamma_2 x_2 (t)
\end{align*}
One could regulate the strength of the coupling by an additional constant i.e. by considering the term
$ \beta_1 a_1 x_2 $ rather than just  $\beta_1 x_2$. This would allow an examination of the strength of coupling while keeping $\beta$'s the same for comparison with a single DDE oscillator.
This seems to be the simplest linear coupling; an alternative might be considering  $\beta_1(x_2), \beta_2(x_1)$ where these are Hill functions.
This coupling results in  a positive mutual feedback in the sense that more $x_1$ makes $g(x_2)$ bigger and vice versa. 

For stable oscillation in a single DDE we need at least one of the functions $v$ or $g$ to be decreasing.
Are there differences if we couple same kind oscillators  (-,+) with (-,+)  or  (+.-) with (+,-) on one hand and two different oscillators (+,-) with (-,+)?
If we mutually couple two DDEs in a way that they form a negative feedback loop, can we identify oscillations that arise due to local negative feedback (i.e. the oscillator type (-,+) and (+,-)) and the global negative feedback, that arises from the mutual coupling of the oscillators? In particular, consider two types of mutually coupled equations. First consider the coupling through  the nonlinearities $g$
\begin{align*}
x_1'(t)  &= \beta_1 e^{-\mu_1\tau_1(t)}\frac{v_1(x_1(t)))}{v_1(x_1(t-\tau_1(t)))}g_1(x_2(t-\tau_1(t)))
- \gamma_1 x_1(t) \nonumber \\
x_2'(t)  &= \beta_2 e^{-\mu_2\tau_2(t)}\frac{v_2(x_2(t)))}{v_2(x_2(t-\tau_2(t)))}g_2(x_1(t-\tau_2t))) - \gamma_2 x_2(t),
\end{align*}
where  both $v_1, v_2$ are decreasing, so that each equation can support oscillations with constant input from the other oscillator (see Section 5.3). If both $g_1, g_2$ are also decreasing this represents mutual  inhibition between two oscillators and suggests the possible  existence of bistability.

Alternatively, one can consider coupling of  the equations through the $v$ functions rather than nonlinearities $g$, with  all functions $v_1, v_2, g_1, g_2$ still decreasing:
\begin{align*}
x_1'(t)  &= \beta_1 e^{-\mu_1\tau_1(t)}\frac{v_1(x_2(t)))}{v_1(x_2(t-\tau_1(t)))}g_1(x_1(t-\tau_1(t)))
- \gamma_1 x_1(t)\\
x_2'(t)  &= \beta_2 e^{-\mu_2\tau_2(t)}\frac{v_2(x_1(t)))}{v_2(x_1(t-\tau_2(t)))}g_2(x_2(t-\tau_2t))) - \gamma_2 x_2(t),
\end{align*}

Does either system support bistability between periodic solutions? If so, does the shape of the periodic solutions reflect the fact that in mutually repressible coupled systems one component is high while the other is low?

One can also consider larger system of equations that are cyclically coupled
\begin{align} \label{eq:x1g-1}
x_1'(t)  &= \beta_1 e^{-\mu_1\tau_1(t)}\frac{v_1(x_1(t)))}{v_1(x_1(t-\tau_1(t)))}g_1(x_N(t-\tau_1(t)))
- \gamma_1 x_1(t)\\ \label{eq:xjg-1}
x_j'(t)  &= \beta_j e^{-\mu_j\tau_j(t)}\frac{v_j(x_j(t)))}{v_j(x_j(t-\tau_j(t)))}g_j(x_{j-1}(t-\tau_j(t))) - \gamma_j x_j(t), \quad j=2,\ldots,N.
\end{align}
We have already tackled a particular problem of this form in \cite{ghmww2020}
where we considered
\begin{align*}
\dfrac{dM\!}{dt}(t) & = \beta_M e^{-\mu \tau_{M}(t)} \dfrac{v_M(E(t))}{v_M(E(t-\tau_{M}(t)))} f(E(t-\tau_M(t))) -\bar\gamma_M M(t), \\
\dfrac{dI}{dt}(t) & = \beta_I e^{-\mu\tau_I(t)} \dfrac{v_I({M(t)})}{v_I(M(t-\tau_I(t)))} M(t-\tau_I(t)) -\bar\gamma_I I(t) , \\
\dfrac{dE}{dt}(t) & = \beta_E I(t) -\bar\gamma_E E(t).
\end{align*}
which is exactly of the form \eqref{eq:x1g-1},\eqref{eq:xjg-1} with $N=3$ and $g_2(x)=g_3(x)=x$ and $\mu_3=\tau_3=0$.

To conclude, the work presented here  opens many interesting questions on how the interplay between the character of the nonlinearity, state dependent delay and coupling affects the local and global dynamics of DDEs.

\appendix
\section*{Appendices}
\renewcommand{\theequation}{\thesection.\arabic{equation}}
\renewcommand\thefigure{\thesection.\arabic{figure}} 
\section{Relation of \eqref{eq:basic}-\eqref{eq:thres} to the model of Gedeon et al \cite{ghmww2020}}\label{app-reduction}

The Goodwin  model \cite{goodwin1965} for operon dynamics    considers a large population of cells, each of which contains
one copy of a particular operon.  $(M,I,E)$ respectively
denote the mRNA, intermediate protein, and effector protein  concentrations. For a generic operon
\citep{Goodwin1963,goodwin1965,Griffith68a,Griffith68b,othmer76,selgrade79}
the dynamics are assumed to be  given by
\begin{align}
    \dfrac{dM\!}{dt}  &= \beta_M {f}(E )
    -\gamma_M M ,\label{eq:mrna}\\
    \dfrac{dI}{dt}  &= \beta_I M  -\gamma_I I  ,\label{eq:intermed}\\
    \dfrac{dE}{dt}  &= \beta_E I  - \gamma_E E .\label{eq:effector}
\end{align}
The production flux $f$   of  mRNA  is assumed to be a function of the effector level $E$.  Furthermore, the model assumes that the flux of protein and metabolite production  are proportional (at rates $\beta_I,\beta_E$ respectively) to the amount of mRNA and
intermediate protein respectively. All three of the components $(M,I,E)$ are subject to degradation at rates $\gamma_M, \gamma_I, \gamma_E$.  The parameters $\beta_I,\beta_E,\gamma_M,\gamma_I$ and
 $\gamma_E$ have dimensions [time$^{-1}$].

In \cite{ghmww2020} we extended the classic Goodwin model
for the regulation of the bacterial operon to a situation in which the cells are growing exponentially at a rate $\mu$ and have finite transcriptional and translational velocities that are potentially dependent on the state of the system. For the full model of \cite{ghmww2020}, retain the notation of the original Goodwin operon model, denote the transcriptional velocity  by $v_M(E(t))$ and the translational velocity by $v_I(M(t))$, and let   $\bar \gamma_i = \gamma_i + \mu$.
Then  the extension in \cite{ghmww2020} is
\begin{align}
\dfrac{dM}{dt}
& = \beta_M e^{-\mu \tau_M(t)} \dfrac{v_M(E)}{v_M(E(t-\tau_M(t)))}
f(E(t-\tau_M(t))) -\bar \gamma_M M, \label{eq:mrna-delay-var}\\
\dfrac{dI}{dt}
& = \beta_I e^{-\mu\tau_I(t)} \dfrac{v_I({M})}{v_I(M(t-\tau_I(t)))} M(t-\tau_I(t)) -\bar \gamma_I I , \label{eq:intermed-delay-var}\\
\dfrac{dE}{dt} & = \beta_E I -\bar \gamma_E E. \label{eq:effector-delay-var}
\end{align}
These equations are supplemented by the two additional equations  which implicitly define the delays  $\tau_M$ and $\tau_I$   by threshold conditions, namely
\begin{align}
a_M &= \int_{t-\tau_M(t)}^{t} v_M(E(s)) ds=\int_{-\tau_M(t)}^{0} v_M(E(t+s)) ds
\label{eq:delay by stateM}\\
a_I &= \int_{t-\tau_I(t)}^{t} v_I(M(s)) ds=\int_{-\tau_I(t)}^{0} v_I(M(t+s)) ds.
\label{eq:delay by stateI}
\end{align}

In our extended model \cite{ghmww2020}, as in the original Goodwin \cite{Goodwin1963} formulation, the function $f$ is a monotone increasing function for an inducible operon, and for a repressible operon $f$ is a monotone decreasing function.

To make the transition from the full  model \eqref{eq:mrna-delay-var}-\eqref{eq:effector-delay-var} presented in \cite{ghmww2020} to the apparently simpler situation of \eqref{eq:basic}-\eqref{eq:thres} considered here is relatively straightforward.  We simply need two assumptions:

\ben
\item We first assume that the translational velocity  $v_I$ is not regulated but is  large with respect to the other characteristic times in the system so the translational delay $\tau_I \simeq 0$.

\item Our second assumption involves the existence of fast and slow variables \citep[Section 2.3]{mackey2016simple} so there is a clear separation of time scales.
\een
With these assumptions and the substitutions:
$E \to x$, $\beta \to \bar \gamma_I \dfrac{\beta_M \beta_I \beta_E }{\bar \gamma_M \bar \gamma_I \bar \gamma_E}$ (or $\beta \to \bar \gamma_E \dfrac{\beta_M \beta_I \beta_E }{\bar \gamma_M \bar \gamma_I \bar \gamma_E}$), $f \to g$, and $\tau_M(t) \to \tau(t)$ we immediately obtain the system \eqref{eq:basic}-\eqref{eq:thres} that we study here.

One might think that \eqref{eq:basic} is somewhat novel but in fact similar formulations are available in different situations.
Below, in Appendix \ref{app-other} we give several  other examples of models in the spirit of \eqref{eq:basic}, but which differ in details that may or may not offer significant differences in behaviour.


\section{Other examples}\label{app-other}  In this appendix, we mention four different examples of previously published studies which can be thought of as extensions of the considerations in this paper.

\subsection*{A bistable genetic switch}

The two types of operon dynamics originally considered were classified as either repressible (in which the production flux function $f$ in \eqref{eq:mrna} is a decreasing function of its argument), or inducible (in which $f$ is an increasing function of its argument).  In the language of dynamical systems these are negative or positive feedback systems.  However, there is a third type of fundamental dynamical entity found in  prokaryotes, namely the so called bistable switch in which the effector produced by operon $X$ inhibits the transcription of DNA from operon $Y$ and vice versa (which we might denote as a $(-,-)$ system).  The paradigmatic molecular biology example of a bistable switch due to reciprocal
negative feedback is the bacteriophage (or phage) $\lambda$, which is a virus capable of
infecting E. coli bacteria. Originally described in \cite{jacob-1961}  and very nicely treated in \cite{ptashne-1986}, it is but one of many examples of mutually
inhibitory bistable switches that have been found since.  Models of this process are to be found in \cite{grigorov1967model}, \cite{Santilan200475}, and \cite{cherry-2000} among others.

Consider a simple model for this process in which the dynamics of the intermediate of both the $X$ and $Y$ operons is fast relative to the dynamics of the corresponding mRNA $M_X$ and $M_Y$, and effectors $E_X$ and $E_Y$.  Then we can write down a reduced version of the model \eqref{eq:mrna-delay-var}-\eqref{eq:effector-delay-var} in the form
\begin{align}
\dfrac{dM_X}{dt}
& = \beta_{M_X} e^{-\mu \tau_{M_X}(t)} \dfrac{v_{M_X}(E_Y)}{v_{M_X}(E_Y(t-\tau_{M_X}(t)))}
f_X(E_Y(t-\tau_{M_X}(t))) -\bar \gamma_{M_X} M_X, \label{eq:mrna-delay-var-X}\\
\dfrac{dE_X}{dt} & = \beta_{E_X} M_X -\bar \gamma_{E_X} E_X \label{eq:effector-delay-var-X}\\
\dfrac{dM_Y}{dt}
& = \beta_{M_Y} e^{-\mu \tau_{M_Y}(t)} \dfrac{v_{M_Y}(E_X)}{v_{M_Y}(E_X(t-\tau_{M_Y}(t)))}
f(E_X(t-\tau_{M_Y}(t))) -\bar \gamma_{M_Y} M_Y, \label{eq:mrna-delay-var-Y}\\
\dfrac{dE_Y}{dt} & = \beta_{E_Y} M_Y -\bar \gamma_{E_Y} E_Y. \label{eq:effector-delay-var-Y}
\end{align}

Analysis of the network dynamics of ordinary differential equation systems~\cite{plahte1995,gouze1998} suggests that the  system
\begin{align*}
\dfrac{dM_X}{dt}
& = \beta_{M_X} f_X(E_Y) -\bar \gamma_{M_X} M_X, \\
\dfrac{dE_X}{dt} & = \beta_{E_X} M_X -\bar \gamma_{E_X} E_X \\
\dfrac{dM_Y}{dt}
& = \beta_{M_Y} f(E_X) -\bar \gamma_{M_Y}\\
\dfrac{dE_Y}{dt} & = \beta_{E_Y} M_Y -\bar \gamma_{E_Y} E_Y,
\end{align*}
corresponding to \eqref{eq:mrna-delay-var-X}-\eqref{eq:effector-delay-var-Y}
in both $(++)$ and $(--)$ cases is a positive cyclic feedback system~\cite{Gedeon98} and will display bistability~\cite{plahte1995,gouze1998}. The recent preprint \cite{richard2023complete} carefully analyzes the boundary of bistability regions in both $(++)$ and $(--)$ cases.

Based on the results from~\cite{ghmww2020}, inclusion of  distributed delays could add additional equilibria, resulting in multistability. Adding delays also can change the range of bistability in the parameter space.

System $(+-)$ is a negative feedback system where we expect that  the trivial equilibrium will lose stability through Hopf bifurcation  if the slope(s) of nonlinearities  are sufficiently steep at the equilibrium. Adding delays may result in additional stable equilibria and secondary Hopf bifurcations, as was observed in~\cite{ghmww2020}.

\subsection*{A forest growth model}

In \citep{magal2017competition}, the authors considered an age structured  model for the growth of a single tree species forest.  They showed that the dynamics of the number of adult trees $A$ is governed by
\be
\label{eq:forest-1}
\dfrac{dA(t)}{dt} = e^{-\mu \tau(t)} \dfrac{f(A(t))}{f(A(t-\tau(t))}  r b(A(t-\tau(t)))  - \gamma A(t)
\ee
in conjunction with the condition
\be
\label{eq:forest-2}
\int_{t-\tau(t)}^t f(A(s)) ds = \bar s.
\ee
In \eqref{eq:forest-1} $\mu$ is the mortality rate of saplings, $f$ is the velocity of maturation of saplings, $r$ is the birth rate, $b$ is the reproduction (birth) function and $\gamma$ is the mortality rate of adults.

\subsection*{The $G_0$ cell cycle model}

The original $G_0$ cell cycle model proposed by Burns and Tannock in \cite{burns1970existence} is illustrated in Figure \ref{fig:ppsc}, and captures the essence of what is known of the cellular replication process at a intermediate (cellular or non-molecular) level of sophistication and knowledge.  The proliferating phase cells are denoted by $P$ while the resting $G_0$ phase cells are denoted by $N$.
\begin{figure}
\centering
\includegraphics[width = 0.5 \linewidth,angle=0]{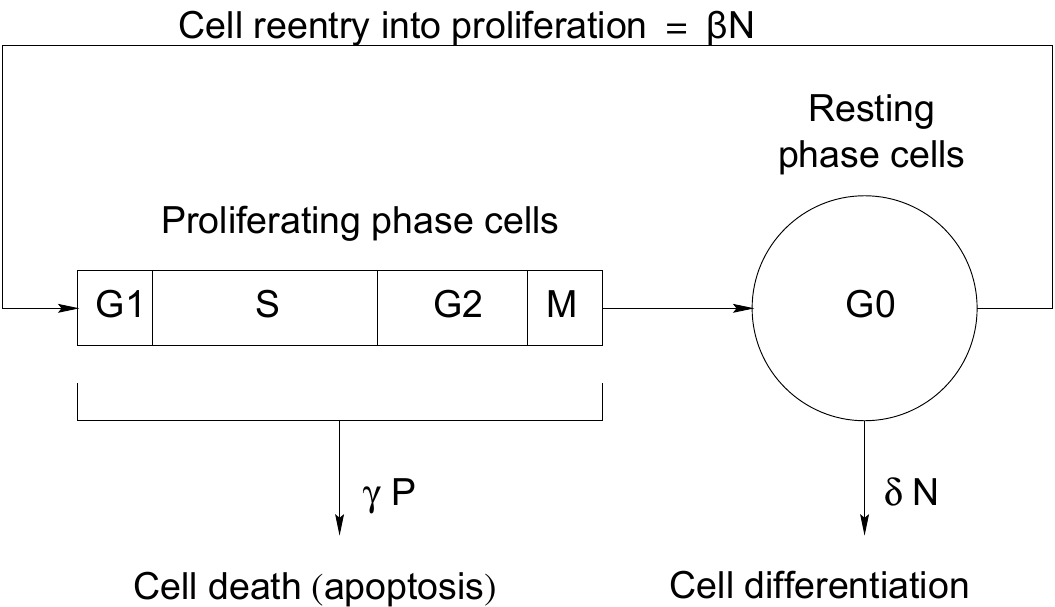}
\caption{The Burns and Tannock \cite{burns1970existence} model for the cell cycle consisting of a resting phase ($G_0$) and the proliferating phase $P$ with the sub-phases $G_1$, $S$ (DNA synthesis), $G_2$, and $M$ (mitosis and cytokinesis).  It is assumed that cells can die from the proliferating phase at a rate $\gamma$ and exit into the differentiation pathway from $G_0$ at a rate $\delta$. Redrawn from \cite{mackey1978unified}.  }
\label{fig:ppsc}
\end{figure}

In the elaboration of \cite{mackey1978unified}, it is presumed that cells die from the proliferative phase at a rate $\gamma$ so the flux of cells to death is $\gamma P$, and differentiate from $G_0$ at a rate $\delta$ so the differentiation flux is $\delta N$.  Cells in the $G_0$ phase can, in addition to differentiating, re-enter the proliferative phase at a rate $\beta$ so the flux of cells into proliferation is $\beta N$.  Cells that enter the proliferative phase are assumed to proceed through the stages  $G_1$, $S$, $G_2$, and $M$ in a lock step fashion that takes a time $\tau$ to complete if death does not intervene.  Once mitosis is completed, cytokinesis produces two daughter cells that then enter $G_0$.

Based on the assumption that the rate $\beta$ of cell entry from $G_0$ into $P$,  is a function of the size $N$ of $G_0$, then it is straightforward to show that the dynamics of the augmented Burns/Tannock cell cycle model are governed by the differential delay equation
\be
\dfrac{dN}{dt} = -[\delta + \beta(N)] N  + 2 e^{- \gamma \tau} N_\tau \beta(N_\tau),
\label{eq:burns-tannock}
\ee
where we have used the notation $N \equiv N(t)$ and $N_\tau \equiv N(t-\tau)$.  Equation \eqref{eq:burns-tannock} has been the subject of an extensive bifurcation analysis in the work of \cite{de2019dynamics}.

Now we consider a slightly modified cell cycle model identical to the Burns/Tannock model with the additional wrinkle that the velocity $v$ with which cells move through the proliferative phase of the cell cycle is  under the control of the number $N$ of non-proliferative cells.

\begin{rem}
One would think that the velocity of movement through the proliferative phase $P$ would be an increasing function of decreased $N$ which would lead to a delay $\tau$ that decreases as $N$ decreases.  This would lead to cell density dependent inter-division times.
\end{rem}

 If we take the number of proliferative phase cells at time $t$ and age $a$ to be $p(t,m)$, the maximum age at cytokinesis to be $a_m$,  and the velocity with which they move through the cell cycle to be $v(N(t))$, then the evolution equation for $p(t,m)$ is given by
\be
\label{eq:BT-new-1}
\dfrac{\partial p(t,m)}{\partial t} + v(N(t)) \dfrac{\partial p(t,m)}{\partial m} = -\gamma p(t,m)
\ee
and we have the initial condition
\be
\label{eq:BT-IC}
p(t,m=0) = N(t) \beta(N(t)).
\ee
Following the same derivation procedure as in  \cite{ghmww2020},  the final equation for the dynamics of the non-proliferative cellular population $N$ is given by
\be
\label{eq:BT-final}
\dfrac{dN(t)}{dt} = 2 e^{-\gamma \tau(t)} \dfrac{v(N(t))}{v(N(t-\tau(t)))} N(t-\tau(t)) \beta (N(t-\tau(t)) - (\delta + \beta(N(t))) N(t).
\ee
We also have the additional condition that
\be
\label{eq:BT-max}
a_m = \int_{t - \tau(t)}^t v(N(s)) ds.
\ee

\begin{rem} Note that, not unsurprisingly, the augmented cell cycle model is identical in formulation with the forest growth model of \cite{magal2017competition}.
\end{rem}

\subsection*{A model for recurrent inhibition}

\begin{figure}
\centering
\includegraphics[width = 0.5\linewidth,angle=0]{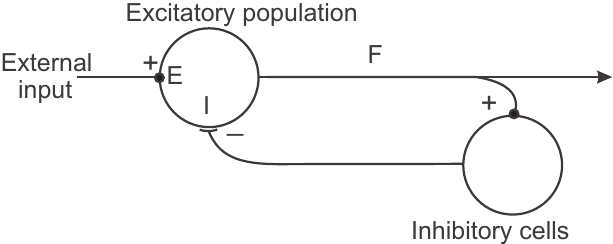}
\caption{A schematic representation of the recurrent inhibition circuit considered in \cite{eurich2002recurrent}.  This figure is modified from the original publication.  }
\label{fig:RI}
\end{figure}

A model for the dynamics of a recurrent inhibitory neural feedback circuit
is considered in \cite{eurich2002recurrent},
however the model was incomplete omitting the velocity ratio term which appears in all of the other models here.
The term $v(x(t))/v(x(t-\tau(t)))$ in \eqref{eq:basic} is essential as shown in
\cite{Bernard2016,Craig2016,ghmww2020} from conservation and flux arguments. Here, we complete 
the model of \cite{eurich2002recurrent}.

To capture the spirit of the recurrent inhibitory neural circuitry and offer a complete formulation 
incorporating the ratio term,
we consider the situation with dynamics described by
\begin{equation}
\frac{dV(t)}{dt} 
= \gamma(E-V(t)) - \cfrac{\Sigma(V(t))}{\Sigma(V(t-\tau(t)))} G(F(V(t-\tau(t)), 
\label{finaldynamics-new}
\end{equation}
where $V= E - I$ is the net potential in the neuron, $E (I) $ is the excitatory (inhibitory) potential, $\Sigma$ is the velocity of propagation of the action potentials around the inhibitory loop back to the soma, and $G$ is the response of the neuron to the recurrent inhibitory drive. $F \approx \alpha V$ ($\alpha > 0$) is the firing frequency of the cell, an increasing function of the postsynaptic potential $V$.  In conjunction with \eqref{finaldynamics-new}, if the length of the recurrent inhibitory pathway is $L$ we have the ancillary  condition
\be
L = \int_{t-\tau(t)}^t \Sigma(V(s))ds.
\label{eq:RI-delay}
\ee

In \cite{eurich2002recurrent} it was argued that $\Sigma$ is an increasing function of the postsynaptic potential $V$ so, as a consequence the delay time $\tau$ is a decreasing function of $V$.  It was further noted that, on physiological grounds, $G$ is an increasing function of the firing frequency $F$ for small values of $F$ and decreasing with increasing $F$ past a certain maximum.

This model can be further generalized to a state dependent DDE with distributed state dependent delays.
This generalization relaxes the assumption that the length $L$ is constant. A more realistic assumption is that the lengths are distributed in population of neurons described by a distribution $K_L$. Note that this is equivalent to having variable firing threshold at the soma that receives the inhibitory input.
The model takes the form
\begin{equation}
\frac{dV(t)}{dt} 
= \gamma(E-V(t)) - \int_0^\infty \cfrac{\Sigma(V(t))}{\Sigma(V(t-\phi(t)))} G(F(V(t-\phi(t)) K_L\left (\int_{t-\phi}^t \Sigma(V(s)) ds \right ) d\phi, 
\label{finaldynamics-new2}
\end{equation}
where $V= E - I$ is the net potential in the neuron, $E (I) $ is the excitatory (inhibitory) potential, $\Sigma$ is the velocity of propagation of the action potentials around the inhibitory loop back to the soma, and $G$ is the response of the neuron to the recurrent inhibitory drive. $F \approx \alpha V$ ($\alpha > 0$) is the firing frequency of the cell, an increasing function of the postsynaptic potential $V$.  To derive this equation we note that the  distribution of the  recurrent loop lengths $L$ is analogous to the distributed maturation times in Kendrick-McCormack age structured models analyzed in~\cite{cassidy2019equivalences}. The analog to the Kendrick-McCormack  model of the age distributed population is
\begin{align*}
& \partial_t E(t,s)   + \partial_s (\Sigma(V(t)) E(s,t)) = -h(s) E(s,t) \\
&\Sigma(V(t))E(t,0) = G(F(V(t))), \qquad E(t_0,s) = f(s), s \in [0,\infty)
\end{align*}
where $s$ is a position along a loop, $E(t,s)$ is the action potential at position $s$ of the loop and the  propagation velocity $\Sigma$ depends on the voltage at time $t$. The  term $h(s)$ is the discharge of the action potential  in the soma  when the circuit has length $s$. The function $h(s)$ can be expressed (see \cite{cassidy2019equivalences}) using the distribution $K_L$ of circuit lengths as
\[ h(s(t)) = \frac{K_L(s(t))}{1-\int_0^\sigma K_L(s(\sigma)) d\sigma} \frac{d}{dt}s(t) .\]
Solving the partial differential equation along the characteristics
\[  \frac{d}{d\phi} t(\phi) = 1, \qquad \frac{d}{dt}s(t) = \Sigma (V(t))\]
 gives the equation (\ref{finaldynamics-new2}). For details of the derivation see~\cite{cassidy2019equivalences}.

\section{Numerical Techniques}\label{app-numerics}

Our numerical techniques are documented in detail in \cite{ghmww2020,wendy-msc,wendy-phd,ifacs22}, so here we will only briefly summarise them, and note some extensions of the previous methods. 

All of our computations are performed using MATLAB \cite{Matlab}. Mainly we perform numerical continuation and bifurcation analyses using the MATLAB package \texttt{ddebiftool} \cite{ddebiftool}. Existing versions of \texttt{ddebiftool} are only implemented for discrete delays, and cannot be 
applied directly to problems with threshold delay. 
Since the largest possible delay is $a/v_0$, to implement threshold delay problems in 
\texttt{ddebiftool} we introduce up to 100 dummy constant delays to discretize the time interval $[t-a/v_0,t]$ on a equally spaced mesh. 
The threshold delay $\tau(t)$ that satisfies \eqref{eq:thres} can then be computed be applying
numerical quadrature. The implementation is explained in detail in \cite{wendy-phd,ifacs22},
and is a refinement of our earlier methods \cite{ghmww2020,wendy-msc}.

One of the features that makes \texttt{ddebiftool}, and numerical continuation in general, so powerful is its ability to follow stable and unstable branches of solutions equally well.
Occasionally we also use numerical simulation of the DDE to find stable solutions. The standard MATLAB solver for state-dependent DDEs is \texttt{ddesd} \cite{ddesd}, but this is also only implemented for discrete delays. To use it to solve \eqref{eq:basic}-\eqref{eq:thres} as an initial value problem we first differentiate the threshold condition \eqref{eq:thres} to obtain
\be \label{eq:thresdiff}
\dfrac{d\tau(t)}{dt} =
 1 - \dfrac{v(x(t))}{v(x(t-\tau(t)))},
\ee
and then solve \eqref{eq:basic},\eqref{eq:thresdiff} using \texttt{ddesd}.
See \cite{ghmww2020,wendy-msc,ifacs22} for implementation details, 
and also \cite{wendy-phd} for an investigation of the differences between 
\eqref{eq:basic},\eqref{eq:thres} and \eqref{eq:basic},\eqref{eq:thresdiff}.

The numerical collocation underpinning the \texttt{ddebiftool} differential equation solver, is best suited to smooth nonlinearities and solutions. When $m$ or $n$ is large and one of the Hill functions
in \eqref{eq:vghill} approaches a step function issues may arise in the computations. Here we mention 
some tricks we applied to compute the figures in those cases.

To circumvent underflow, overflow and division by zero errors in MATLAB when evaluating
$\beta e^{-\mu\tau(\xi)}g(\xi)$ near $\theta_g$ or $\theta_v$ (such as in Figure~\ref{fig:thetagvonelimit}), we rewrite the Hill function as 
\begin{equation*}
	g(x)= \left\{
	\begin {aligned}
	\dfrac{g^{-} + g^{+} (x/\theta_g)^n}{1+(x/\theta_g)^n}, \quad x \leq \theta_g, \\
	\dfrac{g^{-}(\theta_g/x)^n + g^{+}}{(\theta_g/x)^n+1}, \quad x>\theta_g, 
\end{aligned}\right.
\end{equation*}
and similarly for $v(x)$. It can still be delicate to continue branches of fold bifurcations 
when $m$ or $n$ are large.
In practice, when needed, we apply two different techniques for extending branches of fold bifurcations to very large values of $m$ and/or $n$. 
Firstly, the
MATLAB nonlinear system solver \texttt{fsolve} \cite{fsolve} can be used
to solve \eqref{eq:vglam0} for each fixed value of $m$ or $n$,
to extend the branch to large values of these nonlinearity parameters.
Alternatively, since $M(\xi)$ changes sign at the fold bifurcation, these can be revealed 
by performing
a contour plot of $M(\xi)$ 
in the $(\xi, m)$ or $(\xi, n)$ plane and displaying only the $M(\xi)=0$ curve. 
Furthermore, to perform two-parameter continuation of the fold bifurcations in $(\gamma, m)$ or $(\gamma, n)$, we evaluate the right hand side of \eqref{eq:gamvg} to get the corresponding $\gamma$ for the fold bifurcation at $\xi$. We apply such tricks extensively in Section~\ref{sec:gvopposite}.  

We were unable to use \texttt{ddebiftool} subroutines for detecting codimension-two bifurcations along the fold and the Hopf curves. This was likely due to the complexity of the stability computations with the augmented system \texttt{ddebiftool} uses for two-parameter continuation combined with the large number of delays in our implementation. 
To find the codimension-two bifurcation points 
we compute the stability of the points on the codimension-one bifurcation branches and track the number of characteristic values with positive real part along the branch. On the fold branch, a fold-Hopf bifurcation occurs when a complex conjugate pair of characteristic values cross the imaginary axis.
Fold-Hopf points are also found on a branch of Hopf bifurcations at points where a single real characteristic value changes sign. 
A Bogdanov-Takens point is characterised by a single real characteristic value changing sign on a fold branch, at a point in parameter space where a branch of Hopf bifurcations also terminates. 
Note that we do not compute the normal form coefficients of these bifurcations.

\bibliographystyle{siamplain}

\end{document}